\documentclass[12pt,reqno]{amsart}

\usepackage{fixltx2e} 
\usepackage[utf8]{inputenc} 
\usepackage{amssymb, latexsym, stmaryrd, amsthm, dsfont, amsfonts, amsbsy, amsmath, mathrsfs} 
\usepackage{mathtools} 
\usepackage{bm, bbm} 
\usepackage{enumerate} 
\usepackage{verbatim} 
\usepackage{url}   
\usepackage{lscape} 
\usepackage{centernot}
\usepackage[dvipsnames]{xcolor}
\usepackage[colorinlistoftodos]{todonotes} 
\usepackage{nicefrac} 
\usetikzlibrary{calc}

\usepackage{microtype,tikz,tikz-cd}  
\makeatletter 
\def\MT@register@subst@font{\MT@exp@one@n\MT@in@clist\font@name\MT@font@list
 \ifMT@inlist@\else\xdef\MT@font@list{\MT@font@list\font@name,}\fi}
\makeatother

\makeatletter 
\newcommand{\myitem}[1]{%
\item[(#1)]\protected@edef\@currentlabel{#1}%
}
\makeatother

\usepackage[pdftex,bookmarks,bookmarksnumbered,linktocpage,   
         colorlinks,linkcolor=blue,citecolor=blue]{hyperref}
         
\usepackage[capitalize,noabbrev]{cleveref} 

\newcommand{\bit}{\begin{itemize}}    
\newcommand{\eit}{\end{itemize}}
\newcommand{\ben}{\begin{enumerate}}
\newcommand{\een}{\end{enumerate}}

\newcommand{\benroman}{\ben[\normalfont (i)]}  
\let\eroman\een

\newcommand{\bde}{\begin{description}}
\newcommand{\ede}{\end{description}}
\let\oper=\mathbb                               

\newcommand{\III}{\oper{I}}                     
\newcommand{\SSS}{\oper{S}}                     
\newcommand{\UUU}{\oper{U}}                     
\newcommand{\VVV}{\oper{V}}                     
\newcommand{\resLK}{{\upharpoonright}_{\mathscr{L}_{\mathsf{K}}}}
\newcommand{\res}{{\upharpoonright}}

\theoremstyle{theorem}
\newtheorem{Theorem}{Theorem}[section]
\newtheorem{Theorem-n}{Theorem}

\newtheorem{Proposition}[Theorem]{Proposition}
\newtheorem{Modal Sahlqvist Theorem}[Theorem]{Modal Sahlqvist Theorem}
\newtheorem{Intuitionistic Sahlqvist Theorem}[Theorem]{Intuitionistic  Sahlqvist Theorem}
\newtheorem{Esakia Duality}[Theorem]{Esakia Duality}

\newtheorem{Main Lemma}[Theorem]{Main Lemma}
\newtheorem{Compactness Theorem}[Theorem]{Compactness Theorem}
\newtheorem{Los Theorem}[Theorem]{\LL o\'s' Theorem}
\newtheorem{Isbell Theorem}[Theorem]{Isbell's Zigzag Theorem}
\newtheorem{Diagram Lemma}[Theorem]{Diagram Lemma}
\newtheorem{Jonsson}[Theorem]{J{\'o}nsson's Theorem}

\newtheorem{Transfer Lemma}[Theorem]{Transfer Lemma}
\newtheorem{Subdirect Decomposition Theorem}[Theorem]{Subdirect Decomposition Theorem}

\newtheorem{Corollary}[Theorem]{Corollary}
\newtheorem{Claim}[Theorem]{Claim}

\theoremstyle{definition}
\newtheorem{Definition}[Theorem]{Definition}
\newtheorem{exa}[Theorem]{Example}

\theoremstyle{remark}

\newtheorem{Remark}[Theorem]{Remark}

\crefname{Theorem}{Theorem}{Theorems}
\crefname{Proposition}{Proposition}{Propositions}
\crefname{Lemma}{Lemma}{Lemmas}
\crefname{Corollary}{Corollary}{Corollaries}
\crefname{Claim}{Claim}{Claims}
\crefname{Definition}{Definition}{Definitions}
\crefname{exa}{Example}{Examples}
\crefname{Remark}{Remark}{Remarks}
\crefname{Fact}{Fact}{Facts}
\crefname{exer}{Exercise}{Exercises}
\crefname{problem}{Problem}{Problems}
\crefname{Compactness Theorem}{Compactness Theorem}{Compactness Theorems}
\crefname{Los Theorem}{\LL o\'s' Theorem}{\LL o\'s' Theorems}
\crefname{Isbell Theorem}{Isbell's Zigzag Theorem}{Isbell's Zigzag Theorems}
\crefname{Diagram Lemma}{Diagram Lemma}{Diagram Lemmas}
\crefname{Subdirect Decomposition Theorem}{Subdirect Decomposition Theorem}{Subdirect Decomposition Theorems}

\AddToHook{env/Theorem/begin}{\crefalias{Theorem}{Theorem}}
\AddToHook{env/Proposition/begin}{\crefalias{Theorem}{Proposition}}
\AddToHook{env/Lemma/begin}{\crefalias{Theorem}{Lemma}}
\AddToHook{env/Corollary/begin}{\crefalias{Theorem}{Corollary}}
\AddToHook{env/Claim/begin}{\crefalias{Theorem}{Claim}}
\AddToHook{env/Definition/begin}{\crefalias{Theorem}{Definition}}
\AddToHook{env/exa/begin}{\crefalias{Theorem}{exa}}
\AddToHook{env/Remark/begin}{\crefalias{Theorem}{Remark}}
\AddToHook{env/Compactness Theorem/begin}{\crefalias{Theorem}{Compactness Theorem}}
\AddToHook{env/Los Theorem/begin}{\crefalias{Theorem}{\LL o\'s' Theorem}}
\AddToHook{env/Isbell Theorem/begin}{\crefalias{Theorem}{Isbell's Zigzag Theorem}}
\AddToHook{env/Diagram Lemma/begin}{\crefalias{Theorem}{Diagram Lemma}}
\AddToHook{env/Subdirect Decomposition Theorem/begin}{\crefalias{Theorem}{Subdirect Decomposition Theorem}}

\let\leq=\leqslant
\let\nleq=\nleqslant
\let\geq=\geqslant 

\usepackage[mathscr]{euscript}
 \let\mathscr\relax 
\usepackage[scr]{rsfso}

\newcommand{\dom}{\mathsf{dom}}

\renewcommand{\int}{\mathsf{int}\,}

\newcommand{\Z}{\mathbb{Z}}

\bmdefine{\A}{A} 
\bmdefine{\C}{C}                                
                              
\bmdefine{\2}{2}
\bmdefine{\B}{B}
\bmdefine{\D}{D}
\bmdefine{\E}{E}
\newcommand{\I}{{[0,1]}}
\bmdefine{\Term}{T} 
\bmdefine{\Free}{F}
\bmdefine{\Fb}{F}
\newcommand{\?}{\ensuremath{\mkern0.4\thinmuskip}}   

\usepackage{scalerel}

\newcommand{\DL}{\mathsf{DL}}

\newcommand{\bDL}{\mathsf{bDL}}

\newcommand{\V}{\mathsf{V}}

\newcommand{\K}{\mathsf{K}}
\newcommand{\M}{\mathsf{M}}

\newcommand{\HHH}{\mathbb{H}}
\newcommand{\PPP}{\mathbb{P}}
\newcommand{\QQQ}{\mathbb{Q}}

\newcommand{\PPU}{\mathbb{P}_{\!\textsc{\textup{u}}}^{}}

\newcommand{\ext}{\mathsf{ext}}
\newcommand{\extpp}{\mathsf{ext}_{\textsc{pp}}}
\newcommand{\exteq}{\mathsf{ext}_{\textsc{eq}}}
\newcommand{\Mon}{\mathsf{Mon}}

\newcommand{\CCMon}{\mathsf{CCMon}}
\renewcommand{\d}{\mathsf{d}}
\newcommand{\imp}{\mathsf{imp}}
\newcommand{\imppp}{\mathsf{imp}_{\textsc{pp}}}
\newcommand{\impeq}{\mathsf{imp}_{\textsc{eq}}}
\let\LL\L 
\renewcommand{\L}{\mathscr{L}}
\newcommand{\Luk}{\LL}
\newcommand{\F}{\mathcal{F}}
\newcommand{\rsi}{\textsc{rsi}}
\newcommand{\si}{\textsc{si}}
\newcommand{\rfsi}{\textsc{rfsi}}
\newcommand{\fsi}{\textsc{fsi}}
\newcommand{\ac}{\textsc{ac}}
\newcommand{\KAC}{\K_\ac}
\newcommand{\Mcal}{\mathcal{M}}
\newcommand{\Con}{\mathsf{Con}}
\newcommand{\GA}{\mathsf{GA}}
\newcommand{\Cg}{\mathsf{Cg}}
\newcommand{\Ker}{\mathsf{Ker}}
\newcommand{\diag}{\mathsf{diag}}
\newcommand{\fg}{\textup{fg}}
\newcommand{\Kfg}{\K^\fg}
\newcommand{\Krsifg}{\K_\rsi^\fg}
\newcommand{\tfAb}{\mathsf{TFAG}}
\newcommand{\DAb}{\mathsf{DAG}}
\newcommand{\lAb}{{\ell\mathsf{AG}}}
\newcommand{\lDAb}{{\ell\mathsf{DAG}}}
\newcommand{\MV}{\mathsf{MV}}
\newcommand{\DMV}{\mathsf{DMV}}

\DeclareUnicodeCharacter{001B}{}

\makeatletter
\renewenvironment{abstract}
  {%
    \small
    \begin{center}%
      {\bfseries \abstractname\par}%
    \end{center}%
  }
\makeatother

\begin{document}

\title{The theory of implicit operations}

\author{Luca Carai, Miriam Kurtzhals, and Tommaso Moraschini}

\address{Luca Carai: Dipartimento di Matematica ``Federigo Enriques'', Universit\`a degli Studi di Milano, via Cesare Saldini 50, 20133 Milano, Italy}\email{luca.carai.uni@gmail.com}

\address{Miriam Kurtzhals and Tommaso Moraschini: Departament de Filosofia, Facultat de Filosofia, Universitat de Barcelona (UB), Carrer Montalegre, $6$, $08001$ Barcelona, Spain}
\email{mkurtzku7@alumnes.ub.edu and tommaso.moraschini@ub.edu}

\date{\today}

\maketitle

\begin{abstract}
    A family of partial functions of a class of algebras $\K$ is said to be an \emph{implicit operation} of $\K$ when it is defined by a first order formula and it is preserved by homomorphisms. In this work, we develop the theory of implicit operations from an algebraic standpoint.  
    
    Notably, the implicit operations of an elementary class $\K$ are exactly the partial functions on $\K$ definable by existential positive formulas. For instance, ``taking inverses'' is an implicit operation of the class of monoids, defined by the conjunction of the equations $xy \thickapprox 1$ and $yx \thickapprox 1$, saying that $y$ is the inverse of $x$. As this example suggests, implicit operations need not be definable by the terms of the corresponding class of algebras. In fact, the demand that every implicit operation of a universal class $\K$ be interpolated by a family of terms is equivalent to the demand that $\K$ has the strong epimorphism surjectivity property. 
    
    However, implicit operations are always interpolated by families of implicit operations of a simpler kind, namely, those defined by \emph{pp formulas}, i.e., formulas of the form $\exists \vec{x} \varphi$, where $\varphi$ is a conjunction of equations. Motivated by this, we establish an \emph{existential elimination theorem} stating that, when $\K$ is a quasivariety with the amalgamation property, every implicit operation of $\K$ is interpolated by a family of implicit operations defined by conjunctions of equations (i.e., by pp formulas without existential quantifiers). We also provide a series of methods to test whether a concrete class of algebras has the strong epimorphism surjectivity property or, equivalently, to test whether interpolation can be carried out using terms only.

As the implicit operations of a class of algebras $\K$ need not belong to the language of $\K$, it is natural to wonder whether $\K$ can be expanded with its implicit operations. 
The main obstacle is that, in general, implicit operations need not be total. Accordingly, we say that an implicit operation of $\K$ is \emph{extendable} when every member of $\K$ can be extended to one in which the operation is total. For instance, the operation of ``taking inverses'' is not extendable in the class of monoids, but it becomes so in the class of cancellative commutative monoids because every such monoid embeds into an Abelian group.

When expanding a class of algebras $\K$ with its pp definable extendable implicit operations produces a class $\M$ in which every implicit operations is interpolated by a family of terms, we say that $\M$ is a \emph{Beth companion} of $\K$. In the context of quasivarieties, Beth companions are essentially unique, in the sense that all the Beth companions of a quasivariety are term equivalent. However, Beth companions need not exist in general: while Abelian groups are the Beth companion of cancellative commutative monoids, the class of all (commutative) monoids lacks a Beth companion. A series of tools to describe the Beth companion of a concrete class of algebras is also exhibited, drawing connections with absolutely closed, injective, and saturated algebras.

The appeal of Beth companions depends largely on whether the structure theory of a class is improved by moving to its Beth companion. We show that this is indeed the case by proving that, under minimal assumptions, every Beth companion of a relatively congruence distributive quasivariety whose class of relatively finitely subdirectly irreducible members is closed under nontrivial subalgebras is an arithmetical variety with the congruence extension property. 
This theorem applies, for instance, to the quasivariety of reduced commutative rings (which lacks the properties given by the theorem) and its Beth companion, namely, the variety of implicitly closed meadows. 
As a corollary, we  obtain that the Beth companion of a relatively filtral quasivariety must be a discriminator variety.

A variety of examples, together with their Beth companions, are discussed (see \cref{table: Beth companions}). These include both examples from classical algebra such as semigroups and monoids (both with and without commutativity and cancellativity), Abelian $\ell$-groups, torsion-free Abelian groups, and reduced commutative rings, as well as examples with a logical motivation such as (bounded) distributive lattices, Hilbert algebras, Heyting algebras, and MV-algebras.
\end{abstract}

\begin{table}[ht]
\centering
\begin{tabular}{|p{6.5cm}|p{6.5cm}|p{2cm}|} \hline
\textbf{Class} & \textbf{Beth companion}  & \textbf{Location}\\ \hline\hline
 universal classes with the strong   epimorphism surjectivity property & themselves & Thm.~\ref{Thm : Beth companion : examples} \\ \hline
relatively filtral quasivarieties & discriminator varieties & Cor.~\ref{Cor : rel filtral -> discriminator} \\ \hline
monoids & no Beth companion & Thm.~\ref{Thm : M lacks Beth comp} \\ \hline
semigroups & no Beth companion & Rem.~\ref{Rem : semigroups and Beth comp} \\ \hline
commutative monoids & no Beth companion & Thm.~\ref{Thm : M lacks Beth comp} \\ \hline
commutative semigroups & no Beth companion & Rem.~\ref{Rem : semigroups and Beth comp} \\ \hline
cancellative commutative monoids & Abelian groups & Thm.~\ref{Thm : Beth companion : examples} \\ \hline
cancellative commutative semigroups & Abelian groups & Rem.~\ref{Rem : semigroups and Beth comp} \\ \hline
torsion-free Abelian groups & Abelian groups with division & Thm.~\ref{Thm : Beth companion TF Abelian groups} \\ \hline 
Abelian $\ell$-groups &  Abelian $\ell$-groups with division & Thm.~\ref{Thm : div are Beth of tfAb and lAb} \\ \hline 
reduced commutative rings & implicitly closed meadows & Exa.~\ref{Exa : Beth companion : meadows : zero characteristic} \\ \hline 
distributive lattices & relatively complemented distributive lattices & Thm.~\ref{Thm : Beth companion : examples} \\ \hline
bounded distributive lattices & Boolean algebras & Thm.~\ref{Thm : Beth companion : examples} \\ \hline
Hilbert algebras & implicative semilattices & Thm.~\ref{Thm : Beth companion : examples} \\ \hline
pseudocomplemented  distributive lattices & Heyting algebras of depth $\leq 2$ & Thm.~\ref{Thm : Beth companion : examples} \\ \hline
varieties generated by a linearly ordered Heyting algebra $\A$ & no Beth companion if $5 \leq \vert A \vert  < \omega$ and $\VVV(\A)$ otherwise & Thm.~\ref{Thm : V(Cn) lacks Beth comp} \\ \hline
MV-algebras & MV-algebras with division & Thm.~\ref{Thm : Beth comp of MV} \\ \hline
 varieties generated by an MV-algebra of the form $\Luk_n$  & varieties generated by the expansion of $\Luk_n$ with a constant for $\frac{1}{n}$ & Thm.~\ref{Thm : Beth comp of V(MVn)} \\ \hline

\end{tabular}
\vspace{3mm}
\caption{Some classes of algebras and their Beth companions.}
\label{table: Beth companions}
\end{table}

\tableofcontents

\section{Formulas, compactness, and preservation theorems}

Throughout this work, we will assume familiarity with the  notions of an algebra and a homomorphism from universal algebra (see, e.g., \cite{Ber11,BuSa00}), with simple constructions such as direct products or subalgebras, as well as with first order formulas and their interpretation in mathematical structures (see, e.g., \cite{ModCK,HodModTh}). In doing so, we will restrict our attention to algebraic languages. Furthermore, classes of algebras $\K$ will be always assumed to be classes of \emph{similar} algebras, that is, algebras with a common language. Given a pair of algebras $\A$ and $\B$, we write $\A \leq \B$ to indicate that $\A$ is a subalgebra of $\B$. 
We denote the direct product of a family $\{\A_i : i \in I \}$ of similar algebras by $\prod_{i \in I}\A_i$ and the projection maps by $p_j \colon \prod_{i \in I}\A_i \to \A_j$ for every $j \in I$. We also write $\A_1 \times \dots \times \A_n$ for the product of a finite family $\{ \A_1, \dots, \A_n\}$.

By a \emph{formula} we always mean a first order formula.\  Given a formula $\varphi$, we write $\varphi(x_1, \dots, x_n)$ to indicate that the free variables of $\varphi$ are among $x_1, \dots, x_n$.
We denote the conjunction, disjunction, and implication of a pair of formulas $\varphi$ and $\psi$ by $\varphi \sqcap \psi$, $\varphi \sqcup \psi$, and $\varphi \to \psi$, respectively. Moreover, we denote the negation of a formula $\varphi$ by $\lnot \varphi$. When $\varphi$ is an equation $t_1 \thickapprox t_2$, we often write $t_1 \not \thickapprox t_2$ as a shorthand for $\lnot (t_1 \thickapprox t_2)$. 
Given an algebra $\A$, a formula $\varphi(x_1, \dots, x_n)$, and $a_1, \dots, a_n \in A$, we write $\A \vDash \varphi(a_1, \dots, a_n)$ to indicate that $\varphi$ holds in $\A$ of the elements $a_1, \dots, a_n$. When $\A \vDash \varphi(a_1, \dots, a_n)$ for all $a_1, \dots, a_n \in A$, we say that $\varphi$ is \emph{valid} in $\A$
and write $\A \vDash \varphi$. Similarly, if $\Phi$ is a set of formulas, we write $\A \vDash \Phi$, and say that $\A$ is \emph{a model of} $\Phi$, to indicate that $\A \vDash \varphi$ for each $\varphi \in \Phi$. This notion extends to classes of algebras $\K$ as follows: we say that a formula $\varphi$ is \emph{valid} in $\K$ and write $\K \vDash \varphi$ when $\A \vDash \varphi$ for each $\A \in \K$. Similarly, we write $\K \vDash \Phi$ when $\K \vDash \varphi$ for each $\varphi \in \Phi$. We always allow two special formulas $\top$ and $\bot$ without free variables and assume that $\top$ is valid in every algebra, while $\bot$ is not valid in any algebra.

A pair of formulas $\varphi(x_1, \dots, x_n)$ and $\psi(x_1, \dots, x_n)$ is said to be \emph{equivalent} in a class of algebras $\K$ when for all $\A \in \K$ and $a_1, \dots, a_n \in A$,
\[
\A \vDash \varphi(a_1, \dots, a_n) \iff \A \vDash \psi(a_1, \dots, a_n).
\]
When $\varphi$ and $\psi$ are equivalent in every class of algebras, we simply say that they are \emph{equivalent}.

A formula is said to be \emph{existential positive} when it is of the form 
\begin{equation}\label{Eq : existential positive formulas : definition}
\exists x_1, \dots, x_n \varphi,
\end{equation}
where $\varphi$ is built from equations, $\top$, and $\bot$ using only $\sqcap$ and $\sqcup$. 
An existential positive formula of the form in \eqref{Eq : existential positive formulas : definition} is called a \emph{primitive positive formula} (\emph{pp formula} for short) 
when $\varphi$ is a finite conjunction of equations.\footnote{We admit the empty conjunction, which is defined to be $\top$.}
Since existential quantifiers and conjunctions distribute over disjunctions up to equivalence, every existential positive formula is equivalent to a disjunction of pp formulas.
Lastly, a formula is said to be \emph{universal} when it is of the form $\forall x_1, \dots, x_n \varphi$, where $\varphi$ is a quantifier-free formula.

Let $\K$ be a class of algebras. A formula $\varphi(x_1, \dots, x_n)$ is said to be preserved by
\benroman
\item \emph{homomorphisms} in $\K$ when for every homomorphism $h \colon \A \to \B$ with $\A, \B \in \K$ and $a_1, \dots, a_n \in A$,
\[
\text{if }\A \vDash \varphi(a_1, \dots, a_n) \text{, then }\B \vDash \varphi(h(a_1), \dots, h(a_n));
\]
\item \emph{direct products} in $\K$ when for all
$\{\A_i : i \in I \} \subseteq \K$ and 
$a_1, \dots, a_n \in \prod_{i \in I}A_i$,
\[
\text{if }\A_i \vDash \varphi(p_i(a_1), \dots, p_i(a_n)) \text{ for each }i \in I\text{, then }\prod_{i \in I}\A_i \vDash \varphi(a_1, \dots, a_n);
\]
\item \emph{subalgebras} in $\K$ when for all $\A \leq \B \in \K$ and $a_1, \dots, a_n \in A$,
\[
\text{if }\B \vDash \varphi(a_1, \dots, a_n) \text{, then }\A \vDash \varphi(a_1, \dots, a_n).
\]
\eroman

A class of algebras $\mathsf{K}$ is said to be \emph{elementary} when it can be axiomatized by a set of formulas $\Phi$, i.e., $\mathsf{K}$ is the class of models of $\Phi$.
We rely on the following preservation theorem for elementary classes.

\begin{Theorem}\label{Thm : preservation}
Let $\K$ be an elementary class. A formula $\varphi(x_1, \dots, x_n)$ is preserved by  
\benroman
\item\label{item : preservation : ep} homomorphisms in $\K$ if and only if it is equivalent in $\K$ to an existential positive formula; 
\item\label{item : preservation : pp} homomorphisms and direct products in $\K$ if it is a pp formula;
\item\label{item : preservation : universal} subalgebras in $\K$ if and only if it is equivalent in $\K$ to a universal formula.
\eroman
\end{Theorem}

\begin{proof}
\eqref{item : preservation : ep}: This fact is known as the \emph{homomorphism preservation Theorem} and is due to \LL o\'s, Lyndon, and Tarski \cite{Los55,LynHom,Tar55}. Since we have not been able to find 
 this result relativized to elementary classes explicitly stated in the literature, we provide a complete proof here.
The argument requires some basic notions from the model theory of structures in languages containing both function and relation symbols that can be found in any standard book on model theory such as \cite{HodModTh}. 

For every $\A \in \K$ let $R^{\A}$ be the $n$-ary relation on $A$ defined by $\langle a_1, \dots, a_n \rangle \in R^{\A}$ if and only if $\A \vDash \varphi(a_1, \dots, a_n)$. We denote by $\A^*$ the structure obtained by equipping the algebra $\A$ with $R^\A$. Then $\K^* = \{\A^* : \A \in \K \}$ is an elementary class in the language of $\K$ expanded with an $n$-ary relation symbol $R$. Indeed, since $\K$ is an elementary class, an axiomatization of $\K^*$ is obtained by adding to the axiomatization of $\K$ the first order formula $R(x_1, \dots, x_n) \leftrightarrow \varphi(x_1, \dots, x_n)$. 
Since $\K^*$ is elementary,
\cite[Thm.~6.2(3)]{CampVaggSemCon} yields that the following conditions are equivalent:
\begin{enumerate}[(a)]
    \item for every homomorphism $h \colon \A \to \B$ between members of $\K$ and $a_1, \dots, a_n \in A$ we have that $\langle a_1, \dots, a_n\rangle \in R^\A$ implies $\langle h(a_1), \dots, h(a_n)\rangle \in R^\B$;
    \item there exists an existential positive formula $\psi$ in the language of $\K$ such that for all $\A \in \K$ and $a_1, \dots, a_n \in A$ we have $\langle a_1, \dots, a_n\rangle \in R^\A$ if and only if $\A \vDash \psi(a_1, \dots, a_n)$.
\end{enumerate}
It follows immediately from the definition of $R^\A$ and $R^\B$ that condition (a) is equivalent to $\varphi$ being preserved by homomorphisms in $\K$, while condition (b) says that $\varphi$ is equivalent to an existential positive formula in $\K$. We then conclude that $\varphi$ is preserved by homomorphisms in $\K$ if and only if it is equivalent in $\K$ to an existential positive formula, as desired.

\eqref{item : preservation : pp}: Suppose that $\varphi$ is a pp formula. As every pp formula is an existential positive formula, \eqref{item : preservation : ep} implies that $\varphi$ is preserved by homomorphisms in $\K$.
We show that $\varphi$ is also preserved by direct products as well. Since $\varphi(x_1, \dots, x_n)$ is a pp formula, it is of the form
\begin{equation*} 
\exists z_1, \dots, z_m \psi(x_1, \dots, x_n, z_1, \dots, z_m),
\end{equation*} 
where $\psi$ is a finite conjunction of equations. 
Consider $\{\A_i : i \in I \} \subseteq \K$ and $a_1, \dots, a_n \in \prod_{i \in I}A_i$ such that 
$\A_i \vDash \varphi(p_i(a_1), \dots, p_i(a_n))$ for each $i \in I$.
Our goal is to show that $\prod_{i \in I}\A_i \vDash \varphi(a_1, \dots, a_n)$.
For every $i \in I$, from 
$\A_i \vDash \varphi(p_i(a_1), \dots, p_i(a_n))$
it follows that there exists $\langle b_1^i, \dots, b_m^i \rangle \in A_i^m$ such that 
\[
\A_i \vDash \psi(p_i(a_1), \dots, p_i(a_n), b_1^i, \dots, b_m^i).
\] 
It is straightforward to verify that conjunctions of equations and $\top$ are preserved by direct products. Therefore, letting $b_1 = \langle b_1^i : i \in I \rangle, \dots, b_m = \langle b_m^i : i \in I \rangle$, we obtain
\[
\prod_{i \in I}\A_i \vDash \psi(a_1, \dots, a_n, b_1, \dots, b_m).
\]
Hence, we conclude that
$\prod_{i \in I}\A_i \vDash \varphi(a_1, \dots, a_n)$.

\eqref{item : preservation : universal}: See, e.g., \cite[Thm.~6.5.4]{HodModTh} and the subsequent paragraph.
\end{proof}

We recall the Compactness Theorem of first order logic (see, e.g., \cite[Thm.~6.1.1]{HodModTh}).

\begin{Compactness Theorem}\label{Thm : compactness theorem original}
A set of formulas $\Phi$ has a model if every finite subset of $\Phi$ has a model.
\end{Compactness Theorem}

For the present purpose, it is convenient to phrase the Compactness Theorem in terms of infinite conjunctions and disjunctions as well. 
To this end, we denote the conjunction and the disjunction of a (possibly infinite) set of formulas $\Phi$, respectively, by
\[
\bigsqcap \Phi \, \, \text{ and } \, \, \bigsqcup \Phi.
\]
When $\Phi = \emptyset$, we assume that $\bigsqcap \Phi = \top$ and $\bigsqcup \Phi = \bot$. 
Given a pair of sets of formulas $\Phi$ and $\Psi$ with free variables among $\langle x_i : i \in I \rangle$, an algebra $\A$, and a sequence $\vec{a} = \langle a_i : i \in I \rangle$ of elements of $A$, we write
\[
\A \vDash \Big(\bigsqcap \Phi \to \bigsqcup\Psi \Big)(\vec{a})
\]
to indicate that if $\A \vDash \varphi(\vec{a})$ for each $\varphi \in \Phi$, there exists $\psi \in \Psi$ such that $\A \vDash \psi(\vec{a})$. When the above display holds for every $\vec{a}$, we write $\A\vDash \bigsqcap \Phi \to \bigsqcup\Psi$. Similarly, given $\K$ a class of algebras, we write $\K\vDash \bigsqcap \Phi \to \bigsqcup\Psi$ to indicate that $\A \vDash \bigsqcap \Phi \to \bigsqcup\Psi$ for each $\A \in \K$.

A standard argument shows that the Compactness Theorem~\ref{Thm : compactness theorem original} is equivalent to the following.

\begin{Compactness Theorem}\label{Thm : compactness theorem}
Let $\K$ be an elementary class. For each pair of sets of formulas $\Phi$ and $\Psi$,
\[
\text{if }\K\vDash \bigsqcap \Phi \to \bigsqcup\Psi\text{, then }\K\vDash \bigsqcap \Phi' \to \bigsqcup\Psi' \text{ for some finite }\Phi' \subseteq \Phi \text{ and } \Psi' \subseteq \Psi.
\]
\end{Compactness Theorem}

We will make use of the following version of the Compactness Theorem for pp formulas.

\begin{Corollary}\label{Cor : compactness for pp formulas}
Let $\K$ be an elementary class closed under direct products. For each pair of sets of pp formulas $\Phi$ and $\Psi$ with $\Psi \ne \emptyset$,
\[
\text{if }\K\vDash \bigsqcap \Phi \to \bigsqcup\Psi\text{, then }\K\vDash \bigsqcap \Phi' \to\psi \text{ for some finite }\Phi' \subseteq \Phi \text{ and } \psi \in \Psi.
\]
\end{Corollary}

\begin{proof}
Suppose that $\K\vDash \bigsqcap \Phi \to \bigsqcup\Psi$. By the Compactness Theorem \ref{Thm : compactness theorem} there exist finite $\Phi' \subseteq \Phi$ and $\Psi' \subseteq \Psi$ such that
\begin{equation}\label{Eq : finitizing the implication : compactness}
\K\vDash \bigsqcap \Phi' \to \bigsqcup\Psi'.
\end{equation}
As $\Psi \ne \emptyset$, we may assume that $\Psi' \ne \emptyset$. Then consider an enumeration $\Psi' = \{ \psi_1, \dots, \psi_n \}$. To conclude the proof, it suffices to show that $\K\vDash \bigsqcap \Phi' \to \psi_i$ for some $i \leq n$. Suppose the contrary, with a view to contradiction. Then let $x_1, \dots, x_m$ be the free variables of $\bigsqcap\Phi' \to \bigsqcup\Psi'$. For each $i \leq n$ there exist $\A_i \in \K$ and $a_1^i, \dots, a_m^i \in A_i$ such that
\begin{equation}\label{Eq : compactness : premise and conclusion}
\A_i \vDash \bigsqcap\Phi'(a_1^i, \dots, a_m^i) \, \, \text{ and } \, \, \A_i \nvDash \psi_i(a_1^i, \dots, a_m^i).
\end{equation}
Then consider the elements $a_1 = \langle a_1^i : i \leq n \rangle, \dots, a_m = \langle a_m^i : i \leq n \rangle$ of $A_1 \times \dots \times A_n$. From the left hand side of the above display, the assumption that $\Phi'$ is a set of pp formulas, and Theorem \ref{Thm : preservation}(\ref{item : preservation : pp}) it follows that
\[
\A_1 \times \dots \times \A_n \vDash \bigsqcap\Phi'(a_1, \dots, a_m).
\]
As $\A_1, \dots, \A_n \in \K$ and $\K$ is closed under direct products by assumption, we obtain $\A_1 \times \dots \times \A_n \in \K$. Together with (\ref{Eq : finitizing the implication : compactness}) and the above display, this yields
\[
\A_1 \times \dots \times \A_n \vDash \psi_i(a_1, \dots, a_m)
\]
for some $i \leq n$. 
From the above display, Theorem \ref{Thm : preservation}(\ref{item : preservation : pp}) and the assumption that $\psi_i$ is a pp formula it follows that $\A_i \vDash \psi_i(p_i(a_1), \dots, p_i(a_m))$. Together with the definition of $a_1, \dots, a_m$, this amounts to $\A_i \vDash \psi_i(a_1^i, \dots, a_m^i)$, a contradiction with the right hand side of (\ref{Eq : compactness : premise and conclusion}).
\end{proof}

The following may be regarded as a converse to \cref{Thm : preservation}\eqref{item : preservation : pp} under some additional assumptions.

\begin{Theorem}
Let $\K$ be an elementary class closed under direct products and $\varphi$ a formula that is preserved by homomorphisms and direct products in $\K$. Then $\varphi$ is equivalent in $\K$ to a pp formula.    
\end{Theorem}

\begin{proof}
As $\varphi$ is preserved by homomorphisms in $\K$, \cref{Thm : preservation}\eqref{item : preservation : ep} yields an existential positive formula $\psi$ that is equivalent in $\K$ to $\varphi$.
Since $\psi$ is an existential positive formula, it is equivalent to a finite disjunction $\psi_1 \sqcup \dots \sqcup \psi_m$
of pp formulas $\psi_i$. It is then sufficient to show that there exists $i \leq m$ such that $\K \vDash \psi \leftrightarrow \psi_i$.  Since 
$\K \vDash \psi \rightarrow (\psi_1 \sqcup \dots \sqcup \psi_m)$,
by \cref{Cor : compactness for pp formulas} there exists $i \leq m$ such that $\K \vDash \psi \rightarrow \psi_i$. Because $\psi$ is equivalent to 
 $\psi_1 \sqcup \dots \sqcup \psi_m$, 
we have $\K \vDash \psi_i \rightarrow \psi$. Thus, the formula $\psi$, and hence also $\varphi$, is equivalent in $\K$ to $\psi_i$, which is a pp formula.
\end{proof}

We close this introductory section by recalling a fundamental theorem on ultraproducts known as \LL o\'s' Theorem (see, e.g., \cite[Thm.\ V.2.9]{BuSa00}).
To this end, given a family of algebras $\{ \A_i : i \in I \}$, a formula $\varphi(x_1, \dots, x_n)$, and $a_1, \dots, a_n \in \prod_{i \in I}A_i$, let
\[
\llbracket \varphi(a_1, \dots, a_n) \rrbracket = \{ i \in I : \A_i \vDash \varphi(p_i(a_1), \dots, p_i(a_n))\}.
\]

\begin{Los Theorem}\label{Thm : Los}
Let $\{ \A_i : i \in I \}$ be a family of algebras and $U$ an ultrafilter on $I$. For every formula $\varphi(x_1, \dots, x_n)$ and $a_1, \dots, a_n \in \prod_{i \in I}\A_i$ we have
\[
\prod_{i \in I}\A_i / U \vDash \varphi(a_1 / U, \dots, a_n / U) \iff \llbracket \varphi(a_1, \dots, a_n) \rrbracket \in U.
\]
\end{Los Theorem}

We denote by $\PPU$ the class operator of closure under ultraproducts. As a consequence of \LL o\'s' Theorem, every elementary class is closed under $\PPU$.

\section{Universal algebra}

This section reviews the main tools of general algebraic nature that will be employed in this monograph. The reader need not read it  in its entirety before proceeding with the subsequent sections and can come back to it each time they encounter a new notion.

We denote the class operators of closure under isomorphic copies, subalgebras, homomorphic images, direct products, and ultraproducts by $\III, \SSS, \HHH, \PPP$, and $\PPU$, respectively. A class of algebras is said to be:
\benroman
\item a \emph{variety} when it is closed under $\HHH, \SSS$, and $\PPP$;
\item a \emph{quasivariety} when it is closed under $\III, \SSS, \PPP$, and $\PPU$;
\item a \emph{universal class} when it is closed under $\III, \SSS$, and $\PPU$.
\eroman

While every variety is a quasivariety and every quasivariety is a universal class, the converses are not true in general.
We call \emph{proper} the quasivarieties that are not varieties and the universal classes that are not quasivarieties. Examples of a proper quasivariety and a proper universal class are the classes of cancellative commutative monoids (\cref{Exa : CCM : inverses are extendable}) and of fields (\cref{Example : inverses in rings}), respectively.

The next theorem provides an alternative characterization of the above mentioned classes in terms of axiomatizability by certain types of formulas (see, e.g., \cite[Thms.~II.11.9 \& V.2.25 \& V.2.20]{BuSa00}). We recall that a formula is called a \emph{quasiequation} when it is of the form
    \[
    \bigsqcap \Phi \to \varphi,
    \]
    where $\Phi \cup \{ \varphi \}$ is a finite set of equations. When $\Phi = \emptyset$, the above quasiequation is equivalent to the equation $\varphi$. Consequently, every equation is equivalent to a quasiequation.

\begin{Theorem}\label{Thm : classes generation}
The following conditions hold for a class of algebras $\K$:
\benroman
\item\label{item : variety generation} $\K$ is a variety if and only if it can be axiomatized by a set of equations;
\item\label{item : quasivariety generation} $\K$ is a quasivariety if and only if it can be axiomatized by a set of quasiequations;
\item\label{item : universal class generation} $\K$ is a universal class if and only if it can be axiomatized by a set of universal formulas.
\eroman
\end{Theorem}

We denote the least variety, the least quasivariety, and the least universal class containing a class of algebras $\K$ by $\VVV(\K)$, $\QQQ(\K)$, and $\UUU(\K)$, respectively. A variety (resp.\ quasivariety) $\K$ is \emph{finitely generated} when $\K = \VVV(\M)$ (resp.\ $\K = \QQQ(\M)$) for a finite set $\M$ of finite algebras.
The following characterizes $\VVV(\K)$, $\QQQ(\K)$, and $\UUU(\K)$ in terms of the class operators 
(see, e.g., \cite[Thms.~II.9.5 \& V.2.25 \& V.2.20]{BuSa00}).

\begin{Theorem}\label{Thm : quasivariety generation}
For every class of algebras $\K$,
\[
\VVV(\K) = \HHH\SSS\PPP(\K), \quad \QQQ(\K) = \III\SSS\PPP\PPU(\K), \, \, \text{ and } \, \, \UUU(\K) = \III\SSS\PPU(\K).
\]
\end{Theorem}

The following is a straightforward consequence of \cref{Thm : classes generation}.

\begin{Corollary}\label{Cor: VQU and formulas}
The following conditions hold for a class of algebras $\K$:
\benroman
\item\label{Cor: VQU and formulas : V} $\VVV(\K)$ is the class of models of all the equations valid in $\K$;
\item\label{Cor: VQU and formulas : Q} $\QQQ(\K)$ is the class of models of all the quasiequations valid in $\K$;
\item\label{Cor: VQU and formulas : U} $\UUU(\K)$ is the class of models of all the universal formulas valid in $\K$.
\eroman
\end{Corollary}

We will make use of the following closure property of universal classes (see, e.g., \cite[Thm.~3.2.3]{ModCK}).

\begin{Proposition}\label{Prop : universal class : unions of chains}
Universal classes are closed under the formation of unions of chains of algebras.
\end{Proposition}

As quasivarieties need not be closed under $\HHH$, the following concept is often useful.\ Let $\mathsf{K}$ be a quasivariety and $\A$ and algebra (not necessarily in $\K$). A congruence $\theta$ of $\A$ is said to be a $\K$\emph{-congruence} when $\A / \theta \in \K$. Owing to the fact that $\K$ is closed under $\III$ and $\SSS$,  the Homomorphism Theorem \cite[Thm.~II.6.12]{BuSa00} yields that  the \emph{kernel} 
\[
\Ker(h) = \{ \langle a, b \rangle \in A \times A : h(a) = h(b) \}
\]
of every homomorphism $h \colon \A \to \B$ with $\B \in \K$ is a $\K$-congruence of 
 $\A$ such that $\A/\Ker(h) \cong h[\A]$, where $h[\A]$ denotes the subalgebra of $\B$ with universe $h[A]$.
When ordered under the inclusion relation, the set of $\K$-congruences of $\A$ forms an algebraic lattice $\Con_\K(\A)$ in which  meets are intersections (see, e.g., \cite[Prop.~1.4.7 \& Cor.~1.4.11]{Go98a}). 
Given $X \subseteq A \times A$, we denote the least congruence of $\A$ containing $X$ by $\Cg^\A(X)$ and the least $\K$-congruence of $\A$ containing $X$ by $\Cg_\K^\A(X)$. 
We will rely on the following observation.

\begin{Proposition}\label{Prop : variety iff Con=ConK}
A quasivariety $\K$ is a variety if and only if $\mathsf{Con}(\A) = \mathsf{Con}_\K(\A)$ for every $\A \in \K$.
\end{Proposition}

\begin{proof}
First, suppose that  $\K$ is a variety. If $\A \in \K$ and $\theta \in \Con(\A)$, then $\A/\theta \in \HHH(\A) \subseteq \K$. Therefore, $\mathsf{Con}(\A) = \mathsf{Con}_\K(\A)$ for every $\A \in \K$. Assume now that $\mathsf{Con}(\A) = \mathsf{Con}_\K(\A)$ for every $\A \in \K$. Since $\K$ is a quasivariety, it  suffices to prove that $\K$ is closed under homomorphic images. Consider $\B \in \HHH(\K)$. Then there exists a surjective homomorphism $h \colon \A \to \B$ with $\A \in \K$. Since $\Ker(h) \in \Con(\A)$, the assumption yields  $\Ker(h) \in \Con_\K(\A)$. Thus, $\A/\Ker(h) \in \K$. As $\B \cong \A/\Ker(h)$ and $\K$ is closed under $\III$, it follows that $\B \in \K$. 
\end{proof}

The following gives a necessary and sufficient condition for a homomorphism to factor through a quotient (see, e.g., \cite[p.~62]{Gra08}).

\begin{Proposition}\label{Prop : homomorphisms : smaller congruences}
Let $h \colon \A \to \B$ be a homomorphism, $\theta \in \Con(\A)$, and $f \colon \A \to \A/\theta$ the canonical surjection. Then $\theta \subseteq \Ker(h)$ if and only if there exists a homomorphism $g \colon \A / \theta \to \B$ such that $g \circ f = h$.

 \begin{center}
\begin{tikzcd}
  \A \arrow[r, "h"] \arrow[d, "f"'] & \B \\
\A/\theta \arrow[ur, dashed, "g"'] 
\end{tikzcd}
\end{center}
\end{Proposition}

An algebra $\A$ is a \emph{subdirect product} of a family $\{ \B_i : i \in I \}$ when $\A \leq \prod_{i \in I}\B_i$ and for every $i \in I$ the projection map $p_i \colon \A \to \B_i$ is surjective.\ Similarly, an embedding $h \colon \A \to \prod_{i \in I}\B_i$ is called \emph{subdirect} when $h[\A]\leq \prod_{i \in I}\B_i$ is a subdirect product. The next result simplifies the task of constructing subdirect embeddings (see, e.g., \cite[Lem.\ II.8.2]{BuSa00}).

\begin{Proposition}\label{Prop : subdirect embedding}
Let $\A$ be an algebra and $X \subseteq \Con(\A)$. Then the map
\[
h \colon \A / \bigcap X \to \prod_{\theta \in X}\A / \theta 
\]
defined by the rule $h(a / \bigcap X) = \langle a / \theta : \theta \in X \rangle$ is a subdirect embedding.
\end{Proposition}

Notably, every congruence can be viewed as a subdirect product.

\begin{Proposition}\label{Prop : congruences are subalgebras}
Let $\K$ be a quasivariety and $\A \in \K$. Every congruence $\theta$ of $\A$ is the universe of an algebra $\theta^* \in \K$ such that $\theta^* \leq \A \times \A$ is a subdirect product.
\end{Proposition}

\begin{proof}
The fact that $\theta$ is the universe of a subalgebra $\theta^*$ of $\A \times \A$ is an immediate consequence of the definition of a congruence of $\A$. As $\A \in \K$ and $\K$ is closed under subalgebras and direct products (because it is a quasivariety), we obtain $\theta^* \in \K$. To conclude that $\theta^* \leq \A \times \A$ is a subdirect product, it suffices to show that the projection maps $p_1, p_2 \colon \theta^* \to \A$ are surjective. Consider $a \in A$. As $\theta$ is a reflexive relation on $A$, we have $\langle a, a \rangle \in \theta$.\ Consequently, $p_1(\langle a, a\rangle) = p_2(\langle a, a\rangle) = a$.
\end{proof}

Let $\K$ be a quasivariety. An algebra $\A \in \K$ is said to be \emph{subdirectly irreducible relative to $\K$} when for every subdirect embedding $h \colon \A \to \prod_{i \in I}\B_i$ with $\{ \B_i : i \in I \} \subseteq \K$ there exists $i \in I$ such that $p_i \circ h \colon \A \to \B_i$ is an isomorphism. In case this happens whenever the index set $I$ is finite, 
we say that $\A$ is \emph{finitely subdirectly irreducible relative to $\K$}.\footnote{We adopt the convention that the direct product of an empty family of algebras is the trivial algebra in the language under consideration. Consequently, we do not regard the trivial algebra as relatively (finitely) subdirectly irreducible.} 
The classes of algebras that are subdirectly irreducible relatively to $\K$  and finitely subdirectly irreducible relative to $\K$ will be denoted by $\K_\rsi$ and $\K_\rfsi$, respectively. When $\K$ is a variety, the requirement that $\{ \B_i : i \in I \}$ is a subset of $\K$ in the above definitions can be harmlessly dropped and we simply say that $\A$ is \emph{subdirectly irreducible} or \emph{finitely subdirectly irreducible} (i.e., we drop the ``relative to  $\K$''). In this case, we also write $\K_{\textsc{si}}$ and $\K_{\textsc{fsi}}$ instead of $\K_{\textsc{rsi}}$ and $\K_{\textsc{rfsi}}$.

The importance of subdirect embeddings and of algebras that are 
subdirectly irreducible relative to $\K$ derives from the following  representation theorem (see, e.g., \cite[Thm.\ 3.1.1]{Go98a}).

\begin{Subdirect Decomposition Theorem}\label{Thm : Subdirect Decomposition}
Let $\K$ be a quasivariety. For every $\A \in \K$ there exists a subdirect embedding $f \colon \A \to \prod_{i \in I}\B_i$ with $\{ \B_i : i \in I \} \subseteq \K_{\textsc{rsi}}$.
\end{Subdirect Decomposition Theorem}

For instance, an Abelian group is subdirectly irreducible precisely when it is either cyclic of prime-power order or quasicyclic (see, e.g., \cite[Thm.\ 3.29]{Ber11}). Therefore, every Abelian group can be represented as a subdirect product of Abelian groups of this form.

Notably, algebras that are (finitely) subdirectly irreducible relative to $\K$ can be recognized by looking at the  structure of their lattices of $\K$-congruences. More precisely, we recall that an element $a$ of a  lattice $\A$ is said to be:
\benroman
\item \emph{completely meet irreducible} when $a \in X$ for every $X \subseteq A$ such that $a = \bigwedge X$;
\item \emph{meet irreducible} when $a \in X$ for every finite $X \subseteq A$ such that $a = \bigwedge X$. 
\eroman
Notice that every completely meet irreducible element is meet irreducible and that the maximum of a lattice is never meet irreducible because it coincides with $\bigwedge \emptyset$. Given a quasivariety $\K$ and
 $\A \in \K$, let
\begin{align*}
\mathsf{Irr}^\infty_\K(\A) &= \text{the set of completely meet irreducible elements of }\Con_\K(\A);\\
\mathsf{Irr}_{\K}(\A) &= \text{the set of meet irreducible elements of }\Con_\K(\A).
\end{align*}
Furthermore, we denote the identity relation on $\A$ by $\textup{id}_A$.\  The following is a consequence of \cite[Cor.\ 1.4.8]{Go98a} and the Correspondence Theorem \cite[Thm.\ II.6.20]{BuSa00}.

\begin{Proposition}\label{Prop : RFSI}
Let $\A$ be a member of a quasivariety $\K$. For every $\theta \in \Con(\A)$ we have
\begin{align*}
\A / \theta \in \K_{\textup{\textsc{rsi}}} \, \, &\text{ if and only if } \, \, \theta \in \mathsf{Irr}_{\K}^\infty(\A);\\
\A / \theta \in \K_{\textup{\textsc{rfsi}}} \, \, &\text{ if and only if } \, \, \theta \in \mathsf{Irr}_{\K}(\A).
\end{align*}
Therefore, $\A\in \K_{\textup{\textsc{rsi}}}$  \textup{(}resp.\ $\A\in \K_{\textup{\textsc{rfsi}}}$\textup{)} if and only if $\textup{id}_A \in \mathsf{Irr}_{\K}^\infty(\A)$ \textup{(}resp.\ $\textup{id}_A \in \mathsf{Irr}_{\K}(\A)$\textup{)}.
\end{Proposition}

As a consequence, a member $\A$ of a quasivariety $\K$ is relatively subdirectly irreducible precisely when it has a least nonidentity $\K$-congruence, called the \emph{monolith} of $\A$. When it exists, the monolith of $\A$ is always the $\K$-congruence of $\A$ generated by a pair of distinct elements $a, b \in A$, which we denote by $\Cg_\K^\A(a, b)$.

Given two binary relations $R_1$ and $R_2$ on a set $A$, we let
\[
R_1 \circ R_2 = \{ \langle a, b \rangle \in A \times A : \text{there exists }c\in A \text{ s.t. }\langle a, c \rangle \in R_1 \text{ and }\langle c, b \rangle \in R_2 \}.
\]
A variety $\mathsf{K}$ is said to be \emph{congruence permutable} when for all $\A \in \mathsf{K}$ and $\theta, \phi \in \mathsf{Con}(\A)$ we have $\theta \circ \phi = \phi \circ \theta$.  A quasivariety $\K$ is said to be \emph{relatively congruence distributive} when $\Con_\K(\A)$ is a distributive lattice for every $\A \in \K$. If $\K$ is a variety, we simply say that $\K$ is  \emph{congruence distributive}. Notably, every variety whose members have a group (resp.\ lattice) structure is congruence permutable (resp.\ distributive) (see, e.g., \cite[p.~79]{BuSa00}). A variety that is both congruence distributive and congruence permutable is called \emph{arithmetical}.

\begin{Remark}
    Contrarily to the case of congruence distributivity, congruence permutability is usually understood as a property of varieties only (as opposed to arbitrary quasivarieties). The reason is that an algebra is congruence permutable if and only if $\theta \circ \phi$ is the join of $\theta$ and $\phi$ in $\mathsf{Con}(\A)$ for all $\theta,\phi \in \mathsf{Con}(\A)$. However, the sole quasivarieties $\K$ such that  $\theta \circ \phi$ is the join of $\theta$ and $\phi$ in $\mathsf{Con}_\K(\A)$ for all $\A \in \K$ and $\theta, \phi \in \mathsf{Con}_\K(\A)$ are those that are  varieties \cite[Thm.\ 5.5]{MR1268510} (see also \cite{MR1408738}).
\qed
\end{Remark}

The following is a generalization of J\'onsson's Theorem to the setting of finitely subdirectly irreducible algebras, which can be obtained as a straightforward consequence of \cite[Thm.~1.7]{CD90}.

\begin{Jonsson}\label{Thm : Jonsson}
Let $\K$ be a class of algebras such that $\VVV(\K)$ is congruence distributive. Then 
$\VVV(\K)_\fsi \subseteq \HHH\SSS\PPU(\K).$
\end{Jonsson}

We will also utilize the following analogous statement for quasivarieties (see \cite[Thm.~1.5]{CD90}). 

\begin{Theorem}\label{Thm : easy Jonsson quasivarieties}
Let $\K$ be a class of algebras. Then $\QQQ(\K)_{\rfsi} \subseteq \III\SSS\PPU(\K)$.
\end{Theorem}

When the class $\K$ in the above result is a finite set of finite algebras, the class operator $\PPU$ becomes superfluous because of the following observation  (see, e.g., \cite[Thm. 5.6(2)]{Ber11}).

\begin{Proposition} \label{Prop : P_u trivial in finite setting}
    If $\mathsf{K}$ is a finite set of finite algebras, then $\PPU(\mathsf{K}) \subseteq \III(\mathsf{K})$.
\end{Proposition}

Lastly, given an algebra $\A$ and a set $X \subseteq A$, we denote the least subuniverse of $\A$ containing $X$ by $\mathsf{Sg}^\A(X)$. When $\mathsf{Sg}^\A(X) \ne \emptyset$, the subalgebra of $\A$ with universe $\mathsf{Sg}^\A(X)$ will also be denoted by $\mathsf{Sg}^\A(X)$.
When $X = \{a_1, \dots, a_n\}$ is finite, we write $\mathsf{Sg}^\A(a_1, \dots, a_n)$ in place of $\mathsf{Sg}^\A(\{a_1, \dots, a_n\})$.
If $A = \mathsf{Sg}^\A(X)$ for some finite $X \subseteq A$, we say that $\A$ is \emph{finitely generated}. If every finitely generated subalgebra of $\A$ is finite, we call $\A$ \emph{locally finite}. A class of algebras is \emph{locally finite} when its members are. We denote the class of finitely generated members of a class of algebras $\K$ by $\Kfg$.

The following is an immediate consequence of \cite[Thm.~V.2.14]{BuSa00}.

\begin{Proposition} \label{Prop : universal class gen by fin gen}
Let $\K$ be a universal class.
Then $\K = \UUU(\Kfg)$.
\end{Proposition}

The \cref{Thm : Subdirect Decomposition} readily implies that $\K=\QQQ(\K_\rsi)$ for every quasivariety $\K$. It is well known that this result can be improved by restricting to the class $\Krsifg$ of finitely generated members of $\K_\rsi$.
As we were unable to find a reference in the literature, we provide a proof.

\begin{Proposition}\label{Prop : quasivariety = Q fin gen SI}
Let $\K$ be a quasivariety. Then $\K = \QQQ(\K_\rsi^\textup{fg})$. 
\end{Proposition}

\begin{proof}
From \cite[Prop.~2.1.18]{Go98a} it follows that $\K = \QQQ(\K^\textup{fg})$, where $\K^\textup{fg}$ is the class of the finitely generated members of $\K$.
Then let $\A \in \K^\textup{fg}$. By the \cref{Thm : Subdirect Decomposition} there exists a subdirect embedding $f \colon \A \to \prod_{i \in I}\A_i$, where $\A_i \in \K_\rsi$ for every $i \in I$. As $f$ is a subdirect embedding, each $\A_i$ is a homomorphic image of $\A$ and, therefore, finitely generated. So, $\A_i \in \K_\rsi^\textup{fg}$ for every $i \in I$. Since $f$ is an embedding into a direct product of members of $\K_\rsi^\textup{fg}$, we obtain $\A \in \III\SSS\PPP(\K_\rsi^\textup{fg}) \subseteq \QQQ(\K_\rsi^\textup{fg})$. Thus, $\K^\textup{fg} \subseteq \QQQ(\K_\rsi^\textup{fg})$, and hence $\QQQ(\K^\textup{fg}) \subseteq \QQQ(\K_\rsi^\textup{fg})$. Together with $\K = \QQQ(\K^\textup{fg})$, this yields $\K \subseteq \QQQ(\K_\rsi^\textup{fg})$. Since $\K$ is a quasivariety containing $\K_\rsi^\textup{fg}$, we conclude that $\K = \QQQ(\K_\rsi^\textup{fg})$.
\end{proof}

Let $\K$ be a quasivariety. For every pair of algebras $\A$ and $\B$, $\theta \in \mathsf{Con}_\K(\A)$, and $\phi \in \mathsf{Con}_\K(\B)$, the relation
\[
\theta \times \phi = \{ \langle \langle a_1, b_1 \rangle, \langle a_2, b_2 \rangle\rangle \in (A \times B)^2 : \langle a_1, a_2 \rangle \in \theta \text{ and } \langle b_1, b_2 \rangle \in \phi \}
\]
is a $\K$-congruence of the direct product $\A \times \B$. Given a pair of algebras $\A \leq \B$ and $\theta \in \mathsf{Con}_\K(\B)$, we write $\theta{\upharpoonright}_A$ as a shorthand for $\theta \cap (A \times A)$. Notice that $\theta{\upharpoonright}_A$ is a $\K$-congruence of $\A$.  The next result is an effortless generalization to quasivarieties of \cite[Thm.\ 1.2.20]{KP01}.

\begin{Theorem}\label{Thm : CD vs product congruences}
A quasivariety $\mathsf{K}$ is relatively congruence distributive if and only if for every subdirect product $\A \leq \B \times \C$ with $\B, \C \in \mathsf{K}$ and every $\theta \in \mathsf{Con}_\mathsf{K}(\A)$ there exist $\phi \in \mathsf{Con}_\mathsf{K}(\B)$ and $\eta \in \mathsf{Con}_\mathsf{K}(\C)$ such that $\theta = (\phi \times \eta){\upharpoonright}_A$.
\end{Theorem}

As a consequence, we deduce the following.

\begin{Corollary}\label{Cor : CD vs FSI product congruences}
Let $\K$ be a relatively congruence distributive quasivariety, $\A$ an algebra, and $\theta \in \mathsf{Con}_\K(\A)$ such that $\A / \theta$ is either trivial or a member of $\K_\textsc{rfsi}$. Then for every $\B \in \K$ such that  $\A \leq \B \times \B$ is a subdirect product there exists $\phi \in \mathsf{Con}_\K(\B)$ with
\[
\theta \in \{  (\phi \times B^2){\upharpoonright}_{A}, (B^2 \times \phi){\upharpoonright}_{A}\}.
\]
\end{Corollary}

\begin{proof}
Consider $\B \in \K$ such that  $\A \leq \B \times \B$ is a subdirect product. We have two cases: either $\A / \theta$ is trivial or it belongs to $\K_\textsc{rfsi}$. First, suppose that $\A / \theta$ is trivial. Then $\theta  = A \times A$. As $A \subseteq B \times B$ by assumption, we obtain
\[
\theta = A \times A = (B^2 \times B^2){\upharpoonright}_{A}.
\]
Since $B^2$ is a $\K$-congruence of $\B$ (because $\K$ contains all the trivial algebras in the appropriate language), we obtain $B^2 \in \mathsf{Con}_\K(\B)$. Hence, we are done taking $\phi = B^2$.

Next we consider the case where $\A / \theta \in \K_\textsc{rfsi}$. As $\A \leq \B \times \B$ is a subdirect product and $\K$ is relatively congruence distributive by assumption, we can apply Theorem \ref{Thm : CD vs product congruences}, obtaining some $\phi_1, \phi_2 \in \mathsf{Con}_\K(\B)$ such that $\theta = (\phi_1 \times \phi_2){\upharpoonright}_{A}$. Observe that $\phi_1 \times \phi_2 = (\phi_1 \times B^2) \cap (B^2 \times \phi_2)$. Therefore,
\begin{equation}\label{Eq : theta decomposes under RCD : 1}
\theta = (\phi_1 \times \phi_2){\upharpoonright}_{A} = ((\phi_1 \times B^2) \cap (B^2 \times \phi_2)){\upharpoonright}_{A} = (\phi_1 \times B^2){\upharpoonright}_{A} \cap (B^2 \times \phi_2){\upharpoonright}_{A}.
\end{equation}
Observe that $\phi_1 \times B^2, B^2 \times \phi_2 \in \mathsf{Con}_\K(\B \times \B)$ because $\phi_1, \phi_2 \in \mathsf{Con}_\K(\B)$. Together with $\A \leq \B \times \B$, this yields
\begin{equation}\label{Eq : theta decomposes under RCD : 2}
(\phi_1 \times B^2){\upharpoonright}_{A}, (B^2 \times \phi_2){\upharpoonright}_{A} \in \mathsf{Con}_\K(\A).
\end{equation}
Lastly, recall that $\A / \theta \in \K_\textsc{rfsi}$ by assumption. Then $\theta \in \mathsf{Irr}_\K(\A)$ by Proposition \ref{Prop : RFSI}. Therefore, from (\ref{Eq : theta decomposes under RCD : 1}) and (\ref{Eq : theta decomposes under RCD : 2}) it follows that $\theta \in \{ (\phi_1 \times B^2){\upharpoonright}_{A}, (B^2 \times \phi_2){\upharpoonright}_{A} \}$. Since $\phi_1, \phi_2 \in \mathsf{Con}_\K(\B)$, we are done.
\end{proof}

Let $\K$ be a class of algebras and $X$ a nonempty set of variables. The set of all terms built over variables in $X$ can be made into an algebra $\Term(X)$ called the \emph{term algebra over $X$}. The binary relation $\theta_\K$ on $T(X)$ defined by $\langle t,s \rangle \in \theta_\K$ if and only if $\K \vDash t \thickapprox s$ is a congruence of $\Term(X)$. We call the quotient $\Free_\K(X) = \Term(X)/\theta_\K$ the \emph{$\K$-free algebra over $X$}. 
Free algebras have the following universal mapping property: for every map $f \colon X \to A$ with $\A \in \K$ there exists a unique homomorphism $h \colon \Free_{\K}(X) \to \A$ such that $h(x/\theta_\K) = f(x)$ for every $x \in X$. To simplify the notation, we will often denote an element $t/\theta_K$ of $F_\K(X)$ simply by $t$.
The following states that quasivarieties contain all free algebras (see, e.g., \cite[Thm.~II.10.12]{BuSa00}).
\begin{Theorem}\label{thm:quasivariety free algebra}
Let $\K$ be a quasivariety and $X$ a nonempty set. Then $\Free_\K(X) \in \K$.    
\end{Theorem}

A member $\A$ of quasivariety $\K$ is called \emph{finitely presented}  when there exist a finite set of variables $X$ and a finite $Y \subseteq T(X) \times T(X)$ such that $\A \cong \Term(X)/\Cg_\K^{\Term(X)}(Y)$.

We call an algebra in a language $\L$ an \emph{$\L$-algebra}. 
When $\L$ and $\L'$ are two languages such that $\L \subseteq \L'$ we say that $\L'$ is an \emph{expansion} of $\L$. If $\L'$ is an expansion of $\L$, then for every $\mathscr{L'}$-algebra $\A$, we can consider its \emph{$\L$-reduct} $\A \res_\L$ obtained from $\A$ by forgetting the interpretations of all function symbols that are not in $\mathscr{L}$. Given a class $\K$ of $\L'$-algebras, we denote the class of the $\L$-reducts of members of $\K$ by $\K\res_\L$ and we call the members of $\SSS(\K\res_\L)$ the \emph{$\L$-subreducts} of $\K$. For instance, the monoid subreducts of Abelian groups are precisely the cancellative commutative monoids (see, e.g., \cite[pp.~39--40]{Lan84}).
We will often denote the language of a given class of algebras $\K$ by $\L_\K$, and we will refer to the terms of $\L_\K$ simply as the \emph{terms of $\K$}.

Particular cases of language expansions are those obtained by adding to the language names for the elements of a given algebra. More precisely, given an $\mathscr{L}$-algebra $\A$, we consider the language $\L_A$ obtained by adding to $\mathscr{L}$ a set of new constants $\{c_a : a \in A\}$ that is in bijection with the elements of $\A$.
Given a function $h \colon A \to B$ between the universes of a pair of $\L$-algebras $\A$ and $\B$, we denote by $\B_{h[A]}$ the $\L_A$-algebra whose $\L$-reduct is $\B$ and in which each constant $c_a$ is interpreted as $h(a)$. In particular, we denote by $\A_A$ the expansion of $\A$ to an $\L_A$-algebra induced by the identity map on $A$. We define the \emph{diagram} $\diag(\A)$ of an $\L$-algebra $\A$ to be the set of all variable-free $\L_A$-formulas that are equations and negated equations valid in $\A_A$.
The following lemma connects the validity of diagrams with the existence of embeddings (see, e.g., \cite[Prop.~2.1.8]{ModCK}).

\begin{Diagram Lemma}\label{Lem : Diagram Lemma}
Let $\A$ and $\B$ be $\L$-algebras and $h \colon A \to B$ a function. Then $h \colon \A \to \B$ is an embedding if and only if 
$\B_{h[A]} \vDash \diag(\A)$.
\end{Diagram Lemma}

\section{Implicit operations}

An $n$-ary \emph{partial function} on a set $X$ is a function $f \colon Y \to X$ for some $Y \subseteq X^n$. In this case, the set $Y$ will be called the \emph{domain} of $f$ and will be denoted by $\mathsf{dom}(f)$. This notion can be extended to classes of algebras as follows. An $n$-ary \emph{partial function} on a class of algebras $\mathsf{K}$ is a sequence $\langle f^\A : \A \in \mathsf{K}\rangle$, where $f^\A$ is an $n$-ary partial function on $A$ for each $\A \in \mathsf{K}$. By a \emph{partial function} on $\mathsf{K}$ we mean an $n$-ary partial function on $\mathsf{K}$ for some $n \in \mathbb{N}$. When $f$ is a partial function on $\mathsf{K}$ and $\A \in \mathsf{K}$, we denote the $\A$-component of $f$ by $f^\A$. 

\begin{exa}[\textsf{Monoids}]\label{Exa : inverses in monoids 0}
The operation of ``taking inverses'' can be viewed as a partial function on the variety of monoids $\mathsf{Mon}$. More precisely, recall that the inverse of an element $a$ of a monoid $\A = \langle A; \cdot, 1 \rangle$ is unique when it exists, in which case it will be denoted by $a^{-1}$. Then let $f^\A$ be the unary partial function on $A$ with domain
\[
\mathsf{dom}(f^\A)  = \{ a \in A : a \text{ has an inverse in }\A\},
\] 
defined for each $a \in \mathsf{dom}(f^\A)$ as $f^\A(a) = a^{-1}$. The sequence $\langle f^\A : \A \in \mathsf{Mon}\rangle$ is a unary partial function on $\mathsf{Mon}$.
\qed
\end{exa}

\begin{Definition}
A formula $\varphi(x_1, \dots, x_n, y)$ with $n \geq 1$ in the language of a class of algebras $\mathsf{K}$ is said to be \emph{functional} in $\mathsf{K}$ when for all $\A \in \mathsf{K}$ and $a_1, \dots, a_n \in A$ there exists at most one $b \in A$ such that $\A \vDash \varphi(a_1, \dots, a_n, b)$.  When $\K = \{ \A \}$, we often say that $\varphi$ is \emph{functional} in $\A$.
\end{Definition}

In other words, $\varphi$ is functional in a class of algebras $\mathsf{K}$ when
\[
\mathsf{K} \vDash (\varphi(x_1, \dots, x_n, y) \sqcap \varphi(x_1, \dots, x_n, z)) \to y \thickapprox z.
\]
In this case, $\varphi$ induces an $n$-ary partial function $\varphi^\A$ on each $\A \in \mathsf{K}$ with domain 
\[
\mathsf{dom}(\varphi^\A) = \{ \langle a_1, \dots, a_n \rangle \in A^n : \text{there exists $b \in A$ such that }\A \vDash \varphi(a_1, \dots, a_n, b) \},
\]
defined for all $\langle a_1, \dots, a_n \rangle \in \mathsf{dom}(\varphi^\A)$ as
$\varphi^\A(a_1, \dots, a_n) = b$, where $b$ is the unique element of $A$ such that $\A \vDash \varphi(a_1, \dots, a_n, b)$. 
Consequently,
\[
\varphi^{\mathsf{K}} = \langle \varphi^\A : \A \in \mathsf{K}\rangle
\]
is an $n$-ary partial function on $\mathsf{K}$. 

\begin{Definition}
A partial function $f$ on a class of algebras $\mathsf{K}$ is said to be
\benroman
\item \emph{defined by a formula $\varphi$} when $\varphi$ is functional in $\mathsf{K}$ and $f = \varphi^{\mathsf{K}}$;
\item \emph{implicit} when it is defined by some formula.
\eroman
\end{Definition}

We remark that the arity of implicit partial functions is always positive because the definition of a functional formula $\varphi = \varphi(x_1, \dots, x_n, y)$ requires $n$ to be positive.

\begin{exa}[\textsf{Monoids}]\label{Exa : inverses in monoids 1}
We will prove that the partial function of ``taking inverses'' in monoids introduced in Example \ref{Exa : inverses in monoids 0} is defined by the formula
\[
\varphi(x, y) = (x \cdot y \thickapprox 1) \sqcap (y \cdot x \thickapprox 1).
\]
First, observe that for each monoid $\A$ and $a, b \in A$ we have
\[
\A \vDash \varphi(a, b) \iff a \cdot b = 1 = b \cdot a \iff \text{$b = a^{-1}$.}
\]
As a consequence, for all $a, b, c \in A$ such that $\A \vDash \varphi(a, b) \sqcap \varphi(a, c)$ we have $b = a^{-1} = c$, whence $\varphi$ is functional in $\mathsf{Mon}$. Together with the above display, this shows that $\varphi$ defines the partial function of ``taking inverses'' in monoids which, therefore, is implicit.
\qed
\end{exa}

\begin{Definition}
An $n$-ary partial function $f$ on a class of algebras $\mathsf{K}$ is said to be
\benroman
\item an \emph{operation} of $\mathsf{K}$ when it is preserved by homomorphisms in $\K$, that is, for each homomorphism $h \colon \A \to \B$ with $\A,\B \in \mathsf{K}$ and $\langle a_1, \dots, a_n \rangle \in \mathsf{dom}(f^\A)$ we have $\langle h(a_1), \dots, h(a_n) \rangle \in \mathsf{dom}(f^\B)$ and
\[
h(f^\A(a_1,\dots,a_n))=f^\B(h(a_1),\dots,h(a_n));
\]
\item an \emph{implicit operation} of $\mathsf{K}$ when it is both implicit and an operation of $\mathsf{K}$. 
\eroman   
We denote the class of implicit operations of $\mathsf{K}$ by $\mathsf{imp}(\mathsf{K})$. When $\K = \{ \A \}$, we often write $\imp(\A)$ instead of $\imp(\K)$.
\end{Definition}

Partial functions on a class $\K$ of algebras are defined as sequences $\langle f^{\A} : \A \in \mathsf{K} \rangle$ of partial functions indexed by $\K$. Consequently, when $\K$ is a proper class, so is each of the partial functions on $\K$.
Nevertheless, since implicit partial functions
can be identified with their defining formulas, we will always treat $\imp(\K)$ as a set.

\begin{exa}[\textsf{Monoids}]
We will prove the following.

\begin{Theorem}\label{Thm : inverses monoid : implicit operation}
Taking inverses is a unary implicit operation of the variety of monoids which, moreover, can be defined by the conjunction of equations
\[
\varphi = (x \cdot y \thickapprox 1) \sqcap (y \cdot x \thickapprox 1).
\]
\end{Theorem}

\begin{proof}
Recall from Example \ref{Exa : inverses in monoids 1} that the partial function $f$ on the variety of monoids of ``taking inverses'' is implicit and defined by $\varphi$. We will show that $f$ is also an operation. To this end, consider a homomorphism $h \colon \A \to \B$ of monoids and $a \in \mathsf{dom}(f^\A)$. Then $f^\A(a) = a^{-1}$. Since monoid homomorphisms preserve inverses,
we obtain
\[
h(f^\A(a)) = h(a^{-1}) = h(a)^{-1}.
\]
Consequently, $h(a)$ has an inverse, whence $h(a) \in \mathsf{dom}(f^\B)$ and $f^\B(h(a)) = h(a)^{-1}$. Together with the above display, this yields $h(f^\A(a)) = f^\B(h(a))$.
\end{proof}
\end{exa}

\begin{exa}[\textsf{Term functions}]\label{Example : term functions}
Let $\mathsf{K}$ be a class of algebras. Every term of $\K$ can be viewed as an implicit operation, as we proceed to illustrate. Let $t(x_1, \dots, x_n)$ be a term. For each $\A \in \K$, evaluating $t$ on tuples of elements of $A$ induces a function $t^{\A} \colon A^n \to A$. Then the sequence $t^\K = \langle t^\A : \A \in \mathsf{K} \rangle$ is an $n$-ary implicit operation of $\mathsf{K}$ defined by the equation $t(x_1, \dots, x_n) \thickapprox y$. We call $t^\K$ a \emph{term function of $\mathsf{K}$}. These implicit operations are always ``total'', in the sense that each $t^\A$ is a 
total function on $A$. 
\qed
\end{exa}

In elementary classes, implicit operations admit the following description.

\begin{Theorem}\label{Thm : implicit operations vs existential positive formulas}
Let $f$ be a partial function on an elementary class $\mathsf{K}$. Then $f$ is an implicit operation of $\mathsf{K}$ if and only if it is defined by an existential positive formula. 
\end{Theorem}

\begin{proof}
To prove the implication from left to right, suppose that $f$ is an implicit operation on $\mathsf{K}$. Then there exists a formula $\varphi(x_1, \dots, x_n, y)$ that defines $f$. We will prove that $\varphi$ is preserved by homomorphisms in $\mathsf{K}$. To this end, consider a homomorphism $h \colon \A \to \B$ with $\A, \B \in \mathsf{K}$ and $a_1, \dots, a_n, b \in A$ such that $\A \vDash \varphi(a_1, \dots, a_n, b)$. As $\varphi$ defines  $f$, from $\A \vDash \varphi(a_1, \dots, a_n, b)$ it follows that $\langle a_1, \dots, a_n \rangle \in \mathsf{dom}(f^\A)$ and $f^\A(a_1, \dots, a_n) = b$. Together with the assumption that $f$ is an operation of $\mathsf{K}$, this yields $\langle h(a_1), \dots, h(a_n) \rangle \in \mathsf{dom}(f^\B)$ and 
\[
h(b) = h(f^\A(a_1, \dots, a_n)) = f^\B(h(a_1), \dots, h(a_n)).
\]
Since $\varphi$ defines $f$, we conclude that $\B \vDash \varphi(h(a_1), \dots, h(a_n), h(b))$. Hence, $\varphi$ is preserved by homomorphisms in $\mathsf{K}$. Therefore, we can apply Theorem \ref{Thm : preservation}(\ref{item : preservation : ep}), obtaining that $\varphi$ is equivalent to an existential positive formula $\psi$ in $\mathsf{K}$. As $\varphi$ defines  $f$, so does $\psi$. Thus, we conclude that $f$ is defined  by an existential positive formula.

Then we 
proceed to prove the implication from right to left. Suppose that $f$ is defined by an existential positive formula $\varphi(x_1, \dots, x_n, y)$. To conclude that $f$ is an implicit operation of $\mathsf{K}$, it suffices to show that it is an operation of $\mathsf{K}$. Consider a homomorphism $h \colon \A \to \B$ with $\A, \B \in \mathsf{K}$ and $a_1, \dots, a_n\in A$ such that $\langle a_1, \dots, a_n \rangle \in \mathsf{dom}(f^\A)$. As $\varphi$ defines $f$, we have $\A \vDash \varphi(a_1, \dots, a_n, f^\A(a_1, \dots, a_n))$. Since $\varphi$ is an existential positive formula, we can apply~\cref{Thm : preservation}(\ref{item : preservation : ep}), obtaining $\B \vDash \varphi(h(a_1), \dots, h(a_n), h(f^\A(a_1, \dots, a_n)))$. Therefore, $\langle h(a_1), \dots, h(a_n) \rangle \in \mathsf{dom}(f^\B)$ and $h(f^\A(a_1, \dots, a_n)) = f^\B(h(a_1), \dots, h(a_n))$ because $\varphi$ defines $f$. Hence, we conclude that $f$ is an operation of $\mathsf{K}$.
\end{proof}

Let $f, f_1, \dots, f_n$ be $m$-ary partial functions on a set $X$. We write $f = f_1 \cup \dots \cup f_n$ to indicate that $f$ is the union of $f_1, \dots, f_n$ when they are viewed as subsets of $X^m \times X$. This condition is equivalent to the requirements that $\dom(f) = \dom(f_1) \cup \dots \cup \dom(f_n)$ and that $f_i(x)=f(x)$ for all $i$ and $x \in \dom(f_i)$.
As a consequence of \cref{Thm : implicit operations vs existential positive formulas}, we obtain the following.

\begin{Corollary}\label{Cor : each implicit operations splits into pp formulas}
Let $f$ be an implicit operation of an elementary class $\mathsf{K}$. Then there exist some implicit operations $f_1, \dots, f_n$ of $\mathsf{K}$ defined by pp formulas such that for each $\A \in \mathsf{K}$,
\[
f^\A = f_1^\A \cup \dots \cup f_n^\A.
\]
\end{Corollary}

\begin{proof}
As $\mathsf{K}$ is an elementary class, we can apply \cref{Thm : implicit operations vs existential positive formulas}, obtaining that $f$ is defined by an existential positive formula $\varphi$. We may assume that $\varphi = \varphi_1 \sqcup \dots \sqcup \varphi_n$ for some pp formulas $\varphi_1, \dots, \varphi_n$. 
Since $\varphi$ defines $f$, it is functional in $\mathsf{K}$. Together with $\varphi = \varphi_1 \sqcup \dots \sqcup \varphi_n$, this implies that each $\varphi_i$ is functional in $\mathsf{K}$ and, therefore, defines a partial function $f_i$ on $\mathsf{K}$. As $\varphi_i$ is a pp formula, from Theorem \ref{Thm : implicit operations vs existential positive formulas} it follows that $f_i$ is an implicit operation of $\mathsf{K}$. Now, recall that $\varphi = \varphi_1 \sqcup \dots \sqcup \varphi_n$ defines $f$ and $\varphi_i$ defines $f_i$ for each $i \leq n$. Thus, we conclude that $f^\A = f_1^\A \cup \dots \cup f_n^\A$ for each $\A \in \mathsf{K}$.
\end{proof}

In view of Corollary \ref{Cor : each implicit operations splits into pp formulas}, implicit operations of elementary classes are obtained by gluing together implicit operations defined by pp formulas.\ This is the reason why the most fundamental implicit operations in mathematics are defined by pp formulas (as opposed to arbitrary existential positive formulas), as shown by the forthcoming examples.
We denote by $\imppp(\K)$ the set of implicit operations of a given class $\mathsf{K}$ that are defined by pp formulas and, when $\K = \{ \A \}$, we often write $\imppp(\A)$ instead of $\imppp(\K)$.

\begin{Corollary}\label{Cor : functionality in Q(K)}
Let $\mathsf{K}$ be a class of algebras and $\varphi$ an existential positive formula functional in $\mathsf{K}$. Then $\varphi$ defines an implicit operation of $\QQQ(\mathsf{K})$.
\end{Corollary}

\begin{proof}
Suppose that $\varphi(\vec{x}, y)$ is an existential positive formula, where $\vec{x}$ is a finite sequence of variables. Since $\varphi$ is existential positive, it is equivalent to a formula of the form
\[
\bigsqcup_{i=1}^m \exists \vec{z} \psi_i(\vec{x}, y, \vec{z}\,), 
\]
where each $\psi_i(\vec{x}, y, \vec{z})$ is a finite conjunction of equations. We rely on the following observation.

\begin{Claim}\label{Claim : functionality in Q(K) : claim}
The formula $\varphi$ is functional in a class of algebras $\mathsf{M}$ if and only if for all $i, j \leq m$ we have
\[
\mathsf{M} \vDash (\psi_i(\vec{x}, y, \vec{z}\,) \sqcap \psi_j(\vec{x}, y', \vec{z}\,')) \to y \thickapprox y',
\]
where $y'$ is a fresh variable and $\vec{z}\,'$ a sequence of fresh variables of the same length as $\vec{z}$.
\end{Claim}

\begin{proof}[Proof of the Claim]
   The functionality of $\varphi$ in $\mathsf{M}$ amounts to
\[
\mathsf{M} \vDash (\varphi(\vec{x}, y) \sqcap \psi(\vec{x}, y')) \to y \thickapprox y'.
\]
 In turn, this amounts to 
\begin{equation*}
\mathsf{M} \vDash \left( \left(\bigsqcup_{i=1}^m \exists \vec{z} \psi_i(\vec{x}, y, \vec{z}\,)\right) \sqcap \left(\bigsqcup_{i=1}^m \exists \vec{z}\,' \psi_i(\vec{x}, y', \vec{z}\,')\right)\right) \to y \thickapprox y'.
\end{equation*}
The latter amounts to 
\begin{equation*}
\mathsf{M} \vDash \left( \bigsqcup_{i,j=1}^m \exists \vec{z}, \vec{z}\,' (\psi_i(\vec{x}, y, \vec{z}\,) \sqcap \psi_j(\vec{x}, y', \vec{z}\,'))\right) \to y \thickapprox y',
\end{equation*}
which is in turn equivalent to the condition displayed in the statement.
\end{proof}

Now, suppose that $\varphi$ is functional in $\K$. By  Claim \ref{Claim : functionality in Q(K) : claim} we obtain that for all $i, j \leq m$,
\[
\mathsf{K} \vDash (\psi_i(\vec{x}, y, \vec{z}\,) \sqcap \psi_j(\vec{x}, y', \vec{z}\,')) \to y \thickapprox y'.
\]
As the formula in the above display is a quasiequation for all $i,j \leq m$, \cref{Cor: VQU and formulas}\eqref{Cor: VQU and formulas : Q} implies that it is also valid in $\QQQ(\K)$. Together with Claim \ref{Claim : functionality in Q(K) : claim}, this guarantees that $\varphi$ is functional in $\QQQ(\K)$.
\end{proof}

    Given an $n$-ary partial function $g$ and $m$-ary partial functions $f_1, \dots, f_n$ on a class $\mathsf{K}$, their \emph{composition} $g(f_1, \dots, f_n)$ is the $m$-ary partial function on $\mathsf{K}$ such that $\dom(g(f_1, \dots, f_n)^\A)$ is 
    \[
        \bigcap_{i \leq m}\mathsf{dom}(f_i^{\A}) \cap  \{\langle a_1,  \dots, a_m \rangle \in A^m : \langle f_i^{\A}(a_1, \dots, a_m) : i \leq n \rangle \in \mathsf{dom}(g^{\A})\}
    \]
    for all $\A \in \mathsf{K}$ and 
    \[g(f_1, \dots, f_n)^{\A}(a_1, \dots, a_m) = g^{\A}(f_1^{\A}(a_1, \dots, a_m), \dots, f_n^{\A}(a_1, \dots, a_m))\]
    for all $\langle a_1, \dots, a_m \rangle \in \dom(g(f_1, \dots, f_n)^\A).$

\begin{Proposition} \label{Prop : closure under composition for imp}
Let $\mathsf{K}$ be a class of algebras. Then $\mathsf{imp}(\mathsf{K})$ and $\mathsf{imp}_{\textsc{pp}}(\mathsf{K})$ are closed under composition. 
\end{Proposition}

\begin{proof}
Consider $g, f_1, \dots, f_n \in \imp(\K)$, where $g$ is $n$-ary and each $f_i$ is $m$-ary. Let $h = g(f_1, \dots, f_n)$ be their composition, which is a partial $m$-ary function on $\K$. 
To show that $h$ is an operation of $\K$,
consider a homomorphism $k \colon \A \to \B$ with $\A, \B \in \K$ and $\langle a_1, \dots, a_m \rangle \in \dom(h^\A)$. It follows from the definition of $h$ that $\langle a_1, \dots, a_m \rangle \in \dom(f_i^{\A})$ for all $i$  and $\langle f_1^{\A}(a_1, \dots, a_m), \dots, f_n^{\A}(a_1, \dots, a_m) \rangle \in \dom(g^{\A})$. As $g, f_1, \dots, f_n$ are operations of $\K$, we have that
\begin{align*}
    \langle k(a_1), \dots, k(a_m) \rangle \in \dom(f_i^{\B})
\end{align*}
for all $i$ and
\begin{align*}
    \langle f_1^{\B}(k(a_1), \dots, k(a_m)),& \dots, f_n^{\B}(k(a_1), \dots, k(a_m)) \rangle\\
    &= \langle k(f_1^{\A}(a_1, \dots, a_m)), \dots, k(f_n^{\A}(a_1, \dots, a_m)) \rangle \in \dom(g^{\B}).
\end{align*}
Then the definition of $\dom(h^\B)$ implies $\langle k(a_1), \dots, k(a_m) \rangle \in \dom(h^\B)$.
Using again the fact that $g, f_1, \dots, f_n$ are operations of $\K$ yields
\begin{align*}
    k(h^{\A}(a_1, \dots, a_m)) & = k(g^{\A}(f_1^{\A}(a_1, \dots, a_m), \dots, f_n^{\A}(a_1, \dots, a_m)))\\
    & = g^{\B}(k(f_1^{\A}(a_1, \dots, a_m)), \dots, k(f_n^{\A}(a_1, \dots, a_m)))\\
    & = g^{\B}(f_1^{\B}(k(a_1), \dots, k(a_m)), \dots, f_n^{\B}(k(a_1), \dots, k(a_m)))\\
    & = h^{\B}(k(a_1), \dots, k(a_m)).
\end{align*}
Thus, $h$ is an operation of $\K$. We now prove that $h$ is defined by a formula.
Since $g, f_1, \dots, f_n \in \imp(\K)$, 
there exist functional formulas $\psi, \varphi_1, \dots, \varphi_n$ that define $g, f_1, \dots, f_n$, respectively. Therefore, for all $\A \in \K$, $a_1, \dots, a_m, b, b_1, \dots, b_n, c \in A$, and $i \leq n$ we have
\begin{equation} \label{Eq : defining formulas 2}
     \langle a_1, \dots, a_m \rangle \in \dom(f_i^\A) \text{ and } f_i^\A(a_1, \dots, a_m) = b \iff \A \vDash \varphi_i(a_1, \dots, a_m, b)
\end{equation}
and 
\begin{equation} \label{Eq : defining formulas 1}
    \langle b_1, \dots, b_n \rangle \in \dom(g^\A) \text{ and } g^{\A}(b_1, \dots, b_n) = c  \iff \A \vDash \psi(b_1, \dots, b_n, c).
\end{equation} 
Let
\begin{equation} \label{Eq : definition varphi}
\chi(x_1, \dots, x_m,y) = \exists z_1, \dots, z_n \left(\psi(z_1, \dots, z_n,y) \sqcap \bigsqcap_{i=1}^n \varphi_i(x_1, \dots, x_m, z_i)\right).
\end{equation}
We show that $\chi$ defines $h$. Consider $\A \in \K$ and $a_1, \dots, a_m, c \in A$. By \eqref{Eq : definition varphi} we have $\A \vDash \chi(a_1, \dots, a_m,c)$ if and only if there exist $b_1, \dots, b_n \in A$ such that $\A \vDash \varphi_i(a_1, \dots, a_m, b_i)$ for all $i \leq n$ and $\A \vDash \psi(b_1, \dots, b_n, c)$.
By \eqref{Eq : defining formulas 2} and \eqref{Eq : defining formulas 1}, the latter condition is equivalent to
\begin{align*}
\langle a_1, \dots, a_m \rangle &\in \dom(f_i^\A) \text{ and } f_i^\A(a_1, \dots, a_m) = b_i \text{ for every } i \leq n, \text{ and}\\
\langle b_1, \dots, b_n \rangle &\in \dom(g^\A) \text{ and } g^{\A}(b_1, \dots, b_n) = c.
\end{align*}
In turn, this amounts to $\langle a_1, \dots, a_m \rangle \in \dom(h^\A)$ and $h^\A(a_1, \dots, a_m) = c$ by the definition of $h$.
Therefore, $\chi$ defines $h$, and hence $h \in \imp(\K)$.

Suppose now that in addition $g, f_1, \dots, f_n \in \imppp(\K)$. We can then assume that the formulas $\psi, \varphi_1, \dots, \varphi_n$ that define $g, f_1, \dots, f_n$ are pp formulas. Let $\chi$ be defined as in \eqref{Eq : definition varphi}. 
Since $\psi$ and each $\varphi_i$ are pp formulas, they are of the form $\exists \vec{v} \psi'$ and $\exists \vec{u}_i  \varphi_i'$, where $\vec{v}$ and $\vec{u}_i$ are finite sequences of variables and $\psi'$ and $\varphi_i'$ are finite conjunctions of equations. It is straightforward to verify that $\chi$ is equivalent to the formula obtained by pulling out the existential quantifiers $\exists \vec{v}$ and $\exists \vec{u}_i$ from the conjunction in \eqref{Eq : definition varphi},
and hence $\chi$ is equivalent to a pp formula. Thus, $h$ is defined by a pp formula, and so $h \in \imppp(\K)$.
\end{proof}

\begin{exa}[\textsf{Isbell's operations}]\label{Exa : Isbell operations}
A fundamental example of implicit operations of the variety of monoids is due to Isbell \cite{Isb65}\footnote{While these operations are traditionally considered in the variety of semigroups, it is straightforward to verify that all their properties relevant to our discussion continue to hold in the variety of monoids.}.
More precisely, for each $n \geq 1$ let
\[
\psi_n(z_1, \dots, z_n, w_1, \dots, w_n, x_1, \dots, x_{2n+1}, y)
\]
be the conjunction of the following equations in the language of monoids:
\begin{align*}
y &\thickapprox x_1 z_1\\
x_1 & \thickapprox w_1 x_2\\
x_{2i}z_i & \thickapprox x_{2i+1}z_{i+1} \text{ for }i = 1, \dots, n-1\\
w_i x_{2i+1} & \thickapprox w_{i+1}x_{2(i+1)} \text{ for }i = 1, \dots, n-1\\
x_{2n}z_n & \thickapprox x_{2n+1}\\
w_n x_{2n+1}& \thickapprox y.
\end{align*}
Then let $\varphi_0(x,y) = x \thickapprox y$ and for each $n \geq 1$,
\[
\varphi_n(x_1, \dots, x_{2n+1}, y) = \exists z_1, \dots, z_n, w_1, \dots, w_n \psi_n(z_1, \dots, z_n, w_1, \dots, w_n, x_1, \dots, x_{2n+1}, y).
\]
We refer to $\varphi_n$ as to the $n$-th \emph{Isbell's formula}. 
Notice that Isbell's formulas are pp formulas. It follows from \cite[Lem.\ 4.4]{Camper18jsl} that each Isbell's formula is functional in the variety of monoids. Whence,
from \cref{Cor : functionality in Q(K)} we deduce the following.

\begin{Theorem}\label{Thm : Isbell implicit operations}
Every Isbell's formula defines an implicit operation of the variety of monoids.
\end{Theorem}

\noindent Isbell's formulas and the implicit operations they define, which we term \emph{Isbell's operations},  will play a prominent role in the next sections (see 
\cref{ex:dominions monoids}, Isbell's Zigzag Theorem \ref{Thm : Isbell Theorem}, and \cref{Prop: property unary extpp in Mon}).
\qed
\end{exa}

\begin{exa}[\textsf{Reduced commutative rings}]\label{Example : inverses in rings}
Throughout this work, rings will be assumed to possess an identity element. Given an element $a$ of a ring $\langle A; +, \cdot, -, 0, 1\rangle$, we denote its multiplicative inverse (when it exists) by $a^{-1}$. By a \emph{field} we understand a commutative ring $\A$ with $0 \ne 1$ such that $a^{-1}$ exists for each $a \in A - \{ 0 \}$. The class of fields will be denoted by $\mathsf{Field}$.
    
A commutative ring $\A$ is said to be \emph{reduced}  when for each $a \in A$,
\[
a \cdot a = 0 \text{ implies }a = 0
\]
(see, e.g., \cite{MR1322960}). The class of reduced commutative rings forms a quasivariety $\mathsf{RCRing}$ axiomatized relative to commutative rings by the quasiequation $x \cdot x \thickapprox 0 \to x \thickapprox 0$. We remark that this quasivariety is proper because the ring of integers $\mathbb{Z}$ is reduced, while its quotient $\mathbb{Z}_4$ is not. We rely on the next characterization of reduced commutative rings.

\begin{Theorem}\label{Thm : reduced rings are Q fields}
The class of reduced commutative rings coincides with the quasivariety generated by all fields.
\end{Theorem}

\begin{proof}
We will prove that
\[
\mathsf{RCRing} = \III\SSS\PPP(\mathsf{Field}) = \III\SSS\PPP\PPU(\mathsf{Field}) = \QQQ(\mathsf{Field}).
\]
The first equality above holds because a commutative ring is reduced if and only if it embeds into a direct product of fields (see, e.g., \cite[Prop.~3.1]{Mat83}),
the second because $\mathsf{Field}$ is an elementary class and, therefore, closed under $\PPU$, and the third follows from Theorem \ref{Thm : quasivariety generation}.
\end{proof}

Let $\A$ be a field and $a \in A$. The \emph{weak inverse} of $a$ in $\A$ is the element
\begin{equation}\label{eq: to}
    \mathsf{wi}(a) =
    \begin{cases}
        a^{-1}  & \text{if } a \ne 0;\\
        0 & \text{if } a = 0.\\
    \end{cases}        
    \end{equation}
We will prove the following.

\begin{Theorem}\label{Thm : RCRing : implicit operation}
There exists a unary implicit operation $f$ of the quasivariety of reduced commutative rings
such that $f^\A$ is total and $f^\A(a) = \mathsf{wi}(a)$ for all fields $\A$ and $a \in A$.\   Moreover, $f$ can be defined by the conjunction of equations
\[
\varphi = (x^2y \thickapprox x) \sqcap ( xy^2 \thickapprox y).
\]
\end{Theorem}

\begin{proof}
We will prove that for every field $\A$ and $a, b \in A$ we have
\begin{equation}\label{Eq : weak inverses in fields}
\A \vDash \varphi(a, b) \iff b = \mathsf{wi}(a).
\end{equation}
The implication from right to left is straightforward. To prove the reverse implication, suppose that $\A \vDash \varphi(a, b)$, i.e., 
\[
a^2b = a \, \, \text{ and } \, \, ab^2 = b.
\]
We have two cases: either $a \ne 0$ or $a = 0$. First, suppose that $a \ne 0$. Then $\mathsf{wi}(a) = a^{-1}$. Therefore, from 
$a^2b = a$
it follows that
\[
b=a^{-2} a^2b=a^{-2}a=a^{-1},
\]
where $a^{-2}$ abbreviates $(a^{-1})^2$.
Whence $b = a^{-1} = \mathsf{wi}(a)$. Then we consider the case where $a = 0$. In this case, $\mathsf{wi}(a) = 0$.  From $a = 0$ and 
$ab^2=b$
it follows that $b = 0 = \mathsf{wi}(a)$. This establishes (\ref{Eq : weak inverses in fields}). Consequently, $\varphi$ is functional in $\mathsf{Field}$.

Recall from \cref{Thm : reduced rings are Q fields} that $\QQQ(\mathsf{Field})= \mathsf{RCRing}$. As $\varphi$ is a pp formula and is functional in $\mathsf{Field}$, we can apply \cref{Cor : functionality in Q(K)}, obtaining that $\varphi$ defines an implicit operation $f$ of $\mathsf{RCRing}$. Lastly, (\ref{Eq : weak inverses in fields}) ensures that $f^\A(a) = \mathsf{wi}(a)$ for every field $\A$ and $a \in A$. 
\end{proof}
\end{exa}

\begin{exa}[\textsf{Distributive lattices}]\label{Example : complements in lattices}
Given a lattice $\A$ and $b, c \in A$, we let
\[
[b, c] = \{ a \in A : b \leq a \leq c \}.
\]
Moreover, given $a, d \in A$, we say that $d$ is a \emph{complement} of $a$ \emph{relative} to the interval $[b, c]$ when
\[
a \land d = b \, \, \text{ and } \, \, a \lor d = c.
\]

In distributive lattices, relative complements are unique when they exist \cite[Cor.\ IX.1]{BirkLT}. Consequently, with every distributive lattice $\A$ we can associate a ternary partial function $f^\A$ on $A$ with domain
\[
\mathsf{dom}(f^\A)  = \{ \langle a, b, c\rangle \in A^3 : \text{$a$ has a complement relative to $[a \land b\land c, a \lor b \lor c]$ in }\A \text\},
\] 
defined for each $\langle a, b, c\rangle \in \mathsf{dom}(f^\A)$ as
\[
f^\A(a, b, c) = \text{the complement of $a$ relative to }[a \land b\land c, a \lor b \lor c] \text{ in }\A.
\]
Let $\mathsf{DL}$ be the variety of distributive lattices. Then the sequence $f = \langle f^\A : \A \in \mathsf{DL}\rangle$ is a partial function on $\mathsf{DL}$, which captures the idea of ``taking relative complements''. 

This construction acquires special interest in the case of \emph{bounded} distributive lattices. For let $\A = \langle A; \land, \lor, 0, 1\rangle$ be a bounded distributive lattice and $a, b \in A$. Then $b$ is said to be a \emph{complement} of $a$ when
\[
a \land b = 0 \, \, \text{ and } \, \, a \lor b = 1
\]
or, equivalently, when $b$ is a complement of $a$ relative to $[0, 1] = A$. With every bounded distributive lattice $\A$ we can associate a unary partial function $f^\A$ on $A$ with domain
\[
\mathsf{dom}(f^\A)  = \{ a \in A : \text{$a$ has a complement in }\A \text\},
\] 
defined for each $a \in \mathsf{dom}(f^\A)$ as
\[
f^\A(a) = \text{the complement of $a$ in }\A.
\]
Let $\mathsf{bDL}$ be the variety of bounded distributive lattices. Then the sequence $f = \langle f^\A : \A \in \mathsf{bDL}\rangle$ is a partial function on $\mathsf{bDL}$, which captures the idea of ``taking complements''. 

\begin{Theorem}\label{Thm : relative complements is implicit}
The following conditions hold:
\benroman
\item\label{item : relative complements is implicit : 1} taking relative complements is a ternary implicit operation of the variety of distributive lattices which, moreover, can be defined by the conjunction of equations
\[
\varphi = (x_1 \land y \thickapprox  x_1 \land x_2 \land x_3) \sqcap ( x_1 \lor y \thickapprox x_1 \lor x_2 \lor x_3);
\]
\item\label{item : relative complements is implicit : 2} taking complements is a unary implicit operation of the variety of bounded distributive lattices which, moreover, can be defined by the conjunction of equations
\[
\psi =  (x \land y \thickapprox  0) \sqcap ( x \lor y \thickapprox 1).
\]
\eroman
\end{Theorem}

\begin{proof}
(\ref{item : relative complements is implicit : 1}): Observe that the partial function $f$ on $\mathsf{DL}$ of ``taking relative complements'' can be defined by the conjunction of equations
\[
\varphi(x_1, x_2, x_3, y) = (x_1 \land y \thickapprox  x_1 \land x_2 \land x_3) \sqcap ( x_1 \lor y \thickapprox x_1 \lor x_2 \lor x_3).
\]
Therefore, from Theorem \ref{Thm : implicit operations vs existential positive formulas} it follows that $f$ is an implicit operation of $\mathsf{DL}$.

(\ref{item : relative complements is implicit : 2}): Analogous to the proof of (\ref{item : relative complements is implicit : 1}).
\end{proof}
\end{exa}

\begin{exa}[\textsf{Absolute value}]\label{Example : absolute value}
Let $\varphi$ be the pp formula
\[
\varphi(x,y) = \exists z_1, z_2, z_3, z_4 ((y \thickapprox z_1^2 + z_2^2 + z_3^2 + z_4^2) \sqcap (x^2 \thickapprox y^2))
\]
in the language of rings. 
By Lagrange's four squares theorem any nonnegative integer can be written as the sum of four integer squares (see, e.g., \cite[Thm.~11-3]{And94}). Therefore, for all $a,b \in \mathbb{Z}$ we have that $\mathbb{Z} \vDash \varphi(a,b)$ if and only if $b \geq 0$ and $a^2=b^2$, which happens exactly when $b=\lvert a \rvert$. In particular, $\varphi$ is functional in $\mathbb{Z}$. Then \cref{Cor : functionality in Q(K)} implies that $\varphi$ defines an implicit operation $f$ of the quasivariety $\K$ of rings generated by $\mathbb{Z}$ such that $f^\mathbb{Z}$ is the absolute value function.

While $f$ is an implicit operation defined by a pp formula, it is interesting to observe that $f$ cannot be defined by a conjunction of equations. Indeed, suppose, on the contrary, that $f$ is defined on $\K$ by a conjunction of equations $\psi$. In the variety of commutative rings, each equation in variables $x$ and $y$ is equivalent to an equation of the form $p(x,y)\thickapprox 0$, where $p(x,y)$ is a polynomial with integer coefficients. So, we can assume that 
\[
\psi = \bigsqcap_{i=1}^n p_i(x,y)\thickapprox 0,
\]
where each $p_i(x,y)$ is a polynomial with integer coefficients.
Since $\psi$ defines the absolute value function on $\mathbb{Z}$, we have that $\mathbb{Z} \vDash \psi(a,a)$ for every nonnegative integer $a$.
Thus, for every $i$ the polynomial $p_i(x,x)$ in a single variable $x$ vanishes on every nonnegative integer. We recall that the only polynomial in a single variable with rational coefficients that has infinitely many roots is the zero polynomial (see, e.g., \cite[Prop.~12.2.20]{Art10}). Then $p_i(x,x)$ is the zero polynomial, and hence $p_i(-1,-1)=0$ for every $i$. We conclude that $\mathbb{Z} \vDash \psi(-1,-1)$, which contradicts that $f^\mathbb{Z}(-1)=\lvert -1 \rvert = 1$. Therefore, $f$ cannot be defined by a conjunction of equations.
\end{exa}

\section{Existential elimination}\label{Sec : existential elimination}

The idea of interpolating a given family of functions by simpler ones plays a fundamental role in mathematics. For instance, a well-known theorem of Lagrange states that every finite set of pairs of real numbers can be interpolated by a polynomial with real coefficients (see, e.g., \cite[Thm.~6.1]{SM03}). In this section, we will establish a general  interpolation theorem for the implicit operations of a quasivariety $\K$. 

More precisely, recall from Corollary \ref{Cor : each implicit operations splits into pp formulas} that every implicit operation of $\K$ can be obtained by gluing together finitely many implicit operations defined by pp formulas.\ 
The main result of this section states that, if $\K$ has the amalgamation property, the study of its implicit operations can be further simplified by observing that  each implicit operation defined by a pp formula is interpolated by one defined by a conjunction of equations (\cref{Thm : existential elimination}). We term this  phenomenon \emph{existential elimination} because conjunctions of equations are obtained by removing existential quantifiers from pp formulas. 

As we mentioned, the reason for existential elimination is the amalgamation property, whose definition we proceed to recall.

\begin{Definition}
Given a class $\K$ of similar algebras, we say that
\benroman
\item a tuple
$\langle \A, \B, \C, h_1, h_2 \rangle$ is a \emph{span}
in $\K$ when $h_1 \colon \A \to \B$ and $h_2 \colon \A \to \C$ is a pair of embeddings with $\A, \B, \C \in \K$;
\item a span
$\langle \A, \B, \C, h_1, h_2 \rangle$ in $\K$ 
has an \emph{amalgam}
in $\K$ when there exists a pair of embeddings $g_1 \colon \B \to \D$ and $g_2 \colon \C \to \D$ with $\D \in \K$ such that $g_1 \circ h_1 = g_2 \circ h_2$;
     
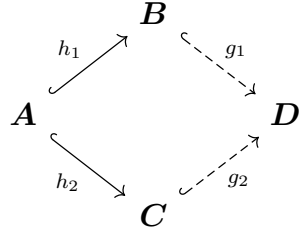
\begin{center}
\begin{tikzcd}
& \B \arrow[dr, dashed, hook, " g_1"] &  \\
\A  \arrow[ur, hook, "h_1"] \arrow[dr, hook, swap, "h_2 "] &  & \D \\
& \C \arrow[ur, dashed, hook, swap, " g_2"] &
\end{tikzcd}
\end{center}  
\item a member $\A$ of $\K$ is an \emph{amalgamation base} for $\K$ when every span
in $\K$ of the form $\langle \A, \B, \C, h_1, h_2 \rangle$ 
has an amalgam
in $\K$;
\item $\K$ has the \emph{amalgamation property} when every 
span
in $\K$ 
has an amalgam
in $\K$.
\eroman 
\end{Definition}

Furthermore, we will rely on the following notion of interpolation.

\begin{Definition}\label{Def: interpolation}
Let $\mathcal{F} \cup \{ g \}$ be a family of $n$-ary implicit operations of a class of algebras $\K$. We say that $g$ is \emph{interpolated} by $\mathcal{F}$ when for all $\A \in \mathsf{K}$ and $\langle a_1, \dots, a_n \rangle \in \mathsf{dom}(g^\A)$ there exists $f \in \mathcal{F}$ such that
\[
\langle a_1, \dots, a_n \rangle \in \mathsf{dom}(f^\A) \, \, \text{ and } \, \, f^\A(a_1, \dots, a_n) = g^\A(a_1, \dots, a_n).
\]
\noindent When $\mathcal{F} = \{ f \}$, we often say that $g$ \emph{is interpolated by} $f$. \end{Definition}

Given a class of algebras $\K$, we denote by $\impeq(\K)$ the set of implicit operations of $\K$ defined by a conjunction of equations. When $\K = \{ \A \}$, we often write $\impeq(\A)$ instead of $\impeq(\K)$.\footnote{Although we will not rely on this fact, we 
 will show in \cref{exa:ccmon AP}
that $\mathsf{imp}_{\textsc{eq}}(\K)$ need not be 
closed under composition 
(cf.\ Proposition \ref{Prop : closure under composition for imp}).
} The aim of this section is to establish the following interpolation result.

\begin{Theorem}\label{Thm : existential elimination}
    The following conditions hold for a quasivariety $\K$ with the amalgamation property:
    \benroman
    \item\label{item : existential elimination : 1} every member of $\mathsf{imp}_{\textsc{pp}}(\mathsf{K})$ can be interpolated by some member of $\mathsf{imp}_{\textsc{eq}}(\mathsf{K})$;
    \item\label{item : existential elimination : 2} every member of $\mathsf{imp}(\mathsf{K})$ can be interpolated by a finite subset of $\mathsf{imp}_{\textsc{eq}}(\mathsf{K})$.
    \eroman
\end{Theorem}

As the variety of monoids lacks the amalgamation property (see, e.g., \cite[p.~100]{SurvKissal}\footnote{As observed in \cite[p.~108]{SurvKissal}, the failure of the amalgamation property for the variety of monoids can be seen as a consequence of the corresponding result for the variety of semigroups proved in \cite{Kim57} (see also \cite[Exa.~1]{Cla83}). 
}), it  falls outside the scope of Theorem \ref{Thm : existential elimination}. This is reflected by the fact that this variety possesses implicit operations defined by pp formulas that cannot be interpolated by any implicit operation defined by a conjunction of equations, an example being  every $n$-th Isbell's operation for $n \geq 1$ (see \cref{Exa : Isbell is not interolable by eq : existential elimination}). On the other hand, the variety of distributive lattices has the amalgamation property (see, e.g., \cite[Thm.~VII.8.4]{BD74}) and, therefore, each of its implicit operations defined by pp formulas can be interpolated by one  defined by a conjunction of equations. 

The rest of this section is devoted to the proof of Theorem \ref{Thm : existential elimination}. The first ingredient of the proof is the following concept, introduced in \cite{Isb65} (see also \cite{Bacsich47}).

\begin{Definition}
Let $\mathsf{K}$ be a class of algebras and $\A \leq \B$ a pair of $\L_\K$-algebras. The \emph{dominion} of $\A$ in $\B$ relative to $\mathsf{K}$ is the set
\begin{align*}
\mathsf{d}_\mathsf{K}(\A, \B) = \{ b \in B : &\text{ for each pair of homomorphisms }g, h \colon \B \to \C \text{ with }\C \in \mathsf{K},\\
& \text{ if }g{\upharpoonright}_A = h{\upharpoonright}_A, \text{ then }g(b) = h(b) \}.
\end{align*}
\end{Definition}

It is straightforward to verify that $\mathsf{d}_\mathsf{K}(\A, \B)$ is the universe of a subalgebra of $\B$ that contains $A$. It follows immediately from its definition that $\mathsf{d}_\mathsf{K}(\A, \B)$ is the intersection of all the equalizers of pairs of homomorphisms $g,h \colon \B \to \C$ with $\C \in \K$ that agree on $A$, where we recall that the equalizer of $g$ and $h$ is $\{b \in B : g(b)=h(b)\}$. Moreover, when $\K$ is closed under direct products, it turns out that $\mathsf{d}_\mathsf{K}(\A, \B)$ is
itself the
equalizer of a pair of homomorphisms from $\B$ into an algebra of $\K$ that agree on $A$.

We rely on the following fact.

\begin{Proposition}\label{Prop: dominions and homs}
Let $\K$ be a class of algebras,
$\A \leq \B \in \K$, and $\A' \leq \B' \in \K$.
If $g \colon \B \to \B'$ is a homomorphism with $g[A] \subseteq A'$, then  $g[\d_\K(\A, \B)] \subseteq \d_\K(\A', \B')$.
\end{Proposition}

\begin{proof}
Let $b \in g[\d_\K(\A, \B)]$. Then there exists $a \in \d_\K(\A, \B)$ with $g(a)=b$. Suppose that $h_1,h_2 \colon \B' \to \C$ are homomorphisms with $\C \in \K$ and $h_1{\upharpoonright}_{A'}=h_2{\upharpoonright}_{A'}$. Since $g[A] \subseteq A'$, the homomorphisms $h_1\circ g,h_2\circ g \colon \B \to \C$ satisfy $(h_1\circ g){\upharpoonright}_{A}=(h_2\circ g){\upharpoonright}_{A}$. As $a \in \d_\K(\A, \B)$, it follows that $h_1(g(a))=h_2(g(a))$. Therefore, $b=g(a) \in  \d_\K(\A', \B')$, as desired.
\end{proof}

As an immediate consequence of the previous proposition we obtain the following result, where, for $\A \leq \B$ and $\theta \in \Con(\B)$, we denote the subalgebra of $\B/\theta$ with universe $\{ a/\theta : a \in A\}$ by $\A/\theta$.

\begin{Corollary}\label{Cor: dominions subalg quot}
The following conditions hold for every class $\K$ of algebras and $\A \leq \B \in \K$.
\benroman
\item\label{Cor: dominions subalg quot: 1} If $\A \leq \A' \leq \B' \in \K$ and $\B \leq \B'$, then $\d_\K(\A, \B) \subseteq \d_\K(\A', \B')$. 
\item\label{Cor: dominions subalg quot: 2} If $\theta \in \Con(\B)$ and $\B/\theta \in \K$, then $b \in \d_\K(\A, \B)$ implies $b/\theta \in \d_\K(\A / \theta, \B / \theta)$.
\eroman 
\end{Corollary}

\begin{proof}
Both statements follow from \cref{Prop: dominions and homs}: for \eqref{Cor: dominions subalg quot: 1} let $g \colon \B \to \B'$ be the inclusion map, and for \eqref{Cor: dominions subalg quot: 2} let $g \colon \B \to \B/\theta$ be the canonical surjection.
\end{proof}

In general, the task of describing dominions for concrete classes of algebras may be hard. However, in some cases a tangible description is within reach.

\begin{exa}[\textsf{Distributive lattices}]
In the variety $\mathsf{DL}$ of distributive lattices dominions can be described as follows 
(see \cite[Thm.\ 2.4]{WmEpiL}).
For each $\A \leq \B \in \mathsf{DL}$ the dominion $\mathsf{d}_\mathsf{DL}(\A, \B)$ is 
the least subset $C$ of $B$ containing $A$ and closed under meets and joins such that for all $a, b, c \in C$ and $d \in B$,
\[
\pushQED{\qed}\text{if $d$ is the complement of $a$ relative to $[b, c]$, then } d \in C.\qedhere \popQED
\]
\end{exa}

\begin{exa}[\textsf{Monoids}]\label{ex:dominions monoids}
For each $n \in \mathbb{N}$ let $\varphi_n(x_1, \dots, x_{2n+1}, y)$ be the $n$-th Isbell's formula defined in Example \ref{Exa : Isbell operations}. Dominions in the varieties of monoids and commutative monoids are described by the following classic result
(see \cite[Thm.\ 1.2]{HowZigzag}).

\begin{Isbell Theorem}\label{Thm : Isbell Theorem}
Let $\mathsf{K}$ be the variety of monoids or the variety of commutative monoids. For each $\A \leq \B \in \mathsf{K}$ and $b \in B$ we have
\[
b \in \mathsf{d}_\mathsf{K}(\A, \B) \iff  \B \vDash \varphi_n(a_1, \dots, a_{2n+1}, b) \text{ for some $n \in \mathbb{N}$ and $a_1, \dots, a_{2n+1} \in A$}.
\]
\end{Isbell Theorem}

This theorem was originally stated for the variety of semigroups in 
\cite{Isb65}. Similar descriptions of dominions have been obtained for the varieties of commutative semigroups, rings, and commutative rings (see \cite{HoIsEpiII,IsbEpiIV}).
\qed
\end{exa}

We will make use of the following description of dominions in terms of implicit operations (see \cite[Thm.~1]{Bacsich47} and \cite[Thm.~3.2]{Camper18jsl}).

\begin{Theorem}\label{Thm : dominions : pp formulas}
Let $\mathsf{K}$ be an elementary class. For every $\A \leq \B \in \mathsf{K}$ we have
\begin{align*}
\mathsf{d}_\mathsf{K}(\A, \B) = \{ b \in B : & \text{ there exist } f \in \imppp(\K) \text{ and }\langle a_1, \dots, a_n \rangle \in \mathsf{dom}(f^\B) \cap A^n\\
&\text{ such that }f^\B(a_1, \dots, a_n) = b \}.
\end{align*}
\end{Theorem}

As shown in the next result, dominions in amalgamation bases are  especially well behaved.

\begin{Proposition}
     \label{Prop : doms computable in subalgebras}
   Let $\mathsf{K}$ be a class of algebras closed under finite direct products
   and $\A \leq \B \leq \C$ with $\B, \C \in \K$.  If $\B$ is an amalgamation base for $\mathsf{K}$,  then 
   \[
   \mathsf{d}_{\mathsf{K}}(\A,\B) = \mathsf{d}_{\mathsf{K}}(\A,\C) \cap B.
   \]
\end{Proposition}

\begin{proof}
We first prove the inclusion from left to right. As $\B \leq \C$, from \cref{Cor: dominions subalg quot}\eqref{Cor: dominions subalg quot: 1} it follows that $\mathsf{d}_{\mathsf{K}}(\A,\B) \subseteq \mathsf{d}_{\mathsf{K}}(\A,\C)$. Since $\mathsf{d}_{\mathsf{K}}(\A,\B) \subseteq B$ by definition, we obtain that $\mathsf{d}_{\mathsf{K}}(\A,\B) \subseteq \mathsf{d}_{\mathsf{K}}(\A,\C) \cap B$.

Next we prove the inclusion from right to left. Suppose, with a view to contradiction, that this inclusion fails. Then there exists $b \in B$ such that
\begin{equation} \label{Eq : doms assumption}
    b \in \mathsf{d}_{\mathsf{K}}(\A, \C) \, \,  \text{ and }\, \, b \notin \mathsf{d}_{\mathsf{K}}(\A, \B).
\end{equation}
As $b \in B$, the right hand side of the above display implies that there exists a pair of homomorphisms $f_1, f_2 \colon \B \to \D$ with $\D \in \mathsf{K}$ such that 
\begin{equation} \label{Eq : different homs from B}
f_1 {\upharpoonright}_A = f_2 {\upharpoonright}_A \, \, \text{ and  }\, \,  f_1(b) \neq  f_2(b).
\end{equation}

We may assume that $f_1$ and $f_2$ are embeddings. Otherwise, we replace each $f_i$ by the embedding $f_i^* \colon \B \to \D \times \B$ defined as $f_i^* (c) = \langle f_i(c), c\rangle$ for every $c \in B$. Observe that $\D \times \B \in \K$ because $\B, \D \in \K$ and $\K$ is closed under finite direct products by assumption. Furthermore, from (\ref{Eq : different homs from B}) and the definition of $f_1^*$ and $f_2^*$ it follows that $f_1^* {\upharpoonright}_A = f_2^* {\upharpoonright}_A$ and $ f_1^*(b) \neq  f_2^*(b)$.  Consequently, from now on we will assume that $f_1$ and $f_2$ are embeddings.

Recall from the assumptions that $\B \leq \C$. Then let $i \colon \B \to \C$ be the inclusion map, which is always an embedding. As $i \colon \B \to \C$ and $f_1 \colon \B \to \D$ are a pair of embeddings with  $\C, \D \in \mathsf{K}$, we can apply the assumption that $\B$ is an amalgamation base for $\mathsf{K}$, 
obtaining a pair of embeddings  $g_1 \colon \C \to \boldsymbol{E}$ and $g_2 \colon \D \to \boldsymbol{E}$ with $\boldsymbol{E} \in \K$ such that 
        \begin{equation} \label{Eq : AP eq1}
            g_1 \circ i = g_2 \circ f_1.
        \end{equation}
        
        
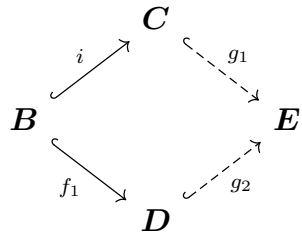
\begin{center}
\begin{tikzcd}
& \C \arrow[dr,hook, dashed, "g_1"] &  \\
\B  \arrow[ur, hook, "i"] \arrow[dr, hook, swap, "f_1"] & & \boldsymbol{E} \\
& \boldsymbol{D}  \arrow[ur, hook, dashed, swap, "g_2"] & 
\end{tikzcd}
\end{center}

Since $i, g_1, g_2$, and $f_2$ are embeddings, so are the compositions $g_1 \circ i \colon \B \to \boldsymbol{E}$ and $g_2 \circ f_2 \colon \B \to \boldsymbol{E}$. Together with $\boldsymbol{E} \in \K$,  we obtain another span 
$\langle \B,\boldsymbol{E}, \boldsymbol{E}, g_1 \circ i, g_2 \circ f_2 \rangle$ in $\mathsf{K}$. 
The assumption that $\B$ is an amalgamation base for $\mathsf{K}$ yields 
a pair of embeddings $h_1, h_2 \colon \boldsymbol{E} \to \boldsymbol{F}$  with $\boldsymbol{F} \in \K$
        such that 
        \begin{equation} \label{Eq : AP eq2}
            h_1 \circ g_1 \circ i = h_2 \circ g_2 \circ f_2.
        \end{equation}
\begin{center}
\begin{tikzcd}
& \C  \arrow[r,hook, "g_1"] & \boldsymbol{E} \arrow[dr,hook, dashed, "h_1"] \\
\B  \arrow[ur, hook, "i"] \arrow[dr, hook, swap, "f_2"]  & & & \boldsymbol{F}  \\
& \D \arrow[r, hook, swap, "g_2"] & \boldsymbol{E} \arrow[ur,hook, swap, dashed, "h_2"] 
\end{tikzcd}
\end{center}
   
As $i$ is the inclusion map from $\B$ to $\C$, from (\ref{Eq : AP eq2}) and (\ref{Eq : AP eq1}) it follows that for each $c \in B \subseteq C$,
\begin{equation}\label{Eq : last AP display in the dominion proof}
    h_1 \circ g_1(c) = h_2 \circ g_2 \circ f_2(c) \, \, \text{ and }\, \, h_2 \circ g_1(c) = h_2 \circ g_2 \circ f_1(c).
\end{equation}
By the left hand side of \eqref{Eq : different homs from B} we have $f_1(a) = f_2(a)$ for every $a \in A$. Together with (\ref{Eq : last AP display in the dominion proof}) and $A \subseteq B$, this yields that for every $a \in A$,
\[
h_1 \circ g_1(a) = h_2 \circ g_2 \circ f_2(a) = h_2 \circ g_2 \circ f_1(a) = h_2 \circ g_1(a).
\]
Hence, $(h_1 \circ g_1) {\upharpoonright}_A = (h_2 \circ g_1) {\upharpoonright}_A$. On the other hand, recall that $f_1(b) \ne f_2(b)$ by the right hand side of  \eqref{Eq : different homs from B}. Since $h_2 \circ g_2 \colon \C \to \boldsymbol{F}$ is an embedding (because so are $h_2$ and $g_2$), we obtain $h_2 \circ g_2\circ f_1(b) \ne h_2 \circ g_2\circ f_2(b)$. Together with $b \in B$ and   (\ref{Eq : last AP display in the dominion proof}), this implies $h_1 \circ g_1(b) \ne h_2 \circ g_1(b)$. Since $h_2\circ g_1 \colon \C \to \boldsymbol{F}$ is a homomorphism with $\boldsymbol{F} \in \K$ such that  $(h_1 \circ g_1) {\upharpoonright}_A = (h_2 \circ g_1) {\upharpoonright}_A$, we conclude that $b \notin \mathsf{d}_\K(\A, \C)$, a contradiction with the left hand side of (\ref{Eq : doms assumption}).
\end{proof}

The second ingredient of the proof of Theorem \ref{Thm : existential elimination} is the following construction, which associates an algebra with every pp formula.
We denote the set of variables occurring in a formula $\varphi$ by $Var(\varphi)$. For instance, if $\varphi = \exists x ( x + y \thickapprox x)$, then $Var(\varphi) = \{ x, y \}$. Moreover, we denote the term algebra with variables in $Var(\varphi)$ by  $\boldsymbol{T}(Var(\varphi))$ and let
\[
\ulcorner\varphi\urcorner = \{ \langle t_1, t_2 \rangle : t_1 \thickapprox t_2 \text{ is an equation occurring in }\varphi \}. 
\]
Observe that $\ulcorner\varphi\urcorner \subseteq T(Var(\varphi)) \times T(Var(\varphi))$.

\begin{Definition}\label{Def: TK}
Let $\K$ be a quasivariety. With every pp formula $\varphi$ we associate the algebra 
\[
\boldsymbol{T}_\mathsf{K}(\varphi) = \boldsymbol{T}(Var(\varphi)) / \theta(\varphi), \text{ where }\theta(\varphi) = \mathsf{Cg}_{\mathsf{K}}^{\boldsymbol{T}(Var(\varphi))}(\ulcorner\varphi\urcorner).
\]
\end{Definition}

Recall that, when $\varphi$ defines an implicit operation of $\mathsf{K}$, we denote this operation by
\[
\varphi^{\mathsf{K}} = \langle \varphi^\A : \A \in \mathsf{K}\rangle.
\]

We rely on the next observation.

\begin{Proposition}\label{Prop : finitely presented}
Let $\varphi(x_1, \dots, x_n, y)$ be a pp formula that defines an implicit operation of a quasivariety $\mathsf{K}$. Then the following conditions hold:
\benroman
\item\label{item : finitely presentable : 1} $\boldsymbol{T}_\mathsf{K}(\varphi)$ is a finitely presented member of $\mathsf{K}$ such that
\[
\langle x_1 / \theta(\varphi), \dots, x_n / \theta(\varphi) \rangle \in \mathsf{dom}(\varphi^{\boldsymbol{T}_\mathsf{K}(\varphi)})\, \, \text{ and } \, \, \varphi^{\boldsymbol{T}_\mathsf{K}(\varphi)}(x_1 / \theta(\varphi), \dots, x_n / \theta(\varphi)) = y / \theta(\varphi);
\]
\item\label{item : finitely presentable : 2} for all $\A \in \mathsf{K}$ and $\langle a_1, \dots, a_n \rangle \in \mathsf{dom}(\varphi^\A)$ there exists a homomorphism $h \colon \boldsymbol{T}_\mathsf{K}(\varphi) \to \A$ such that
\[
h(x_i / \theta(\varphi)) = a_i \text{ for each }i \leq n \, \, \text{ and } \, \, h(y / \theta(\varphi)) = \varphi^\A(a_1, \dots, a_n). 
\]
\eroman
\end{Proposition}

\begin{proof}
Since $\varphi$ is a pp formula, it is of the form $\exists z_1, \dots, z_m \psi$, where $\psi$ is a finite conjunction of equations. Therefore,
\begin{equation}\label{Eq : fin pres 0}
\varphi = \exists z_1, \dots, z_m \bigsqcap_{i \leq k} t_i \thickapprox s_i,
\end{equation}
where 
$t_i$ and $s_i$
are terms in variables $x_1, \dots, x_n, y, z_1, \dots, z_m$.

(\ref{item : finitely presentable : 1}): The algebra $\boldsymbol{T}_\mathsf{K}(\varphi)$ is a finitely presented member of $\mathsf{K}$ by definition. Therefore, it only remains to show that for each $i \leq k$,
\[
\boldsymbol{T}_\mathsf{K}(\varphi) \vDash \varphi(x_1 / \theta(\varphi), \dots, x_n / \theta(\varphi), y /  \theta(\varphi)).
\]
In view of (\ref{Eq : fin pres 0}), it will be enough to prove
\[
\boldsymbol{T}_\mathsf{K}(\varphi) \vDash \bigg(\bigsqcap_{i \leq k} t_i \thickapprox s_i\bigg)(x_1 / \theta(\varphi), \dots, x_n / \theta(\varphi), y /  \theta(\varphi), z_1 / \theta(\varphi), \dots, z_m / \theta(\varphi)).
\]

Let $i \leq k$. From the definitions  of $\ulcorner\varphi\urcorner$ and $\theta(\varphi)$ it follows that 
$\langle t_i, s_i \rangle \in \ulcorner\varphi\urcorner \subseteq \theta(\varphi)$.
Consequently,
\begin{align*}
t_i^{\boldsymbol{T}_\mathsf{K}(\varphi)}&(x_1 / \theta(\varphi), \dots, x_n / \theta(\varphi), y /  \theta(\varphi), z_1 / \theta(\varphi), \dots, z_m / \theta(\varphi))\\
&=t_i(x_1, \dots, x_n, y, z_1, \dots, z_m) / \theta(\varphi)\\
&=s_i(x_1, \dots, x_n, y, z_1, \dots, z_m) / \theta(\varphi)\\
&=s_i^{\boldsymbol{T}_\mathsf{K}(\varphi)}(x_1 / \theta(\varphi), \dots, x_n / \theta(\varphi), y /  \theta(\varphi), z_1 / \theta(\varphi), \dots, z_m / \theta(\varphi)).
\end{align*}

(\ref{item : finitely presentable : 2}): Let $\A \in \mathsf{K}$ and $\langle a_1, \dots, a_n \rangle \in \mathsf{dom}(\varphi^\A)$. Then $\A \vDash \varphi(a_1, \dots, a_n, \varphi^\A(a_1,\dots, a_n))$. In view of (\ref{Eq : fin pres 0}), there exist $b_1, \dots, b_m \in A$ such that for each $i \leq k$,
\begin{equation}\label{Eq : fin pres 2}
t_i^\A(a_1, \dots, a_n, \varphi^\A(a_1,\dots, a_n), b_1, \dots, b_m) = s_i^\A(a_1, \dots, a_n, \varphi^\A(a_1,\dots, a_n), b_1, \dots, b_m).
\end{equation}
Now, let $g \colon \boldsymbol{T}(Var(\varphi)) \to \A$ be the unique homomorphism such that
\begin{equation}\label{Eq : fin pres 3}
g(x_i) = a_i \text{ for each }i \leq n, \quad g(y) = \varphi^\A(a_1, \dots, a_n), \quad \text{ and }\quad g(z_j) = b_j \text{ for each }j \leq m.
\end{equation}
From the above display and (\ref{Eq : fin pres 2}) it follows that $\ulcorner\varphi\urcorner \subseteq \mathsf{Ker}(g)$. As $\A \in \mathsf{K}$ by assumption, we also have $\mathsf{Ker}(g) \in \mathsf{Con}_\mathsf{K}(\boldsymbol{T}(Var(\varphi)))$. Consequently,
\[
\theta(\varphi) = \text{Cg}_\mathsf{K}^{\boldsymbol{T}(\varphi)}(\ulcorner\varphi\urcorner)\subseteq  \mathsf{Ker}(g).
\]
Since $\boldsymbol{T}_\mathsf{K}(\varphi) = \boldsymbol{T}(Var(\varphi)) / \theta(\varphi)$, we can apply Proposition \ref{Prop : homomorphisms : smaller congruences} to the above display, obtaining a homomorphism $h \colon \boldsymbol{T}_\mathsf{K}(\varphi) \to \A$ defined for every $t \in T(Var(\varphi))$ as $h(t/ \theta(\varphi)) = g(t)$. Together with (\ref{Eq : fin pres 3}), this yields $h(x_i / \theta(\varphi)) = a_i$ for each $i \leq n$ and $h(y / \theta(\varphi)) = \varphi^\A(a_1, \dots, a_n)$.
\end{proof}

We are now ready to prove Theorem \ref{Thm : existential elimination}.

\begin{proof} 
(\ref{item : existential elimination : 1}): Let $f$ be an implicit operation of $\K$ defined by a pp formula $\varphi(x_1, \dots, x_n, y)$. Consider the algebra
        \[
        \A = \mathsf{Sg}^{\boldsymbol{T}_{\K}(\varphi)}(x_1 / \theta(\varphi), \dots, x_n / \theta(\varphi)).
        \]
By Proposition \ref{Prop : finitely presented}(\ref{item : finitely presentable : 1}) we have
\[
\langle x_1 / \theta(\varphi), \dots, x_n / \theta(\varphi) \rangle \in \mathsf{dom}(f^{\boldsymbol{T}_{\K}(\varphi)}) \cap A^n \, \, \text{ and } \, \, f^{\boldsymbol{T}_{\K}(\varphi)}(x_1 / \theta(\varphi), \dots, x_n / \theta(\varphi)) = y / \theta(\varphi).
\]
Together with 
Theorem \ref{Thm : dominions : pp formulas}, this yields 
 $        y/\theta(\varphi) \in \mathsf{d}_{\mathsf{K}}(\A, \boldsymbol{T}_{\K}(\varphi))$. Now, let
 \[
 \B = \mathsf{Sg}^{\boldsymbol{T}_{\K}(\varphi)}(x_1 / \theta(\varphi), \dots, x_n / \theta(\varphi), y / \theta(\varphi)).
 \]
 As $\A \leq \B \leq \boldsymbol{T}_{\K}(\varphi) \in \K$ and $\K$ is a quasivariety with the amalgamation property, we can apply Proposition \ref{Prop : doms computable in subalgebras} to $y/\theta(\varphi) \in \mathsf{d}_{\mathsf{K}}(\A, \boldsymbol{T}_{\K}(\varphi)) \cap B$, obtaining 
$
y/\theta(\varphi) \in  \mathsf{d}_{\mathsf{K}}(\A, \B)$. By Theorem \ref{Thm : dominions : pp formulas} there exist an $m$-ary $g \in \mathsf{imp}_{\textsc{pp}}(\mathsf{K})$ and $\langle a_1, \dots, a_m \rangle\in \mathsf{dom}(g^\B) \cap A^m$ such that $g^{\B}(a_1, \dots, a_m) = y/\theta(\varphi)$. 

Since $g \in \mathsf{imp}_{\textsc{pp}}(\mathsf{K})$, there exists a formula $\exists z_1, \dots, z_k \psi(x_1, \dots, x_m, y, z_1, \dots, z_k)$, where $\psi$ is a conjunction of equations, defining $g$. Together with $\langle a_1, \dots, a_m \rangle \in \mathsf{dom}(g^\B)$ and $g^\B(a_1, \dots, a_n) = y / \theta(\varphi)$, this guarantees the existence of $b_1, \dots, b_k \in B$ such that
\begin{equation}\label{Eq : existential elimination - new proof 1}
\B \vDash \psi(a_1, \dots, a_m, y / \theta(\varphi), b_1, \dots, b_k).
\end{equation}
As $a_1, \dots, a_m \in A$, $\A \leq \B$, and $\A$ is generated by $x_1 / \theta(\varphi), \dots, x_n / \theta(\varphi)$ by definition, for each $i \leq m$ there exists a term $t_i(x_1, \dots, x_n)$ such that $a_i = t_i^\B(x_1 / \theta(\varphi), \dots, x_n / \theta(\varphi))$. Similarly, as $b_1, \dots, b_k \in B$ and $\B$ is generated by $x_1 / \theta(\varphi), \dots, x_n / \theta(\varphi), y / \theta(\varphi)$ by definition, for each $j \leq k$ there exists a term $s_j(x_1, \dots, x_n, y)$ such that $b_j = s_j^{\B}(x_1/\theta(\varphi), \dots, x_n / \theta(\varphi), y / \theta(\varphi))$. 

We consider the formula 
\[
\gamma = \psi(t_1(x_1, \dots, x_n), \dots, t_m(x_1, \dots, x_n), y, s_1(x_1, \dots, x_n, y), \dots, s_k(x_1, \dots, x_n, y)).
\]
Notice that $\gamma$ is a conjunction of equations because so is $\psi$. Then observe that  the formula $\exists z_1, \dots, z_k \psi(x_1, \dots, x_m, y, z_1, \dots, z_k)$ is functional in $\K$ because it defines $g$. Together with the definition of $\gamma$,  this guarantees that $\gamma$ is also functional in $\K$. Hence,  $\gamma$ defines some $h \in \mathsf{imp}_{\textsc{eq}}(\mathsf{K})$ by Corollary \ref{Cor : functionality in Q(K)}. Therefore, to conclude the proof, it suffices to show that $h$
 interpolates $f$.

 First, observe that from (\ref{Eq : existential elimination - new proof 1}) and the definitions of $\gamma$ and $t_1, \dots, t_m, s_1, \dots, s_k$ it follows that 
\begin{equation}\label{Eq : existential elimination - new proof 2}
\B \vDash \gamma(x_1 / \theta(\varphi), \dots, x_n / \theta(\varphi), y / \theta(\varphi)).
\end{equation}
As $\gamma$ defines $h$, this yields 
\[
\langle x_1 / \theta(\varphi), \dots, x_n / \theta(\varphi)\rangle \in \mathsf{dom}(h^\B)\, \, \text{ and }\, \, h^\B(x_1 / \theta(\varphi), \dots, x_n / \theta(\varphi)) = y / \theta(\varphi).
\]

We are now ready to prove that $h$ interpolates $f$. To this end, consider $\C \in \K$ and $c_1, \dots, c_n, d \in C$ such that $\langle c_1, \dots, c_n \rangle \in \mathsf{dom}(f^\C)$ and $f^\C(c_1, \dots, c_n) = d$. As $f$ is defined by the pp formula $\varphi(x_1, \dots, x_n, y)$ by assumption, from Proposition \ref{Prop : finitely presented}(\ref{item : finitely presentable : 2}) it follows that there exists a homomorphism $e \colon \boldsymbol{T}_\K(\varphi) \to \C$ such that
\[
e(x_i / \theta(\varphi)) = c_i \text{ for each $i \leq n$} \, \, \text{ and }\, \, e(y/ \theta(\varphi)) = d.
\]
Since $\B \leq \boldsymbol{T}_\K(\varphi)$ and $x_1 / \theta(\varphi), \dots, x_n / \theta(\varphi), y / \theta(\varphi) \in B$ by the definition of $\B$, the above display  still holds if we restrict $e$ to a homomorphism $e \colon \B \to \C$. As $h$ is an implicit operation of $\K$, it is preserved by homomorphism between members of $\K$ and, in particular, by $e$. Together with (\ref{Eq : existential elimination - new proof 2}) and the above display, this yields
\[
\langle c_1, \dots, c_n \rangle = \langle e(x_1 / \theta(\varphi)), \dots, e(x_n / \theta(\varphi))\rangle \in \mathsf{dom}(h^\C)
\]
and 
\begin{align*}
    d &= e (y / \theta(\varphi)) = e(h^\B(x_1 / \theta(\varphi), \dots, x_n / \theta(\varphi)))= h^\C(e(x_1 / \theta(\varphi)), \dots, e(x_n / \theta(\varphi))
) \\
&= h^\C(c_1, \dots, c_n).
\end{align*}
Since $f^\C(c_1, \dots, c_n) = d$, we conclude that $h$ interpolates $f$.

(\ref{item : existential elimination : 2}): Immediate consequence of  (\ref{item : existential elimination : 1}) and Corollary \ref{Cor : each implicit operations splits into pp formulas}.
    \end{proof}

From Theorems \ref{Thm : existential elimination}(\ref{item : existential elimination : 1}) and \ref{Thm : dominions : pp formulas}  we deduce the following (for a similar observation, see \cite[Thm.\ $1^*$, p.\ 475]{Bacsich47} and \cite{Bacs_epi}).

\begin{Corollary}\label{Cor: dominions AP impeq}
Let $\K$ be a quasivariety with the amalgamation property. For every $\A \leq \B \in \mathsf{K}$ we have
\begin{align*}
\mathsf{d}_\mathsf{K}(\A, \B) = \{ b \in B : & \text{ there exist } f \in \imp_{\textsc{eq}}(\K) \text{ and }\langle a_1, \dots, a_n \rangle \in \mathsf{dom}(f^\B) \cap A^n\\
&\text{ such that }f^\B(a_1, \dots, a_n) = b \}.
\end{align*}
\end{Corollary}

We close this section with two examples. The first shows that Isbell's operations cannot be interpolated by any implicit operation defined by a conjunction of equations in the variety of commutative monoids, and the second shows that  condition (\ref{item : existential elimination : 1}) of Theorem \ref{Thm : existential elimination} cannot be improved by requiring the equality $\imppp(\K)= \imp_{\textsc{eq}}(\K)$.

\begin{exa}[\textsf{Isbell's operations}]\label{Exa : Isbell is not interolable by eq : existential elimination}
For every positive $n$ let $f_n$ be the $n$-th Isbell's operation, viewed as an implicit operation of the variety $\Mon$ of monoids (see \cref{Exa : Isbell operations}   and  \cref{Thm : Isbell implicit operations}). Then $f_n \in \imppp(\mathsf{Mon})$. We will prove that $f_n$ cannot be interpolated by any member of $\imp_{\textsc{eq}}(\mathsf{\Mon})$. Suppose the contrary, with a view to contradiction. Then $f_n$ is interpolated by some $g \in \imp_{\textsc{eq}}(\Mon)$.\ Consider the commutative monoids $\mathbb{N} = \langle \mathbb{N}; \cdot, 1 \rangle$ and $\mathbb{Q} = \langle \mathbb{Q}; \cdot, 1 \rangle$. Moreover, let $a_1, \dots, a_n, b_1, \dots, b_n, c_1, \dots, c_{2n+1}, d$ be the sequence of rationals defined as follows: for every
 $1 \leq i \leq n$ and $1 <  j < 2n+1$,
\[a_i = \frac{1}{2} = b_i, \quad c_j = 12, \quad c_1 = 6 = c_{2n+1}, \quad d = 3.\] 
Using the formula $\psi_n$  in \cref{Exa : Isbell operations}, we have
\[
\mathbb{Q} \vDash \psi_n(a_1, \dots, a_n, b_1, \dots, b_n, c_1, \dots, c_{2n+1}, d).
\]
By the definition of $f_n$ this yields $\langle c_1, \dots, c_{2n+1} \rangle \in \dom(f_n^\mathbb{Q})$ and $f_n^\mathbb{Q}(c_1, \dots, c_{2n+1}) = d$.  As $g$ interpolates $f_n$, we obtain $\langle c_1, \dots, c_{2n+1} \rangle \in \dom(g^\mathbb{Q})$ and $g^\mathbb{Q}(c_1, \dots, c_{2n+1}) = d$. Let $\varphi(x_1, \dots, x_{2n+1},y)$ be the conjunction of equations defining $g$. Then $\mathbb{Q} \vDash \varphi(c_1, \dots, c_{2n+1}, d)$. As $\varphi$ is a universal formula and $c_1, \dots, c_{2n+1}, d \in \mathbb{N}$, \cref{Thm : preservation}\eqref{item : preservation : universal} implies  that $\mathbb{N} \vDash \varphi(c_1, \dots, c_{2n+1}, d)$, and hence $\langle c_1, \dots, c_{2n+1} \rangle \in \dom(g^\mathbb{N})$ and $g^\mathbb{N}(c_1, \dots, c_{2n+1}) = d$.
Let $\A$ and $\B$ be the submonoids of $\mathbb{N}$ with universes $A=\{1\} \cup \{2m : m \in \mathbb{N}\}$ and $B = \{0,1\}$, respectively. 
Since $c_1, \dots, c_{2n+1} \in A$ and $g \in \mathsf{imp}_{\textsc{eq}}(\Mon)$, \cref{Thm : dominions : pp formulas} yields $3 = d \in \d_\Mon(\A, \oper{N})$.
Let $h,k \colon \mathbb{N} \to \B$ be given by
\[
h(m) = \begin{cases}
    1 &\text{if } m = 1,\\
    0 &\text{otherwise,}
\end{cases} \hspace{2cm}
k(m) = \begin{cases}
    1 &\text{if $m$ is odd,}\\
    0 &\text{otherwise.}
\end{cases} 
\]
It is immediate to verify that $h$ and $k$ are homomorphisms such that $h\res_{A}=k\res_{A}$ and $h(3) \neq k(3)$, whence $3 \notin \d_\Mon(\A, \mathbb{N})$. As this is false, we conclude that $f_n$ cannot be interpolated by any member of $\imp_{\textsc{eq}}(\mathsf{Mon})$.
\qed
\end{exa} 

\begin{exa}[\textsf{Cancellative commutative monoids}]\label{exa:ccmon AP}
An element $a$ of a monoid $\A = \langle A; \cdot, 1 \rangle$ is said to be \emph{cancellative} when for all $b, c \in A$,
\[
(a b = a c \text{ implies }b = c) \, \, \text{ and } \, \ (b  a = c  a \text{ implies }b = c).
\]
When all the elements of $\A$ are cancellative, we say that $\A$ is \emph{cancellative} (see, e.g., \cite{MR21847244,MR218472}).\ The class of cancellative commutative monoids forms a quasivariety, denoted by $\mathsf{CCMon}$, which is axiomatized relative to commutative monoids by the quasiequation $x y \thickapprox x  z \to  y \thickapprox z$. 

We recall that $\mathsf{CCMon}$ has the amalgamation property (see, e.g., \cite[pp.~100, 108]{SurvKissal}) and, therefore, falls within the scope of  \cref{Thm : existential elimination}. On the other hand, we will show that $\imppp(\mathsf{CCMon}) \ne \imp_{\textsc{eq}}(\mathsf{CCMon})$.

Consider the pp formula
\[
\varphi(x, y) = \exists z ( xz \thickapprox 1 \sqcap y \thickapprox 1).
\]
Notice that $\varphi$ is functional in $\mathsf{CCMon}$. Indeed, $\A \vDash \varphi(a, b)$ implies $b=1$ for all $\A \in \mathsf{CCMon}$ and $a, b\in A$. Since $\varphi$ is a pp formula, we can apply \cref{Cor : functionality in Q(K)}, obtaining that it defines a unary $f \in \imppp(\mathsf{CCMon})$. Moreover, if $\A \in \CCMon$, then $\dom(f^\A)$ consists of the invertible elements of $\A$. We show that $f \notin \mathsf{imp}_{\textsc{eq}}(\mathsf{CCMon})$. Suppose the contrary, with a view to contradiction. Then consider the cancellative commutative monoids $\mathbb{N} = \langle \mathbb{N}; \cdot, 1 \rangle$ and $\mathbb{Q} = \langle \mathbb{Q}; \cdot, 1 \rangle$. We have
\[
2 \in \mathsf{dom}(f^\mathbb{Q}) \, \, \text{ and } \, \, f^\mathbb{Q}(2) = 1.
\]
Since $f$ is defined by a conjunction of equations, from the above display and $1,2 \in \mathbb{N} \leq \mathbb{Q}$ it follows that $2 \in \mathsf{dom}(f^\mathbb{N})$, a contradiction with the fact that $2$ is not invertible in $\mathbb{N}$. Hence, we conclude that $\imppp(\mathsf{CCMon}) \ne \imp_{\textsc{eq}}(\mathsf{CCMon})$.

We conclude this example by showing that $\imp_{\textsc{eq}}(\CCMon)$ is not closed under composition. Observe that $f$ coincides with the composition $h \circ g$, where $h \in \impeq(\CCMon)$ is the unary implicit operation defined by the equation $y \thickapprox 1$ and $g \in \impeq(\CCMon)$ is the implicit operation of ``taking inverses'' in monoids (see \cref{Exa : inverses in monoids 1}) restricted to $\CCMon$. Since $h \circ g = f \notin \imp_{\textsc{eq}}(\CCMon)$, this shows that $\impeq(\CCMon)$ is not closed under composition.
\qed
\end{exa}

\section{The strong Beth definability property}

As the implicit operations of a class of algebras need not be term functions, it is natural to wonder whether they can at least be interpolated by a set of terms, which can be thought of as a way of rendering them ``explicit''. This idea is reminiscent  of the \emph{Beth Definability Theorem} of first order logic (see, e.g., \cite[pp.~301--302]{HodModTh}), a fundamental result stating that every implicit definition can be turned explicit (in the setting of first order theories). However, the notions of implicit and explicit definability typical of first order logic differ from ours. For instance, an explicit definition in first order logic is simply a definition given by a formula. As our implicit operations are defined by a formula by definition, they are already explicitly definable in the sense of first order logic. As a consequence, the Beth Definability Theorem cannot be applied to our implicit operations in a nontrivial way and, in particular, it does not guarantee  they can be interpolated by a set of terms, that is, made explicit in our sense.

The next definition formalizes the idea of interpolating implicit operations by sets of terms and is a particular instance of the notion of interpolation introduced in \cref{Def: interpolation}.

\begin{Definition}
Let $f$ be an $n$-ary implicit operation of a class of algebras $\mathsf{K}$. We say that $f$ is \emph{interpolated} by a set $\{ t_i : i \in I \}$ of $n$-ary terms of $\mathsf{K}$ when it is interpolated by $\{ t_i^\K : i \in I \}$. This means that for all $\A \in \mathsf{K}$ and $\langle a_1, \dots, a_n \rangle \in \mathsf{dom}(f^\A)$ there exists $i \in I$ such that
\[
f^\A(a_1, \dots, a_n) = t_i^\A(a_1, \dots, a_n).
\]
Intuitively, the partial function $f$ is made ``explicit'' by the terms in $\{ t_i : i \in I \}$. 
\end{Definition}

Notice that $f$ is interpolated by a set of terms if and only if $f^{\A}(a_1, \dots, a_n) \in \mathsf{Sg}^{\A}(a_1, \dots, a_n)$ for all $\A \in \K$ and $\langle a_1, \dots, a_n \rangle \in \dom(f^{\A})$.
When $f$ is defined by a formula $\varphi$, the demand that $f$ be interpolated by the set of terms $\{ t_i : i \in I \}$ can be rendered as follows:
\begin{equation}\label{Eq : infinite interpolation}
\mathsf{K} \vDash \varphi(x_1, \dots, x_n, y) \to \bigsqcup_{i \in I}  t_i(x_1, \dots, x_n) \thickapprox y.
\end{equation}
As a consequence of the \cref{Thm : compactness theorem}, we obtain the following.

\begin{Proposition}\label{Prop : interpolation : infinite to finite}
An implicit partial function on an elementary class can be interpolated by a set of terms if and only if it can be interpolated by a finite set of terms.
\end{Proposition}

\begin{proof}
Let $f$ be a partial function on an elementary class $\mathsf{K}$. Assume that $f$ is defined by a formula $\varphi(x_1, \dots, x_n, y)$ and that it can be interpolated by a set of terms $\{ t_i : i \in I \}$. Then condition (\ref{Eq : infinite interpolation}) holds. From the \cref{Thm : compactness theorem} it follows that there exists a finite $T \subseteq \{ t_i : i \in I \}$ such that
\[
\mathsf{K} \vDash \varphi(x_1, \dots, x_n, y) \to \bigsqcup_{t \in T} t(x_1, \dots, x_n) \thickapprox y.
\]
As $f$ is defined by $\varphi$, this means that $f$ is interpolated by the terms in $T$.
\end{proof}

When viewed as an implicit operation on a class of algebras $\mathsf{K}$, every term function $\langle t^\A : \A \in \mathsf{K}\rangle$ of $\mathsf{K}$ is interpolated by a single term, namely, $t$ (see Example \ref{Example : term functions}). However, not all implicit operations can be interpolated by terms.\ For instance, there is no set of terms interpolating the implicit operation of ``taking inverses'' in the variety of monoids as we will see in \cref{Exa : monoids lack the SES}. It is therefore sensible to isolate the cases in which interpolation is always possible, something that indicates a good balance between the expressivity of the language (measured by what can be said in terms of implicit operations) and its actual richness (measured by what can be interpolated, or made explicit, by terms).

\begin{Definition} \label{Def : strong Beth prop}
A class of algebras is said to have the \emph{strong Beth definability property} when each of its implicit operations can be interpolated by a set of terms.

\end{Definition}

The reason why we termed our Beth definability property  ``strong'' is to distinguish it from other weaker definability properties considered in the literature (see \cite{BlHoo06,Kreisel60JSL}).
Although we will not rely on this fact, we remark that, when a quasivariety is the equivalent algebraic semantics of a propositional logic in the sense of \cite{BP89}, the strong Beth definability property is the algebraic counterpart of the so called \emph{projective Beth property} investigated in \cite[p.~76]{BF85} (see also \cite[Sec.~2.2.3]{Hoo01} and \cite{MakpBeth,Mak99,Mak00}).

In the context of elementary classes, the strong Beth definability property can be equivalently formulated by restricting our attention to implicit operations defined by pp formulas and interpolation by a finite set of terms. More precisely, we have the following.

\begin{Proposition}\label{Prop : elementary strong Beth}
The following conditions are equivalent for an elementary class $\mathsf{K}$:
\benroman
\item\label{item : elementary strong Beth 1} $\mathsf{K}$ has the strong Beth definability property;
\item\label{item : elementary strong Beth 2} each implicit operation of $\mathsf{K}$ defined by a pp formula can be interpolated by a finite set of terms.
\eroman
\end{Proposition}

\begin{proof}
The implication (\ref{item : elementary strong Beth 1})$\Rightarrow$(\ref{item : elementary strong Beth 2}) is an immediate consequence of Proposition \ref{Prop : interpolation : infinite to finite}. 
To prove (\ref{item : elementary strong Beth 2})$\Rightarrow$(\ref{item : elementary strong Beth 1}) suppose
that each implicit operation of $\mathsf{K}$ defined  by a pp formula can be interpolated by a finite set of terms. Then let $f$ be an implicit operation of $\mathsf{K}$. By Corollary \ref{Cor : each implicit operations splits into pp formulas} there exist some implicit operations $f_1, \dots, f_n$ of $\mathsf{K}$ defined by pp formulas such that for each $\A \in \mathsf{K}$,
\[
f^\A = f_1^\A \cup \dots \cup f_n^\A.
\]
By assumption each $f_i$ is interpolated by a finite set of terms $T_i$. In view of the above display, we conclude that $f$ is interpolated by the terms in $T_1 \cup \dots \cup T_n$.
\end{proof}

Rephrasing condition (\ref{item : elementary strong Beth 2}) of Proposition \ref{Prop : elementary strong Beth} in terms of the validity of certain formulas in $\mathsf{K}$ yields the following.

\begin{Corollary}
An elementary class $\mathsf{K}$ has the strong Beth definability property if and only if for each pp formula $\varphi(x_1, \dots, x_n, y)$ such that
\[
\mathsf{K} \vDash (\varphi(x_1, \dots, x_n, y) \sqcap \varphi(x_1, \dots, x_n, z)) \to y \thickapprox z
\]
there exist terms $t_1(x_1, \dots, x_n), \dots, t_m(x_1, \dots, x_n)$ such that
\[
\mathsf{K} \vDash \varphi(x_1, \dots, x_n, y) \to \bigsqcup_{i \leq m}t_i(x_1, \dots, x_n) \thickapprox y.
\]
\end{Corollary}

For elementary classes closed under direct products the equivalence in Proposition \ref{Prop : elementary strong Beth} can be refined as follows.

\begin{Proposition}\label{Prop : classes closed under products : strong Beth}
The following conditions are equivalent for an elementary class $\mathsf{K}$ closed under direct products:
\benroman
\item\label{item : classes closed under products 1} $\mathsf{K}$ has the strong Beth definability property;
\item\label{item : classes closed under products 2} each implicit operation of $\mathsf{K}$ defined by a pp formula can be interpolated by a single term.
\eroman
\end{Proposition}

\begin{proof}
(\ref{item : classes closed under products 1})$\Rightarrow$(\ref{item : classes closed under products 2}): Suppose that $\mathsf{K}$ has the strong Beth definability property. Then consider an implicit operation $f$ of $\mathsf{K}$ that can be defined by a pp formula $\varphi(x_1, \dots, x_n, y)$.\ From the strong Beth definability property it follows that $f$ can be interpolated by a set of terms $\{ t_i : i \in I \}$. As $f$ is defined by $\varphi$, this amounts to
\[
\mathsf{K} \vDash \varphi(x_1, \dots, x_n, y)  \to \bigsqcup_{i \in I}  t_i(x_1, \dots, x_n) \thickapprox y.
\]
Since $\varphi$ is a pp formula and $\mathsf{K}$ an elementary class closed under $\PPP$ by assumption, we can apply Corollary \ref{Cor : compactness for pp formulas}, obtaining that there exists $i \in I$ such that
\[
\mathsf{K} \vDash \varphi(x_1, \dots, x_n, y)  \to  t_i(x_1, \dots, x_n) \thickapprox y.
\]
As $\varphi$ defines $f$, we conclude that $f$ is interpolated by $t_i$.

(\ref{item : classes closed under products 2})$\Rightarrow$(\ref{item : classes closed under products 1}):  Immediate from the implication (\ref{item : elementary strong Beth 2})$\Rightarrow$(\ref{item : elementary strong Beth 1}) of Proposition \ref{Prop : elementary strong Beth}.
\end{proof}

\section{The strong epimorphism surjectivity property}

The strong Beth definability property admits a purely algebraic formulation, as we proceed to illustrate.
Let $\mathsf{K}$ be a class of algebras. A homomorphism $f \colon \A \to \B$ with $\A, \B \in \mathsf{K}$ is said to be a $\mathsf{K}$-\emph{epimorphism} when it is right cancellable, that is, when for every pair of homomorphisms $g, h \colon \B \to \C$ with $\C \in \mathsf{K}$,
\[
g \circ f = h \circ f\, \, \text{ implies }\, \, g=h.
\]

While every surjective homomorphism between members of $\mathsf{K}$ is a $\mathsf{K}$-epimorphism, the converse need not hold in general.\ For instance, the inclusion map of $\langle \mathbb{Z}; \cdot, 1 \rangle$ into $\langle \mathbb{Q}; \cdot, 1 \rangle$ is a nonsurjective epimorphism in the variety of monoids. Consequently, a class of algebras $\mathsf{K}$ is said to have the \emph{epimorphism surjectivity property} when every $\mathsf{K}$-epimorphism is surjective.

We will show that, in the setting of universal classes, the strong Beth definability property is equivalent to the following strengthening of the epimorphism surjectivity property (see \cref{Thm : strong ES = strong Beth}).

\begin{Definition}
A class of algebras $\mathsf{K}$ has the \emph{strong epimorphism surjectivity property} when for every homomorphism $f \colon \A \to \B$ with $\A, \B \in \mathsf{K}$ and $b \in B - f[A]$ there exists a pair of homomorphisms $g, h \colon \B \to \C$ with $\C \in \mathsf{K}$ such that $g \circ f = h \circ f$ and $g(b) \ne h(b)$.
\end{Definition}

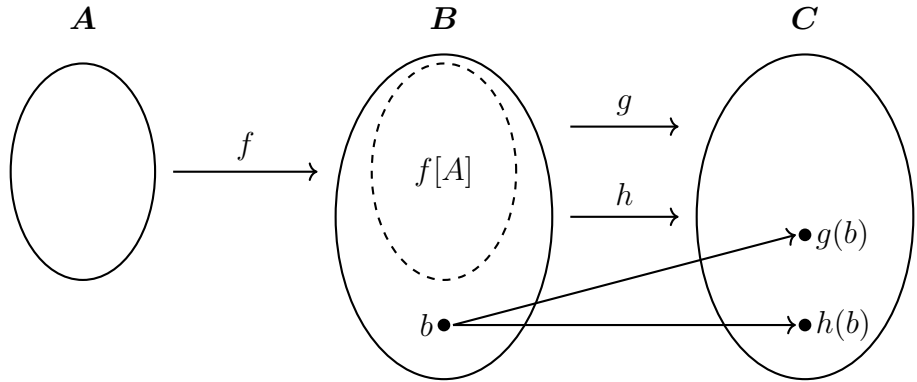
\begin{center}
\begin{tikzpicture}[scale=1.2]
  \draw[thick] (-2,0.5) ellipse (0.8 and 1.2);
  \node at (-2,2.2) {$\A$};
  \draw[->, thick] (-1,0.5) -- (0.6,0.5) node[midway, above] {\(f\)};
  \draw[thick] (2,0) ellipse (1.2 and 1.8);
  \node at (2,2.2) {$\B$};
  \draw[thick, dashed] (2,0.5) ellipse (0.8 and 1.2);
  \node at (2,0.5) {$f[A]$};
  \fill (2,-1.2) circle (2pt);
  \node[left] at (2,-1.2) {$b$};
  \draw[->, thick] (3.4,1) -- (4.6,1) node[midway, above] {$g$};
  \draw[->, thick] (3.4,0) -- (4.6,0) node[midway, above] {$h$};
  \draw[->, thick] (2.1,-1.2) -- (5.9,-0.2);
  \draw[->, thick] (2.1,-1.2) -- (5.9,-1.2);
  \draw[thick] (6,0) ellipse (1.2 and 1.8);
  \node at (6,2.2) {$\C$};
  \fill (6,-0.2) circle (2pt);
  \node[right] at (6,-0.2) {$g(b)$};
  \fill (6,-1.2) circle (2pt);
  \node[right] at (6,-1.2) {$h(b)$};
\end{tikzpicture}
\end{center}

\begin{Remark} \label{Rem : simplification of SES}
When $\mathsf{K}$ is closed under $\III$ and $\SSS$,
we may assume that $\A \leq \B$ and that the map $f \colon \A \to \B$ in the above definition is an inclusion map. In this case, the strong epimorphism surjectivity property simplifies to the demand that for all $\A \leq \B \in \mathsf{K}$ and $b \in B - A$ there exists a pair of homomorphisms $g, h \colon \B \to \C$ with $\C \in \mathsf{K}$ such that $g{\upharpoonright}_A = h{\upharpoonright}_A$ and $g(b) \ne h(b)$.
  Moreover, when $\mathsf{K}$ is a quasivariety, the Subdirect Decomposition Theorem \ref{Thm : Subdirect Decomposition} allows us to assume  $\C \in \mathsf{K}_{\textsc{rsi}}$. 
\qed
\end{Remark}

\begin{exa}[\textsf{Abelian groups}]\label{Exa : abalian groups : SES}
We view groups as algebras $\A = \langle A; \cdot, ^{-1}, 1 \rangle$, i.e., we assume that multiplication, taking inverses,
and the neutral element are all basic operations. 
While it is known that every variety of Abelian groups has the strong epimorphism surjectivity property,
we provide a short proof for the sake of completeness. 
For consider a variety of Abelian groups $\mathsf{V}$, $\A \leq \B \in \mathsf{V}$, and $b \in B - A$. 
Since congruences in Abelian groups correspond to subgroups, there exists a congruence $\theta$ of $\B$
associated with the subgroup $\A$. Clearly, the canonical surjection $g \colon \B \to \B / \theta$ is a homomorphism such that $g^{-1}(1 / \theta) = A$ and $\B / \theta \in \mathsf{V}$. Then let $h \colon \B \to \B / \theta$ be the homomorphism that sends every element of $B$ to $1 / \theta$. From $g^{-1}(1 / \theta) = A$ and the definition of $h$ it follows that $g(a) = 1 / \theta = h(a)$ for each $a \in A$, whence $g{\upharpoonright}_A = h{\upharpoonright}_A$. On the other hand, from $b \notin A = g^{-1}(1 / \theta)$ it follows that $g(b) \ne 1 / \theta$, while $h(b) = 1 / \theta$ by the definition of $h$. Hence, $g(b) \ne h(b)$. We conclude that $\mathsf{V}$ has the strong epimorphism surjectivity property.
\qed
\end{exa}

While every class with the strong epimorphism surjectivity property has the epimorphism surjectivity property, the converse need not hold in general.\ For instance, it is known that only $16$ varieties of Heyting algebras have the strong epimorphism surjectivity property \cite[Thm.\ 8.1]{Mak00}. On the other hand, there exists a continuum of varieties of Heyting algebras with the nonstrong version of this property \cite[p.~199]{BMRES}.
However, the two properties coincide in quasivarieties with the amalgamation property (see \cref{Thm : AP -> WES = SES}).

\begin{Remark}\label{Rem : SES = monos are regular}
In the context of quasivarieties, the strong epimorphism surjectivity property admits a purely categorical formulation. For observe that each quasivariety $\mathsf{K}$ can be viewed as a category whose objects are the members of $\mathsf{K}$ and whose arrows are the homomorphisms between them. 
As quasivarieties contain free algebras (see \cref{thm:quasivariety free algebra}), 
monomorphisms coincide with embeddings in quasivarieties (see, e.g., \cite[Prop.~8.29]{AHS06}).
A monomorphism is said to be \emph{regular} when it is an equalizer. It turns out that a quasivariety $\mathsf{K}$ has the strong epimorphism surjectivity property if and only if all monomorphisms are regular in $\mathsf{K}$ (see, e.g., \cite[Prop.~6.1]{SurvKissal}).
\qed
\end{Remark}

For the present purpose, the interest of the strong epimorphism surjectivity property comes from the fact that it is the algebraic counterpart of the strong Beth definability property. More precisely, we will prove the following 
theorem which generalizes the correspondences between the strong epimorphism surjectivity property and definability properties established in \cite[Thm.~4]{Bacsich47}, \cite[Thm.~3.6]{MakpBeth}, \cite[Thm.~3.1]{Mak99}, and \cite[Thm.~5]{BethHoog}.

\begin{Theorem}\label{Thm : strong ES = strong Beth}
    A universal class has the strong epimorphism surjectivity property if and only if it has the strong Beth definability property. 
\end{Theorem}

As an immediate consequence of \cref{Rem : simplification of SES}, we obtain an alternative formulation of the strong epimorphism surjectivity property in terms of dominions.

\begin{Proposition}\label{Prop : SES and dominions}
Let $\mathsf{K}$ be a class of algebras closed under $\III$ and $\SSS$. 
Then $\mathsf{K}$ has the strong epimorphism surjectivity property if and only if $\mathsf{d}_\mathsf{K}(\A, \B) = A$ for every $\A \leq \B \in \mathsf{K}$.
\end{Proposition}

From \cref{Thm : dominions : pp formulas} and \cref{Prop : SES and dominions} we deduce the following.

\begin{Corollary}\label{Cor : failures of the SES}
A universal class $\mathsf{K}$ has the strong epimorphism surjectivity property if and only if for all $\A \leq \B \in \mathsf{K}$, $f \in \imppp(\K)$, and $\langle a_1, \dots, a_n \rangle \in \mathsf{dom}(f^\B) \cap A^n$ we have $f^{\B}(a_1, \dots, a_n) \in A$.
\end{Corollary}

We are now ready to prove \cref{Thm : strong ES = strong Beth}.

\begin{proof}

Let $\mathsf{K}$ be a universal class. To prove the implication from left to right, we reason by contraposition. Suppose that $\mathsf{K}$ lacks the strong Beth definability property. By Propositions~\ref{Prop : interpolation : infinite to finite} and \ref{Prop : elementary strong Beth} 
there exists $f \in \mathsf{imp}_{\textsc{pp}}(\mathsf{K})$
that cannot be interpolated by any set of terms.
Therefore, there exists $\B \in \mathsf{K}$ and $\langle a_1, \dots, a_n \rangle \in \dom(f^{\B})$ for which there exists no term $t$ such that $f^{\B}(a_1, \dots, a_n) = t^{\B}(a_1, \dots, a_n)$. Then $f^{\B}(a_1, \dots, a_n) \notin \mathsf{Sg}^{\B}(a_1, \dots, a_n)$.
Let $\A=\mathsf{Sg}^{\B}(a_1, \dots, a_n)$.
Since $\langle a_1, \dots, a_n \rangle \in \mathsf{dom}(f^{\B}) \cap A^n$ and $f^{\B}(a_1, \dots a_n) \notin A$, 
\cref{Cor : failures of the SES} implies that $\mathsf{K}$ lacks the strong epimorphism surjectivity property.

Next we prove the implication from right to left. Suppose that $\K$ has the strong Beth definability property. In order to prove that $\K$ has the strong epimorphism surjectivity property, it suffices to show that $\mathsf{d}_{\mathsf{K}}(\A, \B) = A$ for every $\A \leq \B \in \K$ (see \cref{Prop : SES and dominions}). To this end, consider $\A \leq \B \in \mathsf{K}$ and $b \in \mathsf{d}_{\mathsf{K}}(\A, \B)$. By \cref{Thm : dominions : pp formulas} there exist $f \in \mathsf{imp}_{\textsc{pp}}(\mathsf{K})$ and $\langle a_1, \dots, a_n \rangle \in \mathsf{dom}(f^{\B}) \cap A^n$ such that $b = f^{\B}(a_1, \dots, a_n)$. Since $\mathsf{K}$ has the strong Beth definability property, there exists a term $t$ such that $f^{\B}(a_1, \dots, a_n) = t^{\B}(a_1, \dots, a_n)$. Therefore,
\begin{align*}
    b &= f^{\B}(a_1, \dots, a_n) = t^{\B}(a_1, \dots, a_n) = t^{\A}(a_1, \dots, a_n) \in A. 
\end{align*}
This shows that $\mathsf{d}_{\mathsf{K}}(\A, \B) \subseteq A$. As the reverse inclusion always holds, we conclude that $\mathsf{d}_{\mathsf{K}}(\A, \B) = A$.
\end{proof}

We close this section with a series of examples of classes of algebras with and without the strong epimorphism surjectivity property.

\begin{exa}[\textsf{Monoids}]\label{Exa : monoids lack the SES}
As we mentioned, the inclusion map of $\mathbb{Z} = \langle \mathbb{Z}; \cdot, 1 \rangle$ into $\mathbb{Q} = \langle \mathbb{Q}; \cdot,1 \rangle$ is a nonsurjective epimorphism in the variety of monoids $\mathsf{Mon}$, whence $\mathsf{Mon}$ lacks the 
epimorphism surjectivity property, and consequently its strong version as well.

This can also be viewed through the lens of \cref{Cor : failures of the SES}. For let $f$ be the implicit operation of ``taking inverses'' in monoids and recall that it can be defined by a pp formula (see Theorem \ref{Thm : inverses monoid : implicit operation}).
Clearly, $\mathbb{Z} \leq \mathbb{Q}$ and $2 \in \mathsf{dom}(f^\mathbb{Q}) \cap \mathbb{Z}$. Moreover,
\[
f^\mathbb{Q}(2) = 2^{-1} = \frac{1}{2} \notin \mathbb{Z}.
\]
From Corollary \ref{Cor : failures of the SES} it follows that $\mathsf{Mon}$ lacks the strong epimorphism surjectivity property. The same proof yields the same conclusion for the variety of commutative monoids. In view of Theorem \ref{Thm : strong ES = strong Beth}, both varieties lack the strong Beth definability property.
\qed
\end{exa}

\begin{exa}[\textsf{Reduced commutative rings}]
Essentially the same argument shows that the quasivariety $\mathsf{RCRing}$ of reduced commutative rings lacks the strong epimorphism surjectivity property. More precisely, let $\mathbb{Z}$ and $\mathbb{Q}$ be the reduced commutative rings of the integers and the rationals, respectively. Moreover, let $f$ be the implicit operation of $\mathsf{RCRing}$ given by \cref{Thm : RCRing : implicit operation}. As $\mathbb{Q}$ is a field, \cref{Thm : RCRing : implicit operation} guarantees that $f^{\mathbb{Q}}$ is the operation of  ``taking weak inverses'' in $\mathbb{Q}$. Therefore, we can replicate the argument detailed in \cref{Exa : monoids lack the SES}, yielding that $\mathsf{RCRing}$ lacks
the strong epimorphism surjectivity property, and hence also the strong Beth definability property.
\qed
\end{exa}

\begin{exa}[\textsf{Distributive lattices}]
We will show that the variety of distributive lattices $\mathsf{DL}$ lacks the strong epimorphism surjectivity property. For let $f$ be the implicit operation of ``taking relative complements'' in distributive lattices and recall that it can be defined by a pp formula (see Theorem \ref{Thm : relative complements is implicit}).
Moreover, let $\B$ be the four-element Boolean lattice with universe $\{ 0, a, b, 1 \}$, where $0$ and $1$ are the minimum and the maximum of $\B$, respectively. Lastly, let $\A$ be the subalgebra of $\B$ with universe $\{ 0, a, 1 \}$. Since $b$ is the complement of $a$ relative to $[0, 1]$ in $\B$, we have $\langle a, 0, 1 \rangle \in \mathsf{dom}(f^\B) \cap A^3$ and
\[
f^\B(a, 0, 1) = b \notin A.
\]
Hence, we can apply Corollary \ref{Cor : failures of the SES}, obtaining that $\mathsf{DL}$ lacks the strong epimorphism surjectivity 
property, and thus also the strong Beth definability property.
\qed
\end{exa}

The survey \cite{SurvKissal} contains plenty of examples of classes of algebras with and without the epimorphism surjectivity property and the \emph{intersection property of amalgamation} (IPA, for short), which is equivalent to the strong epimorphism surjectivity property in varieties (see \cite[Prop.~4.5]{SurvKissal}). Among these examples, we count the following.
A \emph{semigroup with zero} is an algebra $\A=\langle A; \cdot, 0 \rangle$, where $\langle A; \cdot \rangle$ is a semigroup and $a0 = 0a = 0$ for every $a \in A$.
For every $n \geq 4$ the variety of semigroups with zero satisfying the equation $x^n \thickapprox 0$ has the epimorphism surjectivity property, but not its strong version (see \cite[p.~89]{SurvKissal}). Moreover, the varieties of semilattices and lattices\label{lattices ES} have both the strong epimorphism surjectivity property (see \cite[pp.~99, 102]{SurvKissal}).

\section{Tangible epimorphism surjectivity}

The aim of this section is to facilitate the task of determining whether a given class of algebras has the strong epimorphism surjectivity property. On the one hand, we will show that this problem can often be settled by considering only finitely generated or finitely presented algebras.

\begin{Theorem}\label{Thm : finitely generated : SES}
The following conditions are equivalent for a universal class $\mathsf{K}$:
\benroman
\item\label{item : universal SES : almost total : 1} $\mathsf{K}$ has the strong epimorphism surjectivity property;
\item\label{item : universal SES : almost total : 2} for each $\A \leq \B \in \mathsf{K}$ with $\A$ and $\B$ finitely generated we have $\mathsf{d}_\mathsf{K}(\A, \B) = A$.
\eroman 
In addition, when $\mathsf{K}$ is a quasivariety, we may assume that $\B$ is finitely presented.
\end{Theorem}

 On the other hand, we will provide a criterion for the validity of the strong epimorphism surjectivity property which applies to a large class of quasivarieties. More precisely, a term $t$ of arity $ \geq 3$ is a \emph{near unanimity term} (see, e.g., \cite[Sec.~1.2.3]{KP01})
 for a class of algebras $\mathsf{K}$ when
\[
\mathsf{K} \vDash x \thickapprox t(y, x, \dots, x) \thickapprox t(x, y, x, \dots, x) \thickapprox \dots \thickapprox t(x, \dots, x, y).
\]
Intuitively, the term $t$ returns $x$ when its arguments are almost unanimously $x$. 
Notably, each variety with a near unanimity term is congruence distributive (see \cite[Thm.~2]{Mit78}), although the converse need not hold in general (see \cite[Lem.~3]{Mit78}). Ternary near unanimity terms play a prominent role in algebra and are known as \emph{majority terms} (see, e.g., \cite[Def.\ II.12.8]{BuSa00}). As $t(x, y, z) = (x \land y) \lor (x \land z) \lor (y \land z)$ is a majority term for every class of algebras with a lattice reduct, we have the following.

\begin{Theorem} \label{Thm : lattice reduct -> CD}
Let $\K$ be a class of algebras with a lattice reduct. Then $\VVV(\K)$ has a majority term and is congruence distributive.
\end{Theorem}

We will show that, in the presence of a near unanimity term, the task of determining whether a quasivariety has the strong epimorphism surjectivity property can be simplified as follows.

\begin{Theorem}\label{Thm : unanimity}
The following conditions are equivalent for a quasivariety  $\mathsf{K}$ with a near unanimity term of arity $n$:
\benroman
\item\label{item : unanimity : 1} $\mathsf{K}$ has the strong epimorphism surjectivity property;
\item\label{item : unanimity : 2} for each finitely generated $\A \leq \B_1 \times \dots \times \B_{n-1}$ with $\B_1, \dots, \B_{n-1} \in \mathsf{K}_{\textsc{rfsi}}$ finitely generated we have $\mathsf{d}_\mathsf{K}(\A, \B_1 \times \dots \times \B_{n-1}) = A$.
\eroman
\end{Theorem}

Before proving these results, let us provide an example of how to apply them in practice.

\begin{exa}[\textsf{Relatively complemented distributive lattices}]\label{Exa : relatively complemented DL}
An algebra $\langle A; \land, \lor, r \rangle$ is a \emph{relatively complemented distributive lattice} when $\langle A; \land, \lor \rangle$ is a distributive lattice and $r$ a ternary operation such that $r(a, b, c)$ is the complement of $a$ relative to the interval $[a \land b \land c, a \lor b \lor c]$ for all $a, b, c \in A$. Notice that if $b \leq a \leq c$, then $r(a, b, c)$ is the complement of $a$ relative to $[b, c]$. The class of relatively complemented distributive lattices forms a variety, which we denote by $\mathsf{RCDL}$. We will prove the following.

\begin{Theorem}\label{Thm : SES : Boolean algebras}
The varieties of relatively complemented distributive lattices and of Boolean algebras have the strong epimorphism surjectivity property.
\end{Theorem}

\begin{proof}
We will detail the proof for $\mathsf{RCDL}$ only, as the case of Boolean algebras is analogous and well known (see, e.g., \cite[p.~103]{SurvKissal}). Let $\boldsymbol{D}_2$ be the unique relatively complemented distributive lattice with universe $\{ 0, 1 \}$ and $0 < 1$. We begin with the following observation.

\begin{Claim}\label{Claim : RCDL : SES}
We have $\mathsf{RCDL}_{\textsc{fsi}} = \III(\boldsymbol{D}_2)$.
\end{Claim}

\begin{proof}[Proof of the Claim]
Since every two-element algebra is finitely subdirectly irreducible, we obtain $\III(\boldsymbol{D}_2) \subseteq \mathsf{RCDL}_{\textsc{fsi}}$. To prove the other inclusion, consider $\A \in \mathsf{RCDL}_{\textsc{fsi}}$. Let $\A^-$ be the lattice reduct of $\A$. To conclude the proof, it will be enough to show that $\mathsf{Con}(\A) = \mathsf{Con}(\A^-)$. For suppose that this is the case. Then, as $\A$ is finitely subdirectly irreducible, we can apply Proposition \ref{Prop : RFSI}, obtaining that  $\textup{id}_A$ is meet irreducible in $\mathsf{Con}(\A)$. Since $\mathsf{Con}(\A) = \mathsf{Con}(\A^-)$, this yields that $\textup{id}_A$ is also meet irreducible in $\mathsf{Con}(\A^-)$. Consequently, Proposition \ref{Prop : RFSI} guarantees that $\A^-$ is a finitely subdirectly irreducible distributive lattice. Up to isomorphism, the only such lattice is the lattice reduct of $\boldsymbol{D}_2$.  Hence, we conclude that $\A \cong  \boldsymbol{D}_2$, as desired.

Therefore, it only remains to show that $\mathsf{Con}(\A) = \mathsf{Con}(\A^-)$. As $\A^-$ is a reduct of $\A$, we have $\mathsf{Con}(\A) \subseteq	 \mathsf{Con}(\A^-)$. Then we 
proceed to prove the other inclusion. Consider $\theta \in \mathsf{Con}(\A^-)$. To prove that $\theta \in \mathsf{Con}(\A)$, let $\langle a_1, a_2 \rangle, \langle b_1, b_2 \rangle, \langle c_1, c_2 \rangle \in \theta$. We need to show that $\langle r^\A(a_1, b_1, c_1), r^\A(a_2, b_2, c_2) \rangle \in \theta$. To this end, recall from Theorem \ref{Thm : relative complements is implicit} that ``taking relative complements'' is a ternary implicit operation $f$ of $\mathsf{DL}$. Moreover, consider the canonical surjection $\pi_\theta \colon \A^- \to \A^- / \theta$.

We will prove that for each $i \leq 2$,
\begin{equation}\label{Eq : foijfkjqplasjalksa}
\langle a_i / \theta, b_i / \theta, c_i / \theta \rangle \in \mathsf{dom}(f^{\A^- / \theta}) \, \, \text{ and } \, \, f^{\A^- / \theta}(a_i / \theta, b_i / \theta, c_i / \theta) = r^\A(a_i, b_i, c_i) / \theta.
\end{equation}
First, observe that $\langle a_i, b_i, c_i \rangle \in \mathsf{dom}(f^{\A^-})$ and $f^{\A^-}(a_i, b_i, c_i) = r^\A(a_i, b_i, c_i)$ because $\A^-$ is the reduct of the relatively complemented distributive lattice $\A$ and $f$ is the implicit operation of ``taking relative complements''. Since $f$ is an operation of $\mathsf{DL}$, 
it is preserved by the homomorphism $\pi \colon \A^- \to \A^- / \theta$. Therefore, from  $\langle a_i, b_i, c_i \rangle \in \mathsf{dom}(f^{\A^-})$ and $f^{\A^-}(a_i, b_i, c_i) = r^\A(a_i, b_i, c_i)$ it follows that $\langle a_i / \theta, b_i / \theta, c_i / \theta \rangle = \langle \pi(a_i), \pi(b_i), \pi(c_i) \rangle\in \mathsf{dom}(f^{\A^- / \theta})$ and
\begin{align*}
f^{\A^- / \theta}(a_i / \theta, b_i / \theta, c_i / \theta) &= f^{\A^- / \theta}(\pi(a_i), \pi(b_i), \pi(c_i)) = \pi(f^{\A^-}(a_i, b_i, c_i))\\
& = \pi(r^\A(a_i, b_i, c_i)) = r^\A(a_i, b_i, c_i) / \theta.
\end{align*}
This establishes (\ref{Eq : foijfkjqplasjalksa}). As $a_1 / \theta = a_2 / \theta$,  $b_1 / \theta = b_2 / \theta$, and $c_1 / \theta = c_2 / \theta$ by assumption, this implies
\[
r^\A(a_1, b_1, c_1)/ \theta = f^{\A^- / \theta}(a_1 / \theta, b_1 / \theta, c_1 / \theta) = f^{\A^- / \theta}(a_2 / \theta, b_2 / \theta, c_2 / \theta) = r^\A(a_2, b_2, c_2)/ \theta,
\]
that is, $\langle r^\A(a_1, b_1, c_1), r^\A(a_2, b_2, c_2) \rangle \in \theta$, as desired.
\end{proof}

Now, we 
prove that $\mathsf{RCDL}$ has the strong epimorphism surjectivity property. As $\mathsf{RCDL}$ has a lattice reduct, it possesses a majority term. Therefore, we can apply \cref{Thm : unanimity}, obtaining that $\mathsf{RCDL}$ has the strong epimorphism surjectivity property if and only if for each finitely generated $\A \leq \B \times \C$ with $\B, \C \in \mathsf{RCDL}_{\textsc{fsi}}$ finitely generated we have $\mathsf{d}_\mathsf{K}(\A, \B \times \C) = A$. Together with \cref{Claim : RCDL : SES} the latter specializes to the following: for each $\A \leq \boldsymbol{D}_2 \times \boldsymbol{D}_2$ we have $\mathsf{d}_\mathsf{K}(\A, \boldsymbol{D}_2 \times \boldsymbol{D}_2) = A$.

To prove this, consider $\A \leq \boldsymbol{D}_2 \times \boldsymbol{D}_2$. By inspection for each $b \in (D_2 \times D_2)- A$ one can find an endomorphism $h$ of $\boldsymbol{D}_2 \times \boldsymbol{D}_2$ such that $id{\upharpoonright}_A = h{\upharpoonright}_A$ and $id(b) \ne g(b)$, where $id$ is the identity map on $\boldsymbol{D}_2 \times \boldsymbol{D}_2$. Hence, we conclude that $\mathsf{d}_\mathsf{K}(\A, \boldsymbol{D}_2 \times \boldsymbol{D}_2) = A$.
\end{proof}
\end{exa}

We are now ready to prove Theorem \ref{Thm : finitely generated : SES}.

\begin{proof}
The implication (\ref{item : universal SES : almost total : 1})$\Rightarrow$(\ref{item : universal SES : almost total : 2}) holds by 
\cref{Prop : SES and dominions}.
To prove the implication (\ref{item : universal SES : almost total : 2})$\Rightarrow$(\ref{item : universal SES : almost total : 1}), we reason by contraposition.\ Suppose that $\mathsf{K}$ lacks the strong epimorphism surjectivity property.\ In view of Corollary \ref{Cor : failures of the SES}, there exist $\C \leq \D \in \mathsf{K}$, an implicit operation $f$ of $\mathsf{K}$ defined by a pp formula $\varphi(x_1, \dots, x_n, y)$, and
\begin{equation}\label{Eq : fin presentable : universal : 1}
\langle c_1, \dots, c_n \rangle \in \mathsf{dom}(f^\D) \cap C^n \text{ with }f^\D(c_1, \dots, c_n) \notin C.
\end{equation}

As $\varphi$ defines an implicit operation of $\mathsf{K}$, it also defines an implicit operation $g$ of $\mathbb{Q}(\mathsf{K})$ by \cref{Cor : functionality in Q(K)}. As both $f$ and $g$ are defined by $\varphi$ and $\D \in \mathsf{K} \subseteq \mathbb{Q}(\mathsf{K})$, we have $f^\D = g^\D$. We will make use of this observation without further notice.

\begin{Claim}\label{Claim : finitely presented}
There exist a finitely presented member $\B$ of $\mathbb{Q}(\mathsf{K})$ and $\A \leq \B$ finitely generated with elements $a_1, \dots, a_n \in A$ and a homomorphism $h \colon \B \to \C$ satisfying the following conditions:
\benroman
\item\label{item : claim : finitely presented 1} $\langle a_1, \dots, a_n \rangle \in \mathsf{dom}(g^\B) \cap A^n$ and $g^\B(a_1, \dots, a_n) \notin A$;
\item\label{item : claim : finitely presented 2} $h(a_1) = c_1, \dots, h(a_n) = c_n$, $h(g^\B(a_1, \dots, a_n)) = f^\D(c_1, \dots, c_n)$, and $h[A] \subseteq C$.
\eroman
\end{Claim}

\begin{proof}[Proof of the Claim]
Define $\B = \boldsymbol{T}_{\mathbb{Q}(\mathsf{K})}(\varphi)$, $a_1 = x_1 / \theta(\varphi), \dots, a_n = x_n / \theta(\varphi)$, and $\A = \mathsf{Sg}^\B(a_1, \dots, a_n)$ (for the definition of $\boldsymbol{T}_{\mathbb{Q}(\mathsf{K})}(\varphi)$ see \cref{Def: TK}). Clearly, $\A$ is a finitely generated subalgebra of $\B$. Moreover, from \cref{Prop : finitely presented}\eqref{item : finitely presentable : 1} it follows that $\B$ is finitely presented in $\mathbb{Q}(\mathsf{K})$ and that
\begin{equation}\label{Eq : fin presentable : universal : 2}
\langle a_1, \dots, a_n \rangle \in \mathsf{dom}(g^{\B}) \cap A^n \, \, \text{ and } \, \, g^\B(a_1, \dots, a_n) = y / \theta(\varphi).
\end{equation}
Now, recall from (\ref{Eq : fin presentable : universal : 1}) that $\langle c_1, \dots, c_n \rangle \in \mathsf{dom}(f^\D) = \mathsf{dom}(g^\D)$. Together with $\D \in \mathsf{K} \subseteq \mathbb{Q}(\mathsf{K})$ and \cref{Prop : finitely presented}\eqref{item : finitely presentable : 2}, this guarantees the existence of a homomorphism $h \colon \B \to \D$ such that
\[
h(a_1) = c_1, \dots, h(a_n) = c_n, \text{ and }h(y / \theta(\varphi)) = g^\D(c_1, \dots, c_n) = f^\D(c_1, \dots, c_n).
\]
From the right hand sides of the above display and (\ref{Eq : fin presentable : universal : 2}) it follows that $h(g^\B(a_1, \dots, a_n)) = f^\D(c_1, \dots, c_n)$.

Therefore, it only remains to prove that $g^\B(a_1, \dots, a_n) \notin A$ and $h[A] \subseteq C$. To this end, recall that $\A$ is generated by $a_1, \dots, a_n$ and that $c_1, \dots, c_n$ belong to the subalgebra $\C$ of $\D$. Therefore, from the left hand side of the above display it follows that $h[A] \subseteq C$. Moreover, from the right hand side of (\ref{Eq : fin presentable : universal : 1}) and $h(g^\B(a_1, \dots, a_n)) = f^\D(c_1, \dots, c_n)$  we obtain $h(g^\B(a_1, \dots, a_n)) \notin C$. Together with $h[A] \subseteq C$, this yields $g^\B(a_1, \dots, a_n) \notin A$.
\end{proof}

Let $\A$, $\B$, and $a_1, \dots, a_n$ be as in Claim \ref{Claim : finitely presented}. We have two cases: either $\mathsf{K}$ is a quasivariety or not. We begin with the case where $\mathsf{K}$ is a quasivariety. Then $\mathsf{K} = \mathbb{Q}(\mathsf{K})$. Therefore, $\B$ is a finitely presented member of $\mathsf{K}$ and $\A \leq \B$  finitely generated. From Claim \ref{Claim : finitely presented}(\ref{item : claim : finitely presented 1}) and Theorem \ref{Thm : dominions : pp formulas} it follows that $g^\B(a_1, \dots, a_n) \in \mathsf{d}_\mathsf{K}(\A, \B) - A$, whence $\mathsf{d}_\mathsf{K}(\A, \B) \ne A$, as desired.

It only remains to consider the case where $\mathsf{K}$ is not a quasivariety. Then let $h\colon \B \to \D$ be the homomorphism given by Claim \ref{Claim : finitely presented}. Define $\A' = h[\A]$ and $\B' = h[\B]$. As $\D \in \mathsf{K}$ and $\mathsf{K}$ is a universal class by assumption, from $\A' \leq \B' \leq \D$ it follows that $\A ', \B' \in \mathsf{K}$. Furthermore, since $\A$ and $\B$ are finitely generated by Claim \ref{Claim : finitely presented}, the algebras $\A'$ and $\B'$ are also finitely generated. By condition (\ref{item : claim : finitely presented 1}) of the same claim we have $\B \vDash \varphi(a_1, \dots, a_n, g^\B(a_1, \dots, a_n))$. As $h \colon \B \to \B'$ is a homomorphism and $\varphi$ a pp formula, we can apply Theorem \ref{Thm : preservation}(\ref{item : preservation : pp}), obtaining $\B' \vDash \varphi(h(a_1), \dots, h(a_n), h(g^\B(a_1, \dots, a_n)))$. By Claim \ref{Claim : finitely presented}(\ref{item : claim : finitely presented 2}) this amounts to $\B' \vDash \varphi(c_1, \dots, c_n, f^\D(c_1, \dots, c_n))$, that is,
\[
\langle c_1, \dots, c_n \rangle \in \mathsf{dom}(f^{\B'}) \, \, \text{ and } \, \, f^{\B'}(c_1, \dots, c_n) = f^\D(c_1, \dots, c_n).
\]

Recall from (\ref{Eq : fin presentable : universal : 1}) that $f^\D(c_1, \dots, c_n) \notin C$. Together with the right hand side of the above display and the fact that $A' = h[A] \subseteq C$ (see condition (\ref{item : claim : finitely presented 2}) of Claim \ref{Claim : finitely presented}), this yields $f^{\B'}(c_1, \dots, c_n) \notin A'$. On the other hand, from $a_1, \dots, a_n \in A$ and $h(a_i) = c_i$ for each $i \leq n$ (see    Claim \ref{Claim : finitely presented}(\ref{item : claim : finitely presented 2})) it follows that $c_1, \dots, c_n \in h[A] = A'$. Hence, the left hand side of the above display can be improved to $\langle c_1, \dots, c_n \rangle \in \mathsf{dom}(f^{\B'}) \cap (A')^{n}$. As a consequence, we can apply Theorem \ref{Thm : dominions : pp formulas}, obtaining $\varphi^{\B'}(c_1, \dots, c_n) \in \mathsf{d}_\mathsf{K}(\A', \B') - A'$, whence $\mathsf{d}_\mathsf{K}(\A', \B') \ne A'$.
\end{proof}

Now, we
proceed to prove \cref{Thm : unanimity}. We will make use of the next concept from \cite[Def.~4.4]{CKMWES}.

\begin{Definition}
Let $\mathsf{K}$ be a quasivariety, $\A \in \mathsf{K}$, and $\theta \in \mathsf{Con}_\mathsf{K}(\A)$.\ Given a positive integer $n$, we say that $\theta$ is \emph{$n$-irreducible} in $\mathsf{Con}_\mathsf{K}(\A)$ when $\theta= \theta_1 \cap \dots \cap \theta_n$ with $\theta_1, \dots, \theta_n \in \mathsf{Con}_\mathsf{K}(\A)$ implies $\theta=\theta_1 \cap \dots \cap \theta_{i-1} \cap \theta_{i+1} \cap \dots \cap \theta_n$ for some $i \leq n$. When $\mathsf{K}$ is clear from the context, we will simply say that $\theta$ is \emph{$n$-irreducible}.
\end{Definition}

Notice that the only $1$-irreducible $\mathsf{K}$-congruence of $\A$ is $A \times A$. Moreover, a $\mathsf{K}$-congruence $\theta$ of $\A$ is $2$-irreducible if and only if either $\theta \in \mathsf{Irr}_{\mathsf{K}}(\A)$ or $\theta = A \times A$. We rely on the following observation (see \cite[Prop.~4.5]{CKMWES}).

\begin{Proposition}\label{Prop : optimal two cases}
Let $\mathsf{K}$ be a quasivariety, $\A \in \mathsf{K}$, and $\theta \in \mathsf{Con}_\mathsf{K}(\A)$ $n$-irreducible. Then there exist $\phi_1, \dots, \phi_{n-1} \in \mathsf{Irr}_{\mathsf{K}}(\A)$ such that $\theta = \phi_1 \cap \dots \cap \phi_{n-1}$.
\end{Proposition}

We recall that, for $\A \leq \B$ and $\phi \in \Con(\B)$, we denote by $\A/\phi$ the subalgebra of $\B/\phi$ with universe $\{ a/\phi \in B/\phi : a \in A\}$. We will need the following easy consequence of Zorn's Lemma (see the proof of \cite[Prop.\ 3.7]{CKMWES}).

\begin{Proposition}\label{Prop : full congruence}
Let $\mathsf{K}$ be a quasivariety,  $\A \leq \B \in \mathsf{K}$, and $b \in B - A$. There exists $\phi \in \mathsf{Con}_\mathsf{K}(\B)$ such that $b / \phi \notin A / \phi$ and for each $\theta \in \mathsf{Con}_\mathsf{K}(\B / \phi) - \{ \textup{id}_{B / \phi} \}$ there exists $a \in A$ such that $\langle a / \phi, b / \phi \rangle \in \theta$.
\end{Proposition}

We are now ready to prove Theorem \ref{Thm : unanimity}.
We follow a strategy similar to the one used to establish an analogous result \cite[Thm.~4.3]{CKMWES} in the setting of epimorphisms between finitely generated algebras.

\begin{proof}
As the implication (\ref{item : unanimity : 1})$\Rightarrow$(\ref{item : unanimity : 2}) is straightforward, we only detail the implication (\ref{item : unanimity : 2})$\Rightarrow$(\ref{item : unanimity : 1}).  To this end, we reason by contraposition. Suppose that $\mathsf{K}$ lacks the strong epimorphism surjectivity property. By Theorem \ref{Thm : finitely generated : SES} there exist $\A \leq \B \in \mathsf{K}$ with $\A$ and $\B$ finitely generated and some $b \in \mathsf{d}_\mathsf{K}(\A, \B) - A$.

\begin{Claim}\label{Claim : near unanimity 1}
We may assume that for each $\theta \in \mathsf{Con}_\mathsf{K}(\B) - \{ \textup{id}_B \}$ there exists $a \in  A$ such that $\langle a, b \rangle \in \theta$.
\end{Claim}

\begin{proof}[Proof of the Claim]
As $\A \leq \B \in \mathsf{K}$ and $b \in \mathsf{d}_\mathsf{K}(\A, \B) - A \subseteq B - A$, we can apply \cref{Prop : full congruence}, obtaining $\phi \in \mathsf{Con}_\mathsf{K}(\B)$ satisfying the following requirements: $b / \phi \in B / \phi - A / \phi$ and for each $\theta \in \mathsf{Con}_\mathsf{K}(\B / \phi) - \{ \textup{id}_{B / \phi} \}$ there exists $a \in A$ such that $\langle a / \phi, b / \phi \rangle \in \theta$.

Clearly, $\A / \phi \leq \B / \phi$ is a proper subalgebra. Moreover, $\A / \phi$ and $\B / \phi$ are finitely generated members of $\mathsf{K}$ because so are $\A$ and $\B$ by assumption and $\phi \in \mathsf{Con}_\mathsf{K}(\B)$. \cref{Cor: dominions subalg quot}\eqref{Cor: dominions subalg quot: 2} implies that $b / \phi \in \mathsf{d}_\mathsf{K}(\A / \phi, \B / \phi)$. As $b / \phi \notin A / \phi$, we obtain $b / \phi \in \mathsf{d}_\mathsf{K}(\A / \phi, \B / \phi) - A / \phi$.
Therefore, we may assume that $\phi = \textup{id}_B$ (otherwise we replace $\A$ and $\B$ by $\A / \phi$ and $\B / \phi$, respectively).

When coupled with the assumption that $\phi = \textup{id}_B$, the fact that for each $\theta \in \mathsf{Con}_\mathsf{K}(\B / \phi) - \{ \textup{id}_{B / \phi} \}$ there exists $a \in A$ such that $\langle a / \phi, b / \phi \rangle \in \theta$ implies that for each $\theta \in \mathsf{Con}_\mathsf{K}(\B) - \{ \textup{id}_B \}$ there exists $a \in A$ such that $\langle a, b \rangle \in \theta$.
\end{proof}

We will rely on the next observation.

\begin{Claim}\label{Claim : near unanimity 2}
The congruence $\textup{id}_B$ is $n$-irreducible in $\mathsf{Con}_\mathsf{K}(\B)$.
\end{Claim}

\begin{proof}[Proof of the Claim]
Let $\theta_1, \dots, \theta_n \in \mathsf{Con}_{\mathsf{K}}(\B)$ be such that $\text{id}_B = \theta_1 \cap \dots \cap \theta_n$. Let also $\phi_i = \theta_1 \cap \dots \cap \theta_{i-1} \cap \theta_{i+1} \dots \cap \theta_n$ for each $i \leq n$. We will show that $\phi_i=\text{id}_B$ for some $i \leq n$. Suppose the contrary, with a view to contradiction. 
\cref{Claim : near unanimity 1} yields $a_1, \dots, a_n \in A$ such that $\langle a_i, b \rangle \in \phi_i$ for every $i \leq n$.
By assumption $\mathsf{K}$ has a near unanimity term $t(x_1, \dots, x_n)$. 
We will prove that
\[
\langle t^\B(a_1, \dots, a_n), b \rangle \in  \theta_j
\]
for every $j \leq n$. To this end, consider $j \leq n$. As $\langle a_i, b \rangle \in \phi_i \subseteq \theta_j$ for every $i \leq n$ such that $i \ne j$, we obtain $\langle t^\B(a_1, \dots, a_n), t^\B(b, \dots, b, a_j, b, \dots, b) \rangle \in  \theta_j$.  Furthermore, since $t$ is a near unanimity term, we have  $t^\B(b, \dots, b, a_j, b, \dots, b)=b$.  Hence, $\langle t^\B(a_1, \dots, a_n), b \rangle \in  \theta_j$. This establishes the above display. Together with the assumption that $\textup{id}_B = \theta_1 \cap \dots \cap \theta_n$, this implies $b = t^\B(a_1, \dots, a_n)$. As $a_1, \dots, a_n \in A$ and $\A \leq \B$, we conclude that $b \in A$, which is false. Hence, $\textup{id}_B$ is $n$-irreducible.
\end{proof}

In view of Claim \ref{Claim : near unanimity 2} and Proposition \ref{Prop : optimal two cases}, there exist $\theta_1, \dots, \theta_{n-1} \in \mathsf{Irr}_{\mathsf{K}}(\B)$ such that $\text{id}_B =  \theta_1 \cap \dots \cap \theta_{n-1}$.\ Therefore, we can apply Proposition \ref{Prop : subdirect embedding} obtaining a subdirect embedding $h \colon \B \to \B / \theta_1 \times \dots \times \B / \theta_{n-1}$. Let $\B_i = \B / \theta_i$ for each $i \leq n-1$. By replacing $\A$ and $\B$ by their isomorphic images $h[\A]$ and $h[\B]$, respectively, we may assume that $\A \leq \B \leq \B_1 \times \dots \times \B_{n-1}$.
\ Notice that each $\B_i = \B / \theta$ is finitely generated because so is $\B$. Furthermore, from Proposition \ref{Prop : RFSI} and $\theta_i \in \mathsf{Irr}_{\mathsf{K}}(\B)$ it follows that $\B_i \in \mathsf{K}_{\textsc{rfsi}}$. 
Lastly, as $b \in \mathsf{d}_\mathsf{K}(\A, \B) - A$ and $\B \leq \B_1 \times \dots \times \B_{n-1}$, \cref{Cor: dominions subalg quot}\eqref{Cor: dominions subalg quot: 2} allows us to conclude that $b \in  \mathsf{d}_\mathsf{K}(\A, \B_1 \times \dots \times \B_{n-1}) - A$.
Hence, $\mathsf{d}_\mathsf{K}(\A, \B_1 \times \dots \times \B_{n-1}) \ne A$.
\end{proof}

The literature on epimorphisms contains two variants of Theorem \ref{Thm : unanimity} in which the class $\mathsf{K}$ is required to be an arithmetical variety with the property that the class of its finitely subdirectly irreducible members is closed under ultraproducts and nontrivial subalgebras  \cite[Thm.\ 6.8]{Camper18jsl} or only a congruence permutable variety \cite[Thm.~5.3]{CKMWES}. 
The first variant deals with the demand that all $\K$-epimorphisms be surjective, while the second with the weaker demand that all $\K$-epimorphisms between finitely generated algebras be surjective called the \emph{weak epimorphism surjectivity property}.
In both cases, the conclusion is that failures of the relevant property are witnessed by counterexamples of the form $\A \leq \B$ where $\B$ is a finitely subdirectly irreducible member of $\mathsf{K}$.
The possibility of obtaining similar results 
for
the \emph{strong} epimorphism surjectivity property 
is prevented by the following example. However, we will show in \cref{Cor : CP variety SES} that, under the amalgamation property, 
the above mentioned result for congruence permutable varieties becomes available in the context of the strong epimorphism surjectivity property as well.

\begin{exa}[\textsf{Heyting algebras}]\label{exa: Heyting and SES}
A \emph{Heyting algebra} is an algebra $\langle A; \land, \lor, \to, 0, 1 \rangle$ which comprises a bounded distributive lattice $\langle A; \land, \lor, 0, 1 \rangle$ and a binary operation $\to$ (called \emph{implication}) such that for all $a, b, c \in A$ we have
\[
a \land  b \leq c \iff a \leq b \to c.
\]
This means that $b \to c$ is the largest element $d \in A$ such that $d \wedge (b \to c) \leq c$ (see \cite[p.~173]{BD74}).

As a consequence, Heyting algebras are uniquely determined by their lattice reduct. In particular, every finite distributive lattice $\A$ can be expanded uniquely to a Heyting algebra by letting $0$ and $1$ be the minimum and maximum of $\A$, respectively, and defining
\[
a \to b = \max \{ c \in A : a \land c \leq b \} \text{ for all }a, b \in A.
\]
From a logical standpoint, the importance of Heyting algebras derives from the fact that they algebraize the intuitionistic propositional logic (see, e.g., \cite[Ch.~IX]{MR344067}).

Let $\C$ be the five-element chain, viewed as a Heyting algebra. Then $\mathbb{V}(\C)$ is an arithmetical variety whose class of finitely subdirectly irreducible members is closed under nontrivial subalgebras and ultraproducts (see, e.g., \cite[p.~80]{BuSa00} and \cite[p.~2 \& Thm.~2.3]{CD90}).

While it is known that $\mathbb{V}(\C)$ lacks the strong epimorphism surjectivity property (see \cite[Thm.~4.2]{MakpBeth}),
it is impossible to find counterexamples to this property of the form $\A \leq \B$, where $\B$ is a finitely subdirectly irreducible member of $\mathbb{V}(\C)$, for in this situation we always have $\mathsf{d}_\mathsf{\mathbb{V}(\C)}(\A, \B) = A$.
\qed
\end{exa}

The next result is well known  (see, e.g., \cite[Thm.\ 1.3]{BMRES}).
We provide a novel and short proof using the characterization of dominions in the presence of the amalgamation property established in \cref{Cor: dominions AP impeq}.

\begin{Theorem} \label{Thm : AP -> WES = SES}
    Let $\mathsf{K}$ be a quasivariety with the amalgamation property. Then $\mathsf{K}$ has the strong epimorphism surjectivity property if and only if 
    it has the weak epimorphism surjectivity property.
\end{Theorem}

\begin{proof}
The implication from left to right is straightforward.  So, let us assume that $\K$  has the weak  epimorphism surjectivity property. We will show that $\mathsf{K}$ has the strong epimorphism surjectivity property using Proposition \ref{Prop : SES and dominions}. To this end, consider $\A \leq \B \in \mathsf{K}$ and $b \in \mathsf{dom}_{\mathsf{K}}(\A,\B)$. Since $\K$ has the amalgamation property, by \cref{Cor: dominions AP impeq} there exist $f \in \imp_{\textsc{eq}}(\K)$ and $\langle a_1, \dots, a_n \rangle \in \dom(f^\B) \cap A^n$ such that $f^\B(a_1, \dots, a_n)=b$. Let $\A' = \mathsf{Sg}^{\A}(a_1, \dots, a_n)$ and $\B' = \mathsf{Sg}^{\B}(a_1, \dots, a_n,b)$.  Since $f$ is defined by a conjunction of equations, $\langle a_1, \dots, a_n \rangle \in \dom(f^{\B'}) \cap (A')^n$ and $f^{\B'}(a_1, \dots, a_n)=b$. So, \cref{Cor: dominions AP impeq} implies that $b \in \d_\K(\A',\B')$. As $\B' = \mathsf{Sg}^{\B}(a_1, \dots, a_n,b)$ and $a_1, \dots, a_n,b \in \d_\K(\A',\B')$, we obtain $\d_\K(\A',\B') = B'$. Therefore, the inclusion map $\A' \to \B'$ is an epimorphism. Since $\A', \B'$ are finitely generated members of $\K$  and $\K$ has the weak epimorphism surjectivity property, it follows that $\A' = \B'$, and hence $b \in B' = A' \subseteq A$. We have shown that $\mathsf{dom}_{\mathsf{K}}(\A,\B)=A$. Thus, $\K$ has the strong epimorphism surjectivity property.
\end{proof}

Given a class of algebras $\K$ closed under subalgebras and $\B \in \K$, we say that a subalgebra $\A$ of $\B$ is $\K$-\emph{epic} when the inclusion map $\A \to \B$ is a $\K$-epimorphism. In this case, $\K$ has the epimorphism surjectivity property if and only if every $\A \in \K$ lacks proper $\K$-epic subalgebras.
 We rely on the following result, which is an immediate consequence of \cite[Thm.~5.3]{CKMWES} and the proof of \cite[Thm.~5.4]{MRWepi}.
\begin{Theorem}\label{Thm:weak es fsi}
Let $\K$ be a congruence permutable variety. Then $\K$ has the weak epimorphism surjectivity property if and only if the finitely generated members of $\K_{\fsi}$ lack proper $\K$-epic finitely generated subalgebras.
\end{Theorem}

As mentioned above, the amalgamation property allows us to obtain a result similar to \cite[Thm.\ 6.8]{Camper18jsl} and \cite[Thm.~5.3]{CKMWES} for the strong epimorphism surjectivity property in congruence permutable varieties.

\begin{Corollary} \label{Cor : CP variety SES}
Let $\mathsf{K}$ be a congruence permutable variety with the amalgamation property. Then $\mathsf{K}$ has the strong epimorphism surjectivity property if and only if every finitely generated $\B \in \mathsf{K}_{\textsc{fsi}}$ lacks proper $\K$-epic subalgebras. 
\end{Corollary}
\begin{proof}
The implication from left to right is straightforward. On the other hand, if every finitely generated $\B \in \mathsf{K}_{\textsc{fsi}}$ lacks proper $\K$-epic subalgebras, then \cref{Thm:weak es fsi}
guarantees that $\K$ has the weak epimorphism surjectivity property, which by \cref{Thm : AP -> WES = SES} implies that $\mathsf{K}$ has the strong epimorphism surjectivity property as well. 
\end{proof}

The \emph{join} of a family of  varieties $\K_1, \dots, \K_n$ is the least variety containing them, namely, $\VVV(\K_1 \cup \dots \cup \K_n)$. While the weak and the strong epimorphism surjectivity properties need not be preserved by joins of varieties, in special cases they are, as we proceed to show.

\begin{Theorem} \label{Thm : arithmetical join of WES =  WES}
Let $\mathsf{K}$ be an arithmetical variety. If $\mathsf{K}$ is the join of finitely many varieties with the weak epimorphism surjectivity property, then it has the weak epimorphism surjectivity property. 
\end{Theorem}
\begin{proof}
Assume that $\K = \VVV(\K_1 \cup \dots \cup \K_n)$, where each $\K_i$ is a variety with the weak epimorphism surjectivity property. Suppose, with a view to contradiction, that    
$\mathsf{K}$ lacks this property. By 
\cref{Thm:weak es fsi}
this implies that there exists a finitely generated $\B \in \mathsf{K}_\fsi$ with a finitely generated subalgebra $\A \leq \B$ that is proper and $\K$-epic. 
Applying \cref{Thm : Jonsson}, 
we obtain that 
$\B \in \HHH\SSS\PPU(\mathsf{K}_1 \cup \dots \cup \mathsf{K}_n)$. 
By \cite[Thm.~5.6]{Ber11} we have $\PPU(\K_1 \cup \dots \cup \K_n) = \PPU(\K_1) \cup  \dots \cup \PPU(\K_n)$. Therefore, $\B \in \HHH\SSS\PPU(\K_1) \cup \dots \cup \HHH\SSS\PPU(\K_n) \subseteq\K_1 \cup \dots \cup \K_n$. Then $\B \in \K_i$ for some $i \leq n$.
 But $\mathsf{K}_i$ has the weak epimorphism surjectivity property by assumption, whence $\A \leq \B$ cannot be a $\K_i$-epic subalgebra. 
   As $\mathsf{K}_i \subseteq \mathsf{K}$, in particular it follows that $\A \leq \B$ cannot be a $\K$-epic subalgebra either. But this contradicts the assumption and thus completes the proof.
\end{proof}

The following is an immediate consequence of Theorems \ref{Thm : AP -> WES = SES} and 
\ref{Thm : arithmetical join of WES =  WES}.

\begin{Corollary} \label{Cor : arithmetical AP+ join of WES = SES}
    Let $\mathsf{K}$ be an arithmetical variety with the amalgamation property. If $\mathsf{K}$ is the join of finitely many varieties with the weak epimorphism surjectivity property, then it has the strong epimorphism surjectivity property. 
\end{Corollary}

\section{Extendable implicit operations}

The implicit operations $f$ of a class of algebras $\mathsf{K}$ behave well with respect to extensions, in the sense that if $\A \leq \B$ and $\A, \B \in \mathsf{K}$, then $f^\B$ extends $f^\A$.  More precisely, we have the following.

\begin{Proposition}\label{Prop : implicit operations extend}
Let $f$ be an implicit operation of a class of algebras $\mathsf{K}$. For all $\A, \B \in \mathsf{K}$ with $\A \leq \B$ the partial function $f^\B$ extends $f^\A$, in the sense that for all $\langle a_1, \dots, a_n \rangle \in \mathsf{dom}(f^\A)$ we have
\[
\langle a_1, \dots, a_n \rangle \in \mathsf{dom}(f^\B) \, \, \text{ and } \, \, f^\A(a_1, \dots, a_n) = f^\B(a_1, \dots, a_n).
\]
\end{Proposition}

\begin{proof}
Let $i \colon \A \to \B$ be the inclusion map and consider $\langle a_1, \dots, a_n \rangle \in \mathsf{dom}(f^\A)$. Since $i$ is a homomorphism, $\A, \B \in \mathsf{K}$, and $f$ an implicit operation of $\mathsf{K}$, we obtain 
\[
\langle i(a_1), \dots, i(a_n) \rangle \in \mathsf{dom}(f^\B) \, \, \text{ and }\, \, i(f^\A(a_1, \dots, a_n)) = f^\B(i(a_1), \dots, i(a_n)).
\]
As $i$ is the inclusion map, this yields the desired conclusion.
\end{proof}

Let $f$ be an implicit operation of a class of algebras $\mathsf{K}$. While Proposition \ref{Prop : implicit operations extend} guarantees that $f^\B$ extends $f^\A$ whenever $\A, \B \in \mathsf{K}$ and $\A \leq \B$, there is no reason to expect that we can extend $f^\A$ to a total function in this way. More precisely, there may be no extension $\B$ of $\A$ in $\mathsf{K}$ for which $f^\B$ is a total function. This makes the following definition attractive.

\begin{Definition}\label{Def:extendable operation}
Let $\M$ and $\K$ be classes of algebras with $\M \subseteq \K$. An $n$-ary implicit operation $f$ of $\mathsf{K}$ is said to be \emph{extendable relative to $\M$} when for all $\A \in \mathsf{M}$ and $a_1, \dots, a_n \in A$ there exists $\B \in \mathsf{K}$ such that 
\[
\A \leq \B \, \, \text{ and } \, \, \langle a_1, \dots, a_n \rangle \in \mathsf{dom}(f^\B).
\]
The set of implicit operations of $\mathsf{K}$ that are extendable relative to $\M$ will be denoted by $\textsf{ext}(\M,\K)$. We also let
\[
\extpp(\M,\K) = \ext(\M, \K) \cap \imppp(\K) \, \, \text{ and } \, \, \ext_{\textsc{eq}}(\M,\K) = \ext(\M, \K) \cap \imp_{\textsc{eq}}(\K).
\]
When $\M=\K$, we write $\ext(\K),\extpp(\K)$, and $\ext_{\textsc{eq}}(\K)$ instead of $\ext(\K,\K),\extpp(\K,\K)$, and $\ext_{\textsc{eq}}(\K,\K)$. Moreover, when an implicit operation is in $\ext(\K)$, we simply say it is \emph{extendable}. 
\end{Definition}

\begin{Remark}\label{rem: inclusions ext}
Let $\M_1,\M_2,\K_1,\K_2$ be classes of algebras with $\M_1 \subseteq \M_2 \subseteq \K_2$ and $\K_1 \subseteq \K_2$. Then the definition of an extendable implicit operation immediately yields that $\ext(\M_2,\K_1) \subseteq  \ext(\M_1,\K_2)$.\ In particular, if $\M$ and $\K$ are classes of algebras such that $\M \subseteq \K$, then $\ext(\K) \subseteq \ext(\M, \K)$. 
\qed
\end{Remark}

The relation between extendable implicit operations and the idea of ``extending partial functions to total ones'' is made precise by the next result. 

\begin{Theorem}\label{Thm : extendable 1}
Let $\mathsf{K}$ be a universal class and $\A \in \K$. Then there exists $\B \in \mathsf{K}$ with $\A \leq \B$ such that $f^\B$ is total for each $f \in \ext(\K)$. When, in addition, $\mathsf{K}$ is a quasivariety and $\A \in \mathsf{K}_\textsc{rsi}$, the algebra $\B$ can be chosen in $\mathsf{K}_\textsc{rsi}$.
\end{Theorem}

The proof of \cref{Thm : extendable 1} hinges on the following observation.

\begin{Proposition}\label{Prop : extendability in quasivarieties : si trick}
Let $\mathsf{K}$ be a quasivariety, $\A \in \mathsf{K}_{\textsc{rsi}}$, and $\B \in \mathsf{K}$ with $\A \leq \B$. Then there exist $\C \in \mathsf{K}_\textsc{rsi}$ with $\A \leq \C$ and a  surjective homomorphism $h \colon \B \to \C$.
\end{Proposition}

\begin{proof}
By the Subdirect Decomposition Theorem \ref{Thm : Subdirect Decomposition} there exists a subdirect embedding $g \colon \B \to \prod_{i \in I}\B_i$ for some family $\{ \B_i : i \in I \} \subseteq \mathsf{K}_\textsc{rsi}$. From $\A \leq \B$ it follows that $g \colon \A \to \prod_{i \in I}p_i[g[\A]]$ is also a subdirect embedding. As $\A \in \mathsf{K}_\textsc{rsi}$, there exists $j \in I$ such that $p_j \circ g \colon \A \to p_j[g[\A]]$ is an isomorphism. Together with $p_j[g[\A]] \leq \B_j$, this yields that $p_j \circ g \colon \A \to \B_j$ is an embedding.\ Since $\mathsf{K}_{\textsc{rsi}}$ is closed under $\III$, there exist $\C \in \mathsf{K}_\textsc{rsi}$ isomorphic to $\B_j$ such that $\A \leq \C$ and a  surjective homomorphism $h \colon \B \to \C$ 
(the latter is obtained by 
composing $p_j \circ g \colon \B \to \B_j$ with the isomorphism between $\B_j$ and $\C$).
\end{proof}

We are now ready to prove \cref{Thm : extendable 1}.

\begin{proof}
We begin with the following observation.

\begin{Claim}\label{Claim : extendable}
Let $\A \in \mathsf{K}$. Then there exists $\B \in \mathsf{K}$ with $\A \leq \B$ such that $A^n \subseteq  \mathsf{dom}(f^\B)$ for each $n$-ary $f \in \ext(\K)$.
\end{Claim}

\begin{proof}[Proof of the Claim]
Recall from \cref{Thm : implicit operations vs existential positive formulas} that each implicit operation $f$ of $\mathsf{K}$ is defined by an existential positive formula $\varphi_f$. Then consider the following set of formulas in the language of $\mathsf{K}$ expanded with fresh constants $\{ c_a : a \in A \}$ for the elements of $A$:
\[
\Sigma = \{  \exists y \varphi_f (c_{a_{1}}, \dots, c_{a_{n}}, y) :  n \in \mathbb{N}, \ a_1, \dots, a_n \in A, \text{ and }f\in  \ext(\K) \text{ is $n$-ary} \}.
\]
Moreover, let $\Gamma$ be a set of axioms for $\mathsf{K}$ (which is an elementary class by assumption) and define 
\[
\Delta = \mathsf{diag}(\A) \cup \Sigma \cup \Gamma.
\]
We will prove that $\Delta$ has a model. 

By the Compactness Theorem \ref{Thm : compactness theorem original}
it suffices to show that so does each finite subset of $\Delta$. To this end, consider $a_1, \dots, a_k \in A$ and 
$f_1, \dots, f_m \in \ext(\K)$ such that $f_i$ has arity $n_i$ for each
$i \leq m$. 
Moreover, for each 
$i \leq m$ let $a_1^i, \dots, a_{n_i}^i \in \{ a_1, \dots, a_k \}$. We need to prove that the following set has a model:

\begin{equation}\label{Eq : extendable : finite compactness}
\mathsf{diag}(\textsf{Sg}^\A(a_1, \dots, a_k)) \cup \{ \exists y \varphi_{f_i} (c_{a_{1}^i}, \dots, c_{a_{n_i}^i}, y) : 1 \leq i \leq m \} \cup \Gamma.
\end{equation}

To this end, we shall define a sequence $\A_0 \leq   \A_1 \leq  \dots \leq  \A_m$ of members of $\mathsf{K}$. First, let $\A_0 = \textsf{Sg}^\A(a_1, \dots, a_k)$. Clearly, $\A_0 \in \mathsf{K}$ because $\A_0 \leq \A \in \mathsf{K}$ and $\mathsf{K}$ is a universal class by assumption. Then suppose that the sequence $\A_0 \leq \dots \leq \A_i$ has already been defined for $i < m$. Since $\A_0 \leq \A_i$ we have $a_1^{i+1}, \dots, a_{n_{i+1}}^{i+1} \in \{ a_1, \dots, a_k \} \subseteq A_0 \subseteq A_{i}$. As $f_{i+1} \in \ext(\K)$ is $n_{i+1}$-ary and $\A_i \in \mathsf{K}$, there exists $\A_{i+1} \in \mathsf{K}$ such that $\langle a_1^{i+1}, \dots, a_{n_{i+1}}^{i+1} \rangle \in \mathsf{dom}(f_{i+1}^{\A_{i+1}})$. Clearly, $\A_0 \leq \dots \leq \A_{i+1}$ is still a sequence of members of $\mathsf{K}$. This concludes the definition of $\A_0 \leq   \A_1 \leq  \dots \leq  \A_m$.

Observe that $\textsf{Sg}^\A(a_1, \dots, a_k) = \A_0 \leq \A_m$. Then let $\A_m^+$ be the expansion of $\A_m$ with constants in $\{ c_a : a \in \textsf{Sg}^\A(a_1, \dots, a_k)\}$ in which each $c_a$ is interpreted as $a$. We will prove that $\A_m^+$  is a model of the set of formulas in (\ref{Eq : extendable : finite compactness}). From $\textsf{Sg}^\A(a_1, \dots, a_k) = \A_0 \leq \A_m$ and the Diagram Lemma \ref{Lem : Diagram Lemma} it follows that $\A_m^+$ is a model of $\mathsf{diag}(\textsf{Sg}^\A(a_1, \dots, a_k))$. Furthermore, $\A_m^+ \vDash \Gamma$ because $\A_m \in \mathsf{K}$ and $\Gamma$ axiomatizes $\mathsf{K}$. Therefore, it only remains to show that $\A_m^+ \vDash \exists y \varphi_{f_i} (c_{a_{1}^i}, \dots, c_{a_{n_i}^i}, y)$ for each $i \leq m$. As each $c_a$ is interpreted as $a \in A_0 \subseteq A_m$ in $\A_m^+$, this amounts to
\[
\A_m \vDash \exists y \varphi_{f_i} (a_{1}^i, \dots, a_{n_i}^i, y)\text{ for each  }i \leq m.
\]
Consider $i \leq m$. The construction of $\A_i$ guarantees that $\langle a_1^{i}, \dots, a_{n_{i}}^{i}\rangle \in \mathsf{dom}(f_i^{\A_i})$. As $f_i$ is defined by $\varphi_{f_i}$, this yields $\A_i \vDash \varphi_{f_i}(a_1^{i}, \dots, a_{n_{i}}^{i}, f_i^{\A_i}(a_1^{i}, \dots, a_{n_{i}}^{i}))$. Since $\A_i \leq \A_m$ and $\varphi_{f_i}$ is an existential positive formula, we can apply Theorem \ref{Thm : preservation}(\ref{item : preservation : ep}), obtaining $\A_m \vDash \varphi_{f_i}(a_1^{i}, \dots, a_{n_{i}}^{i}, f_i^{\A_i}(a_1^{i}, \dots, a_{n_{i}}^{i}))$, whence $\A_m \vDash \exists y \varphi_{f_i}(a_1^{i}, \dots, a_{n_{i}}^{i}, y)$. Thus, we conclude that $\A_m^+$ is a model of the set of formulas in (\ref{Eq : extendable : finite compactness}).

As we mentioned, from the fact that the set of formulas in (\ref{Eq : extendable : finite compactness}) has a model it follows that $\Delta$ 
also has
a model $\B^+$.
Let $\B$ be the $\mathscr{L}_\mathsf{K}$-reduct of $\B^+$. Since $\B^+$ is a model of $\Gamma$, so is $\B$. Together with the assumption that $\Gamma$ axiomatizes $\mathsf{K}$, this yields $\B \in \mathsf{K}$. Furthermore, as $\B^+$ is a model of $\mathsf{diag}(\A)$, we can apply the Diagram Lemma \ref{Lem : Diagram Lemma}, obtaining that $\A$ embeds into $\B$ via the map that sends $a$ to the interpretation of $c_a$ in $\B^+$. Since $\mathsf{K}$ is an elementary class, it is closed under $\III$. Therefore, we may assume that $\A \leq \B$ and that $c_a$ is interpreted as $a$ in $\B^+$.

To conclude the proof of 
the claim
it only remains to show that $A^n \subseteq  \mathsf{dom}(f^{\B})$ for each $n$-ary $f \in \ext(\K)$. To this end, consider an $n$-ary $f \in \ext(\K)$ and $a_1, \dots, a_n \in A$. As $\B^+$ is a model of $\Sigma$, we obtain $\B \vDash \exists y \varphi_f(a_1, \dots, a_n, y)$. Since $\varphi_f$ defines $f$, this amounts to $\langle a_1, \dots, a_n \rangle \in \mathsf{dom}(f^\B)$.
\end{proof}

Now,
we proceed to prove the first part of the statement of \cref{Thm : extendable 1}. Consider $\A \in \mathsf{K}$. We will define a sequence $\{ \A_i : i \in \mathbb{N} \}$ of members of $\mathsf{K}$. First, let $\A_0 = \A$. Then suppose $\A_i \in \mathsf{K}$ has already been defined. By Claim \ref{Claim : extendable} there exists $\A_{i+1} \in \mathsf{K}$ with $\A_i \leq \A_{i+1}$ and $A_i^n \subseteq \mathsf{dom}(f^{\A_{i+1}})$ for each $n$-ary $f \in \ext(\K)$. 
By definition the sequence $\{ \A_i : i \in \mathbb{N} \}$ constructed in this way is such that
\[
\A = \A_0 \leq \A_1 \leq \A_2 \leq \cdots
\]

Now, as $\mathsf{K}$ is a universal class, it is closed under unions of chains of algebras by \cref{Prop : universal class : unions of chains}. Therefore, the union $\B$ of the chain in the above display belongs to $\mathsf{K}$. Furthermore, $\A = \A_0 \leq \B$. To conclude the proof of the first part of the statement, it only remains to show that $f^\B$ is total for each $f \in \ext(\K)$. To this end, let $f \in \ext(\K)$ be $n$-ary and $b_1, \dots, b_n \in B$. As $\B$ is the union of the chain in the above display, there exists $i \in \mathbb{N}$ such that $b_1, \dots, b_n \in A_i$. By the definition of $\A_{i+1}$ we have $\langle b_1, \dots, b_n \rangle \in A_i^n \subseteq \mathsf{dom}(f^{\A_{i+1}})$. As $\A_{i+1} \leq \B$, we can apply Proposition \ref{Prop : implicit operations extend}, obtaining $\langle b_1, \dots, b_n \rangle \in \mathsf{dom}(f^\B)$. Hence, $f^\B$ is a total operation, as desired.

To prove the second part of the statement of \cref{Thm : extendable 1}, suppose that $\mathsf{K}$ is a quasivariety and consider $\A \in \mathsf{K}_\textsc{rsi}$. In view of the first part of the statement of  \cref{Thm : extendable 1}, there exists $\B \in \mathsf{K}$ with $\A \leq \B$ such that $f^\B$ is total for each $f \in \ext(\K)$. By Proposition \ref{Prop : extendability in quasivarieties : si trick} there also exist $\C \in \mathsf{K}_\textsc{rsi}$ with $\A \leq \C$ and a surjective homomorphism $h \colon \B \to \C$.
To conclude the proof, it only remains to show that $f^\C$ is total for each $f \in \ext(\K)$. To this end, consider an $n$-ary $f \in \ext(\K)$ and $c_1, \dots, c_n \in C$. Since $h \colon \B \to \C$ is surjective, there exist $b_1, \dots, b_n \in B$ such that $h(b_j) = c_j$ for each $j \leq n$. 
Recall
that $f^\B$ is total because $f \in \ext(\Kfg, \K)$. Therefore, $\langle b_1, \dots, b_n \rangle \in \mathsf{dom}(f^\B)$. 
As $f$ is an implicit operation and $h$ a homomorphism, we conclude that $\langle c_1, \dots, c_n \rangle = \langle h(b_1), \dots, h(b_n) \rangle \in \mathsf{dom}(f^\C)$.
\end{proof}

The following is a consequence of \cref{Thm : extendable 1}.

\begin{Corollary} \label{Cor : closure under composition for ext}
Given a universal class $\mathsf{K}$, the classes $\mathsf{ext}(\mathsf{K})$ and $\mathsf{ext}_{\textsc{pp}}(\mathsf{K})$ are closed under composition. 
\end{Corollary}
\begin{proof}
    Consider an $n$-ary $g \in \ext(\K)$ and $m$-ary $f_1, \dots, f_n \in \ext(\K)$. 
    Let $\A \in \K$. By \cref{Thm : extendable 1} there exists $\B \in \K$ with $\A \leq \B$ such that $g^\B, f_1^\B, \dots, f_n^\B$ are total. 
    It then follows from the definition of composition that $g(f_1, \dots, f_n)^\B$ is also total.
    Therefore, $\mathsf{ext}(\mathsf{K})$ is closed under composition. 
    As $\imppp(\K)$ is closed under composition by \cref{Prop : closure under composition for imp}, we obtain that $\mathsf{ext}_{\textsc{pp}}(\mathsf{K})$ is also closed under composition 
    because
    $\mathsf{ext}_{\textsc{pp}}(\mathsf{K}) = \mathsf{ext}(\mathsf{K}) \cap \mathsf{imp}_{\textsc{pp}}(\mathsf{K})$. 
\end{proof}

\begin{exa}[\textsf{Cancellative commutative monoids}]\label{Exa : CCM : inverses are extendable}
Recall from \cref{exa:ccmon AP} that the
class of cancellative commutative monoids forms a quasivariety,
which we denote by $\mathsf{CCMon}$.
The importance of cancellative commutative monoids is due to the following well-known result (see, e.g., \cite[pp.~39--40]{Lan84}).

\begin{Theorem}\label{Thm : CCM subreducts of Abelian groups}
The quasivariety of cancellative commutative monoids is the class of monoid subreducts of Abelian groups.
\end{Theorem}

Recall from Theorem \ref{Thm : inverses monoid : implicit operation} that ``taking inverses'' is an implicit operation of the variety of monoids, definable by the conjunction of equations $\varphi = (x \cdot y \thickapprox 1) \sqcap (y \cdot x \thickapprox 1)$. Clearly, its restriction to $\mathsf{CCMon}$ is an implicit operation of $\mathsf{CCMon}$, which is defined by the equation $x \cdot y \thickapprox 1$. We will prove the following.

\begin{Theorem}\label{Thm : inverses in monoids : extendable}
Taking inverses is a unary extendable implicit operation of the quasivariety of cancellative commutative monoids, which, moreover, can be defined by the equation $x \cdot y \thickapprox 1$. 
\end{Theorem}

\begin{proof}
It suffices to prove that the implicit operation $f$ of ``taking inverses'' in $\mathsf{CCMon}$ is extendable. To this end, consider $\A \in \mathsf{CCMon}$ and $a \in A$. In view of Theorem \ref{Thm : CCM subreducts of Abelian groups}, $\A$ is a subreduct of an Abelian group $\B$. Let $\C$ be the monoid reduct of $\B$. Since $\B$ is an Abelian group, $\C$ is a cancellative commutative monoid by Theorem \ref{Thm : CCM subreducts of Abelian groups}. Therefore, $\C \in \mathsf{CCMon}$. Furthermore, $a \in A \subseteq C$ has an inverse in $\C$ because $\C$ is the reduct of a group. Therefore, $a \in \mathsf{dom}(f^\C)$. Hence, we conclude that $f$ is extendable.
\end{proof}

On the other hand, the implicit operation $f$ of ``taking inverses'' in the variety of all monoids
is not extendable. For suppose the contrary, with a view to contradiction. By Theorem \ref{Thm : extendable 1} this implies that for each monoid $\A$ there exists a monoid $\B$ such that $f^\B$ is total, that is, such that $\B$ is the reduct of a group. As a consequence, we obtain that every monoid embeds into the monoid reduct of a group. But this is false because monoid subreducts of groups need to be cancellative and noncancellative monoids exist (e.g., full transformation monoids). We conclude that  $f$ is not extendable. An analogous argument shows that the restriction of $f$ to the variety of commutative monoids is also not extendable.
\qed
\end{exa}

The next results simplify the task of proving that an implicit operation is extendable.

\begin{Proposition}\label{Prop : extendable : sufficient conditions}
Let $\K$ be an elementary class and $\M \subseteq \K$.
The following conditions hold:
\benroman
\item\label{item : extendable : sufficient conditions : 1} if $\mathsf{K} \subseteq \mathbb{U}(\mathsf{M})$, then $\mathsf{ext}(\mathsf{K}) = \ext(\M,\K)$;
\item\label{item : extendable : sufficient conditions : 2} if $\PPP(\mathsf{K}) \subseteq \mathsf{K} \subseteq \mathbb{Q}(\mathsf{M})$, then $\extpp(\mathsf{K}) = \extpp(\M,\mathsf{K})$.
\eroman
\end{Proposition}

\begin{proof}
We begin with the following observation.

\begin{Claim}\label{Claim : extendable : sufficient : Pu}
We have $\ext(\M,\K) \subseteq \ext(\PPU(\M),\K)$.
\end{Claim}

\begin{proof}[Proof of the Claim]
Consider $f \in \ext(\M,\K)$, $\A \in \PPU(\mathsf{M})$ and $a_1, \dots, a_n \in A$. Then there exist $\{ \A_i : i \in I \} \subseteq \mathsf{M}$ and an ultrafilter $U$ on $I$ such that $\A = \prod_{i \in I}\A_i / U$. Moreover, there exist $a_1^*, \dots, a_n^* \in \prod_{i \in I}A_i$ such that $a_i = a_i^* / U$ for each $i \leq n$. Lastly, since $f$ is an  implicit operation of $\mathsf{K}$, it is defined by a formula $\varphi$.

As $\{ \A_i : i \in I \} \subseteq \mathsf{M}$, from the assumptions it follows that for each $i \in I$ there exists $\B_i \in \mathsf{K}$ with $\A_i \leq \B_i$ such that $\langle a_1^*(i), \dots, a_n^*(i) \rangle \in \mathsf{dom}(f^{\B_i})$. Since $\varphi$ defines $f$, this yields $\B_i \vDash \exists y \varphi(a_1^*(i), \dots, a_n^*(i), y)$ for each $i \in I$. Let $\B = \prod_{i \in I}\B_i / U$. By \LL o\'s' Theorem \ref{Thm : Los} we have
\[
\B \vDash \exists y \varphi(a_1^* / U, \dots, a_n^* / U, y).
\]
Observe that $\B \in \mathsf{K}$ because $\{ \B_i : i \in I \} \subseteq \mathsf{K}$ and $\mathsf{K}$ is an elementary class by assumption and, therefore, closed under $\PPU$. Together with the fact that $\varphi$ defines $f$ and the above display, this yields $\langle a_1^* / U, \dots, a_n^* / U \rangle \in \mathsf{dom}(f^\B)$. Lastly, recall that $\A_i \leq \B_i$ for each $i \in I$. As a consequence, the map $h \colon \A \to \B$ defined by the rule $h(a / U) = a / U$ is an embedding and
\[
\langle h(a_1), \dots, h(a_n) \rangle = \langle h(a_1^*/U), \dots, h(a_n^*/U) \rangle = \langle a_1^* / U, \dots, a_n^* / U \rangle \in \mathsf{dom}(f^\B).
\]

As $\mathsf{K}$ is closed under $\III$ (because it is an elementary class), we may assume that $h \colon \A \to \B$ is the inclusion map. Therefore, we obtain that $\A \leq \B \in \mathsf{K}$ and $\langle a_1, \dots, a_n \rangle \in \mathsf{dom}(f^\B)$, as desired.
\end{proof}

(\ref{item : extendable : sufficient conditions : 1}): Suppose that $\mathsf{K} \subseteq \mathbb{U}(\mathsf{M})$. The inclusion from left to right follows from \cref{rem: inclusions ext}. To prove the other inclusion consider $f \in \ext(\M,\K)$,
$\A \in \mathsf{K}$, and $a_1, \dots, a_n \in A$. By \cref{Thm : quasivariety generation} we have $\UUU(\mathsf{M}) = \III\SSS\PPU(\mathsf{M})$. Together with $\A \in \mathsf{K} \subseteq \mathbb{U}(\mathsf{M})$, this yields $\A \in \III\SSS\PPU(\mathsf{M})$. Therefore, there exist $\B \in \PPU(\mathsf{M})$ and an embedding $h \colon \A \to \B$. By Claim \ref{Claim : extendable : sufficient : Pu} there exists also $\C \in \mathsf{K}$ such that $\B \leq \C$ and $\langle h(a_1), \dots, h(a_n) \rangle \in \mathsf{dom}(f^\C)$. 

Since $\C \in \mathsf{K}$ and $\mathsf{K}$ is closed under $\III$, we may assume that $h \colon \A \to \C$ is the inclusion map. Consequently, we obtain that $\A \leq \B \leq \C \in \mathsf{K}$ and $\langle a_1, \dots, a_n \rangle \in \mathsf{dom}(f^\C)$. Hence, we conclude that $f$ is extendable.

(\ref{item : extendable : sufficient conditions : 2}): Suppose that $\PPP(\mathsf{K}) \subseteq \mathsf{K} \subseteq \mathbb{Q}(\mathsf{M})$. The inclusion from left to right follows from \cref{rem: inclusions ext}. To prove the other inclusion consider $f \in \extpp(\M,\K)$, $\A \in \mathsf{K}$, and $a_1, \dots, a_n \in A$. Assume that $f$ is defined by a pp formula $\varphi$.
By Theorem \ref{Thm : quasivariety generation} we have $\mathbb{Q}(\mathsf{M}) = \III\SSS\PPP\PPU(\mathsf{M})$. Together with $\A \in \mathsf{K} \subseteq \mathbb{Q}(\mathsf{M})$, this yields $\A \in \III\SSS\PPP\PPU(\mathsf{M})$. Therefore, there exist $\{\B_i : i \in I \} \subseteq \PPU(\mathsf{M})$ and an embedding $h \colon \A \to \prod_{i \in I}\B_i$. By Claim \ref{Claim : extendable : sufficient : Pu} there exists also $\{ \C_i : i \in I \} \subseteq \mathsf{K}$ such that $\B_i \leq \C_i$ and $\langle p_i(h(a_1)), \dots, p_i(h(a_n)) \rangle \in \mathsf{dom}(f^{\C_i})$ for each $i \in I$. 

Let $\C = \prod_{i \in I}\C_i$. From $\{ \C_i : i \in I \} \subseteq \mathsf{K}$ and the assumption that $\PPP(\mathsf{K}) \subseteq \mathsf{K}$ it follows that $\C \in \mathsf{K}$. Furthermore, $\prod_{i \in I}\B_i \leq \C$ because $\B_i \leq \C_i$ for each $i \in I$. Therefore, $h \colon \A \to \prod_{i \in I}\B_i$ can be viewed as an embedding  $h \colon \A \to \C$. Lastly, recall that $\langle p_i(h(a_1)), \dots, p_i(h(a_n)) \rangle \in \mathsf{dom}(f^{\C_i})$ for each $i \in I$. As $\varphi$ defines $f$, this amounts to $\C_i \vDash \exists y \varphi(p_i(h(a_1)), \dots, p_i(h(a_n)), y)$ for each $i \in I$. As $\varphi$ is a pp formula by assumption, we can apply Theorem \ref{Thm : preservation}(\ref{item : preservation : pp}) to the definition of $\C$, obtaining $\C \vDash \exists y \varphi(h(a_1), \dots, h(a_n), y)$. Since $\C \in \mathsf{K}$, this amounts to $\langle h(a_1), \dots, h(a_n)\rangle \in \mathsf{dom}(f^\C)$.

Since $\C \in \mathsf{K}$ and $\mathsf{K}$ is closed under $\III$, we may assume that $h \colon \A \to \C$ is the inclusion map. Consequently, we obtain that $\A \leq \prod_{i \in I}\B_i \leq \C \in \mathsf{K}$ and $\langle a_1, \dots, a_n \rangle \in \mathsf{dom}(f^\C)$. Hence, we conclude that $f$ is extendable.
\end{proof}

Recall that, given a class $\K$ of algebras, we denote the class of finitely generated members of $\K$ by $\Kfg$.

\begin{Theorem}\label{Thm : ext ext is ext ext for U and Q}
The following conditions hold for a class $\K$ of algebras:
\benroman
\item\label{item : Thm : ext ext is ext ext for U and Q : 1} if $\K$ is a universal class, then $\mathsf{ext}(\mathsf{K}) = \ext(\Kfg,\K)$;
\item\label{item : Thm : ext ext is ext ext for U and Q : 2} if $\K$ is a quasivariety, then $\extpp(\K)=\extpp(\K_\textsc{rsi}^{\textup{fg}},\K)$.
\eroman
\end{Theorem}

\begin{proof}
(\ref{item : Thm : ext ext is ext ext for U and Q : 1}):   Let $\K$ be a universal class and recall from  \cref{Prop : universal class gen by fin gen} that $\K = \UUU(\Kfg)$. Therefore, $\K$ is elementary and $\Kfg \subseteq \K \subseteq \UUU(\Kfg)$. Consequently, we can apply Proposition 
   \ref{Prop : extendable : sufficient conditions}(\ref{item : extendable : sufficient conditions : 1})   to the case where $\M = \Kfg$, obtaining $\ext(\K) = \ext(\Kfg, \K)$.

(\ref{item : Thm : ext ext is ext ext for U and Q : 2}):   The inclusion from left to right is straightforward. 
To prove the other inclusion, observe that
\cref{Prop : quasivariety = Q fin gen SI}  guarantees that $\mathsf{K} = \mathbb{Q}(\K_\textsc{rsi}^{\textup{fg}})$. Moreover, $\PPP(\mathsf{K}) \subseteq \mathsf{K}$ because $\mathsf{K}$ is a quasivariety. Therefore, we can apply \cref{Prop : extendable : sufficient conditions}(\ref{item : extendable : sufficient conditions : 2}), obtaining $\extpp(\K)=\extpp(\K_\textsc{rsi}^{\textup{fg}},\K)$.
\end{proof}

\begin{Corollary}\label{Cor : extendability : FSI : examples}
A pp formula $\varphi(x_1, \dots, x_n, y)$ defines an extendable implicit operation of a quasivariety $\mathsf{K}$ if and only if for each 
 $\A \in \Krsifg$ 
there exists $\B \in \mathsf{K}$ with $\A \leq \B$ such that for all $a_1, \dots, a_n \in B$ there exists a unique $b \in B$ 
such that $\B \vDash \varphi(a_1, \dots, a_n, b)$. The equivalence still holds if we require $\B$ to be a member of $\mathsf{K}_\textsc{rsi}$. 
\end{Corollary}

\begin{proof}
The implication from left to right and the last part of the statement follow from Theorem \ref{Thm : extendable 1}. To prove the implication from right to left, assume that for each 
$\A \in \Krsifg$ 
there exists $\A^* \in \mathsf{K}$ with $\A \leq \A^*$ such that for all $a_1, \dots, a_n \in A^*$ there exists a unique $b \in A^*$ such that $\A^* \vDash \varphi(a_1, \dots, a_n, b)$.

We begin by showing that $\varphi$ defines an implicit operation of $\mathsf{K}$. Let
\[
\mathsf{M} = \{ \A^* :  \A \in \Krsifg  \}.
\]
Observe that $\varphi$ is functional in $\mathsf{M}$ by assumption. As $\varphi$ is a pp formula, we can apply \cref{Cor : functionality in Q(K)}, obtaining that $\varphi$ is functional in $\QQQ(\mathsf{M})$ as well. Since $\mathsf{M} \subseteq \mathsf{K}$ and 
 $\Krsifg \subseteq \SSS(\mathsf{M})$, 
we have $\mathsf{K} = \QQQ(\mathsf{M})$ by Proposition \ref{Prop : quasivariety = Q fin gen SI}. Hence, we conclude that $\varphi$ defines an implicit operation $f$ of $\mathsf{K}$ which, moreover is extendable by Theorem \ref{Thm : ext ext is ext ext for U and Q}(\ref{item : Thm : ext ext is ext ext for U and Q : 2}).
\end{proof}

So far, the only concrete example of an extendable implicit operation that we have encountered is that of ``taking inverses'' in the quasivariety of cancellative commutative monoids (see Example \ref{Exa : CCM : inverses are extendable}). We close this section with five additional examples related to filtral quasivarieties, reduced commutative rings, distributive lattices, Hilbert algebras, and pseudocomplemented distributive lattices.

\begin{exa}[\textsf{Filtral quasivarieties}]
A quasivariety $\mathsf{K}$ is said to be \emph{relatively filtral} when for every subdirect product $\A \leq \prod_{i \in I}\A_i$ with $\{ \A_i : i \in I \} \subseteq \mathsf{K}_\textsc{rsi}$ 
and every $\theta \in \mathsf{Con}_{\mathsf{K}}(\A)$ there exists a filter $F$ on $I$ such that
\[
\theta = \{ \langle a, b \rangle \in A \times A : \llbracket a \thickapprox b \rrbracket \in F\}.
\]
When $\K$ is a variety, we simply say that it is \emph{filtral}. This notion originated in the context of varieties   \cite{Magari1969} and was extended to quasivarieties in \cite{CampRaf}. Examples of filtral varieties include the variety of (bounded) distributive lattices (see, e.g., \cite[Ex.~3]{Ber87}).

 We recall that the \emph{quaternary discriminator} function on a set $A$ is the function $d_A \colon A^4 \to A$ defined for all $a, b, c, d \in A$ as 
\[
d_A(a, b, c, d) = \begin{cases}
			c & \text{ if } a = b;\\
			d & \text{ otherwise.}
		\end{cases}
		\]

We will prove the following.

\begin{Theorem}\label{Thm : relatively filtral quasivarieties : extendability}
Let $\mathsf{K}$ be a relatively filtral quasivariety. Then there exists a quaternary $f \in \exteq(\K)$ such that $f^\A$ is total and coincides with the quaternary discriminator function on $A$ for each $\A \in \mathsf{K}_{\textsc{rsi}}$. 
\end{Theorem}

\begin{proof}
Consider a relatively filtral quasivariety  $\mathsf{K}$. 
From  \cite[Thm.~6.3]{CampRaf} and the implication (4)$ \Rightarrow$(1) in \cite[Thm.~4.1]{CampVaggDisc}
it follows that there exists a conjunction of equations $\varphi(x_1, x_2, x_3, x_4, y)$ such that for all  $\A \in \mathsf{K}_\textsc{rsi}$ and $a, b, c, d, e \in A$,
\[
\A \vDash \varphi(a, b, c, d, e) \iff d_A(a, b, c, d) = e.
\]
Therefore, for all $\A \in \mathsf{K}_\textsc{rsi}$ and $a,b,c,d  \in A$ there exists a unique $e \in A$
such that $\A \vDash \varphi(a,b,c,d,e)$. Then \cref{Cor : extendability : FSI : examples} implies that 
$\varphi$ defines an extendable implicit operation $f$ of $\mathsf{K}$. In view of the above display, $f^\A$ coincides with $d_A$ for each $\A \in \mathsf{K}_\textsc{rsi}$.
\end{proof}
\end{exa}

\begin{exa}[\textsf{Reduced commutative rings}]\label{Exa : reduced rings : extendable} In \cref{Example : inverses in rings} we proved that there exists an implicit operation of the quasivariety $\mathsf{RCRing}$ of reduced commutative rings that coincides with the operation of ``taking weak inverses'' in fields. We now show that this operation is extendable.
\begin{Theorem}\label{Thm : RCRing : extendable implicit operation}
There exists a unary $f \in \exteq(\mathsf{RCRing})$
such that $f^{\A}$ is total and coincides with the operation of taking weak inverses for each field $\A$. 
\end{Theorem}

\begin{proof}
Let $f$ be the implicit operation of 
$\mathsf{RCRing}$
given by \cref{Thm : RCRing : implicit operation}. Then $f^\A$ is total and coincides with the operation of taking weak inverses for every field $\A$. Moreover, $f$ is defined by a conjunction of equations, whence $f \in \imp_{\textsc{eq}}(\mathsf{RCRing})$. Therefore, it suffices to prove that $f$ is extendable. To this end, recall from \cref{Thm : reduced rings are Q fields} that $\mathsf{RCRing} = \QQQ(\mathsf{Field})$, where $\mathsf{Field}$ is the class of fields. As $f^\A$ is total for each field $\A$, we can apply \cref{Prop : extendable : sufficient conditions}(\ref{item : extendable : sufficient conditions : 2}) (taking $\mathsf{M} = \mathsf{Fields}$), obtaining that $f$ is extendable.
\end{proof}
\end{exa}

\begin{exa}[\textsf{Distributive lattices}]\label{Exa : DL : extendable}
Recall from \cref{Example : complements in lattices} that ``taking relative complements'' defines an implicit operation of the variety $\DL$ of distributive lattices and that ``taking complements'' defines an implicit operation of the variety $\bDL$ of bounded distributive lattices.
We show that these operations are extendable.

\begin{Theorem}\label{Thm : distributive lattice : expandable}
The following conditions hold:
\benroman
\item\label{item : DL extendable 1} the operation of taking relative complements in $\DL$ is a ternary member of $\ext_{\textsc{eq}}(\DL)$;
\item\label{item : DL extendable 2} the operation of taking complements in $\bDL$  is a unary member of $\ext_{\textsc{eq}}(\bDL)$.
\eroman
\end{Theorem}

\begin{proof}
We detail only the proof of (\ref{item : DL extendable 1}), as the proof of (\ref{item : DL extendable 2}) is analogous. In view of \cref{Thm : relative complements is implicit}, it suffices to show that the implicit operation $f$ of ``taking relative complements'' of 
$\DL$
is extendable. Let $\boldsymbol{D}_2$ be the two-element lattice and observe that $f^{\boldsymbol{D}_2}$ it total. As $\boldsymbol{D}_2$ is (up to isomorphism) the only subdirectly irreducible member of $\mathsf{DL}$ (see, e.g., \cite[Ex.~3.19]{Ber11}), we obtain that $f^\A$ it total for each $\A \in \mathsf{DL}_\si$. Therefore, we can apply Theorem \ref{Thm : ext ext is ext ext for U and Q}(\ref{item : Thm : ext ext is ext ext for U and Q : 2}), obtaining that $f$ is extendable.
\end{proof}
\end{exa}

\begin{exa}[\textsf{Hilbert algebras}]
 
The implication subreducts of Heyting algebras are known as \emph{Hilbert algebras} (see, e.g., \cite{DHilbA}).
 Hilbert algebras  form a variety (see \cite[Thm.\ 3]{DHilbA}) that we denote by $\mathsf{Hilbert}$.
Every Hilbert algebra $\langle A; \to \rangle$ possesses a term-definable constant $1 = x \to x$ and can be endowed with a partial order $\leq$ defined for all $a, b \in A$ as
\[
a \leq b \iff  a \to b = 1.
\]

We denote the implication reduct of a Heyting algebra $\A$ by $\A_\to$. Notably, the order of $\A_\to$ coincides with the lattice order of $\A$. We will prove the following.

\begin{Theorem}\label{Thm : Hilbert algebras : extendable}
There exists a binary $f \in \ext_{\textsc{eq}}(\mathsf{Hilbert})$
such that $f^{\A_\to}$ is total and coincides with $\land^\A$ for each Heyting algebra $\A$. 
\end{Theorem}

\begin{proof}
Consider the conjunction of equations
\[
\varphi = (y \to x_1 \thickapprox 1) \sqcap (y \to x_2 \thickapprox 1) \sqcap (x_1 \to (x_2 \to y) \thickapprox 1).
\]
We begin with the following observation.

\begin{Claim}\label{Claim : Hilbert algebras : functional} For all Heyting algebras $\A$ and $a, b, c \in A$,
    \[
 \A_\to \vDash \varphi(a,b,c) \iff a \land b = c.   
    \]
    \end{Claim}
    \begin{proof}[Proof of the Claim]
Observe that
\[
\A \vDash (c \to a \thickapprox 1) \sqcap (c \to b \thickapprox 1) \iff (c \leq a \textrm{ and } c \leq b) \iff c \leq a \land b.
\]
    As the equation $x \to (y \to z) \thickapprox x \land y \to z$ holds in every Heyting algebra, we also have 
    \[
    \A \vDash a \to (b \to c) \thickapprox 1 \iff a \land b \to c \thickapprox 1 \iff a \land b \leq c.
    \]
From the definition of $\varphi$ and the two displays above it follows that
\[
        \A \vDash \varphi(a,b,c)  \iff  a \land b = c.
\]
Since $\varphi$ is a formula in the language $\langle \to \rangle$, the demand that $\A \vDash \varphi(a, b, c)$ is equivalent to $\A_\to \vDash \varphi(a, b, c)$. Together with the above display, this yields the desired conclusion.
    \end{proof}
    
Let $\mathsf{M}$ be the class of implication reducts of Heyting algebras. Since 
the variety $\mathsf{Hilbert}$
is the class of implication subreducts of Heyting algebras, we have $\mathsf{Hilbert} = \QQQ(\mathsf{M})$.\ Moreover, observe that $\varphi$ is functional in $\mathsf{M}$ by Claim \ref{Claim : Hilbert algebras : functional}.\ As $\varphi$ is a conjunction of equations, we can apply \cref{Cor : functionality in Q(K)}, obtaining that $\varphi$ defines an implicit operation $f$ of $\mathsf{Hilbert}$. In view of Claim \ref{Claim : Hilbert algebras : functional}, it only remains to show that $f$ is extendable. To this end, consider a Hilbert algebra $\A$ and $a, b \in A$. Then $\A$ is a subreduct of a Heyting algebra $\B$, that is, $\A \leq \B_\to$. Since $f^{\B_\to}$ is total by Claim \ref{Claim : Hilbert algebras : functional}, we conclude that $f$ is extendable.
\end{proof}

\end{exa}

\begin{exa}[\textsf{Pseudocomplemented distributive lattices}]
A \emph{pseudocomplemented distributive lattice} is an algebra $\langle A; \land, \lor, \neg, 0, 1 \rangle$ which comprises a bounded distributive lattice $\langle A; \land, \lor, 0, 1 \rangle$ and a unary operation $\neg$ such that for all $a, b \in A$ we have
\[
a \leq \neg b \iff a \wedge b = 0
\]
(see, e.g., \cite[Sec.~VIII]{BD74}).

This means that $\neg b$ is the largest $a \in A$ such that $a \wedge b = 0$. Consequently, pseudocomplemented distributive lattices are uniquely determined by their lattice reduct.

Given a Heyting algebra $\A$ and $a \in A$, we define $\lnot a = a \to 0$.  The interest of pseudocomplemented distributive lattices derives from the fact that they coincide with the $\langle \land, \lor, \lnot, 0, 1 \rangle$-subreducts of Heyting algebras  (see, e.g., \cite[Proof of Thm.~2.6]{BP89}).

It is well known that the class of pseudocomplemented distributive lattices forms a locally finite variety, which we denote by $\mathsf{PDL}$ (see, e.g., \cite[Thm.~VIII.3.1]{BD74} and \cite[Thm.~4.55]{Ber11}). 
The finitely generated members of $\mathsf{PDL}_{\textsc{si}}$ are precisely the pseudocomplemented distributive lattices whose lattice reduct is a finite Boolean lattice adjoined with a new top element (see \cite[Thm.~2]{LakPDL}). Being a finite distributive lattice, every finitely generated member $\A = \langle A; \land, \lor, \lnot, 0, 1 \rangle$ of $\mathsf{PDL}_{\textsc{si}}$ can be expanded with an implication $\to^\A$ such that $\langle A; \land, \lor, \to^\A, 0, 1 \rangle$ is a Heyting algebra. We will show that this expansion is witnessed by an extendable operation of $\mathsf{PDL}$. More precisely, we will establish the following.

\begin{Theorem}\label{Thm : PDL : extendable implicit operation}
There exists a binary $f \in \ext_{\textsc{eq}}(\mathsf{PDL})$ such that $f^\A$ is total and coincides with $\to^\A$ for each finitely generated $\A \in \mathsf{PDL}_{\textsc{si}}$. 
\end{Theorem}

\begin{proof}
Let $\varphi(x_1,x_2,y)$ be the conjunction of the following equations:
\begin{align}
x_1 \wedge y & \leq x_2;\label{eq:1}\\ 
\neg x_1 \vee x_2 & \leq y;\label{eq:2}\\ 
\neg (\neg x_1 \vee x_2) &= \neg y;\label{eq:3}\\ 
y \vee x_1 &= \neg \neg y \vee x_1.\label{eq:4}
\end{align}
It will be enough to show that for all finitely generated $\A \in \mathsf{PDL}_{\textsc{si}}$ and $a, b, c \in A$,
\begin{equation}\label{Eq : PDL : definition in finite SI}
\A \vDash \varphi(a, b, c) \iff a \to^\A b = c.
\end{equation}
For suppose this is true.\ Then $\varphi$ defines an extendable implicit operation $f$ of $\mathsf{PDL}$ by Corollary \ref{Cor : extendability : FSI : examples}. In addition, the above display guarantees that $f^\A$ is total and coincides with $\to^\A$ for each finitely generated $\A \in \mathsf{PDL}_\textsc{si}$, as desired.

 We proceed to prove 
(\ref{Eq : PDL : definition in finite SI}).   Let $\A \in \mathsf{PDL}_\textsc{si}$ be finitely generated. Then the lattice reduct of $\A$ is a finite Boolean lattice $\B$ adjoined with a new top element. We denote the maximum of $\B$ by $\top$, while the minimum and the maximum of $\A$ are $0$ and $1$, respectively. We also write $a'$ for the complement of $a \in B$ in the Boolean lattice $\B$.
For all $a,b\in A$ we have
    \begin{equation}\label{eq: to}
    a \to^\A b =
    \begin{cases}
        a' \vee b  & \text{if } a \in B \text{ and } a \nleq b;\\
        1 & \text{if } a \leq b;\\
        b  & \text{if } a = 1.\\
    \end{cases}        
    \end{equation}
It follows that
\begin{equation}\label{eq: neg and negneg}
\begin{aligned}
    \neg a &=
    \begin{cases}
        a' & \text{if } a \in B - \{0\};\\
        1  & \text{if } a=0;\\
        0  & \text{if } a = 1.
    \end{cases}
    \hspace{2cm} &
    \neg \neg a &=
    \begin{cases}
        a & \text{if } a \in B - \{\top\};\\
        1  & \text{if } a \in \{\top,1\}.
    \end{cases}
\end{aligned}    
\end{equation}
Therefore,
\begin{equation}\label{eq: neg vee}
    \neg a \vee b =
    \begin{cases}
        a' \vee b & \text{if } a \in B - \{0\};\\
        1  & \text{if } a=0;\\
        b  & \text{if } a = 1.
    \end{cases}    
\end{equation}

To prove the equivalence in (\ref{Eq : PDL : definition in finite SI}), consider $a,b,c \in A$. We begin with the implication from right to left. It suffices to verify that $\A \vDash \varphi(a,b,a \to^\A b)$, which amounts to check that \eqref{eq:1}--\eqref{eq:4} 
hold in $\A$ once evaluated in $a,b,a \to^\A b$. It is a straightforward consequence of the properties of implications in Heyting algebras that $a \wedge (a \to^\A b) \leq b$ and $\neg a \vee b \leq a \to^\A b$
(see, e.g., \cite[Thm.~IX.2.3(i, iv, v)]{BD74}). So, \eqref{eq:1} and \eqref{eq:2} hold in $\A$. 
From \cite[Thm.~IX.2.3(ix, xi)]{BD74} it follows that $\neg (\neg a \vee b) = \neg \neg a \wedge \neg b =  \neg (a \to^\A b)$, and hence \eqref{eq:3} holds.
It only remains to verify \eqref{eq:4}, which states that $(a \to^\A b) \vee a = \neg \neg (a \to^\A b) \vee a$. First observe that the equation clearly holds when $a \in \{0,1\}$. So, we can assume that $a \notin \{0,1\}$.
By \eqref{eq:3}, \eqref{eq: neg vee}, and \eqref{eq: neg and negneg} we have 
\[
\neg \neg (a \to^\A b)=\neg \neg (\neg a \vee b) = \neg \neg (a' \vee b) =
    \begin{cases}
        a' \vee b & \text{if } a \nleq b;\\
        1 & \text{if } a \leq b,
    \end{cases}
\]
which coincides with $a \to^\A b$ by \eqref{eq: to} because $a \neq 1$. Thus, $ (a \to^\A b) \vee a = \neg \neg (a \to^\A b) \vee a$, and hence \eqref{eq:4} holds.

We proceed to prove the implication from left to right in (\ref{Eq : PDL : definition in finite SI}). To this end,
assume that $\A \vDash \varphi(a,b,c)$. We will show that 
$a \to^\A b = c$. From~\eqref{eq:1} it follows that $a \wedge c \leq b$, which yields $c \leq a \to^\A b$.
It then remains to show that $a \to^\A b \leq c$. We consider different cases separately. 
If $a=1$, then \eqref{eq:2} implies that $a \to^\A b = b=\neg a \vee b \leq c$. So, we can 
assume that $a \neq 1$. Consider the case in which $a \leq b$. Then
\[
c \vee a = \neg \neg c \vee a = \neg \neg (\neg a \vee b) \vee a = 1 \vee a = 1,
\]
where the first and second equalities follow from \eqref{eq:4} and \eqref{eq:3}, and the third from \eqref{eq: neg and negneg} because $\neg a \vee b \in \{\top,1\}$ as $a \leq b$. We have thus obtained that $c \vee a = 1$. Since $\A$ has a second largest element $\top$, we have that $1$ is join irreducible in $\A$. So, $c \vee a = 1$ and $a \neq 1$ imply
$c=1$, and hence $c=a \to^\A b$, because $a \leq b$. Finally, we can assume that $a \neq 1$ and $a \nleq b$. Then $a,b \neq 1$, and so $a,b \in B$. Thus, $\neg a \vee b \notin \{\top,1\}$ because $a \nleq b$. Then \eqref{eq: neg and negneg} yields $\neg \neg (\neg a \vee b)=\neg a \vee b$. 
So, \eqref{eq:3} implies that $\neg \neg c=\neg a \vee b$. Therefore, $c \leq \neg a \vee b \in B - \{ \top\}$, and hence $c \in B - \{\top\}$. Then $c = \neg \neg c = \neg a \vee b$. Since $a \nleq b$ and $a \neq 1$, we have $a \in B - \{0\}$ and, therefore, $\neg a \vee b = a' \vee b = a \to^\A b$ by \eqref{eq: neg vee} and~\eqref{eq: to}. 
Consequently, $c = \neg a \vee b = a \to^\A b$. This establishes that $c = a \to^\A b$ in all possible cases and concludes the proof.
\end{proof}
\end{exa}

\section{Adding implicit operations}\label{Sec: adding implicit operations}

A fundamental question 
which arises in relation to implicit operations 
is the following:\ 
is it possible to expand the language of a given class of algebras $\mathsf{K}$ with new function symbols for some of its implicit operations so that every implicit operation of $\mathsf{K}$ becomes interpolable by a set of terms in a class $\mathsf{M}$ of algebras in the expanded language?
Obviously, the interest of this possibility depends on whether 
$\M$ meets some basic desiderata:
 in addition to the demand that every implicit operation of $\mathsf{K}$ can be interpolated by a set of terms 
of $\M$,
we shall demand from our theory that
\benroman
\renewcommand{\labelenumi}{(\theenumi)}
\renewcommand{\theenumi}{D\arabic{enumi}}
\item\label{D1 : desiderata 1} every member of $\mathsf{K}$ extends to one of 
$\mathsf{M}$;
\item\label{D1 : desiderata 2} when such a class $\M$ 
exists, it is unique.
\eroman

A familiar example of this situation is given by
\begin{align*}
\mathsf{K} &= \text{the quasivariety of cancellative commutative monoids};\\
\mathsf{M} &= \text{the variety of Abelian groups}.
\end{align*}
In this case, $\mathsf{M}$ is obtained by adding the implicit operation of ``taking inverses'' to $\mathsf{K}$.
The fact that every implicit operation of $\mathsf{K}$ is
interpolated by a set of terms of $\mathsf{M}$ is
a consequence of $\mathsf{M}$ having the strong Beth definability property (see \cref{Exa : abalian groups : SES}). Furthermore,  (\ref{D1 : desiderata 1}) holds because every cancellative commutative monoid extends to an Abelian group by \cref{Thm : CCM subreducts of Abelian groups}.
Lastly, (\ref{D1 : desiderata 2}) will be a consequence of the general theory (Theorem \ref{Thm : Beth companions : term equivalence}). On the other hand, we will show that the variety of all commutative monoids lacks an expansion with the desired properties (\cref{Thm : M lacks Beth comp}).

In this section, we begin to set the stage for this theory by describing how to expand a class of algebras with a given family of implicit operations.

\begin{Definition}
Let $\mathsf{K}$ be a class of algebras and $\mathcal{F} \subseteq \mathsf{imp}(\mathsf{K})$.  An $\mathcal{F}$-\emph{expansion} of $\mathscr{L}_\mathsf{K}$ is a language $\L_\K \cup \{ g_f : f \in \mathcal{F} \}$, where $g_f$ is a new function symbol of the same arity as that of $f$ for each  $f \in \mathcal{F}$. 
We will often denote an $\mathcal{F}$-expansion of $\L_\K$ by $\mathscr{L}_\F$. When $\mathcal{F} = \{f\}$ for some $f \in \mathsf{imp}(\mathsf{K})$, we drop the braces
and just write $\mathscr{L}_f$ and call it an \emph{$f$-expansion}.
\end{Definition}

\begin{Definition}
Let $\mathsf{K}$ be a class of algebras, $\mathcal{F} \subseteq \mathsf{imp}(\mathsf{K})$, and  $\mathscr{L}_\F$ an $\mathcal{F}$-expansion of $\L_\K$.
\benroman
\item For each $\A \in \mathsf{K}$ such that $f^{\A}$ is total for every $f \in \mathcal{F}$,  let
\begin{align*}
\A[\mathscr{L}_\mathcal{F}] = &\text{ the unique $\mathscr{L}_\mathcal{F}$-algebra whose $\L_\K$-reduct is $\A$ and in which}\\
&\text{ $g_f$ is interpreted as $f^\A$ for each $f \in \mathcal{F}$}.
\end{align*}
\item For each $\M \subseteq \K$ let 
\[
\M[\mathscr{L}_\mathcal{F}] = \{ \A[\mathscr{L}_\mathcal{F}] : \A \in \M \text{ and }f^\A \text{ is total for each }f \in \mathcal{F}\}.
\]
\eroman
\end{Definition}

The class $\mathsf{K}[\mathscr{L}_\mathcal{F}]$ can be viewed as the natural expansion of $\mathsf{K}$ induced by the implicit operations in $\mathcal{F}$.  The next result provides an alternative description of the class $\mathsf{K}[\mathscr{L}_\mathcal{F}]$.

\begin{Proposition}\label{Prop : how to axiomatize K[F]}
Let $\mathsf{K}$ be a class of algebras, $\mathcal{F} \subseteq \mathsf{imp}(\mathsf{K})$, and $\mathscr{L}_\mathcal{F} = \mathscr{L}_\mathsf{K} \cup \{ g_f : f \in \mathcal{F} \}$ an $\mathcal{F}$-expansion of $\mathscr{L}_\mathsf{K}$. Assume that each $f  \in \mathcal{F}$ is defined by a formula $\varphi_f$. Then
\begin{align*}
\mathsf{K}[\mathscr{L}_\mathcal{F}] = \{ \B : &\text{ }\B \text{ is an 
$\mathscr{L}_\mathcal{F}$-algebra}
\text{ such that }\B{\upharpoonright}_{\mathscr{L}_\mathsf{K}} \in \mathsf{K} \text{ and}\\
&\text{ $\B \vDash \varphi_f(x_1, \dots, x_n, g_f(x_1, \dots, x_n))$ for each $n$-ary $f \in \mathcal{F}$} \}.
\end{align*}
\end{Proposition}

\begin{proof}
To prove the inclusion from left to right, consider $\B \in \mathsf{K}[\mathscr{L}_\mathcal{F}]$.
The definition of $\mathsf{K}[\mathscr{L}_\mathcal{F}]$ guarantees that
there exists $\A \in \mathsf{K}$ such that $f^\A$ is total for all $f \in \mathcal{F}$ and $\B = \A[\mathscr{L}_\mathcal{F}]$.  In particular, $\B$ is an $\mathscr{L}_\mathcal{F}$-algebra.
As $\A$ is the $\mathscr{L}_\mathsf{K}$-reduct of $\A[\mathscr{L}_\mathcal{F}]$ by the definition of $\A[\mathscr{L}_\mathcal{F}]$, this yields $\B{\upharpoonright}_{\mathscr{L}_\mathsf{K}} = \A \in \mathsf{K}$. Then consider an $n$-ary $f \in \mathcal{F}$. Since $f$ is defined by $\varphi_f$ and $f^\A$ is total, for all $a_1, \dots, a_n \in A$ we have
\[
\langle a_1, \dots, a_n \rangle \in \mathsf{dom}(f^\A) \, \, \text{ and } \, \, \A \vDash \varphi_f(a_1, \dots, a_n, f^\A(a_1, \dots, a_n)).
\]
By the definition of $\A[\mathscr{L}_\mathcal{F}]$ the operation $g_f$ is interpreted in $\A[\mathscr{L}_\mathcal{F}]$ as $f^\A$. As $\A$ is the $\mathscr{L}_\mathsf{K}$-reduct of $\A[\mathscr{L}_\mathcal{F}]$, from the above display it follows that $\A[\mathscr{L}_\mathcal{F}] \vDash \varphi_f(x_1, \dots, x_n, g_f(x_1, \dots, x_n))$. Since $\B = \A[\mathscr{L}_\mathcal{F}]$, we conclude that $\B \vDash \varphi_f(x_1, \dots, x_n, g_f(x_1, \dots, x_n))$, as desired.

 Then we proceed to prove the inclusion from right to left. Consider an $\mathscr{L}_\mathcal{F}$-algebra $\B$   
such that $\B{\upharpoonright}_{\mathscr{L}_\mathsf{K}} \in \mathsf{K}$ and  $\B \vDash \varphi_f(x_1, \dots, x_n, g_f(x_1, \dots, x_n))$ for each $n$-ary $f \in \mathcal{F}$. For the sake of readability, let $\A = \B{\upharpoonright}_{\mathscr{L}_\mathsf{K}}$ and observe that $\A\in \mathsf{K}$ by assumption. We will prove that the algebra $\A[\mathscr{L}_\mathcal{F}]$ is defined and coincides with $\B$, whence $\B \in \mathsf{K}[\mathscr{L}_\mathcal{F}]$, as desired. Since $\A = \B{\upharpoonright}_{\mathscr{L}_\mathsf{K}}$ and $\A \in \mathsf{K}$, it suffices to show that for each $f \in \mathcal{F}$ the function $f^\A$ is total and coincides with the interpretation of $g_f$ in $\B$. To this end, consider an $n$-ary $f \in \mathcal{F}$ and $a_1, \dots, a_n \in A$. We need to prove that 
\[
\langle a_1, \dots, a_n \rangle \in \mathsf{dom}(f^\A) \, \, \text{ and } \, \, g_f^\B(a_1, \dots, a_n) = f^\A(a_1, \dots, a_n).
\]
First, from the assumption that $\B \vDash \varphi_f(x_1, \dots, x_n, g_f(x_1, \dots, x_n))$ it follows that $\B \vDash \varphi_f(a_1, \dots, a_n, g_f^\B(a_1, \dots, a_n))$. As $\varphi_f$ is a formula of $\mathscr{L}_\mathsf{K}$ and $\A = \B{\upharpoonright}_{\mathscr{L}_\mathsf{K}}$, this amounts to $\A \vDash \varphi_f(a_1, \dots, a_n, g_f^\B(a_1, \dots, a_n))$. Since $\varphi_f$ defines $f$, the above display holds.
\end{proof}

As a consequence of Proposition \ref{Prop : how to axiomatize K[F]}, we obtain the following.

\begin{Corollary}\label{Cor : K[F] is elementary if K is}
Let $\mathsf{K}$ be a class of algebras, $\mathcal{F} \subseteq \mathsf{imp}(\mathsf{K})$, and $\mathscr{L}_\mathcal{F}$ an $\mathcal{F}$-expansion of $\mathscr{L}_\mathsf{K}$. If $\mathsf{K}$ is an elementary class, then $\mathsf{K}[\mathscr{L}_\mathcal{F}]$ is an elementary class.
\end{Corollary}

A useful feature of $\mathcal{F}$-expansions is that a map between members of $\mathsf{K}[\mathscr{L}_\mathcal{F}]$ 
preserves the operations in $\F$ if it preserves the operations in $\L_\K$.
This is made precise in the next proposition.

\begin{Proposition}\label{Prop : small homs are big homs}
Let $\mathsf{K}$ be a class of algebras, $\mathcal{F} \subseteq \mathsf{imp}(\mathsf{K})$, and $\mathscr{L}_\mathcal{F}$ an $\mathcal{F}$-expansion of $\mathscr{L}_\mathsf{K}$. Every homomorphism $h \colon \A{\upharpoonright}_{\mathscr{L}_{\mathsf{K}}} \to \B{\upharpoonright}_{\mathscr{L}_{\mathsf{K}}}$ with $\A, \B \in \mathsf{K}[\mathscr{L}_\mathcal{F}]$ is also a homomorphism $h \colon \A \to \B$.
\end{Proposition}

\begin{proof}
As $\mathscr{L}_\mathcal{F}$ is an $\mathcal{F}$-expansion of $\mathscr{L}_\mathsf{K}$, it is of the form $\mathscr{L}_\mathsf{K} \cup \{ g_f : f \in \mathcal{F} \}$. Then let $h \colon \A{\upharpoonright}_{\mathscr{L}_{\mathsf{K}}} \to \B{\upharpoonright}_{\mathscr{L}_{\mathsf{K}}}$ be a homomorphism with $\A, \B \in \mathsf{K}[\mathscr{L}_\mathcal{F}]$. It suffices to prove that $h$ preserves $g_f$ for each $f \in \mathcal{F}$. To this end, consider $f \in \mathcal{F}$. Since $\A, \B \in \mathsf{K}[\mathscr{L}_\mathcal{F}]$, we have $\A = \A{\upharpoonright}_{\mathscr{L}_{\mathsf{K}}}[\mathscr{L}_\mathcal{F}]$ and $\B = \B{\upharpoonright}_{\mathscr{L}_{\mathsf{K}}}[\mathscr{L}_\mathcal{F}]$.
Therefore,
\[
g_f^\A = f^{\A{\upharpoonright}_{\mathscr{L}_{\mathsf{K}}}} \, \, \text{ and } \, \, g_f^\B = f^{\B{\upharpoonright}_{\mathscr{L}_{\mathsf{K}}}}.
\]
As $h \colon \A{\upharpoonright}_{\mathscr{L}_{\mathsf{K}}} \to \B{\upharpoonright}_{\mathscr{L}_{\mathsf{K}}}$ is a homomorphism between members of $\mathsf{K}$ and $f$ an implicit operation of $\mathsf{K}$, we know that $h$ preserves $f$. In view of the above display, we conclude that $h$ preserves $g_f$.
\end{proof}

In general, there is no guarantee that each member of $\mathsf{K}$ is a subreduct of a member of $\mathsf{K}[\mathscr{L}_\mathcal{F}]$ or, equivalently, that condition (\ref{D1 : desiderata 1}) is met.\ In the setting of universal classes, this is the case exactly when the members of $\mathcal{F}$ are extendable.

\begin{Proposition}\label{Prop : pp expansions : subreducts : extendable}
Let $\mathsf{K}$ be a universal class, $\mathcal{F} \subseteq \mathsf{imp}(\mathsf{K})$, and $\mathscr{L}_\mathcal{F}$ an $\mathcal{F}$-expansion of $\mathscr{L}_\mathsf{K}$. Then $\mathsf{K}$ is the class of $\mathscr{L}_\mathsf{K}$-subreducts of $\mathsf{K}[\mathscr{L}_\mathcal{F}]$ if and only if $\mathcal{F} \subseteq \mathsf{ext}(\mathsf{K})$.
\end{Proposition}

\begin{proof}
We begin with the implication from left to right.\ Suppose that $\mathsf{K}$ is the class of $\mathscr{L}_\mathsf{K}$-subreducts of $\mathsf{K}[\mathscr{L}_\mathcal{F}]$ and consider an $n$-ary $f \in \mathcal{F}$. We need to prove that $f$ is extendable. To this end, consider $\A \in \mathsf{K}$. By assumption $\A$ is a subreduct of some $\C \in \mathsf{K}[\mathscr{L}_\mathcal{F}]$. By the definition of $\mathsf{K}[\mathscr{L}_\mathcal{F}]$ there exists $\B \in \mathsf{K}$ in which $g^\B$ is total for each $g \in \mathcal{F}$ such that $\C = \B[\mathscr{L}_\mathcal{F}]$. As $\A$ is an $\mathscr{L}_\mathsf{K}$-subreduct of $\C = \B[\mathscr{L}_\mathcal{F}]$ and $\B$ is the $\mathscr{L}_\mathsf{K}$-reduct of $\B[\mathscr{L}_\mathcal{F}]$, we obtain $\A \leq \B$. Furthermore, $f^\B$ is total because $f \in \mathcal{F}$. Since $\A \leq \B \in \mathsf{K}$, we conclude that $f$ is extendable.

Then
 we proceed  to prove the implication from right to left. Suppose that $\mathcal{F} \subseteq \mathsf{ext}(\mathsf{K})$ and consider $\A \in \mathsf{K}$. As $\mathsf{K}$ is a universal class, we can apply Theorem \ref{Thm : extendable 1}, obtaining some $\B \in \mathsf{K}$ with $\A \leq \B$ such that $f^\B$ is total for each $f \in \mathcal{F}$. By the definition of $\mathsf{K}[\mathscr{L}_\mathcal{F}]$ we get $\B[\mathscr{L}_\mathcal{F}] \in \mathsf{K}[\mathscr{L}_\mathcal{F}]$. Since $\B$ is the $\mathscr{L}_\mathsf{K}$-reduct of $\B[\mathscr{L}_\mathcal{F}]$ and $\A \leq \B$, the algebra $\A$ is an $\mathscr{L}_\mathsf{K}$-subreduct of a member of $\mathsf{K}[\mathscr{L}_\mathcal{F}]$.
\end{proof}

We close this section with some observations
governing the behavior of  $\mathsf{K}[\mathscr{L}_\mathcal{F}]$ with respect to class operators.

\begin{Proposition}\label{Prop : closure under op}
Let $\mathsf{K}$ be a class of algebras and $\mathcal{F} \subseteq \mathsf{imp}_{\textsc{pp}}(\mathsf{K})$.
Moreover, let
\[
\mathsf{M} = \{ \A \in \mathsf{K} : f^\A \text{ is total for each }f \in \mathcal{F} \}.
\] 
Then for each $\mathbb{O} \in \{ \HHH, \PPP, \PPU\}$ we have 
\[
\mathbb{O}(\mathsf{M}) \cap \mathsf{K} \subseteq \mathsf{M}.
\]
\end{Proposition}
\begin{proof}
In order to prove that $\mathbb{O}(\mathsf{M}) \cap \mathsf{K} \subseteq \mathsf{M}$,  
consider an $n$-ary $f \in \mathcal{F}$, $\A \in \mathbb{O}(\mathsf{M}) \cap \mathsf{K}$, and $a_1, \dots, a_n \in A$. Since $\mathcal{F} \subseteq \imppp(\K)$ by assumption, there exists a pp formula 
 $\varphi(x_1, \dots, x_n, y)$ defining $f$. 
We need to show that $\langle a_1, \dots, a_n \rangle \in \mathsf{dom}(f^\A)$,
which is equivalent to $\A \vDash \exists y \varphi(a_1, \dots, a_n, y)$.
The definition of $\mathsf{M}$ guarantees that
\begin{equation}\label{Eq : pp expansions : skjdsklqdslkjqdwqqq}
\mathsf{M} \vDash \exists y \varphi(x_1, \dots, x_n, y).
\end{equation}

We have three cases depending on whether $\mathbb{O}$ is $\HHH$, $\PPP$, or $\PPU$. 
First consider the case where $\mathbb{O} = \mathbb{H}$. Then $\A \in \mathbb{H}(\mathsf{M})$ implies that there exist $\B \in \mathsf{M}$ and a surjective homomorphism $h \colon \B \to \A$. Let $b_1, \dots, b_n \in B$ be such that $h(b_i) = a_i$ for each $i \leq n$. From \eqref{Eq : pp expansions : skjdsklqdslkjqdwqqq} and $\B \in \mathsf{M}$ it follows that $\B \vDash \exists y \varphi(b_1, \dots, b_n,y)$. Since $\varphi$ is a pp formula by assumption, so is $\exists y \varphi$. Therefore, we can apply Theorem \ref{Thm : preservation}(\ref{item : preservation : pp}) to obtain that $\A \vDash \exists y \varphi(h(b_1), \dots, h(b_n),y)$. As $h(b_i) = a_i$ for each $i \leq n$, it follows that $\A \vDash \exists y \varphi(a_1, \dots, a_n,y)$, as desired.

Then we consider the case where $\mathbb{O} = \PPP$. From $\A \in \PPP(\mathsf{M})$ it follows that $\A = \prod_{i \in I}\A_i$ for some family $\{ \A_i : i \in I \} \subseteq \mathsf{M}$. In view of (\ref{Eq : pp expansions : skjdsklqdslkjqdwqqq}), we have
\[
\A_i \vDash \exists y \varphi(p_i(a_1), \dots, p_i(a_n), y) \text{ for every }i \in I.
\]
Since $\varphi$ is a pp formula by assumption, so is $\exists y \varphi$. Therefore, we can apply Theorem \ref{Thm : preservation}(\ref{item : preservation : pp}) to the above display, 
obtaining $\A \vDash \exists y \varphi(a_1, \dots, a_n, y)$, as desired. 

Lastly, we consider the case where $\mathbb{O} = \PPU$. From $\A \in \PPU(\mathsf{M})$ it follows that $\A = \prod_{i \in I}\A_i / U$ for some family $\{ \A_i : i \in I \} \subseteq \mathsf{M}$ and ultrafilter $U$ on $I$. Since $\{ \A_i : i \in I \} \subseteq \mathsf{M}$, we can apply \LL o\'s' Theorem \ref{Thm : Los} to (\ref{Eq : pp expansions : skjdsklqdslkjqdwqqq}), 
obtaining $\A \vDash \exists y \varphi(x_1, \dots, x_n, y)$. Thus, $\A \vDash \exists y \varphi(a_1, \dots, a_n, y)$. This concludes the proof.
\end{proof}

\begin{Proposition}\label{Prop : K[F] class operators}
Let $\mathsf{K}$ be a class of algebras,  $\mathcal{F} \subseteq \mathsf{imp}_{\textsc{pp}}(\mathsf{K})$, and $\mathscr{L}_\mathcal{F}$ an $\mathcal{F}$-expansion of $\mathscr{L}_\mathsf{K}$. Moreover, let
\[
\mathsf{M} = \{ \A \in \mathsf{K} : f^\A \text{ is total for each }f \in \mathcal{F} \}.
\] 
The following conditions hold for all $\mathsf{N} \subseteq \mathsf{M}$ and class operators $\mathbb{O}$ such that $\mathbb{O}(\mathsf{K}) \subseteq \mathsf{K}$:
\benroman
\item \label{item : K[K] closed under I Pu P} if $\mathbb{O} \in \{ \III, \HHH, \PPP, \PPU\}$, then
\[
\mathbb{O}(\mathsf{N}[\mathscr{L}_\mathcal{F}]) = (\mathbb{O}(\mathsf{N}))[\mathscr{L}_\mathcal{F}];
\]
\item\label{item : K[F] closed SSS} if $\mathbb{O} = \SSS$ and $\F \subseteq \impeq(\K)$, then
\[
\SSS(\mathsf{N}[\mathscr{L}_\mathcal{F}]) \subseteq (\SSS(\mathsf{N}))[\mathscr{L}_\mathcal{F}].
\]
\eroman
\end{Proposition}
\begin{proof}
As $\mathscr{L}_\mathcal{F}$ is an $\mathcal{F}$-expansion of $\mathscr{L}_\mathsf{K}$, it is of the form $\mathscr{L}_\mathsf{K} \cup \{ g_f : f \in \mathcal{F} \}$. 
This fact will be used repeatedly without further notice.

\eqref{item : K[K] closed under I Pu P}: Assume that $\mathsf{N} \subseteq \mathsf{M}$ and let $\mathbb{O} \in \{\III, \HHH, \PPP, \PPU\}$ be such that $\mathbb{O}(\mathsf{K}) \subseteq \mathsf{K}$. 
We first establish the following. 

\begin{Claim} \label{Claim : Well defined expansion}
For every $\A \in \mathsf{N} \cup \mathbb{O}(\mathsf{N})$ the algebra $\A[\mathscr{L}_\mathcal{F}]$ is defined. 
\end{Claim}
\begin{proof} [Proof of the Claim]
The definition of $\mathsf{M}$ guarantees that $\A[\mathscr{L}_\mathcal{F}]$ is defined for each $\A \in \mathsf{M}$. It then suffices to show that $\mathsf{N} \cup \mathbb{O}(\mathsf{N}) \subseteq \M$. As $\mathsf{N} \subseteq \mathsf{M}$ holds by assumption, it remains to prove that $\mathbb{O}(\mathsf{N}) \subseteq \M$.
From $\mathsf{N} \subseteq \M \subseteq \K$ and $\mathbb{O}(\K) \subseteq \K$ it follows that $\mathbb{O}(\mathsf{N}) \subseteq \mathbb{O}(\M)$ and $\mathbb{O}(\mathsf{N}) \subseteq \mathbb{O}(\K) \subseteq \K$. Therefore, $\mathbb{O}(\mathsf{N}) \subseteq \mathbb{O}(\M) \cap \K$. Since \cref{Prop : closure under op} yields $\mathbb{O}(\M) \cap \K \subseteq \M$, we conclude that $\mathbb{O}(\mathsf{N}) \subseteq \M$.
\end{proof}

We have to consider four
cases depending on whether $\mathbb{O}$ is $\III$, $\HHH$, $\PPP$, or $\PPU$.  We will start with the cases where $\mathbb{O}=\III$ and $\mathbb{O} = \HHH$, which can be treated simultaneously. Suppose that $\mathbb{O} \in \{\III,\HHH\}$.
To prove that $\mathbb{O}(\mathsf{N}[\mathscr{L}_\mathcal{F}]) \subseteq (\mathbb{O}(\mathsf{N}))[\mathscr{L}_\mathcal{F}]$ consider $\A \in \mathbb{O}(\mathsf{N}[\mathscr{L}_{\mathcal{F}}])$.
Then there exist $\B \in \mathsf{N}[\mathscr{L}_{\mathcal{F}}]$ and a surjective homomorphism $h \colon \B \to \A$, which we can assume to be an isomorphism when $\mathbb{O} = \mathbb{I}$. 
As $\B\resLK \in \mathsf{N}$ and $h$ is also a homomorphism $h \colon \B\resLK \to \A\resLK$, which is an isomorphism when $\mathbb{O}=\III$, we obtain $\A\resLK \in \mathbb{O}(\mathsf{N})$. Hence, $\A\resLK[\L_\F]$ is defined by \cref{Claim : Well defined expansion} and, therefore, $\A\resLK[\L_\F] \in (\mathbb{O}(\mathsf{N}))[\mathscr{L}_{\mathcal{F}}]$.
To show that $\A \in (\mathbb{O}(\mathsf{N}))[\mathscr{L}_{\mathcal{F}}]$, it is then sufficient to prove that $\A = \A\resLK[\L_\F]$. 
Let $f \in \F$ be $n$-ary, $a_1, \dots, a_n \in A$, and $b_1, \dots, b_n \in B$ be such that $h(b_i) = a_i$ for each $i \leq n$.
We have
\begin{flalign*}
    g_f^{\A}(a_1, \dots, a_n) & = g_f^{\A}(h(a_1), \dots, h(a_n)) = h(g_f^{\B}(b_1, \dots, b_n))\\
    & = h(f^{\B \resLK}(b_1, \dots, b_n)) = f^{\A \resLK}(h(b_1), \dots, h(b_n))\\
    & = f^{\A \resLK}(a_1, \dots, a_n),
\end{flalign*}
where the first and last equalities hold because $a_i = h(b_i)$ for each $i \leq n$, the second and fourth hold because $h \colon \B \to \A$ and $h \colon \B\resLK \to \A\resLK$ are homomorphisms and $f \in \imp(\K)$, and the third holds because $g_f^{\B} = f^{\B \resLK}$ by the definition of $g_f^{\B}$. 
Together with the fact that $\A$ and $\A \resLK[\mathscr{L}_{\mathcal{F}}]$ have the same $\L_\K$-reduct, the above display yields $\A = \A \resLK[\mathscr{L}_{\mathcal{F}}]$, as desired.

For the reverse inclusion, consider $\A[\mathscr{L}_{\mathcal{F}}] \in \mathbb{O}(\mathsf{N})[\mathscr{L}_{\mathcal{F}}]$.
Then there exist $\B \in \mathsf{N}$ and a surjective homomorphism $h \colon \B \to \A$, which is an isomorphism when $\mathbb{O}=\III$. 
Since $\B \in \mathsf{N}$, \cref{Claim : Well defined expansion} yields that $\B[\mathscr{L}_{\mathcal{F}}]$ is defined and belongs to $\mathsf{N}[\mathscr{L}_{\mathcal{F}}]$. 
By \cref{Prop : small homs are big homs}, $h \colon \B[\mathscr{L}_{\mathcal{F}}] \to \A[\mathscr{L}_{\mathcal{F}}]$ is a homomorphism. Thus, we conclude that $\A[\mathscr{L}_{\mathcal{F}}] \in \mathbb{O}(\mathsf{N}[\mathscr{L}_{\mathcal{F}}])$.

Then we consider the case where $\mathbb{O} = \PPP$. From the definitions of $\PPP(\mathsf{N}[\mathscr{L}_\mathcal{F}])$ and $(\PPP(\mathsf{N}))[\mathscr{L}_\mathcal{F}]$ it follows that 
\begin{align*}
\A \in \PPP(\mathsf{N}[\mathscr{L}_\mathcal{F}])&\iff \A = \prod_{i \in I}(\A_i[\mathscr{L}_\mathcal{F}]) \text{ for some }\{ \A_i : i \in I \} \subseteq \mathsf{N};\\
\A \in (\PPP(\mathsf{N}))[\mathscr{L}_\mathcal{F}] & \iff \A = \Big(\prod_{i \in I}\A_i\Big)[\mathscr{L}_\mathcal{F}] \text{ for some }\{ \A_i : i \in I \} \subseteq \mathsf{N}.
\end{align*}
Therefore, to conclude that $\PPP(\mathsf{N}[\mathscr{L}_\mathcal{F}]) = (\PPP(\mathsf{N}))[\mathscr{L}_\mathcal{F}]$, it suffices to show that for every family $\{ \A_i : i \in I \} \subseteq \mathsf{N}$,
\[
\prod_{i \in I}(\A_i[\mathscr{L}_\mathcal{F}]) = \Big(\prod_{i \in I}\A_i\Big)[\mathscr{L}_\mathcal{F}],
\]
 where the algebras in the above display are defined by \cref{Claim : Well defined expansion}. 
Observe that $\prod_{i \in I}(\A_i[\mathscr{L}_\mathcal{F}])$ and $(\prod_{i \in I}\A_i)[\mathscr{L}_\mathcal{F}]$ have the same $\mathscr{L}_\mathsf{K}$-reduct, namely, $\prod_{i \in I}\A_i$.\  
It will then be enough to prove that for all $n$-ary $f \in \mathcal{F}$ and $a_1, \dots, a_n, b \in \prod_{i \in I}A_i$,
\begin{equation}\label{Eq : pp expansions : class operators 1}
g_f^{\prod_{i \in I}(\A_i[\mathscr{L}_\mathcal{F}])}(a_1, \dots, a_n) = b \iff g_f^{(\prod_{i \in I}\A_i)[\mathscr{L}_\mathcal{F}]}(a_1, \dots, a_n) = b.
\end{equation}
To this end, recall from the assumptions that $f$ is defined 
by a pp formula $\varphi$. Observe that $\prod_{i \in I}\A_i \in \PPP(\K) \subseteq \K$, and hence $f^{\prod_{i \in I}\A_i}$ is defined. We will prove that
\begin{align*}
g_f^{\prod_{i \in I}(\A_i[\mathscr{L}_\mathcal{F}])}(a_1, \dots, a_n) = b &\iff g_f^{\A_i[\mathscr{L}_\mathcal{F}]}(p_i(a_1), \dots, p_i(a_n)) = p_i(b) \text{ for every } i \in I\\
&\iff f^{\A_i}(p_i(a_1), \dots, p_i(a_n)) = p_i(b)\text{ for every } i \in I\\
&\iff \A_i \vDash \varphi(p_i(a_1), \dots, p_i(a_n), p_i(b)) \text{ for every }i \in I\\
&\iff \prod_{i \in I}\A_i \vDash \varphi(a_1, \dots, a_n, b)\\
&\iff f^{\prod_{i \in I}\A_i}(a_1, \dots, a_n)= b\\
&\iff g_f^{(\prod_{i \in I}\A_i)[\mathscr{L}_\mathcal{F}]}(a_1, \dots, a_n)= b.
\end{align*}
The above equivalences are justified as follows. The first holds by the definition of a direct product, the second by the definition of $\A_i[\mathscr{L}_\mathcal{F}]$, the third and the fifth because $f$ is defined by $\varphi$, the fourth follows from an application of \cref{Thm : preservation}(\ref{item : preservation : pp}) made possible by the assumption that $\varphi$ is a pp formula, and the last one holds by the definition of $(\prod_{i \in I}\A_i)[\mathscr{L}_\mathcal{F}]$.\ This establishes (\ref{Eq : pp expansions : class operators 1}), thus concluding the proof that $\PPP(\mathsf{N}[\mathscr{L}_\mathcal{F}]) = (\PPP(\mathsf{N}))[\mathscr{L}_\mathcal{F}]$.

Lastly, we consider the case where $\mathbb{O} = \PPU$.
\ From the definitions of $\PPU(\mathsf{N}[\mathscr{L}_\mathcal{F}])$ and $(\PPU(\mathsf{N}))[\mathscr{L}_\mathcal{F}]$ it follows that 
\begin{align*}
\A \in \PPU(\mathsf{N}[\mathscr{L}_\mathcal{F}])\iff & \text{ there exist }\{ \A_i : i \in I \} \subseteq \mathsf{N} \text{ and an ultrafilter $U$ on $I$}\\
&\text{ such that }\A = \prod_{i \in I}(\A_i[\mathscr{L}_\mathcal{F}]) / U;\\
\A \in (\PPU(\mathsf{N}))[\mathscr{L}_\mathcal{F}]  \iff & \text{ there exist }\{ \A_i : i \in I \} \subseteq \mathsf{N} \text{ and an ultrafilter $U$ on $I$}\\
& \text{ such that }\A = \Big(\prod_{i \in I}\A_i / U\Big)[\mathscr{L}_\mathcal{F}].
\end{align*}
Therefore, to conclude that $\PPU(\mathsf{N}[\mathscr{L}_\mathcal{F}]) = (\PPU(\mathsf{N}))[\mathscr{L}_\mathcal{F}]$, it suffices to show that for every family $\{ \A_i : i \in I \} \subseteq \mathsf{N}$ and ultrafilter $U$ on $I$,
\[
\prod_{i \in I}(\A_i[\mathscr{L}_\mathcal{F}]) / U = \Big(\prod_{i \in I}\A_i / U\Big)[\mathscr{L}_\mathcal{F}],
\]
where the algebras in the above display are defined by \cref{Claim : Well defined expansion}.

Similarly to the case $\mathbb{O} = \PPP$ it suffices to show that for all $n$-ary $f \in \mathcal{F}$ and $a_1, \dots, a_n, b \in \prod_{i \in I}A_i/U$
\begin{equation}\label{Eq : pp expansions : class operators 2}
g_f^{\prod_{i \in I}(\A_i[\mathscr{L}_\mathcal{F}]) / U}(a_1/ U, \dots, a_n/ U) = b/ U \iff g_f^{(\prod_{i \in I}\A_i/ U)[\mathscr{L}_\mathcal{F}]}(a_1/ U, \dots, a_n/ U) = b/ U.
\end{equation}
To this end, recall from the assumptions that $f$ is defined by a pp formula $\varphi$. Observe that $\prod_{i \in I}\A_i / U \in \PPU(\K) \subseteq \K$, and hence $f^{\prod_{i \in I}\A_i / U}$ is defined. We will prove that
\begin{align*}
g_f^{\prod_{i \in I}(\A_i[\mathscr{L}_\mathcal{F}]) / U}(a_1/ U, \dots, a_n/ U) = b/ U &\iff g_f^{\prod_{i \in I}(\A_i[\mathscr{L}_\mathcal{F}])}(a_1, \dots, a_n) / U = b/ U\\
&\iff \llbracket g_f^{\prod_{i \in I}(\A_i[\mathscr{L}_\mathcal{F}])}(a_1, \dots, a_n) \thickapprox b \rrbracket \in U\\
& \iff \{ i \in I : g_f^{\A_i[\mathscr{L}_\mathcal{F}]}(p_i(a_1), \dots, p_i(a_n)) = p_i(b) \} \in U \\
& \iff \{ i \in I : f^{\A_i}(p_i(a_1), \dots, p_i(a_n)) = p_i(b) \} \in U \\
& \iff \{ i \in I : \A_i \vDash \varphi(p_i(a_1), \dots, p_i(a_n), p_i(b)) \} \in U \\
& \iff \llbracket \varphi(a_1, \dots, a_n, b) \rrbracket \in U\\
& \iff \prod_{i \in I}\A_i / U \vDash \varphi(a_1 / U, \dots, a_n / U, b / U)\\
& \iff f^{\prod_{i \in I}\A_i / U}(a_1 / U, \dots, a_n / U) = b/ U\\
& \iff g_f^{(\prod_{i \in I}\A_i / U)[\mathscr{L}_\mathcal{F}]}(a_1 / U, \dots, a_n / U) = b/ U.
\end{align*}
The above equivalences are justified as follows.\ The first holds by the definition of a quotient algebra, the second by the definition of an ultraproduct, the third and the sixth are straightforward, the fourth holds by the definition of $\A_i[\mathscr{L}_\mathcal{F}]$, the fifth and the eighth because $\varphi$ defines $f$, the seventh follows from \LL o\'s' Theorem \ref{Thm : Los}, and the last one from the definition of $(\prod_{i \in I}\A_i / U)[\mathscr{L}_\mathcal{F}]$. This establishes (\ref{Eq : pp expansions : class operators 2}), thus concluding the proof that $\PPU(\mathsf{N}[\mathscr{L}_\mathcal{F}]) = (\PPU(\mathsf{N}))[\mathscr{L}_\mathcal{F}]$.

(\ref{item : K[F] closed SSS}):  Suppose that $\mathbb{O} = \SSS$ and that each $f \in \mathcal{F}$ is defined by a conjunction of equations $\varphi_f$. We need to prove that $\SSS(\mathsf{N}[\mathscr{L}_\mathcal{F}]) \subseteq (\SSS(\mathsf{N}))[\mathscr{L}_\mathcal{F}]$. Consider $\A \in \SSS(\mathsf{N}[\mathscr{L}_\mathcal{F}])$. 
Then there exists $\B \in \mathsf{N}$ such that $\A \leq \B[\mathscr{L}_\mathcal{F}]$. 
As $\B$ is the $\mathscr{L}_\mathsf{K}$-reduct of $\B[\mathscr{L}_\mathcal{F}]$ and $\A \leq \B[\mathscr{L}_\mathcal{F}]$, 
 we obtain $\A\resLK \leq \B\in \mathsf{N}$. Therefore, $\A\resLK \in \SSS(\mathsf{N})$. 
We will prove that the algebra  
$\A\resLK[\L_\F]$ 
is defined and coincides with $\A$, whence $\A \in (\SSS(\mathsf{N}))[\mathscr{L}_\mathcal{F}]$, as desired.

Since $\mathsf{K}$ is closed under $\SSS$ by assumption and $\mathsf{N} \subseteq \mathsf{K}$, we obtain $\A\resLK \in \SSS(\mathsf{N}) \subseteq \mathsf{K}$.
 It then suffices to show that for each $f \in \mathcal{F}$ the function 
$f^{\A\resLK}$ 
is total and coincides with the interpretation of $g_f$ in $\A$. To this end, consider an $n$-ary $f \in \mathcal{F}$ and $a_1, \dots, a_n \in A$.
We need to prove that 
\begin{equation}\label{Eq : K[F] is closed under S}
\langle a_1, \dots, a_n \rangle \in \mathsf{dom}(f^{\A\resLK}) \, \, \text{ and } \, \, g_f^\A(a_1, \dots, a_n) = f^{\A\resLK}(a_1, \dots, a_n).
\end{equation}
Observe that 
\begin{equation*}
f^{\B}(a_1, \dots, a_n) = g_f^{\B[\mathscr{L}_{\mathsf{K}}]}(a_1, \dots, a_n) =  g_f^{\A}(a_1, \dots, a_n),    
\end{equation*}
where the first equality holds by $a_1, \dots, a_n \in A \subseteq B$ and the definition of $\B[\mathscr{L}_{\mathsf{K}}]$, and the second holds because $\A \leq \B[\mathscr{L}_{\mathsf{K}}]$. Since $f$ is defined by a conjunction of equations $\varphi_f$ by assumption, the above display yields
\[
\B \vDash \varphi_f(a_1, \dots, a_n, g_f^{\A}(a_1, \dots, a_n)).
\]
From $\A \leq \B[\L_\F]$ it follows that $\A\resLK \leq \B$ because $\B$ is the $\L_\K$-reduct of $\B[\L_\F]$.
As $\varphi_f$ is a conjunction of equations and $\A\resLK \leq \B$, we can apply Theorem \ref{Thm : preservation}(\ref{item : preservation : universal}) to the above display obtaining 
\[
\A\resLK \vDash \varphi_f(a_1, \dots, a_n, g_f^{\A}(a_1, \dots, a_n)).
\]
Since $\varphi_f$ defines $f$ and $\A\resLK \in \K$, we conclude that (\ref{Eq : K[F] is closed under S}) holds.
\end{proof}

\begin{Proposition}\label{Prop : pp expansions : generation}
Let $\mathsf{K}$ be a class of algebras,  $\mathcal{F} \subseteq \mathsf{imp}_{\textsc{pp}}(\mathsf{K})$, and $\mathscr{L}_\mathcal{F}$ an $\mathcal{F}$-expansion of $\mathscr{L}_\mathsf{K}$. Moreover, let 
\[
\M = \{ \A \in \mathsf{K} : f^\A \text{ is total for each }f \in \mathcal{F} \}.
\]
Then for all $\mathsf{N} \subseteq \M$ and class operator $\mathbb{O} \in \{ \UUU, \QQQ, \III\SSS\PPP\}$,
\[
\text{if $\mathsf{K}$ is closed under $\mathbb{O}$, then $\mathbb{O}(\mathsf{N}[\mathscr{L}_\mathcal{F}]) = \SSS((\mathbb{O}(\mathsf{N}))[\mathscr{L}_\mathcal{F}])$}.
\]
\end{Proposition}

\begin{proof}
We will detail the case in which $\mathbb{O} = \UUU$, as the proof of the case in which $\mathbb{O} \in \{\QQQ, \III\SSS\PPP \}$ is analogous. Assume that $\mathsf{K}$ is closed under $\UUU$. Then it is  closed under $\III$ and $\PPU$ as well. Moreover, by assumption
\begin{equation}\label{Eq : pp expansions : IPu = PuI : 0}
\mathsf{N} \subseteq \M
\end{equation}
Therefore, from Proposition \ref{Prop : K[F] class operators} it follows that
\begin{equation}\label{Eq : pp expansions : IPu = PuI}
\III\PPU(\mathsf{N}[\mathscr{L}_\mathcal{F}]) = (\III\PPU(\mathsf{N}))[\mathscr{L}_\mathcal{F}].
\end{equation}

We will show that
\[
\UUU(\mathsf{N}[\mathscr{L}_\mathcal{F}]) = \III\SSS\PPU(\mathsf{N}[\mathscr{L}_\mathcal{F}]) = \SSS\III\PPU(\mathsf{N}[\mathscr{L}_\mathcal{F}]) = \SSS((\III\PPU(\mathsf{N}))[\mathscr{L}_\mathcal{F}])\subseteq \SSS((\UUU(\mathsf{N}))[\mathscr{L}_\mathcal{F}]).
\]
The first equality above holds by \cref{Thm : quasivariety generation}, 
the second because $\III\SSS = \SSS\III$, the third follows from (\ref{Eq : pp expansions : IPu = PuI}), and the last inclusion holds because $\III\PPU(\mathsf{N}) \subseteq \UUU(\mathsf{N})$. This establishes the inclusion $\UUU(\mathsf{N}[\mathscr{L}_\mathcal{F}]) \subseteq \SSS((\UUU(\mathsf{N}))[\mathscr{L}_\mathcal{F}])$.

Therefore, it only remains to prove the reverse inclusion $\SSS((\UUU(\mathsf{N}))[\mathscr{L}_\mathcal{F}]) \subseteq \UUU(\mathsf{N}[\mathscr{L}_\mathcal{F}])$. Consider $\A \in \SSS((\UUU(\mathsf{N}))[\mathscr{L}_\mathcal{F}])$. Then there exists $\B \in \UUU(\mathsf{N})$ such that $\B[\mathscr{L}_\mathcal{F}]$ is defined and $\A \leq \B[\mathscr{L}_\mathcal{F}]$. From $\B \in \UUU(\mathsf{N})$ and 
 \cref{Thm : quasivariety generation} 
it follows that $\B \in \III\SSS\PPU(\mathsf{N})$. Therefore, there exist a family $\{ \B_i : i \in I \} \subseteq \mathsf{N}$, an ultrafilter $U$ on $I$, and an embedding $h \colon \B \to \prod_{i \in I}\B_i / U$. In view of (\ref{Eq : pp expansions : IPu = PuI : 0}), we have 
\begin{equation}\label{Eq : pp expansions : fgrtenayfobq}
\{ \B_i : i \in I \} \subseteq \M
\end{equation}
Since the hypotheses of \cref{Prop : K[F] class operators} are satisfied, 
we can apply
  \cref{Claim : Well defined expansion} 
to the assumptions that $\mathsf{K}$ is closed under $\PPU$ and the above display, obtaining that the algebra $(\prod_{i \in I}\B_i / U)[\mathscr{L}_\mathcal{F}]$ is defined. Observe that $\B, \prod_{i \in I}\B_i / U \in \III\SSS\PPU(\mathsf{N}) \subseteq \mathsf{K}$ because $\mathsf{N} \subseteq \mathsf{K}$ and $\mathsf{K}$ is closed under $\UUU$ by assumption. Consequently,
\[
\B[\mathscr{L}_\mathcal{F}], \Big(\prod_{i \in I}\B_i / U\Big)[\mathscr{L}_\mathcal{F}] \in \mathsf{K}[\mathscr{L}_\mathcal{F}].
\]
Since $h \colon \B \to \prod_{i \in I}\B_i / U$ is an embedding between members of $\mathsf{K}$, from the above display and Proposition \ref{Prop : small homs are big homs} it follows that $h$ can be regarded as an embedding $h \colon \B[\mathscr{L}_\mathcal{F}] \to (\prod_{i \in I}\B_i / U)[\mathscr{L}_\mathcal{F}]$. Lastly, (\ref{Eq : pp expansions : fgrtenayfobq}) and the assumption that $\mathsf{K}$ is closed under $\PPU$ allow us to apply 
\cref{Prop : K[F] class operators}\eqref{item : K[K] closed under I Pu P}, obtaining 
\[
\prod_{i \in I}(\B_i[\mathscr{L}_\mathcal{F}]) / U \in (\PPU(\{ \B_i : i \in I \}))[\L_\F].
\]
As the $\L_\K$-reduct of $\prod_{i \in I}(\B_i[\mathscr{L}_\mathcal{F}]) / U$ is $\prod_{i \in I}\B_i / U$, the above display yields
\[
\prod_{i \in I}(\B_i[\mathscr{L}_\mathcal{F}]) / U = \Big(\prod_{i \in I}\B_i / U\Big)[\mathscr{L}_\mathcal{F}].
\]
Therefore, the map $h \colon \B[\mathscr{L}_\mathcal{F}] \to \prod_{i \in I}(\B_i[\mathscr{L}_\mathcal{F}]) / U$ is an embedding.\ 
Since
$\{\B_i : i \in I \} \subseteq \mathsf{N}$, 
 it follows that 
$\B[\mathscr{L}_\mathcal{F}] \in \III\SSS\PPU(\mathsf{N}[\mathscr{L}_\mathcal{F}])$. As $\A \leq \B[\mathscr{L}_\mathcal{F}]$ by assumption, we conclude that $\A \in \SSS\III\SSS\PPU(\mathsf{N}[\mathscr{L}_\mathcal{F}]) \subseteq \UUU(\mathsf{N}[\mathscr{L}_\mathcal{F}])$.
\end{proof}

\section{Primitive positive expansions}

As we mentioned, our aim is to expand the language of a given elementary class of algebras $\K$ by adding to it enough implicit operations so that every implicit operation of $\K$ becomes interpolable by a set of terms in a class $\M$ of algebras in the expanded language. In view of \cref{Cor : each implicit operations splits into pp formulas}, the latter can be stated as the demand that implicit operations of $\K$ defined by pp formulas be interpolated by terms of $\M$. Because of this, from now on we shall restrict our attention to implicit operations defined by pp formulas.  Furthermore, we require the implicit operations under consideration to be extendable in order to guarantee the validity of condition (\ref{D1 : desiderata 1}) (see Proposition \ref{Prop : pp expansions : subreducts : extendable}).

However, even when the implicit operations in $\mathcal{F}$ are  defined by pp formulas and extendable, 
the class $\mathsf{K}[\mathscr{L}_\mathcal{F}]$ may lack some desirable closure properties. More precisely, there is no guarantee that if $\mathsf{K}$ is a universal class or a quasivariety, then so is $\mathsf{K}[\mathscr{L}_\mathcal{F}]$. This problem is easily overcome by closing $\mathsf{K}[\mathscr{L}_\mathcal{F}]$ under $\SSS$, leading to the core of this section, namely, the notion of a pp expansion. 
\begin{Definition}
Let $\mathsf{K}$ and $\mathsf{M}$ be a pair of classes of algebras.\ Then $\mathsf{M}$ is said to be a \emph{primitive positive expansion} (\emph{pp expansion} for short) of $\mathsf{K}$ when $\mathsf{M} = \SSS(\mathsf{K}[\mathscr{L}_\mathcal{F}])$ for some $\mathcal{F} \subseteq \mathsf{ext}_{\textsc{pp}}(\mathsf{K})$ and $\mathcal{F}$-expansion $\mathscr{L}_\mathcal{F}$ of $\mathscr{L}_{\mathsf{K}}$. In this case, we say that $\mathsf{M}$ is \emph{induced} by $\mathcal{F}$ and $\mathscr{L}_\mathcal{F}$.
\end{Definition}

From Proposition \ref{Prop : pp expansions : subreducts : extendable} we deduce that pp expansions satisfy condition (\ref{D1 : desiderata 1}).

\begin{Proposition}\label{Prop : pp expansions and subreducts}
Let $\mathsf{M}$ be a pp expansion of a universal class $\mathsf{K}$. Then $\mathsf{K}$ is the class of $\mathscr{L}_\mathsf{K}$-subreducts of $\mathsf{M}$.
\end{Proposition}

\begin{proof}
Assume that $\mathsf{M}$ is a pp expansion of $\mathsf{K}$ of the form $\SSS(\mathsf{K}[\mathscr{L}_\mathcal{F}])$. As $\mathsf{K}$ is a universal class, we can apply \cref{Prop : pp expansions : subreducts : extendable}, obtaining that $\mathsf{K}$ is the class of  $\mathscr{L}_\mathsf{K}$-subreducts of  $\mathsf{K}[\mathscr{L}_\mathcal{F}]$. Hence, $\mathsf{K}$ is also the class of $\mathscr{L}_\mathsf{K}$-subreducts of $\SSS(\mathsf{K}[\mathscr{L}_\mathcal{F}]) = \mathsf{M}$.
\end{proof}

As we announced, the following holds true.

\begin{Theorem}\label{Thm : pp expansion : still quasivariety}
Let $\K$ be a class of algebras. The following conditions hold for a pp expansion $\mathsf{M}$ of $\K$  induced by $\mathcal{F}$ and $\L_\F$:
\benroman
\item\label{item : pp expansion : universal class} if $\mathsf{K}$ is a universal class, then $\mathsf{M}$ is a universal class;
\item\label{item : pp expansion : quasivariety} if $\mathsf{K}$ is a quasivariety, then $\mathsf{M}$ is a quasivariety such that $\mathsf{M}_{\textsc{rsi}} \subseteq \SSS(\mathsf{K}_\textsc{rsi}[\mathscr{L}_\mathcal{F}])$;
\item\label{item : pp expansion : variety} if $\mathsf{K}$ is a variety and $\mathcal{F} \subseteq \mathsf{ext}_{\textsc{eq}}(\K)$, then $\mathsf{M}$ is a variety.
\eroman
\end{Theorem}
\noindent As shown in \cite[Thm.~2.1]{CKMEXAv2}, the hypothesis that $\mathcal{F} \subseteq \mathsf{ext}_{\textsc{eq}}(\K)$ in \cref{Thm : pp expansion : still quasivariety}\eqref{item : pp expansion : variety} cannot be dispensed with.

Besides the above theorem, the main result of this section consists of four observations 
which facilitate the task of detecting the pp expansions of a given class of algebras. On the one hand, we will establish the following description of pp expansions induced by implicit operations definable by conjunctions of equations (for a similar result, see \cite[Lem.\ 2.1]{CCV25}).

\begin{Theorem}\label{Thm : axiomatization of pp expansion (almost always)}
Let $\K$ be a universal class axiomatized by a set of formulas $\Sigma$ and $\M$ a pp expansion of $\K$  induced by $\mathcal{F} \subseteq \mathsf{ext}_{\textsc{eq}}(\K)$ and $\mathscr{L}_\mathcal{F} = \mathscr{L}_\mathsf{K} \cup \{ g_f : f \in \mathcal{F} \}$. Then
$\M = \K[\mathscr{L}_\mathcal{F}]$ and $\M$ is axiomatized by
\[
\Sigma \cup \{ \varphi_f(x_1, \dots, x_n, g_f(x_1, \dots, x_n)) : f\text{ is an $n$-ary member of } \mathcal{F}  \},
\]
where $\varphi_f$ denotes the conjunction of equations defining $f \in \mathcal{F}$.
\end{Theorem}

On the other hand, we will establish the next description of pp expansions in terms of the class operators of universal class and quasivariety generation.

\begin{Theorem}\label{Thm : pp expansion : description in terms of Q and U}
Let $\mathsf{K}$ be a class of algebras,  $\mathcal{F}\subseteq \mathsf{ext}_{\textsc{pp}}(\mathsf{K})$, and $\mathscr{L}_\mathcal{F}$ an $\mathcal{F}$-expansion of $\mathscr{L}_\mathsf{K}$.\ Moreover, let $\mathsf{N} \subseteq \mathsf{K}$ and assume that $f^\A$ is total for all $\A \in  \mathsf{N}$ and $f \in \mathcal{F}$. Then for each  $\mathbb{O} \in \{  \UUU, \QQQ \}$ such that $\mathsf{K} = \mathbb{O}(\mathsf{N})$ the class $\mathbb{O}(\mathsf{N}[\mathscr{L}_\mathcal{F}])$ is a pp expansion of $\mathsf{K}$ that coincides with $\SSS(\mathsf{K}[\mathscr{L}_\mathcal{F}])$.
\end{Theorem}

We will then
show that the relation ``being a pp expansion of'' is transitive.

\begin{Theorem}\label{Thm : pp expansion of pp expansions}
Every pp expansion of a pp expansion of a class of algebras $\mathsf{K}$ is a pp expansion of $\mathsf{K}$.
\end{Theorem}

Lastly, we will show that enlarging the set of implicit operations inducing a pp expansion of a class of algebras also yields a pp expansion of the same class.

\begin{Theorem}\label{Thm : pp expansions : LG expands LF}
Let $\K$ be a universal class and $\F \subseteq \mathcal{G} \subseteq \extpp(\K)$.\ Let also $\L_\F$ be an $\F$-expansion of $\L_\K$ and $\L_\mathcal{G}$ a $\mathcal{G}$-expansion of $\L_\K$ such that $\L_\F \subseteq \L_\mathcal{G}$. Then $\SSS(\K[\L_\mathcal{G}])$ is a pp expansion of $\SSS(\K[\L_\F])$.
\end{Theorem}

Before proving Theorems \ref{Thm : pp expansion : still quasivariety}, \ref{Thm : axiomatization of pp expansion (almost always)}, \ref{Thm : pp expansion : description in terms of Q and U},  \ref{Thm : pp expansion of pp expansions}, and \ref{Thm : pp expansions : LG expands LF}, let us illustrate how these results can be used to describe pp expansions of familiar classes of algebras.

\begin{exa}[\textsf{Lazy pp expansions}]\label{Exa : pp expansions : lazy}
Every class of algebras $\mathsf{K}$ closed under $\SSS$ is a pp expansion of itself. For let $\mathcal{F} = \emptyset$. 
Then $\mathcal{F} = \emptyset \subseteq \mathsf{ext}_{\textsc{pp}}(\mathsf{K})$. Furthermore, let $\mathscr{L}_\mathcal{F} = \mathscr{L}_\mathsf{K}$ and observe that $\mathscr{L}_\mathcal{F}$ is an $\mathcal{F}$-expansion of $\mathscr{L}_\mathsf{K}$ 
because $\mathcal{F} = \emptyset$. Therefore, the class $\mathsf{K}[\mathscr{L}_\mathcal{F}]$ 
coincides 
with $\mathsf{K}$. As $\mathsf{K}$ is closed under $\SSS$, we conclude that $\mathsf{K} = \SSS(\mathsf{K}[\mathscr{L}_\mathcal{F}])$, whence $\mathsf{K}$ is a pp expansion of itself.
\qed
\end{exa}

\begin{exa}[\textsf{Cancellative commutative monoids}]
We will prove the following.

\begin{Theorem}\label{Thm : pp expansion : CCM}
The variety of Abelian groups is a pp  expansion of the quasivariety of cancellative commutative monoids.
\end{Theorem}

\begin{proof}
Let $f$ be the unary implicit operation of the quasivariety of cancellative commutative monoids $\mathsf{CCMon}$ given by Theorem \ref{Thm : inverses in monoids : extendable}.\ Recall from the same theorem that  $f$ is extendable and 
defined by the equation $x \cdot y  \thickapprox 1$. 
Moreover, let 
$(\,)^{-1}$
be a unary function symbol and denote by $t^{-1}$ the result of applying $(\,)^{-1}$ to a term $t$.
Then the language $\mathscr{L}_{f} = \mathscr{L}_\mathsf{CCMon} \cup \{ (\,)^{-1}
\}$ is an $f$-expansion of $\mathscr{L}_\mathsf{CCMon}$, in which the role of $g_f$ is played by $(\,)^{-1}$.

As $\mathsf{CCMon}$ is a universal class, Theorem \ref{Thm : axiomatization of pp expansion (almost always)} yields that $\mathsf{CCMon}[\mathscr{L}_{ f}]$ is a pp expansion of $\mathsf{CCMon}$ axiomatized by the axioms for cancellative commutative monoids plus the equation $x \cdot x^{-1} \thickapprox 1$.
 
Clearly, every member of $\mathsf{CCMon}[\mathscr{L}_{f}]$ is an Abelian group.
On the other hand, every Abelian group can be obtained by adding the implicit operation $f$ to its monoid reduct, which is a cancellative commutative monoid. Therefore, $\mathsf{CCMon}[\mathscr{L}_f]$ coincides with the variety of Abelian groups.
\end{proof}
\end{exa}

\begin{exa}[\textsf{Distributive lattices}]

Our aim is to establish the next result.

\begin{Theorem}\label{Thm : pp expansions : DL}
The following conditions hold:
\benroman
\item\label{item : pp expansions : DL 1} the variety of relatively complemented distributive lattices is a pp expansion of the variety of distributive lattices;
\item\label{item : pp expansions : DL 2} the variety of Boolean algebras is a pp expansion of the variety of bounded distributive lattices.
\eroman
\end{Theorem}

\begin{proof}
We detail only the proof of (\ref{item : pp expansions : DL 1}), as the proof of (\ref{item : pp expansions : DL 2}) is analogous. Let $f$ be the ternary implicit operation of the variety of distributive lattices $\mathsf{DL}$ given by \cref{Thm : distributive lattice : expandable}. Recall from the same theorem that $f$ is extendable and defined by the conjunction of equations
\[
\varphi = (x_1 \land y \thickapprox  x_1 \land x_2 \land x_3) \sqcap ( x_1 \lor y \thickapprox x_1 \lor x_2 \lor x_3).
\]
Moreover, let 
 $r$ 
be a ternary function symbol. Then the language $\mathscr{L}_{f} = \mathscr{L}_\mathsf{DL} \cup \{ r
\}$ is an $f$-expansion of $\mathscr{L}_\mathsf{DL}$, in which the role of $g_f$ is played by $r$.

As $\mathsf{DL}$ is a universal class, we can apply Theorem \ref{Thm : axiomatization of pp expansion (almost always)}, obtaining that $\mathsf{DL}[\mathscr{L}_{f}]$ is a pp expansion of $\mathsf{DL}$ axiomatized by the axioms for distributive lattices plus the equations
\[
x_1 \land r(x_1, x_2, x_3) \thickapprox x_1 \land x_2 \land x_3 \, \, \text{ and } \, \, x_1 \lor r(x_1, x_2, x_3) \thickapprox x_1 \lor x_2 \lor x_3.
\]
As a consequence, $\mathsf{DL}[\mathscr{L}_{f}]$ coincides with the variety of relatively complemented distributive lattices (see Example \ref{Exa : relatively complemented DL}).
\end{proof}
\end{exa}

\begin{exa}[\textsf{Reduced commutative rings}]\label{exa:meadows}
A \emph{meadow} is an algebra 
$\langle A; +, \cdot, -, (\,)^*,0, 1\rangle$ 
which comprises a commutative ring $\langle A; +, \cdot, -,0, 1 \rangle$ and a unary operation 
 $(\,)^*$
such that for each $a \in A$,
\[
(a \cdot a^*) \cdot a = a \, \ \text{ and } \, \, a = a^{**}
\]
(see, e.g., \cite{BT07}). As a consequence, the class of meadows forms a variety.

Prototypical examples of meadows arise by adding the operation of ``taking weak inverses'' to fields. More precisely, a \emph{zero-totalized field} is an algebra 
$\langle A; +, \cdot, -, (\,)^*,0, 1\rangle$ 
comprising a field $\langle A; +, \cdot, -,0, 1\rangle$ and a unary operation 
 $(\,)^*$ defined as $a^* = \mathsf{wi}(a)$ for each $a \in A$, where $\mathsf{wi}(a)$ is the weak inverse of $a$ (see Example \ref{Example : inverses in rings}).
Meadows were introduced in response to the desire to construct an equational theory that captures the essence of fields.
The following representation theorem (see \cite[Sec.~3.2]{MedBHT}) shows that they fulfill this expectation.

\begin{Theorem}\label{Thm : meadows : zero-totalized fields}
An algebra is a meadow if and only if it is isomorphic to a subdirect product of zero-totalized fields.
\end{Theorem}

We will prove the following.

\begin{Theorem} \label{Thm : Meadow pp exp of RCRing}
The variety of meadows is a pp expansion of the quasivariety of reduced commutative rings.
\end{Theorem}

\begin{proof}
Let $f$ be the unary implicit operation of the quasivariety $\mathsf{RCRing}$ of reduced commutative rings given by Theorem \ref{Thm : RCRing : extendable implicit operation} and recall from the same theorem that $f \in \mathsf{ext}_{\textsc{pp}}(\mathsf{RCRing})$. 
Moreover, let 
 $(\,)^*$
be a unary function symbol.
Then the language $\mathscr{L}_{f} = \mathscr{L}_\mathsf{RCRing} \cup \{ 
(\,)^*
\}$ is an $f$-expansion of $\mathscr{L}_\mathsf{RCRing}$, in which the role of $g_f$ is played by $(\,)^*$.

From Theorem \ref{Thm : RCRing : extendable implicit operation} it follows that for each field $\A$ the function $f^\A$ is total and the algebra $\A[\mathscr{L}_{f}]$ coincides with the zero-totalized field obtained by adding the operation of ``taking weak inverses'' to $\A$. Therefore, letting  $\mathsf{Field}$ and $\mathsf{Field}^*$ be the classes of fields and of zero-totalized fields, respectively, we obtain $\mathsf{Field}[\mathscr{L}_{f}] = \mathsf{Field}^*$. In view of Theorems \ref{Thm : reduced rings are Q fields} and \ref{Thm : meadows : zero-totalized fields}, we also have
\[
\mathsf{RCRing} = \QQQ(\mathsf{Field}) \, \, \text{ and } \, \, \mathsf{Meadow} = \QQQ(\mathsf{Field}^*),
\]
where $\mathsf{Meadow}$ is the variety of meadows. 
Lastly, recall that $f \in \mathsf{ext}_{\textsc{pp}}(\mathsf{RCRing})$ and that $f^\A$ is total for each $\A \in \mathsf{Field}$. Together with the left hand side of the above display, this allows us to apply \cref{Thm : pp expansion : description in terms of Q and U}, obtaining that $\QQQ(\mathsf{Field}[\mathscr{L}_{f}])$ is a pp expansion of $\mathsf{RCRing}$. 
By applying in succession the equality $\mathsf{Field}[\mathscr{L}_{f}] = \mathsf{Field}^*$ and the right hand side of the above display, 
we obtain
\[
\QQQ(\mathsf{Field}[\mathscr{L}_{f}]) = \QQQ(\mathsf{Field}^*) = \mathsf{Meadow}.
\]
Hence, we conclude that $\mathsf{Meadow}$ is a pp expansion of $\mathsf{RCRing}$.
\end{proof}

\end{exa}

\begin{exa}[\textsf{Hilbert algebras}]
An \emph{implicative semilattice} is an algebra $\langle A; \land, \to \rangle$ which comprises a semilattice $\langle A; \land \rangle$ and a binary operation $\to$ such that for all $a, b, c \in A$,
\[
a \land b \leq c \iff a \leq b \to c,
\]
where $\leq$ is the meet order of $\langle A; \land \rangle$. The class of implicative semilattices forms a variety (see, e.g., \cite[pp.~105--106]{Koh81}) which coincides with the class of $\langle \land, \to \rangle$-subreducts of Heyting algebras (see \cite[Thms.~5.1 \& 9.1]{Hor62}). We will prove the following.

\begin{Theorem}\label{Thm : pp expansion : Hilbert algebras}
The variety of implicative semilattices is a pp expansion of the variety of Hilbert algebras.
\end{Theorem}

\begin{proof}
First, let
 \begin{align*}
\mathsf{Heyting}_\to &= \text{the class of $\langle \to \rangle$-reducts of Heyting algebras};\\
\mathsf{Heyting}_{\land,\to} &= \text{the class of $\langle \land, \to \rangle$-reducts of Heyting algebras}.
\end{align*}
As the variety $\mathsf{Hilbert}$ of Hilbert algebras is the class of $\langle \to \rangle$-subreducts of Heyting algebras and the variety $\mathsf{ISL}$ of implicative semilattices is the class of $\langle \land, \to \rangle$-subreducts of Heyting algebras, we have
\begin{equation}\label{Eq : Hilbert and ISL : generation}
\mathsf{Hilbert} = \QQQ(\mathsf{Heyting}_\to) \, \, \text{ and } \, \, \mathsf{ISL} = \QQQ(\mathsf{Heyting}_{\land, \to}).
\end{equation}

Now, let $f \in \mathsf{ext}_{\textsc{pp}}(\mathsf{Hilbert})$ be the binary implicit operation given by Theorem \ref{Thm : Hilbert algebras : extendable}. 
Moreover, let 
$\wedge$
be a binary function symbol.
Then the language $\mathscr{L}_{f} = \mathscr{L}_\mathsf{Hilbert} \cup \{ 
\wedge
\}$ is an $f$-expansion of $\mathscr{L}_\mathsf{Hilbert}$, in which the role of $g_f$ is played by $\wedge$.

Recall from Theorem \ref{Thm : Hilbert algebras : extendable} that for every Heyting algebra $\A$ with implication reduct $\A_\to$ the operation $f^{\A_\to}$ is total and coincides with the meet operation of $\A$. Therefore, the algebra $\A_\to[\mathscr{L}_{f}]$ is defined and coincides with the $\langle \land, \to \rangle$-reduct of $\A$. Consequently,
\[
\mathsf{Heyting}_\to[\mathscr{L}_{f}] = \mathsf{Heyting}_{\land,\to}.
\]
Lastly, recall that $f \in \mathsf{ext}_{\textsc{pp}}(\mathsf{Hilbert})$ and that $f^\A$ is total for each $\A \in \mathsf{Heyting}_\to$. Together with the left hand side of (\ref{Eq : Hilbert and ISL : generation}), this allows us to apply \cref{Thm : pp expansion : description in terms of Q and U}, obtaining that $\QQQ(\mathsf{Heyting}_\to[\mathscr{L}_{f}])$ is a pp expansion of $\mathsf{Hilbert}$. By applying in succession the above display and the right hand side of (\ref{Eq : Hilbert and ISL : generation}), we obtain
\[
\QQQ(\mathsf{Heyting}_\to[\mathscr{L}_{f}]) = \QQQ(\mathsf{Heyting}_{\land,\to}) = \mathsf{ISL}.
\]
Hence, we conclude that $\mathsf{ISL}$ is a pp expansion of $\mathsf{Hilbert}$.
\end{proof}

\end{exa}

\begin{exa}[\textsf{Pseudocomplemented distributive lattices}]\label{Exa : depth 2 example : PDL}

A Heyting algebra $\A$ is said to have \emph{depth} $\leq 2$ when the chains in the poset of prime filters of $\A$ have size at most two. The class of all Heyting algebras of depth $\leq 2$ forms a variety  (see, e.g., \cite[Thm.~4.1]{BMRES} and the references therein). We will prove the following.

\begin{Theorem}\label{thm: HA2 pp expansion PDL}
The variety of Heyting algebras of depth $\leq 2$ is a pp expansion of the variety of pseudocomplemented distributive lattices.
\end{Theorem}

\begin{proof}
Let $\mathsf{PDL}$ and $\mathsf{Heyting}_2$ be the varieties of pseudocomplemented distributive lattices and of Heyting algebras of depth $\leq 2$, respectively. 
We recall that the members of $(\mathsf{Heyting}_{2})_\si^\fg$ are precisely the Heyting algebras whose lattice reduct is a finite Boolean lattice adjoined with a new top element(see \cite[Thm.\ 2]{LakPDL}).
It follows from the characterization of subdirectly irreducible Heyting algebras (see, e.g., \cite[Thm.~IX.4.5]{BD74}) and the local finiteness of $\mathsf{Heyting}_{2}$ (see \cite{Kom75,Kuz74})
that $\mathsf{PDL}_\si^\fg$ is the class of $\langle \land, \lor, \lnot, 0, 1 \rangle$-reducts of $(\mathsf{Heyting}_{2})_\si^\fg$. Lastly, from Theorem \ref{Prop : quasivariety = Q fin gen SI} it follows that
\begin{equation}\label{Eq : PDL and Heyting2 : generation}
\mathsf{PDL} = \QQQ(\mathsf{PDL}_\si^\fg) \, \, \text{ and } \, \, \mathsf{Heyting}_{2} = \QQQ((\mathsf{Heyting}_{2})_\si^\fg).
\end{equation}

Now, let $f \in \mathsf{ext}_{\textsc{pp}}(\mathsf{PDL})$ be the binary implicit operation of $\mathsf{PDL}$ given by \cref{Thm : PDL : extendable implicit operation}.
Moreover, let 
$\to$
be a binary function symbol.
Then the language $\mathscr{L}_{f} = \mathscr{L}_\mathsf{PDL} \cup \{ 
\to
\}$ is an $f$-expansion of $\mathscr{L}_\mathsf{PDL}$, in which the role of $g_f$ is played by 
$\to$.

Recall from Theorem \ref{Thm : PDL : extendable implicit operation} that for each $\A \in \mathsf{PDL}_\si^\fg$ the function $f^\A$ is total and coincides with the implication $\to$ of the unique Heyting algebra with the same lattice reduct as $\A$. As $\mathsf{PDL}_\si^\fg$ is the class of $\langle \land, \lor, \lnot, 0, 1 \rangle$-reducts of $(\mathsf{Heyting}_{2})_\si^\fg$, this yields
\[
\mathsf{PDL}_\si^\fg[\mathscr{L}_{f}] = (\mathsf{Heyting}_{2})_\si^\fg.
\]
Finally, recall that $f \in \mathsf{ext}_{\textsc{pp}}(\mathsf{PDL})$ and that $f^\A$ is total for each $\A \in \mathsf{PDL}_\si^\fg$. Together with the left hand side of (\ref{Eq : PDL and Heyting2 : generation}), this allows us to apply \cref{Thm : pp expansion : description in terms of Q and U}, obtaining that $\QQQ(\mathsf{PDL}_\si^\fg[\mathscr{L}_{f}])$ is a pp expansion of $\mathsf{PDL}$. By applying in succession the above display and the right hand side of (\ref{Eq : PDL and Heyting2 : generation}), we obtain
\[
\QQQ(\mathsf{PDL}_\si^\fg[\mathscr{L}_{f}]) = \QQQ((\mathsf{Heyting}_{2})_\si^\fg) = \mathsf{Heyting}_2.
\]
Hence, we conclude that $\mathsf{Heyting}_2$ is a pp expansion of $\mathsf{PDL}$.
\end{proof}
\end{exa}

The rest of this section is devoted to proving Theorems \ref{Thm : pp expansion : still quasivariety}, \ref{Thm : axiomatization of pp expansion (almost always)}, \ref{Thm : pp expansion : description in terms of Q and U}, \ref{Thm : pp expansion of pp expansions}, and \ref{Thm : pp expansions : LG expands LF}.
 We postpone the proof of \cref{Thm : pp expansion : still quasivariety} and begin by proving \cref{Thm : axiomatization of pp expansion (almost always)}.

\begin{proof}
\cref{Prop : how to axiomatize K[F]} yields that 
\begin{align*}
\mathsf{K}[\mathscr{L}_\mathcal{F}] = \{ \B : &\text{ }\B \text{ is an 
$\mathscr{L}_\mathcal{F}$-algebra}
\text{ such that }\B{\upharpoonright}_{\mathscr{L}_\mathsf{K}} \in \mathsf{K} \text{ and}\\
&\text{ $\B \vDash \varphi_f(x_1, \dots, x_n, g_f(x_1, \dots, x_n))$ for each $n$-ary $f \in \mathcal{F}$} \}.
\end{align*}
Since $\Sigma$ axiomatizes $\K$, for every $\mathscr{L}_\mathcal{F}$-algebra $\B$ we have that $\B{\upharpoonright}_{\mathscr{L}_\mathsf{K}} \in \K$ if and only if $\B \vDash \Sigma$. Therefore, from the above display it follows that the set of formulas
\[
\Sigma \cup \{ \varphi_f(x_1, \dots, x_n, g_f(x_1, \dots, x_n)) : f\text{ is an $n$-ary member of } \mathcal{F}  \}
\]
axiomatizes $\mathsf{K}[\mathscr{L}_\mathcal{F}]$.
To show that $\M=\mathsf{K}[\mathscr{L}_\mathcal{F}]$ it is sufficient to prove that $\mathsf{K}[\mathscr{L}_\mathcal{F}]$ is closed under $\SSS$. As $\K$ is a universal class, \cref{Thm : classes generation}\eqref{item : universal class generation} allows us to assume that $\Sigma$ consists of universal formulas. Together with the fact that each $\varphi_f(x_1, \dots, x_n, g_f(x_1, \dots, x_n))$ is a conjunction of equations, this implies that $\mathsf{K}[\mathscr{L}_\mathcal{F}]$ can be axiomatized by a set of universal formulas. Therefore, $\mathsf{K}[\mathscr{L}_\mathcal{F}]$ is a universal class by \cref{Thm : classes generation}\eqref{item : universal class generation}, and hence it is closed under $\SSS$.
\end{proof}

We are now ready to prove \cref{Thm : pp expansion : still quasivariety}.

\begin{proof}
We first prove that if $\K$ is a universal class or quasivariety, then so is $\M$. It
suffices to show that for each class operator $\mathbb{O} \in \{ \UUU, \QQQ \}$ if $\mathsf{K}$ is closed under $\mathbb{O}$, then $\SSS(\mathsf{K}[\mathscr{L}_\mathcal{F}])$ is closed under $\mathbb{O}$ as well. 

To this end, let $\mathbb{O} \in \{ \UUU, \QQQ \}$ and assume that $\mathsf{K}$ is closed under $\mathbb{O}$. Furthermore, let 
\[
\mathsf{N} = \{ \A \in \mathsf{K} : f^\A \text{ is total for each }f \in \mathcal{F} \}.
\]
We begin with the next observation.

\begin{Claim}\label{Claim : pp expansions : N generates K}
We have $\mathsf{K} = \mathbb{O}(\mathsf{N})$.
\end{Claim}

\begin{proof}[Proof of the Claim]
As $\mathbb{O} \in \{ \UUU, \QQQ \}$ and  $\mathsf{K}$ is closed under $\mathbb{O}$, we known that $\mathsf{K}$ is a universal class. Therefore, we can apply Theorem \ref{Thm : extendable 1}, obtaining that for each $\A \in \mathsf{K}$ there exists $\B \in \mathsf{K}$ with $\A \leq \B$ such that $f^\B$ is total for each $f \in \mathsf{ext}(\mathsf{K})$. Since $\mathcal{F} \subseteq \mathsf{ext}_{\textsc{pp}}(\mathsf{K}) \subseteq \mathsf{ext}(\mathsf{K})$ by assumption, for each $\A \in \mathsf{K}$ there exists $\B \in \mathsf{K}$ with $\A \leq \B$ such that $f^\B$ is total for each $f \in \mathcal{F}$. Together with the definition of $\mathsf{N}$, this yields $\mathsf{K} \subseteq \SSS(\mathsf{N})$. As $\mathbb{O} \in \{ \UUU, \QQQ \}$, we obtain $\mathsf{K} \subseteq \mathbb{O}(\mathsf{N})$.  On the other hand, from $\mathsf{N} \subseteq \mathsf{K}$ and the assumption that $\mathsf{K}$ is closed under $\mathbb{O}$ it follows that $\mathbb{O}(\mathsf{N}) \subseteq \mathsf{K}$, whence $\mathsf{K} = \mathbb{O}(\mathsf{N})$.
\end{proof}

Since $\mathsf{K}$ is closed under $\mathbb{O}$ by assumption, the definition of $\mathsf{N}$ allows us to apply Proposition \ref{Prop : pp expansions : generation}, obtaining $\mathbb{O}(\mathsf{N}[\mathscr{L}_\mathcal{F}]) = \SSS((\mathbb{O}(\mathsf{N}))[\mathscr{L}_\mathcal{F}])$. Together with Claim \ref{Claim : pp expansions : N generates K}, this yields $\mathbb{O}(\mathsf{N}[\mathscr{L}_\mathcal{F}]) = \SSS(\mathsf{K}[\mathscr{L}_\mathcal{F}])$. Hence, $\SSS(\mathsf{K}[\mathscr{L}_\mathcal{F}])$ is closed under $\mathbb{O}$, as desired. 

Now, we prove the last part of \eqref{item : pp expansion : quasivariety}. 
Assume that $\mathsf{K}$ is a quasivariety. We need to show that $\mathsf{M}_{\textsc{rsi}} \subseteq \SSS(\mathsf{K}_\textsc{rsi}[\mathscr{L}_\mathcal{F}])$. Let
\[
\mathsf{N}_{\textsc{rsi}} = \{ \A \in \mathsf{K}_\textsc{rsi} : f^\A \text{ is total for each }f \in \mathcal{F} \}.
\]
We rely on the following observation.

\begin{Claim}\label{Claim : pp expansions : N generates K : 2}
We have $\mathsf{K} = \III\SSS\PPP(\mathsf{N}_{\textsc{rsi}})$.
\end{Claim}

\begin{proof}[Proof of the Claim]
Since $\mathsf{K}$ is a quasivariety, it is closed under $\III, \SSS$, and $\PPP$. Therefore, it suffices to prove the inclusion $\mathsf{K} \subseteq \III\SSS\PPP(\mathsf{N}_{\textsc{rsi}})$. To this end, consider $\A \in \mathsf{K}$. In view of the Subdirect Decomposition Theorem \ref{Thm : Subdirect Decomposition} there exist a family $\{ \A_i : i \in I \} \subseteq \mathsf{K}_{\textsc{rsi}}$ and an embedding $h \colon \A \to \prod_{i \in I}\A_i$. By Theorem \ref{Thm : extendable 1} for each $i \in I$ there exists $\B_i \in \mathsf{K}_\textsc{rsi}$ with $\A_i \leq \B_i$ such that $f^{\B_i}$ is total for each $f \in \mathsf{ext}(\mathsf{K})$. As $\mathcal{F} \subseteq \mathsf{ext}(\mathsf{K})$ by assumption, this guarantees that $\{ \B_i : i \in I \} \subseteq \mathsf{N}_{\textsc{rsi}}$. Furthermore, from $\A_i \leq \B_i$ for each $i \in I$ it follows that $\prod_{i \in I}\A_i \leq \prod_{i \in I}\B_i$. Therefore, $h$ can be viewed as an embedding of $\A$ into $\prod_{i \in I}\B_i$. Thus, we conclude that $\A \in \III\SSS\PPP(\mathsf{N}_{\textsc{rsi}})$.
\end{proof}

Since the quasivariety $\mathsf{K}$ is closed under $\III\SSS\PPP$, the definition of $\mathsf{N}_{\textsc{rsi}}$ allows us to apply \cref{Prop : pp expansions : generation}, obtaining $\III\SSS\PPP(\mathsf{N}_{\textsc{rsi}}[\mathscr{L}_\mathcal{F}]) = \SSS((\III\SSS\PPP(\mathsf{N}_{\textsc{rsi}}))[\mathscr{L}_\mathcal{F}])$.  By Claim \ref{Claim : pp expansions : N generates K : 2} this amounts to $\III\SSS\PPP(\mathsf{N}_{\textsc{rsi}}[\mathscr{L}_\mathcal{F}]) = \SSS(\mathsf{K}[\mathscr{L}_\mathcal{F}])$. As $\mathsf{M} = \SSS(\mathsf{K}[\mathscr{L}_\mathcal{F}])$, we conclude that $\III\SSS\PPP(\mathsf{N}_{\textsc{rsi}}[\mathscr{L}_\mathcal{F}]) = \mathsf{M}$.

We are now ready to prove that $\mathsf{M}_{\textsc{rsi}} \subseteq \SSS(\mathsf{K}_\textsc{rsi}[\mathscr{L}_\mathcal{F}])$.\ Consider $\A \in \mathsf{M}_{\textsc{rsi}}$.\ Since  $\mathsf{M} = \III\SSS\PPP(\mathsf{N}_{\textsc{rsi}}[\mathscr{L}_\mathcal{F}])$, there exist a family $\{ \A_i : i \in I \} \subseteq \mathsf{N}_{\textsc{rsi}}$ and an embedding $h \colon \A \to \prod_{i \in I}(\A_i[\mathscr{L}_\mathcal{F}])$. 
Observe that $h$ can be viewed as a subdirect embedding $h \colon \A \to \prod_{i \in I}p_i[h[\A]]$ whose factors belong to $\mathsf{M}$ because for each $i \in I$ we have
\[
p_i[h[\A]] \in \SSS(\A_i[\mathscr{L}_\mathcal{F}]) \subseteq \SSS(\mathsf{N}_{\textsc{rsi}}[\mathscr{L}_\mathcal{F}]) \subseteq \III\SSS\PPP(\mathsf{N}_{\textsc{rsi}}[\mathscr{L}_\mathcal{F}]) =  \mathsf{M}.
\]
As $\A \in \mathsf{M}_{\textsc{rsi}}$, this implies that there exists $i \in I$ such that the map $p_i \circ h \colon \A \to p_i[h[\A]]$ is an isomorphism. Since $p_i[h[\A]] \leq \A_i[\mathscr{L}_\mathcal{F}]$ and $\A_i \in \mathsf{N}_{\textsc{rsi}}$, we obtain $\A \in \III\SSS(\mathsf{N}_{\textsc{rsi}}[\mathscr{L}_\mathcal{F}])$. Lastly, applying in succession $\III\SSS = \SSS\III$, \cref{Prop : K[F] class operators}(\ref{item : K[K] closed under I Pu P}) for the case where $\mathbb{O}= \III$, the fact that $\mathsf{N}_{\textsc{rsi}}$ is closed under $\III$, and the inclusion $\mathsf{N}_{\textsc{rsi}} \subseteq \mathsf{K}_\textsc{rsi}$, we conclude that 
\[
\A \in \III\SSS(\mathsf{N}_{\textsc{rsi}}[\mathscr{L}_\mathcal{F}]) = \SSS\III(\mathsf{N}_{\textsc{rsi}}[\mathscr{L}_\mathcal{F}]) = \SSS((\III(\mathsf{N}_{\textsc{rsi}}))[\mathscr{L}_\mathcal{F}]) = \SSS(\mathsf{N}_{\textsc{rsi}}[\mathscr{L}_\mathcal{F}]) \subseteq \SSS(\mathsf{K}_{\textsc{rsi}}[\mathscr{L}_\mathcal{F}]),
\]
as desired.

It remains to prove \eqref{item : pp expansion : variety}. Assume that $\K$ is a variety and $\F \subseteq \exteq(\K)$. By \cref{Thm : classes generation}\eqref{item : variety generation} there exists a set of equations $\Sigma$ that axiomatizes $\K$. Then \cref{Thm : axiomatization of pp expansion (almost always)} yields that $\mathsf{M}$ is axiomatized by 
\[
\Sigma \cup \{ \varphi_f(x_1, \dots, x_n, g_f(x_1, \dots, x_n)) : f\text{ is an $n$-ary member of } \mathcal{F}  \},
\] 
where $\varphi_f$ is the conjunction of equations defining $f \in \mathcal{F}$. Consequently, $\mathsf{M}$ is axiomatized by a set of equations. Thus, $\M$ is a variety by \cref{Thm : classes generation}\eqref{item : variety generation}.
\end{proof}

Then we prove Theorem \ref{Thm : pp expansion : description in terms of Q and U}.

\begin{proof}
We detail only the case in which $\mathbb{O} = \UUU$, as the case in which $\mathbb{O} = \QQQ$ is handled analogously. Accordingly, assume that $\mathsf{K} = \mathbb{U}(\mathsf{N})$.  As $\mathcal{F} \subseteq \mathsf{ext}_{\textsc{pp}}(\mathsf{K})$ and $\mathscr{L}_\mathcal{F}$ is an $\mathcal{F}$-expansion of $\mathscr{L}_\mathsf{K}$ by assumption, the class $\SSS(\mathsf{K}[\mathscr{L}_\mathcal{F}])$ is a pp expansion of $\mathsf{K}$. Therefore, it only remains to show that $\SSS(\mathsf{K}[\mathscr{L}_\mathcal{F}]) = \UUU(\mathsf{N}[\mathscr{L}_\mathcal{F}])$. Since 
\[
\mathsf{N} \subseteq \{ \A \in \mathsf{K} : f^\A \text{ is total for each }f \in \mathcal{F} \}
\]
by assumption, we can apply \cref{Prop : pp expansions : generation}, obtaining $\UUU(\mathsf{N}[\mathscr{L}_\mathcal{F}]) = \SSS((\UUU(\mathsf{N}))[\mathscr{L}_\mathcal{F}])$. As $\mathsf{K} = \mathbb{U}(\mathsf{N})$, this amounts to $\UUU(\mathsf{N}[\mathscr{L}_\mathcal{F}]) = \SSS(\mathsf{K}[\mathscr{L}_\mathcal{F}])$. 
\end{proof}

It only remains to prove \cref{Thm : pp expansion of pp expansions}. The proof hinges on the next observation, in which a \emph{nonconstant term} is simply a term that is not a constant.

\begin{Proposition}\label{Prop : equivalence in K[F]}
Let $\mathsf{M}$ be a pp expansion of a class of algebras $\mathsf{K}$ induced by $\mathcal{F}$ and $\mathscr{L}_\mathcal{F}$.\ Then the following conditions hold:
\benroman
\item\label{item : equivalence in K[F] : 1} for every constant $c$ of $\mathsf{M}$ there exists a unary $f_c \in \mathsf{ext}_{\textsc{eq}}(\mathsf{K})$ such that $f_c^{\A{\upharpoonright}_{\mathscr{L}_\mathsf{K}}}$ is total and $c^\A = f_c^{\A{\upharpoonright}_{\mathscr{L}_\mathsf{K}}}(a)$ for all $\A \in \mathsf{K}[\mathscr{L}_{\mathcal{F}}]$ and $a \in A$; 
\item\label{item : equivalence in K[F] : 2} for every nonconstant term $t$ of $\mathsf{M}$ 
there exists $f_t \in \mathsf{ext}_{\textsc{pp}}(\mathsf{K})$ such that $t^\A = f_t^{\A{\upharpoonright}_{\mathscr{L}_\mathsf{K}}}$ for each $\A \in \mathsf{K}[\mathscr{L}_\mathcal{F}]$;
\item\label{item : equivalence in K[F] : 3} for every $f \in \mathsf{imp}_{\textsc{pp}}(\mathsf{M})$ there exists $f_* \in \mathsf{imp}_{\textsc{pp}}(\mathsf{K})$ such that $f^\A = f_*^{\A{\upharpoonright}_{\mathscr{L}_\mathsf{K}}}$ for each $\A \in \mathsf{K}[\mathscr{L}_\mathcal{F}]$. Furthermore, if $f \in \mathsf{ext}_{\textsc{pp}}(\mathsf{M})$, then $f_*$ can be chosen in $\mathsf{ext}_{\textsc{pp}}(\mathsf{K})$.
\eroman
\end{Proposition}

\begin{proof}
By assumption $\mathsf{M}$ is a pp expansion of $\mathsf{K}$ induced by $\mathcal{F}$ and $\mathscr{L}_\mathcal{F}$. Therefore, $\mathscr{L}_\mathcal{F}$ is an $\mathcal{F}$-expansion of $\mathscr{L}_\mathsf{K}$. Consequently, $\mathcal{F} \subseteq \mathsf{ext}_{\textsc{pp}}(\mathsf{K})$ and $\mathscr{L}_\mathcal{F}$ is of the form $\mathscr{L}_\mathsf{K} \cup \{ g_f : f \in \mathcal{F} \}$. This fact will be used repeatedly in the proof.

(\ref{item : equivalence in K[F] : 1}): 
Let $c$ be a constant of $\M$. Then $c$ belongs to $\L_\F$, which is an $\F$-expansion of $\L_\K$. As $\F$-expansions are obtained by adding only functions symbols of positive arity, it follows that $c \in \L_\K$.
Then the equation $\varphi(x,y) = y \thickapprox c$ defines a unary $f_c \in \mathsf{ext}_{\textsc{eq}}(\mathsf{K})$ with the desired properties.

(\ref{item : equivalence in K[F] : 2}): We proceed by induction on the construction of the  nonconstant term $t(x_1, \dots, x_n)$. In the base case, $t = p(x_1, \dots, x_n)$ for a basic operation $p$ of $\mathsf{M}$. Therefore, $p \in \mathscr{L}_\mathcal{F}$. As $\mathscr{L}_\mathcal{F} = \mathscr{L}_\mathsf{K} \cup \{ g_f : f \in \mathcal{F} \}$, we have two cases: either $p \in \mathscr{L}_\mathsf{K}$ or there exists $h \in \mathcal{F} \subseteq \mathsf{ext}_{\textsc{pp}}(\mathsf{K})$ such that $p = g_h$. In the first case, we let $f_t = \langle p^\A : \A \in \mathsf{K} \rangle$ and, in the second,  $f_t = h$. In both cases, $f_t \in \mathsf{ext}_{\textsc{pp}}(\mathsf{K})$ and $t^\A = f_t^{\A{\upharpoonright}_{\mathscr{L}_\mathsf{K}}}$ for each $\A \in \mathsf{K}[\mathscr{L}_\mathcal{F}]$.

In the inductive step, $t = p(t_1, \dots, t_m)$ for some $m$-ary $p \in \mathscr{L}_\mathcal{F}$ and terms $t_i(x_1, \dots, x_n)$ of $\mathsf{M}$. By the inductive hypothesis there exist $f_p, f_{t_1}, \dots, f_{t_m} \in \mathsf{ext}_{\textsc{pp}}(\mathsf{K})$ satisfying the condition in the statement for $p, t_1, \dots, t_m$, respectively. Define $f_t$ as the composition $f_p(f_{t_1}, \dots, f_{t_m})$. As $f_p, f_{t_1}, \dots, f_{t_m} \in \mathsf{ext}_{\textsc{pp}}(\mathsf{K})$, from  \cref{Cor : closure under composition for ext} it follows that $f_t \in \mathsf{ext}_{\textsc{pp}}(\mathsf{K})$ as well. Then consider
$\A \in \mathsf{K}[\mathscr{L}_\mathcal{F}]$ and $a_1, \dots, a_n \in A$. By the inductive hypothesis $p^\A = f_p^{\A{\upharpoonright}_{\mathscr{L}_\mathsf{K}}}$ and $t_i^\A = f_{t_i}^{\A{\upharpoonright}_{\mathscr{L}_\mathsf{K}}}$ for each $i \leq m$. Therefore, for each $i \leq m$,
\[
\mathsf{dom}(f_{p}^{\A{\upharpoonright}_{\mathscr{L}_\mathsf{K}}}) = \mathsf{dom}(p^\A) = A^m \, \, \text{ and } \, \, \mathsf{dom}(f_{t_i}^{\A{\upharpoonright}_{\mathscr{L}_\mathsf{K}}}) = \mathsf{dom}(t_i^\A) = A^n.
\]
Together with $f_t = f_p(f_{t_1}, \dots, f_{t_m})$, this yields $\langle a_1, \dots, a_n \rangle \in \mathsf{dom}(f_t^{\A{\upharpoonright}_{\mathscr{L}_\mathsf{K}}})$ and
\begin{align*}
f_t^{\A{\upharpoonright}_{\mathscr{L}_\mathsf{K}}}(a_1, \dots, a_n) &= f_p^{\A{\upharpoonright}_{\mathscr{L}_\mathsf{K}}}(f_{t_1}^{\A{\upharpoonright}_{\mathscr{L}_\mathsf{K}}}(a_1, \dots, a_n), \dots, f_{t_m}^{\A{\upharpoonright}_{\mathscr{L}_\mathsf{K}}}(a_1, \dots, a_n))\\
& = p^\A(t_1^\A(a_1, \dots, a_n), \dots, t_m^\A(a_1, \dots, a_n)) \\
&= t^\A(a_1, \dots, a_n).
\end{align*}
Hence, we conclude that $t^\A = f_t^{\A{\upharpoonright}_{\mathscr{L}_\mathsf{K}}}$.

(\ref{item : equivalence in K[F] : 3}): Let $f$ be an $n$-ary implicit operation of $\mathsf{M}$ defined by a pp formula 
\begin{equation}\label{Eq : equivalence in K[F] : 2}
\varphi(x_1, \dots, x_n, y) = \exists z_1, \dots, z_k \bigsqcap_{j \leq m}t_{j} \thickapprox s_{j},
\end{equation}
where each $t_{j} \thickapprox s_{j}$ is an equation of $\mathsf{M}$ in variables $x_1, \dots, x_n, z_1, \dots, z_k, y$. For each $j \leq m$ let $\varphi_{t_j}(x_1, \dots, x_{n+k+1}, y)$ and $\varphi_{s_j}(x_1, \dots, x_{n+k+1}, y)$ be 
pp formulas of $\mathsf{K}$ defining the implicit operations $f_{t_j}$ and $f_{s_j}$ of $\mathsf{K}$ given by \eqref{item : equivalence in K[F] : 1} and \eqref{item : equivalence in K[F] : 2}.\footnote{When $t_j$ is a constant $c$, we let the pp formula $\varphi_{t_j}$ be $y \thickapprox c$ (see the proof of \eqref{item : equivalence in K[F] : 1}) and think of $f_{t_j}$ as a constant operation of arity $n+k+1$. Similarly for $s_j$ when it is a constant.}
Moreover, for each $j \leq m$ define
\begin{align*}
\alpha_j &= \varphi_{t_j} (x_1, \dots, x_n, z_1, \dots, z_k, y, v_j) \sqcap \varphi_{s_j}(x_1, \dots, x_n, z_1, \dots, z_k, y, v_j) ;\\
\psi & = \exists z_1, \dots, z_k, v_1, \dots, v_m \bigsqcap_{j \leq m} \alpha_j.
\end{align*}

\begin{Claim}\label{Claim : equivalence in K[F]} For all $\A \in \mathsf{K}[\mathscr{L}_\mathcal{F}]$ and $a_1, \dots, a_n, b \in A$,
\[
\A \vDash \varphi(a_1, \dots, a_n, b) \iff \A{\upharpoonright}_{\mathscr{L}_\mathsf{K}} \vDash \psi(a_1, \dots, a_n, b).
\]
\end{Claim}

\begin{proof}[Proof of the Claim]
As $\varphi$ is the formula in (\ref{Eq : equivalence in K[F] : 2}), we have $\A \vDash \varphi(a_1, \dots, a_n, b)$ if and only if there exist $c_1, \dots, c_k \in A$ such that $t_j^\A(a_1, \dots, a_n, c_1, \dots, c_k, b) = s_j^\A(a_1, \dots, a_n, c_1, \dots, c_k, b)$ for each $j \leq m$. In view of \eqref{item : equivalence in K[F] : 1} and \eqref{item : equivalence in K[F] : 2}, the latter is equivalent to the demand that there exist $c_1, \dots, c_k \in A$ such that for each $j \leq m$,
\[
\langle a_1, \dots, a_n, c_1, \dots, c_k, b \rangle \in \mathsf{dom}(f_{t_j}^{\A{\upharpoonright}_{\mathscr{L}_\mathsf{K}}}) \cap \mathsf{dom}(f_{s_j}^{\A{\upharpoonright}_{\mathscr{L}_\mathsf{K}}})
\] 
and
\[
f_{t_j}^{\A{\upharpoonright}_{\mathscr{L}_\mathsf{K}}}(a_1, \dots, a_n, c_1, \dots, c_k, b) = f_{s_j}^{\A{\upharpoonright}_{\mathscr{L}_\mathsf{K}}}(a_1, \dots, a_n, c_1, \dots, c_k, b).
\]
Since the formulas $\varphi_{t_j}$ and $\varphi_{s_j}$ define $f_{t_j}$ and $f_{s_j}$, respectively, the definition of $\alpha_j$ guarantees that this happens precisely when there exist $c_1, \dots, c_k \in A$ such that
\[
\A{\upharpoonright}_{\mathscr{L}_\mathsf{K}} \vDash \exists v_1, \dots, v_m \bigsqcap_{j \leq m}\alpha_j(a_1, \dots, a_n, c_1, \dots, c_k, b, v_j).
\]
By the definition of $\psi$ this amounts to the demand that $\A{\upharpoonright}_{\mathscr{L}_\mathsf{K}} \vDash \psi(a_1, \dots, a_n, b)$.
\end{proof}

Recall that $\varphi$ defines the implicit operation $f$ of $\mathsf{M}$. Therefore, $\varphi$ is functional in $\mathsf{M}$ and, in particular, in $\mathsf{K}[\mathscr{L}_\mathcal{F}]$. Together with Claim \ref{Claim : equivalence in K[F]}, this yields that $\psi$ is functional in $\{ \A{\upharpoonright}_{\mathscr{L}_\mathsf{K}} : \A \in \mathsf{K}[\mathscr{L}_\mathcal{F}] \}$. As $\psi$ is equivalent to a pp formula by definition, we can apply \cref{Cor : functionality in Q(K)}, obtaining that $\psi$ defines some $f_* \in \mathsf{imp}_{\textsc{pp}}(\SSS(\{ \A{\upharpoonright}_{\mathscr{L}_\mathsf{K}} : \A \in \mathsf{K}[\mathscr{L}_\mathcal{F}] \}))$. Then recall that $\mathsf{M}$ is a pp expansion of $\mathsf{K}$ induced by $\mathcal{F}$ and $\mathscr{L}_\mathcal{F}$ and, therefore, $\mathcal{F} \subseteq \mathsf{ext}_{\textsc{pp}}(\mathsf{K})$. Consequently, we can apply Proposition \ref{Prop : pp expansions : subreducts : extendable}, obtaining that $\mathsf{K} = \SSS(\{ \A{\upharpoonright}_{\mathscr{L}_\mathsf{K}} : \A \in \mathsf{K}[\mathscr{L}_\mathcal{F}] \})$. Hence,
\[
f_* \in \mathsf{imp}_{\textsc{pp}}(\SSS(\{ \A{\upharpoonright}_{\mathscr{L}_\mathsf{K}} : \A \in \mathsf{K}[\mathscr{L}_\mathcal{F}] \})) = \mathsf{imp}_{\textsc{pp}}(\mathsf{K}).
\]

Since $\varphi$ and $\psi$ define $f$ and $f_*$, respectively, from Claim \ref{Claim : equivalence in K[F]} it follows that for all $\A \in \mathsf{K}[\mathscr{L}_\mathcal{F}]$ and $a_1, \dots, a_n, b \in A$,
\begin{align*}
& \langle a_1, \dots, a_n \rangle \in \mathsf{dom}(f^\A) \text{ and }f^\A(a_1, \dots, a_n) = b \\
& \quad \iff  \text{ } \A \vDash \varphi(a_1, \dots, a_n, b)\\
& \quad \iff  \text{ }\A{\upharpoonright}_{\mathscr{L}_\mathsf{K}} \vDash \psi(a_1, \dots, a_n, b)\\
& \quad  \iff  \text{ } \langle a_1, \dots, a_n \rangle \in \mathsf{dom}(f_*^{\A{\upharpoonright}_{\mathscr{L}_\mathsf{K}}}) \text{ and }f_*^{\A{\upharpoonright}_{\mathscr{L}_\mathsf{K}}}(a_1, \dots, a_n) = b.
\end{align*}
Hence, we conclude that $f^\A = f_*^{\A{\upharpoonright}_{\mathscr{L}_\mathsf{K}}}$. This concludes the proof  of the first half of (\ref{item : equivalence in K[F] : 3}).

Therefore, it only remains to prove that if $f \in \mathsf{ext}_{\textsc{pp}}(\mathsf{M})$, then $f_* \in \mathsf{ext}_{\textsc{pp}}(\mathsf{K})$. Accordingly, suppose that $f \in \mathsf{ext}_{\textsc{pp}}(\mathsf{M})$. Since we already proved that $f_* \in \mathsf{imp}_{\textsc{pp}}(\mathsf{K})$, it suffices to show that $f_* \in \mathsf{ext}(\mathsf{K})$. To this end, consider $\A \in \mathsf{K}$ and $a_1, \dots, a_n \in A$. As $\mathsf{K}$ is the class of $\mathscr{L}_\mathsf{K}$-subreducts of $\mathsf{M}$ by Proposition \ref{Prop : pp expansions and subreducts}, there exists $\B \in \mathsf{M}$ such that $\A \leq \B{\upharpoonright}_{\mathscr{L}_\mathsf{K}}$. Since $f \in \mathsf{ext}(\mathsf{M})$, $\B \in \mathsf{M}$, and $a_1, \dots, a_n \in A \subseteq B$, there exists $\C \in \mathsf{M}$ such that $\B \leq \C$ and $\langle a_1, \dots, a_n \rangle \in \mathsf{dom}(f^\C)$. From $\C \in \mathsf{M} = \SSS(\mathsf{K}[\mathscr{L}_\mathcal{F}])$ it follows that there also exists $\D \in \mathsf{K}[\mathscr{L}_\mathcal{F}]$ with $\C \leq \D$. Since $\A \leq \B{\upharpoonright}_{\mathscr{L}_\mathsf{K}}$ and $\B \leq \C \leq \D \in \mathsf{M}$ and $\mathsf{K}$ is the class of $\mathscr{L}_\mathsf{K}$-subreducts of $\mathsf{M}$, we have $\A \leq \D{\upharpoonright}_{\mathscr{L}_\mathsf{K}} \in \mathsf{K}$. Moreover, by applying Proposition \ref{Prop : implicit operations extend} to $\langle a_1, \dots, a_n \rangle \in \mathsf{dom}(f^\C)$ and $\C \leq \D$, we obtain $\langle a_1, \dots, a_n \rangle \in \mathsf{dom}(f^\D)$. Hence,
\[
\D \in \mathsf{K}[\mathscr{L}_\mathcal{F}], \quad \langle a_1, \dots, a_n \rangle \in \mathsf{dom}(f^\D), \, \ \text{ and } \, \, \A \leq \D{\upharpoonright}_{\mathscr{L}_\mathsf{K}} \in \mathsf{K}.
\]
Together with the first half of condition (\ref{item : equivalence in K[F] : 3}), the first two items in the above display imply $\langle a_1, \dots, a_n \rangle \in \mathsf{dom}(f_*^{\D{\upharpoonright}_{\mathscr{L}_\mathsf{K}}})$. 
 As $\A \leq \D{\upharpoonright}_{\mathscr{L}_\mathsf{K}} \in \mathsf{K}$,
we conclude that $f_* \in \mathsf{ext}(\mathsf{K})$.
\end{proof}

We are now ready to prove Theorem \ref{Thm : pp expansion of pp expansions}.

\begin{proof}
Let $\mathsf{M}_2$ be a pp expansion of a pp expansion $\mathsf{M}_1$ of $\mathsf{K}$. We will prove that $\mathsf{M}_2$ is also a pp expansion of $\mathsf{K}$. First, as $\mathsf{M}_1$ is a pp expansion of $\mathsf{K}$, there exist $\mathcal{F}_1 \subseteq \mathsf{ext}_{\textsc{pp}}(\mathsf{K})$ and an $\mathcal{F}_1$-expansion $\mathscr{L}_{\mathcal{F}_1}$ of $\mathscr{L}_\mathsf{K}$ such that $\mathsf{M}_1 = \SSS(\mathsf{K}[\mathscr{L}_{\mathcal{F}_1}])$.\ Similarly, as $\mathsf{M}_2$ is a pp expansion of $\mathsf{M}_1$, there exist $\mathcal{F}_2 \subseteq \mathsf{ext}_{\textsc{pp}}(\mathsf{M}_1)$ and an $\mathcal{F}_2$-expansion $\mathscr{L}_{\mathcal{F}_2}$ of $\mathscr{L}_{\mathcal{F}_1}$ such that $\mathsf{M}_2 = \SSS(\mathsf{M}_1[\mathscr{L}_{\mathcal{F}_2}])$.\ Hence,
\begin{equation}\label{Eq : the classes M[K2] explained in detail}
\mathsf{M}_1 = \SSS(\mathsf{K}[\mathscr{L}_{\mathcal{F}_1}]) \, \, \text{ and } \, \, \mathsf{M}_2 = \SSS(\mathsf{M}_1[\mathscr{L}_{\mathcal{F}_2}]).
\end{equation}
Since $\mathscr{L}_{\mathcal{F}_1}$ is an $\mathcal{F}_1$-expansion of $\mathscr{L}_\mathsf{K}$ and $\mathscr{L}_{\mathcal{F}_2}$ an $\mathcal{F}_2$-expansion of $\mathscr{L}_{\mathcal{F}_1}$, we may assume that 
\[
\mathscr{L}_{\mathcal{F}_1} = \mathscr{L}_\mathsf{K} \cup \{ g_f : f \in \mathcal{F}_1 \} \, \, \text{ and } \, \, \mathscr{L}_{\mathcal{F}_2} = \mathscr{L}_{\mathcal{F}_1} \cup \{ g_f : f \in \mathcal{F}_2 \}.
\]
Consequently,
\begin{equation}\label{Eq : the language LF2 explained in detail}
\mathscr{L}_{\mathcal{F}_2} = \mathscr{L}_\mathsf{K} \cup \{ g_f : f \in \mathcal{F}_1 \} \cup \{ g_f : f \in \mathcal{F}_2 \}.
\end{equation}

As $\mathcal{F}_2 \subseteq \mathsf{ext}_{\textsc{pp}}(\mathsf{M}_1)$ and $\mathsf{M}_1$ is a pp expansion of $\mathsf{K}$,  Proposition \ref{Prop : equivalence in K[F]}(\ref{item : equivalence in K[F] : 3}) guarantees that for every $f \in \mathcal{F}_2$ there exists $f_* \in \mathsf{ext}_{\textsc{pp}}(\mathsf{K})$ such that
\begin{equation}\label{Eq : the functions f and f ast in pp exp}
f^\A = f_*^{\A{\upharpoonright}_{\mathscr{L}_\mathsf{K}}} \text{ for each }\A \in \mathsf{K}[\mathscr{L}_{\mathcal{F}_1}].
\end{equation}
 Since $\mathcal{F}_1 \subseteq \mathsf{ext}_{\textsc{pp}}(\mathsf{K})$ by assumption, the set
\[
\mathcal{F} = \mathcal{F}_1 \cup \{ f_* : f \in \mathcal{F}_2 \}
\]
is a subset of $\mathsf{ext}_{\textsc{pp}}(\mathsf{K})$. Define $g_{f_*} = g_f$ for each $f \in \mathcal{F}_2$ and consider the following $\mathcal{F}$-expansion of  $\mathscr{L}_\mathsf{K}$:
\[
\mathscr{L}_\mathcal{F} = \mathscr{L}_\mathsf{K} \cup \{ g_f : f \in \mathcal{F} \}.
\]
Then the pair $\mathcal{F}$ and $\mathscr{L}_\mathcal{F}$ induces a pp expansion $\SSS(\mathsf{K}[\mathscr{L}_\mathcal{F}])$ of $\mathsf{K}$.
To conclude the proof, it will be enough to show that $\mathsf{M}_2 = \SSS(\mathsf{K}[\mathscr{L}_\mathcal{F}])$, for in this case $\mathsf{M}_2$ would also be a pp expansion of $\mathsf{K}$.

First, observe that 
\begin{align*}
\mathscr{L}_{\mathcal{F}_2} &= \mathscr{L}_\mathsf{K} \cup \{ g_f : f \in \mathcal{F}_1 \} \cup \{ g_f : f \in \mathcal{F}_2 \}\\
&= \mathscr{L}_\mathsf{K} \cup \{ g_f : f \in \mathcal{F}_1 \} \cup \{ g_{f_*} : f \in \mathcal{F}_2 \}\\
&= \mathscr{L}_\mathsf{K} \cup \{ g_f : f \in \mathcal{F} \}\\
&=\mathscr{L}_{\mathcal{F}}.
\end{align*}
The above equalities are justified as follows.\ The first is (\ref{Eq : the language LF2 explained in detail}), the second holds by the definition of $g_{f_*}$ for $f \in \mathcal{F}_2$, the third by the definition of $\mathcal{F}$, and the fourth by that of $\mathscr{L}_{\mathcal{F}}$. This establishes $\mathscr{L}_{\mathcal{F}_2} =  \mathscr{L}_{\mathcal{F}}$. Since $\mathscr{L}_{\mathcal{F}_2}$ and $\mathscr{L}_{\mathcal{F}}$ are, respectively, the languages of $\mathsf{M}_2$ and $\SSS(\mathsf{K}[\mathscr{L}_\mathcal{F}])$, we conclude that these classes have the same language.

In view of the right hand side of (\ref{Eq : the classes M[K2] explained in detail}), in order to prove that $\mathsf{M}_2 = \SSS(\mathsf{K}[\mathscr{L}_\mathcal{F}])$, it suffices to show that
\begin{equation}\label{Eq : double pp expansions : the critical equality : 4}
\mathsf{M}_1[\mathscr{L}_{\mathcal{F}_2}] \subseteq \SSS(\mathsf{K}[\mathscr{L}_\mathcal{F}]) \, \, \text{ and } \, \, \mathsf{K}[\mathscr{L}_\mathcal{F}] \subseteq \mathsf{M}_1[\mathscr{L}_{\mathcal{F}_2}].
\end{equation}
We split the proof of the above display in two claims.

\begin{Claim}\label{Claim : double pp expansions : 1}
We have $\mathsf{M}_1[\mathscr{L}_{\mathcal{F}_2}] \subseteq \SSS(\mathsf{K}[\mathscr{L}_\mathcal{F}])$.
\end{Claim}

\begin{proof}[Proof of the Claim]
Consider $\A \in \mathsf{M}_1[\mathscr{L}_{\mathcal{F}_2}]$. By the definition of $\mathsf{M}_1[\mathscr{L}_{\mathcal{F}_2}]$ we have $\A{\upharpoonright}_{\mathscr{L}_{\mathsf{M}_1}} \in \mathsf{M}_1$.\ As $\mathsf{M}_1 = \SSS(\mathsf{K}[\mathscr{L}_{\mathcal{F}_1}])$ by the left hand side of (\ref{Eq : the classes M[K2] explained in detail}), there exists $\B \in \mathsf{K}[\mathscr{L}_{\mathcal{F}_1}]$ with $\A{\upharpoonright}_{\mathscr{L}_{\mathsf{M}_1}} \leq \B$. Furthermore, $\B{\upharpoonright}_{\mathscr{L}_{\mathsf{K}}} \in \mathsf{K}$ because $\B \in \mathsf{K}[\mathscr{L}_{\mathcal{F}_1}]$. Since $\SSS(\mathsf{K}[\mathscr{L}_\mathcal{F}])$ is a pp expansion of $\mathsf{K}$ and $\B{\upharpoonright}_{\mathscr{L}_{\mathsf{K}}}	\in \mathsf{K}$, we can apply Proposition \ref{Prop : pp expansions : subreducts : extendable}, obtaining some $\C \in \mathsf{K}[\mathscr{L}_\mathcal{F}]$ such that $\B{\upharpoonright}_{\mathscr{L}_{\mathsf{K}}} \leq \C{\upharpoonright}_{\mathscr{L}_{\mathsf{K}}}$. 

We will prove that $\A \leq \C$. Since $\C \in \mathsf{K}[\mathscr{L}_\mathcal{F}]$, this will yield $\A \in \SSS(\mathsf{K}[\mathscr{L}_\mathcal{F}])$, thus concluding the proof of the inclusion $\mathsf{M}_1[\mathscr{L}_{\mathcal{F}_2}] \subseteq \SSS(\mathsf{K}[\mathscr{L}_\mathcal{F}])$. First, as $\mathscr{L}_\mathcal{F} = \mathscr{L}_{\mathcal{F}_2} = \mathscr{L}_{\mathcal{F}_1} \cup \{ g_f : f \in \mathcal{F}_2 \}$, we have $\mathscr{L}_{\mathcal{F}_1} \subseteq \mathscr{L}_\mathcal{F}$.\ Therefore, from $\C \in \mathsf{K}[\mathscr{L}_\mathcal{F}]$ it follows that $\C{\upharpoonright}_{\mathscr{L}_{\mathcal{F}_1}} \in \mathsf{K}[\mathscr{L}_{\mathcal{F}_1}]$. By applying Proposition \ref{Prop : small homs are big homs} to $\B,  \C{\upharpoonright}_{\mathscr{L}_{\mathcal{F}_1}} \in \mathsf{K}[\mathscr{L}_{\mathcal{F}_1}]$ and $\B{\upharpoonright}_{\mathscr{L}_{\mathsf{K}}} \leq \C{\upharpoonright}_{\mathscr{L}_{\mathsf{K}}} = (\C{\upharpoonright}_{\mathscr{L}_{\mathcal{F}_1}}){\upharpoonright}_{\mathscr{L}_{\mathsf{K}}}$, we obtain $\B \leq \C{\upharpoonright}_{\mathscr{L}_{\mathcal{F}_1}}$. Together with $\A{\upharpoonright}_{\mathscr{L}_{\mathsf{M}_1}} \leq \B$ and $\mathscr{L}_{\mathsf{M}_1} =\mathscr{L}_{\mathcal{F}_1}$, this yields $\A{\upharpoonright}_{\mathscr{L}_{\mathcal{F}_1}} \leq \C{\upharpoonright}_{\mathscr{L}_{\mathcal{F}_1}}$.

Since the language of $\A$ and $\C$ is $\mathscr{L}_\mathcal{F} = \mathscr{L}_{\mathcal{F}_1} \cup \{ g_f : f \in \mathcal{F}_2 \}$, in order to prove that $\A \leq \C$, it only remains to show that for all $n$-ary $f \in \mathcal{F}_2$ and $a_1, \dots, a_n \in A$,
\begin{equation}\label{Eq : double pp expansions : the critical equality}
g_f^\A(a_1, \dots, a_n) = g_f^\C(a_1, \dots, a_n).
\end{equation}
To this end, consider an $n$-ary $f \in \mathcal{F}_2$ and $a_1, \dots, a_n \in A$. We will show that
\begin{align*}
g_f^\A(a_1, \dots, a_n) &= f^{\A{\upharpoonright}_{\mathscr{L}_{\mathsf{M}_1}}}(a_1, \dots, a_n)\\
&=f^{\B}(a_1, \dots, a_n)\\
&=f_*^{\B{\upharpoonright}_{\mathscr{L}_\mathsf{K}}}(a_1, \dots, a_n)\\
&=f_*^{\C{\upharpoonright}_{\mathscr{L}_\mathsf{K}}}(a_1, \dots, a_n)\\
&=g_{f_*}^{\C}(a_1, \dots, a_n)\\
&=g_f^{\C}(a_1, \dots, a_n).
\end{align*}
The equalities above are justified as follows. To prove the first, recall that $\A \in \mathsf{M}_1[\mathscr{L}_{\mathcal{F}_2}]$ and $f \in \mathcal{F}_2$, whence $f^{\A{\upharpoonright}_{\mathscr{L}_{\mathsf{M}_1}}}$ is a total function and $g_f^\A(a_1, \dots, a_n) = f^{\A{\upharpoonright}_{\mathscr{L}_{\mathsf{M}_1}}}(a_1, \dots, a_n)$. To prove the second, observe that the assumption that $f \in \mathcal{F}_2 \subseteq \mathsf{imp}(\mathsf{M}_1)$, $\A{\upharpoonright}_{\mathscr{L}_{\mathsf{M}_1}} \leq \B$,  and $\A{\upharpoonright}_{\mathscr{L}_{\mathsf{M}_1}}, \B \in \mathsf{M}_1$ allows us to apply \cref{Prop : implicit operations extend} to the fact that $\langle a_1, \dots, a_n \rangle \in \mathsf{dom}(f^{\A{\upharpoonright}_{\mathscr{L}_{\mathsf{M}_1}}})$, obtaining $\langle a_1, \dots, a_n \rangle \in \mathsf{dom}(f^{\B})$ and $f^{\A{\upharpoonright}_{\mathscr{L}_{\mathsf{M}_1}}}(a_1, \dots, a_n) = f^{\B}(a_1, \dots, a_n)$. The third equality follows from $\B \in \mathsf{K}[\mathscr{L}_{\mathcal{F}_1}]$ and (\ref{Eq : the functions f and f ast in pp exp}). To prove the fourth, observe that  $f_*\in \mathsf{imp}(\mathsf{K})$, $\B{\upharpoonright}_{\mathscr{L}_\mathsf{K}},\C{\upharpoonright}_{\mathscr{L}_{\mathsf{K}}} \in \mathsf{K}$, and $\B{\upharpoonright}_{\mathscr{L}_{\mathsf{K}}} \leq \C{\upharpoonright}_{\mathscr{L}_{\mathsf{K}}}$. Therefore, we can apply \cref{Prop : implicit operations extend} to the fact that $\langle a_1, \dots, a_n \rangle \in \mathsf{dom}(f_*^{\B{\upharpoonright}_{\mathscr{L}_\mathsf{K}}})$, obtaining $\langle a_1, \dots, a_n \rangle \in \mathsf{dom}(f_*^{\C{\upharpoonright}_{\mathscr{L}_\mathsf{K}}})$ and $f_*^{\B{\upharpoonright}_{\mathscr{L}_\mathsf{K}}}(a_1, \dots, a_n) = f_*^{\C{\upharpoonright}_{\mathscr{L}_\mathsf{K}}}(a_1, \dots, a_n)$. The fifth holds because $\C \in \mathsf{K}[\mathscr{L}_\mathcal{F}]$ and $f_* \in \mathcal{F}$. Lastly, the sixth equality holds because $g_{f_*} = g_f$ by assumption. This concludes the proof of (\ref{Eq : double pp expansions : the critical equality}). Hence, we obtain $\A \leq \C$, as desired.
\end{proof}

\begin{Claim}\label{Claim : double pp expansions : 2}
We have $\mathsf{K}[\mathscr{L}_\mathcal{F}] \subseteq \mathsf{M}_1[\mathscr{L}_{\mathcal{F}_2}]$.
\end{Claim}

\begin{proof}[Proof of the Claim]
Consider $\A \in \mathsf{K}[\mathscr{L}_\mathcal{F}]$. Then $\A{\upharpoonright}_{\mathscr{L}_\mathsf{K}} \in \mathsf{K}$ and $f^{\A{\upharpoonright}_{\mathscr{L}_\mathsf{K}}}$ is total for each $f \in \mathcal{F}$. As $\mathcal{F}_1 \subseteq \mathcal{F}$, this implies that the algebra $\A{\upharpoonright}_{\mathscr{L}_\mathsf{K}}[\mathscr{L}_{\mathcal{F}_1}]$ is defined and belongs to $\mathsf{K}[\mathscr{L}_{\mathcal{F}_1}]$ and, therefore, to $\mathsf{M}_1$ as well. We will prove that $f^{\A{\upharpoonright}_{\mathscr{L}_\mathsf{K}}[\mathscr{L}_{\mathcal{F}_1}]}$ is total for each $f \in \mathcal{F}_2$. To this end, consider $f \in \mathcal{F}_2$. By first  applying (\ref{Eq : the functions f and f ast in pp exp}) to $\A{\upharpoonright}_{\mathscr{L}_\mathsf{K}}[\mathscr{L}_{\mathcal{F}_1}] \in \mathsf{K}[\mathscr{L}_{\mathcal{F}_1}]$ and then observing that $\A{\upharpoonright}_{\mathscr{L}_\mathsf{K}}$ is the $\mathscr{L}_\mathsf{K}$-reduct of  $\A{\upharpoonright}_{\mathscr{L}_\mathsf{K}}[\mathscr{L}_{\mathcal{F}_1}]$, 
we obtain
\[
f^{\A{\upharpoonright}_{\mathscr{L}_\mathsf{K}}[\mathscr{L}_{\mathcal{F}_1}]} = f_*^{(\A{\upharpoonright}_{\mathscr{L}_\mathsf{K}}[\mathscr{L}_{\mathcal{F}_1}]){{\upharpoonright}_{\mathscr{L}_\mathsf{K}}}} = f_*^{\A{\upharpoonright}_{\mathscr{L}_\mathsf{K}}}.
\]
The function on the right hand side of the above display is total because $\A \in \mathsf{K}[\mathscr{L}_\mathcal{F}]$ and $f_* \in \mathcal{F}$ (the latter by $f \in \mathcal{F}_2$ and the definition of $\mathcal{F}$). Hence, we conclude that the left hand side of the above display is also total, as desired. Since  $\A{\upharpoonright}_{\mathscr{L}_\mathsf{K}}[\mathscr{L}_{\mathcal{F}_1}] \in \mathsf{M}_1$ and $f^{\A{\upharpoonright}_{\mathscr{L}_\mathsf{K}}[\mathscr{L}_{\mathcal{F}_1}]}$ is total for each $f \in \mathcal{F}_2$, the algebra $(\A{\upharpoonright}_{\mathscr{L}_\mathsf{K}}[\mathscr{L}_{\mathcal{F}_1}])[\mathscr{L}_{\mathcal{F}_2}]$ is defined and belongs to $\mathsf{M}_1[\mathscr{L}_{\mathcal{F}_2}]$. Therefore, in order to conclude the proof, it suffices to show that
\begin{equation}\label{Eq : double pp expansions : the critical equality : 3}
\A = (\A{\upharpoonright}_{\mathscr{L}_\mathsf{K}}[\mathscr{L}_{\mathcal{F}_1}])[\mathscr{L}_{\mathcal{F}_2}].
\end{equation}

To this end, recall that the language of these algebras is $\mathscr{L}_{\mathcal{F}_1} \cup \{ g_f : f \in \mathcal{F}_2 \}$ and their universe is $A$. Moreover, recall that $\A \in \mathsf{K}[\mathscr{L}_\mathcal{F}]$. Then $\A = \A{\upharpoonright}_{\mathscr{L}_\mathsf{K}}[\mathscr{L}_\mathcal{F}]$ and $\A{\upharpoonright}_{\mathscr{L}_\mathsf{K}} \in \mathsf{K}$. Together with $\mathcal{F}_1 \subseteq \mathcal{F}$, this yields that the $\mathscr{L}_{\mathcal{F}_1}$-reduct of $\A$ is $\A{\upharpoonright}_{\mathscr{L}_\mathsf{K}}[\mathscr{L}_{\mathcal{F}_1}]$. On the other hand, the $\mathscr{L}_{\mathcal{F}_1}$-reduct of $(\A{\upharpoonright}_{\mathscr{L}_\mathsf{K}}[\mathscr{L}_{\mathcal{F}_1}])[\mathscr{L}_{\mathcal{F}_2}]$ is also $\A{\upharpoonright}_{\mathscr{L}_\mathsf{K}}[\mathscr{L}_{\mathcal{F}_1}]$ by construction. Therefore, in order to prove the above display, it only remains to show that for all $n$-ary $f \in \mathcal{F}_2$ and $a_1, \dots, a_n \in A$,
\begin{equation}\label{Eq : double pp expansions : the critical equality : 2}
g_f^\A(a_1, \dots, a_n) = g_f^{(\A{\upharpoonright}_{\mathscr{L}_\mathsf{K}}[\mathscr{L}_{\mathcal{F}_1}])[\mathscr{L}_{\mathcal{F}_2}]}(a_1, \dots, a_n).
\end{equation}
Consider an $n$-ary $f \in \mathcal{F}_2$ and $a_1, \dots, a_n \in A$. We will prove that
\begin{align*}
g_f^\A(a_1, \dots, a_n) &= g_{f_*}^\A(a_1, \dots, a_n)\\
&= f_*^{\A{\upharpoonright}_{\mathscr{L}_\mathsf{K}}}(a_1, \dots, a_n)\\
&=f_*^{(\A{\upharpoonright}_{\mathscr{L}_\mathsf{K}}[\mathscr{L}_{\mathcal{F}_1}]){\upharpoonright}_{\mathscr{L}_\mathsf{K}}}(a_1, \dots, a_n)\\
&=f^{\A{\upharpoonright}_{\mathscr{L}_\mathsf{K}}[\mathscr{L}_{\mathcal{F}_1}]}(a_1, \dots, a_n)\\
&=g_f^{(\A{\upharpoonright}_{\mathscr{L}_\mathsf{K}}[\mathscr{L}_{\mathcal{F}_1}])[\mathscr{L}_{\mathcal{F}_2}]}(a_1, \dots, a_n).
\end{align*}
The above equalities are justified as follows. The first holds because $g_{f_*} = g_f$ for each $f \in \mathcal{F}_2$ by definition, the second because $\A \in \mathsf{K}[\mathscr{L}_\mathcal{F}]$ and $f_* \in \mathcal{F}$,  the third because $\A{\upharpoonright}_{\mathscr{L}_\mathsf{K}} = (\A{\upharpoonright}_{\mathscr{L}_\mathsf{K}}[\mathscr{L}_{\mathcal{F}_1}]){\upharpoonright}_{\mathscr{L}_\mathsf{K}}$, the fourth follows from (\ref{Eq : the functions f and f ast in pp exp}) and $\A{\upharpoonright}_{\mathscr{L}_\mathsf{K}}[\mathscr{L}_{\mathcal{F}_1}] \in \mathsf{K}[\mathscr{L}_{\mathcal{F}_1}]$, and the fifth from $f \in \mathcal{F}_2$ and the definition of $(\A{\upharpoonright}_{\mathscr{L}_\mathsf{K}}[\mathscr{L}_{\mathcal{F}_1}])[\mathscr{L}_{\mathcal{F}_2}]$. This concludes the proof of (\ref{Eq : double pp expansions : the critical equality : 2}) and, therefore, of (\ref{Eq : double pp expansions : the critical equality : 3}).
\end{proof}

As (\ref{Eq : double pp expansions : the critical equality : 4}) is an immediate consequence of Claims \ref{Claim : double pp expansions : 1} and \ref{Claim : double pp expansions : 2}, we are done.
\end{proof}

The proof of \cref{Thm : pp expansions : LG expands LF} hinges on the following result, which describes how to lift implicit operations to pp expansions.

\begin{Proposition} \label{Prop : lifting implicit operations}
Let $\mathsf{M}$ be a pp expansion of a class of algebras $\mathsf{K}$ and $f \in \imp(\K)$ defined by a formula $\varphi$ of $\mathscr{L}_\mathsf{K}$. Let also $f_\bullet$ be the partial function on $\M$ given by $f_\bullet =\langle f^{\A{\upharpoonright}_{\mathscr{L}_{\mathsf{K}}}} : \A \in \mathsf{M} \rangle$. Then the following conditions hold:
\benroman
\item\label{Prop : lifting implicit operations: impM} $f_\bullet$ is defined by $\varphi$ and belongs to $\imp(\M)$;
\item\label{Prop : lifting implicit operations: impppM} if $f \in \imppp(\K)$, then $f_\bullet \in \imppp(\M)$;
\item\label{Prop : lifting implicit operations: extM} if $\K$ is a universal class and $f \in \mathsf{ext}(\mathsf{K})$, then $f_\bullet \in \mathsf{ext}(\M)$. 
\eroman
\end{Proposition}

\begin{proof}
Throughout the proof we assume that $f$ is $n$-ary.

\eqref{Prop : lifting implicit operations: impM}: Observe that \cref{Prop : pp expansions and subreducts} implies that $\A{\upharpoonright}_{\mathscr{L}_{\mathsf{K}}} \in \K$ for every $\A \in \M$. Hence, $f_\bullet$ is a well-defined partial function on $\M$. We show that $f_\bullet$ is preserved by homomorphisms in $\M$. Consider a homomorphism $h \colon \A \to \B$ with $\A, \B \in \M$ and let $a_1, \dots, a_n \in A$ be such that $\langle a_1, \dots, a_n \rangle \in \mathsf{dom}(f_\bullet^\A)$. Since $f_\bullet^\A = f^{\A{\upharpoonright}_{\mathscr{L}_{\mathsf{K}}}}$, we have
\[
\langle a_1, \dots, a_n \rangle \in \mathsf{dom}(f^{\A{\upharpoonright}_{\mathscr{L}_{\mathsf{K}}}}).
\] 
As $h \colon \A{\upharpoonright}_{\mathscr{L}_{\mathsf{K}}} \to \B{\upharpoonright}_{\mathscr{L}_{\mathsf{K}}}$ is a homomorphism and $f$ is an implicit operation of $\mathsf{K}$, the above display implies that 
\[
\langle h(a_1), \dots, h(a_n) \rangle \in \mathsf{dom}(f^{\B{\upharpoonright}_{\mathscr{L}_{\mathsf{K}}}}) \, \, \text{ and } \, \, h(f^{\A{\upharpoonright}_{\mathscr{L}_{\mathsf{K}}}}(a_1, \dots, a_n)) = f^{\B{\upharpoonright}_{\mathscr{L}_{\mathsf{K}}}}(h(a_1), \dots, h(a_n)).
\]
Together with $f_\bullet^\A = f^{\A{\upharpoonright}_{\mathscr{L}_{\mathsf{K}}}}$ and $f_\bullet^\B = f^{\B{\upharpoonright}_{\mathscr{L}_{\mathsf{K}}}}$, this yields
\[
\langle h(a_1), \dots, h(a_n) \rangle \in \mathsf{dom}(f_\bullet^\B) \, \, \text{ and } \, \, h(f_\bullet^\A(a_1, \dots, a_n)) = f_\bullet^\B(h(a_1), \dots, h(a_n)).
\]
Hence, $f_\bullet$ is preserved by homomorphisms, as desired.

Lastly, we will prove that for all $\A \in \M$ and $a_1, \dots, a_n,b \in A$,
\begin{align*}
& \langle a_1, \dots, a_n \rangle \in \dom(f_\bullet^\A) \text{ and } f_\bullet^\A(a_1, \dots, a_n)=b\\
& \quad \iff  \text{ }
    \langle a_1, \dots, a_n \rangle \in \dom(f^{\A{\upharpoonright}_{\mathscr{L}_{\mathsf{K}}}}) \text{ and } f^{\A{\upharpoonright}_{\mathscr{L}_{\mathsf{K}}}}(a_1, \dots, a_n)=b \\
& \quad \iff  \text{ }  \A{\upharpoonright}_{\mathscr{L}_{\mathsf{K}}} \vDash \varphi(a_1, \dots, a_n,b)\\
& \quad \iff  \text{ }  \A \vDash \varphi(a_1, \dots, a_n,b).
\end{align*}
The first equivalence above holds by the definition of $f_\bullet$, the second because $\varphi$ defines $f$, and the third because $\varphi$ is an $\L_\K$-formula. In view of the above series of equivalences, $\varphi$ defines $f_\bullet$. As $f_\bullet$ is preserved by homomorphisms, we conclude that $f_\bullet \in \imp(\M)$.

\eqref{Prop : lifting implicit operations: impppM}: Suppose that $f \in \imppp(\K)$. Then we may assume that $\varphi$ is a pp formula. It follows from \eqref{Prop : lifting implicit operations: impM} that $f_\bullet$ is defined by $\varphi$ and belongs to $\imp(\M)$, whence $f_\bullet \in \imppp(\M)$.

\eqref{Prop : lifting implicit operations: extM}:
Since $\M$ is a pp expansion of $\K$, it is of the form $\SSS(\K[\L_\mathcal{F}])$ for some $\F \subseteq \mathsf{ext}_{\textsc{pp}}(\mathsf{K})$ and $\L_\F$.  Suppose that $f \in \ext(\K)$.
To show that $f_\bullet \in \mathsf{ext}(\mathsf{M})$, consider $\A \in \mathsf{M}$ and $a_1, \dots, a_n \in A$.
Since $\M = \SSS (\K[\L_\F])$, there exists $\B[\L_\F] \in \K[\L_\F]$ such that $\A \leq \B[\L_\F]$.\ 
As $\B \in \K$, \cref{Thm : extendable 1} yields $\C \in \mathsf{K}$ such that $\B \leq \C$ and $g^\C$ is total for every $g \in \ext(\K)$. 
In particular,
\begin{equation} \label{Eq : props of C}
    \C[\L_\F] \text{ is defined and } \langle a_1, \dots, a_n \rangle \in \mathsf{dom}(f^{\C}).
\end{equation}
By \cref{Prop : small homs are big homs} from $\B \leq \C$ it follows that $\B[\L_\F] \leq \C[\L_\F]$.
Then $\A \leq \B[\L_\F] \leq \C[\L_\F]$ with $\C[\L_\F] \in \mathsf{K}[\L_\F] \subseteq \mathsf{M}$. 
Since $\C[\L_\F]\resLK = \C$, the definition of $f_\bullet$ yields  $f_\bullet^{\C[\L_\F]} = f^{\C}$. 
From \eqref{Eq : props of C} it follows that $\langle a_1, \dots, a_n \rangle \in \mathsf{dom}(f_\bullet^{\C[\L_\F]})$ and, therefore,  $f_\bullet \in \ext(\M)$.
\end{proof}

We now proceed to prove \cref{Thm : pp expansions : LG expands LF}.

\begin{proof}
Let $\M = \SSS(\K[\L_\F])$ and $\L_{\mathcal{G}} = \L_\K \cup \{ g_f : f \in \mathcal{G}\}$. Define $\mathcal{G}_\bullet = \{ f_\bullet : f \in \mathcal{G}-\F\}$, where $f_\bullet =\langle f^{\A{\upharpoonright}_{\L_\K}} : \A \in \mathsf{M} \rangle$.
Since $\mathcal{G} \subseteq \extpp(\K)$, \cref{Prop : lifting implicit operations} implies that $\mathcal{G}_\bullet \subseteq \extpp(\M)$.
Consider the pp expansion $\SSS(\M[\L_{\mathcal{G}_\bullet}])$ of $\M$ induced by $\mathcal{G}_\bullet$ and $\L_{\mathcal{G}_{\bullet}}$, where $\L_{\mathcal{G}_{\bullet}}=\L_{\mathcal{G}}$ and $g_f^{\A[\L_{\mathcal{G}_{\bullet}}]}=f_\bullet^\A$ for all $f \in \mathcal{G}-\F$ and $\A[\L_{\mathcal{G}_{\bullet}}] \in \M[\L_{\mathcal{G}_{\bullet}}]$.
Since $\SSS(\M[\L_{\mathcal{G}_{\bullet}}])$ is a pp expansion of $\M$, 
to prove that $\SSS(\K[\L_{\mathcal{G}}])$ is a pp expansion of $\M$, it suffices to show that 
\begin{equation} \label{Eq : pp exp on K and M coincide}
\SSS(\mathsf{M}[\mathscr{L}_{\mathcal{G}_{\bullet}}]) = \SSS(\mathsf{K}[\mathscr{L}_{\mathcal{G}}]).
\end{equation}

For the inclusion from left to right of \eqref{Eq : pp exp on K and M coincide} it is enough to show that $\mathsf{M}[\mathscr{L}_{\mathcal{G}_{\bullet}}] \subseteq \SSS(\mathsf{K}[\mathscr{L}_{\mathcal{G}}])$. 
To this end, consider $\A \in \mathsf{M}[\mathscr{L}_{\mathcal{G}_{\bullet}}]$. 
Then $\A {\upharpoonright}_{\L_\mathcal{F}} \in \mathsf{M} = \SSS(\K[\L_{\mathcal{F}}])$. Consequently, there exists $\B \in \K[\L_{\mathcal{F}}]$ such that $\A  {\upharpoonright}_{\L_\mathcal{F}} \leq \B$.
As $\mathcal{G} \subseteq \ext(\K)$, \cref{Prop : pp expansions : subreducts : extendable} yields that $\mathsf{K}$ is the class of subreducts of $\mathsf{K}[\mathscr{L}_{\mathcal{G}}]$. Together with $\B \resLK \in \K$ (which holds because $\B \in \K[\L_\mathcal{F}]$), this entails that there exists $\C[\mathscr{L}_{\mathcal{G}}] \in \mathsf{K}[\mathscr{L}_{\mathcal{G}}]$ such that $\B \resLK \leq \C$. 

We will prove that $\A \leq \C[\mathscr{L}_{\mathcal{G}}]$. First, from $\L_\K \subseteq \L_\mathcal{F}$, $\A  {\upharpoonright}_{\L_\mathcal{F}} \leq \B$ and $\B \resLK \leq \C$ it follows that $\A\resLK \leq \C$. Therefore, it only remains to show that $g_f^{\A}(a_1, \dots, a_n) = g_f^{\C[\mathscr{L}_{\mathcal{G}}]}(a_1, \dots, a_n)$ for all $n$-ary $f \in \mathcal{G}$ and $a_1, \dots, a_n \in A$.
We have that
\begin{align*}
    g_f^{\A}(a_1, \dots, a_n)
    & = f_\bullet^{\A}(a_1, \dots, a_n)\\
    & = f^{\A\resLK}(a_1, \dots, a_n)\\
    & = f^{\C}(a_1, \dots, a_n)\\
    & = g_f^{\C[\mathscr{L}_{\mathcal{G}}]}(a_1, \dots, a_n).
\end{align*}
The first equality above holds because $\A \in \mathsf{M}[\mathscr{L}_{\mathcal{G}_{\bullet}}]$ and the second is a consequence of the definition of $f_\bullet$. 
The third equality follows from \cref{Prop : implicit operations extend} because $\A \resLK \leq \C$, and the last from the interpretation of $g_f$ in $\mathsf{K}[\mathscr{L}_{\mathcal{G}}]$. Therefore, $\A \leq \C[\mathscr{L}_{\mathcal{G}}] \in \mathsf{K}[\mathscr{L}_{\mathcal{G}}]$. We conclude that  $\mathsf{M}[\mathscr{L}_{\mathcal{G}_{\bullet}}] \subseteq \SSS(\mathsf{K}[\mathscr{L}_{\mathcal{G}}])$, as desired.

We now prove the inclusion from right to left of \eqref{Eq : pp exp on K and M coincide}. It suffices to show that $\mathsf{K}[\mathscr{L}_{\mathcal{G}}] \subseteq \mathsf{M}[\mathscr{L}_{\mathcal{G}_\bullet}]$.
To this end, consider $\A \in \mathsf{K}[\mathscr{L}_{\mathcal{G}}]$. Then $f^{\A\resLK}$ is a total function for every $f \in \mathcal{G}$. In particular, we have that $f^{\A\resLK}$ is total for every $f \in \mathcal{F}$ because $\mathcal{F} \subseteq \mathcal{G}$.
It follows that $\A\res_{\L_\F} \in \K[\L_\F] \subseteq \M$ and $f_\bullet^{\A\res_{\L_\F}}$ is total for every $f \in \mathcal{G}-\F$. Then $\A\res_{\L_\F}[\L_{\mathcal{G}_\bullet}]$ is defined and belongs to $\M[\L_{\mathcal{G}_\bullet}]$.
So, it only remains to show that $\A=\A\res_{\L_\F}[\L_{\mathcal{G}_\bullet}]$. Clearly, $\A$ and $\A\res_{\L_\F}$ have the same $\L_\mathcal{F}$-reduct. Moreover,
for every $f \in \mathcal{G}-\F$ we have that
\[
g_f^{\A} = f^{\A \resLK} = f_\bullet^{\A\res_{\L_\F}} = g_f^{\A\res_{\L_\F}[\mathscr{L}_{\mathcal{G}}]},
\] 
where the first equality is a consequence of the interpretation of $g_f$ in $\mathsf{K}[\mathscr{L}_{\mathcal{G}}]$, the second follows from the definition of $f_\bullet$ and the fact that $(\A\res_{\L_\F})\resLK = \A\resLK$, and the last one from the interpretation of $g_f$ in $\mathsf{M}[\mathscr{L}_{\mathcal{G}_\bullet}]$. Thus, $\mathsf{K}[\mathscr{L}_{\mathcal{G}}] \subseteq \mathsf{M}[\mathscr{L}_{\mathcal{G}_\bullet}]$, as desired.
\end{proof}

\section{The Beth companion}

 Recall that the strong Beth definability property is the demand that every implicit operation be interpolated by a set of terms  (see Definition \ref{Def : strong Beth prop}). We shall now extend the idea of interpolation to accommodate for situations in which the implicit operation and the set of terms belong to different classes of algebras.

\begin{Definition}
Let $\mathsf{K}$ and $\mathsf{M}$ be a pair of classes of algebras with $\mathscr{L}_\mathsf{K} \subseteq \mathscr{L}_\mathsf{M}$ such that the $\mathscr{L}_\mathsf{K}$-reducts of $\mathsf{M}$ belong to $\mathsf{K}$.\ We say that an $n$-ary  implicit operation 
$f$ of $\mathsf{K}$ is 
\benroman
\item\label{item 1 : interpolation extended : def} \emph{interpolated} in $\mathsf{M}$ by a set of $n$-ary terms $\{ t_i : i \in I \}$ of $\mathsf{M}$ when for all $\A \in \mathsf{M}$ and $\langle a_1, \dots, a_n \rangle \in \mathsf{dom}(f^{\A{\upharpoonright}_{\mathscr{L}_\mathsf{K}}})$ there exists $i \in I$ such that 
\[
f^{\A{\upharpoonright}_{\mathscr{L}_\mathsf{K}}}(a_1, \dots, a_n) = t_i^\A(a_1, \dots, a_n);
\]
\item\label{item 2 : interpolation extended : def}\emph{interpolated} in $\mathsf{M}$ by a set of $n$-ary partial functions $\{ g_i : i \in I \}$ of $\mathsf{M}$ when for all $\A \in \mathsf{M}$ and $\langle a_1, \dots, a_n \rangle \in \mathsf{dom}(f^{\A{\upharpoonright}_{\mathscr{L}_\mathsf{K}}})$ there exists $i \in I$ such that 
\[
\langle a_1, \dots, a_n \rangle \in \mathsf{dom}(g_i^\A) \, \, \text{ and } \, \, f^{\A{\upharpoonright}_{\mathscr{L}_\mathsf{K}}}(a_1, \dots, a_n) = g_i^\A(a_1, \dots, a_n).
\]
\eroman
\end{Definition}

\begin{Remark}
When $\mathsf{K} = \mathsf{M}$, part (\ref{item 1 : interpolation extended : def}) of the above definition specializes to the familiar demand that the  implicit 
operation 
$f$ be interpolated by the set of terms $\{ t_i : i \in I \}$.\ Notice that part (\ref{item 2 : interpolation extended : def}) subsumes part (\ref{item 1 : interpolation extended : def}), for term functions can be viewed as implicit operations (see \cref{Example : term functions}). However, we opted for splitting the definition in two halves for the sake of clarity. Lastly,  we remark that the above definition applies to the situation where $\mathsf{M}$ is a pp expansion of $\mathsf{K}$ because in this case $\mathscr{L}_\mathsf{K} \subseteq \mathscr{L}_\mathsf{M}$ and the $\mathscr{L}_\mathsf{K}$-reducts of $\mathsf{M}$ belong to $\mathsf{K}$ (see Proposition \ref{Prop : pp expansions and subreducts}).
\qed
\end{Remark}

Recall that the notion of a pp expansion was introduced to address the following question:\ is 
it possible to expand a given class of algebras $\mathsf{K}$ by introducing new function symbols for some of its implicit operations so that 
\benroman
\item every implicit operation of $\mathsf{K}$ becomes interpolable by a set of terms in the resulting expansion $\mathsf{M}$, and
\item the basic desiderata  \eqref{D1 : desiderata 1} and \eqref{D1 : desiderata 2}\footnote{See the first paragraph of \cref{Sec: adding implicit operations}.} are met?
\eroman
This idea is made precise by the following definition.

\begin{Definition}
A pp expansion $\mathsf{M}$ of a class of algebras $\mathsf{K}$ is said to be a \emph{Beth companion} of $\mathsf{K}$ when every implicit operation of $\mathsf{K}$ is interpolated in $\mathsf{M}$ by a set of terms of $\mathsf{M}$.
\end{Definition}

The aim of this section is to prove a triplet of results on Beth companions. The first governs the interplay between pp expansions and Beth companions.

\begin{Theorem}\label{Thm : Beth companions vs pp expansions}
Let $\mathsf{K}$ be a universal class, $\mathsf{M}_1$ a pp expansion of $\mathsf{K}$, and $\mathsf{M}_2$ a pp expansion of $\mathsf{M}_1$. Then the following conditions hold:
\benroman
\item\label{item : Beth companions vs pp expansions : 1} if $\mathsf{M}_1$ is a Beth companion of $\mathsf{K}$, then $\mathsf{M}_2$ is a Beth companion of $\mathsf{K}$;
\item\label{item : Beth companions vs pp expansions : 2} if $\mathsf{M}_2$ is a Beth companion of $\mathsf{M}_1$, then $\mathsf{M}_2$ is a Beth companion of $\mathsf{K}$.
\eroman
\end{Theorem}

In general, a quasivariety $\mathsf{K}$ need not possess a Beth companion  (see Section \ref{Sec : Classes without Beth comp}). However, when a Beth companion of $\mathsf{K}$ exists, the above result yields a description of a concrete Beth companion of $\mathsf{K}$.

\begin{Corollary}\label{Cor : Beth companion : ext(K)}
Let $\mathsf{K}$ be a  universal class, 
$\mathcal{F} = \mathsf{ext}_{\textsc{pp}}(\mathsf{K})$, and $\mathscr{L}_\mathcal{F}$ an $\mathcal{F}$-expansion of $\mathscr{L}_\mathsf{K}$. Then $\mathsf{K}$ has a Beth companion if and only if $\SSS(\mathsf{K}[\mathscr{L}_\mathcal{F}])$ is a Beth companion of $\mathsf{K}$.
\end{Corollary}

The second result connects Beth companions with the strong Beth definability and the strong epimorphism surjectivity properties as follows.

\begin{Theorem}\label{Thm : Beth companions vs strong Beth}
The following conditions are equivalent for a pp expansion $\mathsf{M}$ of a universal class $\mathsf{K}$:
\benroman
\item\label{item : Beth companions vs strong Beth : 1} $\mathsf{M}$ is a Beth companion of $\mathsf{K}$;
\item\label{item : Beth companions vs strong Beth : 2} $\mathsf{M}$ has the strong Beth definability property;
\item\label{item : Beth companions vs strong Beth : 3} $\mathsf{M}$ has the strong epimorphism surjectivity property;
\item\label{item : Beth companions vs strong Beth : 4} every member of $\mathsf{imp}_{\textsc{pp}}(\mathsf{K})$ is interpolated in $\mathsf{M}$ by a set of terms of $\mathsf{M}$.
\eroman
In addition, when $\mathsf{K}$ is a quasivariety, we can add the following equivalent condition:
\benroman
\setcounter{enumi}{4}
\item\label{item : Beth companions vs strong Beth : 5} every member of $\mathsf{imp}_{\textsc{pp}}(\mathsf{K})$ is interpolated in $\mathsf{M}$ by a single term of $\mathsf{M}$.
\eroman
\end{Theorem}

The last result in this section
states that, in the setting of quasivarieties, Beth companions are essentially unique (when they exist).
To make this precise, we adapt the notion of term equivalence (see, e.g., \cite[p.~131]{Ber11}) to expresses that two pp expansions of $\K$ in possibly distinct languages are essentially indistinguishable.
Let $\M_1$ and $\M_2$ be a pair of pp expansions of a class of algebras $\K$.
For $i = 1,2$ let $T_i$ be the set of terms of $\M_i$ with variables in $\{x_n : n \in \mathbb{N}\}$. Let $\rho \colon \L_{\M_2} \to T_1$ be a map that preserves the arities. For each $\L_{\M_1}$-algebra $\A$ let $\rho(\A)$ be the $\L_{\M_2}$-algebra with universe $A$ such that $f^{\rho(\A)}=\rho(f)^\A$ for each function symbol $f$ in $\L_{\M_2}$. Similarly, given an arity-preserving map $\tau \colon \L_{\M_1} \to T_2$ and an $\L_{\M_2}$-algebra $\B$, we define an $\L_{\M_1}$-algebra $\tau(\B)$.
We say that $\M_1$ and $\M_2$ are \emph{faithfully term equivalent relative to $\K$} if there exist arity-preserving maps $\tau \colon \L_{\M_1} \to T_2$ and $\rho \colon \L_{\M_2} \to T_1$ such that
$\tau(f)=f(x_1, \dots, x_n)$ and $\rho(f)=f(x_1, \dots, x_n)$ for each $n$-ary function symbol $f$ in $\L_\K$, and
for all $\A \in \M_1$ and $\B \in \M_2$ we have
\benroman
\item\label{item:term eq 1} $\rho(\A) \in \M_2$;
\item\label{item:term eq 2} $\tau(\B) \in \M_1$;
\item\label{item:term eq 3} $\tau \rho (\A)=\A$;
\item\label{item:term eq 4} $\rho \tau(\B)=\B$.
\eroman

In view of the following theorem, from now on, we will talk about \emph{the} Beth companion of a quasivariety.

\begin{Theorem}\label{Thm : Beth companions : term equivalence}
All the Beth companions of a quasivariety $\K$ are  faithfully term equivalent relative to $\mathsf{K}$.
\end{Theorem}

Before proving these results, we shall illustrate their applicability by describing Beth companions of familiar classes of algebras.

\begin{exa}[\textsf{Beth companions}]
Our aim is to prove the next result, which describes Beth companions of some familiar classes of algebras. Further examples will be given once sufficient portions of the theory of Beth companions will become available. The curious reader may consult \cref{table: Beth companions}, which summarizes compactly all the examples considered in this work.

\begin{Theorem}\label{Thm : Beth companion : examples}
The following conditions hold:
\benroman
\item\label{item : Beth companion : example : 1} the variety of Abelian groups is the Beth companion of the quasivariety of cancellative commutative monoids;
\item\label{item : Beth companion : example : 2} the variety of Boolean algebras is the Beth companion of the variety of bounded distributive lattices;
\item\label{item : Beth companion : example : 3} the variety of relatively complemented distributive lattices is the Beth companion of the variety of distributive lattices;
\item\label{item : Beth companion : example : 4} the variety of implicative semilattices is the Beth companion of the variety of Hilbert algebras;
\item\label{item : Beth companion : example : 5} the variety of Heyting algebras of depth $\leq 2$ is the Beth companion of the variety of pseudocomplemented distributive lattices;
\item\label{item : Beth companion : example : 6} every universal class with the strong epimorphism surjectivity property is a Beth companion of itself.
\eroman
\end{Theorem}

\begin{proof}
(\ref{item : Beth companion : example : 1}): By Theorem \ref{Thm : pp expansion : CCM} the variety $\mathsf{AG}$ of Abelian groups is a pp expansion of the quasivariety $\mathsf{CCMon}$ of cancellative commutative monoids. Recall from Example \ref{Exa : abalian groups : SES} that  $\mathsf{AG}$ has the strong epimorphism surjectivity property. Hence, we can apply Theorem \ref{Thm : Beth companions vs strong Beth}, obtaining that $\mathsf{AG}$ is the Beth companion of $\mathsf{CCMon}$.

(\ref{item : Beth companion : example : 2})--(\ref{item : Beth companion : example : 5}): 
Analogous to the proof of (\ref{item : Beth companion : example : 1}). For (\ref{item : Beth companion : example : 2}), use Theorems \ref{Thm : pp expansions : DL}(\ref{item : pp expansions : DL 2}) and \ref{Thm : SES : Boolean algebras}. For (\ref{item : Beth companion : example : 3}), use Theorems \ref{Thm : pp expansions : DL}(\ref{item : pp expansions : DL 1}) and \ref{Thm : SES : Boolean algebras}. For (\ref{item : Beth companion : example : 4}), use \cref{Thm : pp expansion : Hilbert algebras} and the fact that the variety of implicative semilattices has the strong epimorphism surjectivity property (see \cite[Props.~81 \&~82]{Dek20}). For (\ref{item : Beth companion : example : 5}), use \cref{thm: HA2 pp expansion PDL} and the fact that the variety of Heyting algebras of depth~$\leq 2$ has the strong epimorphism surjectivity property  (see \cite[Thm.~8.1(3)]{Mak00}).

(\ref{item : Beth companion : example : 6}): Consider a universal class $\mathsf{K}$ with the strong epimorphism surjectivity property. In view of Example \ref{Exa : pp expansions : lazy}, the class $\mathsf{K}$ is a pp expansion of itself. Therefore, we can apply \cref{Thm : Beth companions vs strong Beth}, obtaining that $\mathsf{K}$ is a Beth companion of itself. 
\end{proof}
\end{exa}

\begin{Remark}
The following shows that the hypothesis that $\K$ is a quasivariety in \cref{Thm : Beth companions : term equivalence} cannot be replaced with the requirement that $\K$ is a universal class. Let $\A = \langle A; \land, \lor, 0, 1\rangle$ be the
two-element bounded lattice and $\K=\UUU(\A)$. Since $\A$ is finite and has no proper subalgebras, from \cref{Thm : quasivariety generation} and \cref{Prop : P_u trivial in finite setting} it follows that
 $\K=  \UUU(\A) = \III\SSS\PPU(\A)=\III(\A)$. Therefore, every member of $\K$ lacks proper subalgebras because it is isomorphic to $\A$. Therefore, for all $\B \leq \C \in \K$ we have $\d_\K(\B,\C)=\d_\K(\C,\C)=\C$. Together with \cref{Prop : SES and dominions}, this yields that $\K$ has the strong epimorphism surjectivity property. Therefore, \cref{Thm : Beth companion : examples}\eqref{item : Beth companion : example : 6} implies that $\K$ is a Beth companion of itself. 

We will show that there exists another Beth companion $\M$ of $\K$ such that $\K$ and $\M$ are not faithfully term equivalent relative to $\K$.
Since $\K$ is a class of bounded distributive lattices, it follows from \cref{Thm : relative complements is implicit}\eqref{item : relative complements is implicit : 2} that taking complements yields a unary implicit operation $f$ of $\K$ defined by a conjunction of equations. As $f^\B$ is total for every 
 $\B \in \III(\A) = \K$, 
we have  $f \in \exteq(\K)$. Let $\L_f = \L_\K \cup \{\neg\}$ be an $f$-expansion of $\L_\K$ and $\M=\SSS(\K[\L_f])$ the pp expansion of $\K$ induced by $f$ and $\L_f$. As every member of $\K$ lacks proper subalgebras, so does every member of $\K[\L_f]$, and hence $\M=\K[\L_f]$. Together with $\K = \III(\A)$ and the fact that $\A[\L_f]$ is defined (because $f^\A$ is total), this yields $\M = \III(\A[\L_f])$.
 By arguing as above we obtain that $\M$ has the strong epimorphism surjectivity property. Therefore, \cref{Thm : Beth companions vs strong Beth} implies that $\M$ is a Beth companion of $\K$.

 It only remains to show that $\M$ and $\K$ are not faithfully term equivalent relative to $\K$. Suppose the contrary, with a view to contradiction. Then let $\tau \colon \L_\K \to T_1$ and $\rho \colon \L_f \to T_2$ be the maps witnessing the faithful term equivalence, where $T_1$ and $T_2$ are the sets of terms of $\L_f$ and $\L_\K$, respectively, with variables in $\{ x_n : n \in \mathbb{N} \}$. Recall  that $\K = \III(\A)$ and  $\M = \III(\A[\L_f])$. Therefore, $\A[\L_f] \cong \rho(\A)$ because  $\rho(\A) \in \M$. Let $h \colon  \rho(\A) \to \A[\L_f]$ be an isomorphism. For every $a \in A$ we have
 \[
  h(\rho(\lnot)^\A(a)) = h(\lnot^{\rho(\A)}a) = \lnot^{\A[\L_f]}h(a) = f^\A(h(a)).
 \]
Consequently, there exists a term $t(x)$ of $\K$ (namely, $\rho(\lnot)$) such that $h(t^\A(a)) = f^\A(h(a))$ for every $a \in A$.
Since $f^\A(h(a))$ is the complement of $h(a)$ in the two-element bounded lattice $\A$, we have $f^\A(h(a)) \neq h(a)$ for every $a \in A$. Therefore, for every $a \in A$ we obtain $h(t^\A(a)) \neq h(a)$, and 
hence
$t^\A(a) \neq a$. From $A=\{0,1\}$ it follows that $t^\A(1) = 0$ and $t^\A(0) = 1$. Since all the basic operations of $\A$ are order preserving in each component, $t^\A$ is order preserving in each component as well. Therefore, $t^\A(0) \leq t^\A(1)$, a contradiction with the fact that $t^\A(1) = 0$ and $t^\A(0) = 1$.
 Hence, $\K$ and $\M$ are not faithfully term equivalent relative to $\K$. In fact, the same argument shows that $\K$ and $\M$ are not even term equivalent (see, e.g., \cite[p.~131]{Ber11} for the definition of term equivalence).
\qed
\end{Remark}

Now, we turn our attention to proving Theorem \ref{Thm : Beth companions vs pp expansions}.

\begin{proof}
Recall from the assumptions that $\mathsf{M}_1$ is a pp expansion of $\mathsf{K}$ and that $\mathsf{M}_2$ is a pp expansion of $\mathsf{M}_1$. Therefore, $\mathsf{M}_2$ is also a pp expansion of $\mathsf{K}$ by \cref{Thm : pp expansion of pp expansions}. This fact will be used repeatedly in the proof.

(\ref{item : Beth companions vs pp expansions : 1}): Suppose that $\mathsf{M}_1$ is a Beth companion of $\mathsf{K}$. We will prove that so is $\mathsf{M}_2$. Since $\mathsf{M}_2$ is a pp expansion of $\mathsf{K}$, it suffices to show that every implicit operation of $\mathsf{K}$ is interpolated in $\mathsf{M}_2$ by a set of terms of $\mathsf{M}_2$. Accordingly, consider an $n$-ary $f \in \mathsf{imp}(\mathsf{K})$. As $\mathsf{M}_1$ is a Beth companion of $\mathsf{K}$, there exists a family $\{ t_i : i \in I \}$ of $\mathsf{M}_1$ that interpolates $f$ in $\mathsf{M}_1$. Since $\mathsf{M}_2$ is a pp expansion of $\mathsf{M}_1$, we know that $\{ t_i : i \in I \}$ is also a set of terms of $\mathsf{M}_2$. We will prove that it interpolates $f$ in $\mathsf{M}_2$.

To this end, consider $\A \in \mathsf{M}_2$ and $a_1, \dots, a_n \in A$ such that $\langle a_1, \dots, a_n \rangle \in \mathsf{dom}(f^{\A{\upharpoonright}_{\mathscr{L}_\mathsf{K}}})$. From $\mathscr{L}_\mathsf{K} \subseteq \mathscr{L}_{\mathsf{M}_1}$ it follows that $(\A{\upharpoonright}_{\mathscr{L}_{\mathsf{M}_1}}){\upharpoonright}_{\mathscr{L}_\mathsf{K}} = 
\A{\upharpoonright}_{\mathscr{L}_\mathsf{K}}$. Therefore, 
\[
f^{\A{\upharpoonright}_{\mathscr{L}_\mathsf{K}}} = f^{(\A{\upharpoonright}_{\mathscr{L}_{\mathsf{M}_1}}){\upharpoonright}_{\mathscr{L}_\mathsf{K}}} \, \, \text{ and } \, \, \langle a_1, \dots, a_n \rangle \in \mathsf{dom}(f^{(\A{\upharpoonright}_{\mathscr{L}_{\mathsf{M}_1}}){\upharpoonright}_{\mathscr{L}_\mathsf{K}}}).
\]
We will prove that $\A{\upharpoonright}_{\mathscr{L}_{\mathsf{M}_1}} \in \mathsf{M}_1$. Recall from the assumptions that  $\K$ is a universal class. Therefore, so is its pp expansion $\M_1$ by \cref{Thm : pp expansion : still quasivariety}(\ref{item : pp expansion : universal class}). Together with the assumption that $\M_2$ is a pp expansion of $\M_1$, this allows us to apply \cref{Prop : pp expansions : subreducts : extendable}, obtaining $\A{\upharpoonright}_{\mathscr{L}_{\mathsf{M}_1}} \in \mathsf{M}_1$, as desired.
Together with the above display and the assumption that $\{ t_i : i \in I \}$ interpolates $f$ in $\mathsf{M}_1$, this implies that there exists $i \in I$ such that $f^{\A{\upharpoonright}_{\mathscr{L}_\mathsf{K}}}(a_1, \dots, a_n) = t_i^{\A{\upharpoonright}_{\mathscr{L}_{\mathsf{M}_1}}}(a_1, \dots, a_n) = t_i^{\A}(a_1, \dots, a_n)$.

(\ref{item : Beth companions vs pp expansions : 2}): Assume that $\mathsf{M}_2$ is a Beth companion of $\mathsf{M}_1$. We will prove that $\mathsf{M}_2$ is a Beth companion of $\mathsf{K}$ as well. As in the previous case, it 
suffices to show that every implicit operation of $\mathsf{K}$ is interpolated in $\mathsf{M}_2$ by a family of 
terms of $\mathsf{M}_2$. Accordingly, consider an $n$-ary $f \in \mathsf{imp}(\mathsf{K})$. 
By \cref{Prop : lifting implicit operations} there exists $g \in \mathsf{imp}(\mathsf{M}_1)$ such that $g^\A = f^{\A{\upharpoonright}_{\mathscr{L}_{\mathsf{K}}}}$ for each $\A \in \mathsf{M}_1$. Observe that $g$ is 
interpolated in $\mathsf{M}_2$ by a family $\{ t_i : i \in I \}$ of terms of $\mathsf{M}_2$ because $\mathsf{M}_2$ 
is a Beth companion of $\mathsf{M}_1$. We will prove that the same family interpolates $f$ in $\mathsf{M}_2$. To 
this end, consider $\A \in \mathsf{M}_2$ and $a_1, \dots, a_n \in A$ such that $\langle a_1, \dots, a_n \rangle \in \mathsf{dom}(f^{\A{\upharpoonright}_{\mathscr{L}_\mathsf{K}}})$. 
 Since $\K$ is a universal class, \cref{Thm : pp expansion : still quasivariety}\eqref{item : pp expansion : universal class} implies that $\M_1$ is a universal class as well. 
As $\A \in \mathsf{M}_2$ and $\mathsf{M}_2$ 
is a pp expansion of $\mathsf{M}_1$
we have $\A{\upharpoonright}_{\mathscr{L}_{\mathsf{M}_1}} \in \mathsf{M}_1$ 
by Proposition \ref{Prop : pp expansions : subreducts : extendable}. Furthermore, $(\A{\upharpoonright}_{\mathscr{L}_{\mathsf{M}_1}}){\upharpoonright}_{\mathscr{L}_{\mathsf{K}}} = \A{\upharpoonright}_{\mathscr{L}_{\mathsf{K}}}$ because $\mathscr{L}_\mathsf{K} \subseteq \mathscr{L}_{\mathsf{M}_1}$. Therefore,
\[
g^{\A{\upharpoonright}_{\mathscr{L}_{\mathsf{M}_1}}} = f^{(\A{\upharpoonright}_{\mathscr{L}_{\mathsf{M}_1}}){\upharpoonright}_{\mathscr{L}_{\mathsf{K}}}} = f^{\A{\upharpoonright}_{\mathscr{L}_{\mathsf{K}}}} \, \, \text{ and } \, \, \langle a_1, \dots, a_n \rangle \in \mathsf{dom}(g^{\A{\upharpoonright}_{\mathscr{L}_{\mathsf{M}_1}}}).
\]
As $\{ t_i : i \in I \}$ interpolates $g$ in $\mathsf{M}_2$, the above display guarantees the existence of some $i \in I$ such that $f^{\A{\upharpoonright}_{\mathscr{L}_{\mathsf{K}}}}(a_1, \dots, a_n) = g^{\A{\upharpoonright}_{\mathscr{L}_{\mathsf{M}_1}}}(a_1, \dots, a_n) = t_i^\A(a_1, \dots, a_n)$.
\end{proof}
Then we prove Corollary \ref{Cor : Beth companion : ext(K)}.

\begin{proof}
It suffices to prove the implication from left to right, as the other one is straightforward. Accordingly, let 
 $\M$ be a Beth companion of $\K$.

\begin{Claim}\label{Claim : corollary : term equivalent to a Beth companion}
The class $\SSS(\mathsf{K}[\mathscr{L}_\mathcal{F}])$ is a pp expansion of $\mathsf{K}$ which, moreover, is term equivalent to a Beth companion of $\mathsf{K}$.
\end{Claim}

\begin{proof}[Proof of the Claim]
Since $\mathsf{M}$ is a pp expansion of $\mathsf{K}$, we may assume that it is induced by some $\mathcal{G} \subseteq \mathsf{ext}_{\textsc{pp}}(\mathsf{K})$ and $\mathcal{G}$-expansion $\mathscr{L}_\mathcal{G}$ of $\mathscr{L}_\mathsf{K}$. As $\mathcal{F} = \mathsf{ext}_{\textsc{pp}}(\mathsf{K})$ by assumption, we have $\mathcal{G} \subseteq \mathsf{ext}_{\textsc{pp}}(\mathsf{K}) = \mathcal{F}$. Then there exists an $\mathcal{F}$-expansion $\mathscr{L}_\mathcal{F}'$ of the form $\mathscr{L}_\mathcal{G} \cup \{ g_f : f \in \mathcal{F} - \mathcal{G} \}$. Since $\mathcal{G} \subseteq \mathcal{F}$ and $\mathscr{L}_\mathcal{G} \subseteq \mathscr{L}_\mathcal{F}'$, \cref{Thm : pp expansions : LG expands LF} implies that $\SSS(\mathsf{K}[\mathscr{L}_\mathcal{F}'])$ is a pp expansion of $\mathsf{M} = \SSS(\mathsf{K}[\mathscr{L}_\mathcal{G}])$. Together with the assumption that $\mathsf{M}$ is a Beth companion of $\mathsf{K}$, this allows us to apply Theorem \ref{Thm : Beth companions vs pp expansions}(\ref{item : Beth companions vs pp expansions : 1}), obtaining that $\SSS(\mathsf{K}[\mathscr{L}_\mathcal{F}'])$ is a Beth companion of $\mathsf{K}$ as well. Lastly, the definition of $\mathscr{L}_\mathcal{F}'$ guarantees that the classes $\SSS(\mathsf{K}[\mathscr{L}_\mathcal{F}])$ and $\SSS(\mathsf{K}[\mathscr{L}_\mathcal{F}'])$ are term equivalent. 
\end{proof}

Recall that $\SSS(\mathsf{K}[\mathscr{L}_\mathcal{F}])$ is a pp expansion of $\mathsf{K}$ by assumption. To prove that it is also a Beth companion of $\mathsf{K}$, consider some $f \in \mathsf{imp}(\mathsf{K})$. By Claim \ref{Claim : corollary : term equivalent to a Beth companion} the class $\SSS(\mathsf{K}[\mathscr{L}_\mathcal{F}])$ is term equivalent to a class in which $f$ is interpolated by a set of terms. By the definition of term equivalence, this guarantees that $f$ is also interpolated in $\SSS(\mathsf{K}[\mathscr{L}_\mathcal{F}])$ by a set of terms.
\end{proof}

Next we prove Theorem \ref{Thm : Beth companions vs strong Beth}.

\begin{proof}
Let $\mathsf{M}$ be a pp expansion of a universal class $\mathsf{K}$ of the form $\SSS(\K[\L_\F])$. 
\cref{Thm : pp expansion : still quasivariety}\eqref{item : pp expansion : universal class} implies that $\M$ is a universal class. So, conditions (\ref{item : Beth companions vs strong Beth : 2}) and (\ref{item : Beth companions vs strong Beth : 3}) are equivalent by \cref{Thm : strong ES = strong Beth}. Furthermore, the implication (\ref{item : Beth companions vs strong Beth : 1})$\Rightarrow$(\ref{item : Beth companions vs strong Beth : 4}) is straightforward.

(\ref{item : Beth companions vs strong Beth : 4})$\Rightarrow$(\ref{item : Beth companions vs strong Beth : 2}):
To prove that $\mathsf{M}$ has the strong Beth definability property, it suffices to show that every implicit operation of $\mathsf{M}$ defined by a pp formula is interpolated by a set of terms. For suppose that this is the case.  As $\mathsf{M}$ is an elementary class, 
we can apply Propositions~\ref{Prop : interpolation : infinite to finite} and~\ref{Prop : elementary strong Beth}, obtaining that $\mathsf{M}$ has the strong Beth definability property, as desired.
Then 
consider an $n$-ary $f \in \mathsf{imp}_{\textsc{pp}}(\mathsf{M})$. By \cref{Prop : equivalence in K[F]}(\ref{item : equivalence in K[F] : 3}) there exists $g \in \mathsf{imp}_{\textsc{pp}}(\mathsf{K})$ such that 
\begin{equation}\label{Eq : Beth companion = strong Beth : 1}
f^\A = g^{\A{\upharpoonright}_{\mathscr{L}_\mathsf{K}}}\text{ for each }\A \in \mathsf{K}[\mathscr{L}_\F]. 
\end{equation}
 Applying (\ref{item : Beth companions vs strong Beth : 4}) to $g \in \mathsf{imp}_{\textsc{pp}}(\mathsf{K})$
yields 
a set  $\{ t_i : i \in I \}$ of terms of $\mathsf{M}$ 
that interpolates $g$ in $\mathsf{M}$. We will show that $\{ t_i : i \in I \}$  interpolates 
$f$ as well. To this end, consider $\A \in \mathsf{M}$ and $a_1, \dots, a_n \in A$ such that $\langle a_1, \dots, a_n \rangle \in \mathsf{dom}(f^\A)$. Since $\A \in \mathsf{M} = \SSS(\K[\L_\F])$, we have $\A \leq \B$ for some $\B \in \mathsf{K}[\mathscr{L}_\F]$.
As $f \in \mathsf{imp}(\mathsf{M})
$ and $\A, \B \in \mathsf{M}$ are such that $\A \leq \B$, from \cref{Prop : implicit operations extend}  and $\langle a_1, \dots, a_n \rangle \in \mathsf{dom}(f^\A)$ it follows that 
\[
\langle a_1, \dots, a_n \rangle \in \mathsf{dom}(f^\B) \, \, \text{ and } \, \, f^\A(a_1, \dots, a_n) = f^\B(a_1, \dots, a_n).
\]
Furthermore, by applying (\ref{Eq : Beth companion = strong Beth : 1}) to the assumption that $\B \in \mathsf{K}[\mathscr{L}_\F]$,
we obtain $f^\B = g^{\B{\upharpoonright}_{\mathscr{L}_\mathsf{K}}}$. Together with the above display, this yields
\[
\langle a_1, \dots, a_n \rangle \in \mathsf{dom}(g^{\B{\upharpoonright}_{\mathscr{L}_\mathsf{K}}}) \, \, \text{ and } \, \, f^\A(a_1, \dots, a_n) = g^{\B{\upharpoonright}_{\mathscr{L}_\mathsf{K}}}(a_1, \dots, a_n).
\]
Since $\{ t_i : i \in I \}$  interpolates $g$ in $\mathsf{M}$, from the left hand side of the above display and $\B \in \mathsf{K}[\mathscr{L}_\F] \subseteq \mathsf{M}$
it follows that there exists $i \in I$ such that $g^{\B{\upharpoonright}_{\mathscr{L}_\mathsf{K}}}(a_1, \dots, a_n) = t_i^\B(a_1, \dots, a_n)$. Together with the right hand side of the above display and the fact that $\A \leq \B$, this yields
\[
f^\A(a_1, \dots, a_n) = g^{\B{\upharpoonright}_{\mathscr{L}_\mathsf{K}}}(a_1, \dots, a_n) = t_i^{\B}(a_1, \dots, a_n) = t_i^\A(a_1, \dots, a_n).
\]
Hence, we conclude that $\{ t_i : i \in I \}$ interpolates $f$, as desired.

(\ref{item : Beth companions vs strong Beth : 2})$\Rightarrow$(\ref{item : Beth companions vs strong Beth : 1}): Suppose that $\mathsf{M}$ has the strong Beth definability property and consider an $n$-ary $f \in \mathsf{imp}(\mathsf{K})$. We need to prove that $f$ is interpolated in $\mathsf{M}$ by a set of terms of $\mathsf{M}$. In view of Proposition \ref{Prop : lifting implicit operations}, there exists $g \in \mathsf{imp}(\mathsf{M})$ such that $g^\A = f^{\A{\upharpoonright}_{\mathscr{L}_{\mathsf{K}}}}$ for each $\A \in \mathsf{M}$. As $\mathsf{M}$ has the strong Beth definability property, the implicit operation $g$ is interpolated by a set of terms $\{ t_i : i \in I \}$ of $\mathsf{M}$.
We will show that this family interpolates $f$ in $\mathsf{M}$ as well. Consider $\A \in \mathsf{M}$ and $a_1, \dots, a_n \in A$ such that $\langle a_1, \dots, a_n \rangle \in \mathsf{dom}(f^{\A{\upharpoonright}_{\mathscr{L}_\mathsf{K}}})$. Together with $g^\A = f^{\A{\upharpoonright}_{\mathscr{L}_{\mathsf{K}}}}$, this yields
\[
\langle a_1, \dots, a_n \rangle \in \mathsf{dom}(g^{\A}) \, \, \text{ and } \, \, f^{\A{\upharpoonright}_{\mathscr{L}_\mathsf{K}}}(a_1, \dots, a_n) = g^\A(a_1, \dots, a_n).
\]
Since $g$ is interpolated by $\{ t_i : i \in I \}$, there exists $i \in I$ such that $g^\A(a_1, \dots, a_n) = t_i^\A(a_1, \dots, a_n)$. By the right hand side of the above display this amounts to $f^{\A{\upharpoonright}_{\mathscr{L}_\mathsf{K}}}(a_1, \dots, a_n) = t_i^\A(a_1, \dots, a_n)$. Hence, we conclude that $f$ is interpolated by $\{ t_i : i \in I \}$ in $\mathsf{M}$.

Next we prove the last part of the statement. To this end, in the rest of the proof, we assume that $\mathsf{K}$ is a quasivariety.

 (\ref{item : Beth companions vs strong Beth : 5})$\Rightarrow$(\ref{item : Beth companions vs strong Beth : 4}): Straightforward.
 
 (\ref{item : Beth companions vs strong Beth : 2})$\Rightarrow$(\ref{item : Beth companions vs strong Beth : 5}): Consider $f \in \mathsf{imp}_{\textsc{pp}}(\mathsf{K})$. In view of \cref{Prop : lifting implicit operations}, there exists $g \in \mathsf{imp}_{\textsc{pp}}(\mathsf{M})$ such that $g^\A = f^{\A{\upharpoonright}_{\mathscr{L}_{\mathsf{K}}}}$ for each $\A \in \mathsf{M}$. Observe that $\mathsf{M}$ is a quasivariety by \cref{Thm : pp expansion : still quasivariety}(\ref{item : pp expansion : quasivariety}). Together with the assumption that $\mathsf{M}$ has the strong Beth definability property and $g \in \mathsf{imp}_{\textsc{pp}}(\mathsf{M})$, this allows us to apply \cref{Prop : classes closed under products : strong Beth}, obtaining that $g$ is interpolated by a single term $t$ of $\mathsf{M}$. An argument analogous to the one detailed in the proof of the implication (\ref{item : Beth companions vs strong Beth : 2})$\Rightarrow$(\ref{item : Beth companions vs strong Beth : 1}) shows that $t$ interpolates $f$ in $\mathsf{M}$ as well (the only difference is that, in this case, the role of the family $\{ t_i : i \in I \}$ is taken over by the single term $t$).
\end{proof}

It only remains to prove \cref{Thm : Beth companions : term equivalence}. The proof hinges on the next observation.

\begin{Remark} \label{Rem : properties preserved by term equivalence}
The strong connection provided by a faithful term equivalence guarantees the preservation of many important properties.
    Let $\M_1$ and $\M_2$ be two pp expansions of $\K$ that are faithfully term equivalent relative to $\K$ and suppose that the term equivalence is witnessed by $\tau \colon \L_{\M_1} \to T_2$ and $\rho \colon \L_{\M_2} \to T_1$. Then it is straightforward to verify that the following conditions hold:
    \benroman
    \item\label{item : preservation term eq : homs} if $h \colon \A \to \B$ is a homomorphism  between $\mathscr{L}_{\M_1}$-algebras, then  $h \colon \rho(\A) \to \rho(\B)$ is a homomorphism between $\mathscr{L}_{\M_2}$-algebras; 
    \item\label{item : preservation term eq : subalg} if $\A, \B$ are $\mathscr{L}_{\M_1}$-algebras such that $\A \leq \B$, then $\rho(\A) \leq \rho(\B)$;
    \item\label{item: preservation term eq : reduct} if $\A \in \M_1$, then $\A\resLK=\rho(\A)\resLK$;
    \item\label{item : preservation term eq : Con lattices} if $\M_1$ and  $\M_2$ are quasivarieties, then $\mathsf{Con}_{\M_1}(\A) = \mathsf{Con}_{\M_2}(\rho(\A))$ for every $\A \in \M_1$; 
    \item\label{item : preservation term eq : (quasi)variety} if $\M_1$ is a universal class, quasivariety, or variety, then so is $\M_2$;
    \item\label{item : preservation term eq : Beth companion} $\M_1$ is a Beth companion of $\K$ if and only if $\M_2$ is a Beth companion of $\K$.
    \eroman
Since the roles of $\rho$ and $\tau$ are interchangeable, $\M_1$ and $\M_2$ may be swapped in the conditions above.
\qed
\end{Remark}

We are now ready to prove \cref{Thm : Beth companions : term equivalence}.

\begin{proof}
Let $\mathsf{M}_1$ and $\mathsf{M}_2$ be a pair of Beth companions of a quasivariety $\mathsf{K}$. We will prove that $\mathsf{M}_1$ and $\mathsf{M}_2$ are faithfully term equivalent relative to $\K$.
For $i = 1,2$ let $\mathcal{F}_i \subseteq \mathsf{ext}_{\textsc{pp}}(\mathsf{K})$ be such that $\mathsf{M}_i = \SSS(\mathsf{K}[\mathscr{L}_{\mathcal{F}_i}])$. We may assume that 
\[
\L_{\mathcal{F}_1} = \L_\K \cup \{g_f : f \in \mathcal{F}_1 \} \, \, \text{ and } \, \, \L_{\mathcal{F}_2} = \L_\K \cup \{g_f : f \in \mathcal{F}_2 \}.
\]
Then let $T_i$ be the set of terms of $\mathscr{L}_{\F_i}$ with variables in $\{ x_n : n \in \mathbb{N} \}$. 
Since $\K$ is a quasivariety, \cref{Thm : Beth companions vs strong Beth} implies that for every $n$-ary $f \in \mathsf{imp}_{\textsc{pp}}(\mathsf{K})$ and $i = 1, 2$ there exists an $n$-ary term  $F_i(f) \in T_i$
that interpolates $f$ in $\M_i$. Therefore, for every $n$-ary $f \in \imppp(\K)$, $\A \in \M_i$ with $i = 1, 2$, and $a_1, \dots, a_n \in A$,
\begin{equation} \label{Eq : f and F_i}
\langle a_1, \dots, a_n \rangle \in \mathsf{dom}(f^{\A \resLK}) \text{ implies }    f^{\A \resLK}(a_1, \dots, a_n) = F_i(f)^{\A}(a_1, \dots, a_n).
\end{equation}

Define $\tau \colon \mathscr{L}_{\F_1} \to T_2$ and 
$\rho \colon \mathscr{L}_{\F_2} \to T_1$ by setting 
 $\tau(f)=f(x_1, \dots, x_n)$ and $\rho(f)=f(x_1, \dots, x_n)$ for every $n$-ary $f \in \L_\K$, and 
\[
\tau(g_{f_1})=F_2(f_1) \quad \text{and} \quad \rho(g_{f_2})=F_1(f_2)
\]
for each $f_1 \in \F_1$ and $f_2 \in \F_2$.  We proceed to show that $\tau$ and $\rho$ witness that $\M_1$ and $\M_2$ are faithfully term equivalent relative to $\K$. Let $\F=\F_1 \cup \F_2$ and let $\L_\F$ be an $\F$-expansion of $\L_\K$ such that $\L_{\F_1},\L_{\F_2} \subseteq \L_\F$. Recall that $\mathcal{F} = \mathcal{F}_1 \cup \mathcal{F}_2 \subseteq \extpp(\K)$ by assumption. Therefore,
\cref{Thm : pp expansions : LG expands LF} implies that  $\SSS(\mathsf{K}[\mathscr{L}_{\mathcal{F}}])$ is a pp expansion of both $\mathsf{M}_1 = \SSS(\mathsf{K}[\mathscr{L}_{\mathcal{F}_1}])$ and $\mathsf{M}_2 = \SSS(\mathsf{K}[\mathscr{L}_{\mathcal{F}_2}])$. To verify conditions \eqref{item:term eq 1} and \eqref{item:term eq 3} in the definition of faithful term equivalence, consider $\A \in \M_1$. We need to prove that $\rho(\A) \in \M_2$ and $\tau(\rho(\A)) =\A$. Since $\SSS(\mathsf{K}[\mathscr{L}_{\mathcal{F}}])$ is a pp expansion of $\mathsf{M}_1$, \cref{Prop : pp expansions and subreducts} implies that $\A$ is an $\L_{\F_1}$-subreduct of $\SSS(\mathsf{K}[\mathscr{L}_{\mathcal{F}}])$. As every member of $\SSS(\mathsf{K}[\mathscr{L}_{\mathcal{F}}])$ is a subalgebra of a member of  $\mathsf{K}[\mathscr{L}_{\mathcal{F}}]$, it follows that there exists $\B \in \mathsf{K}[\mathscr{L}_{\mathcal{F}}]$ such that $\A \leq \B\res_{\L_{\F_1}}$. 
 Since $\M_2$ is a universal class by \cref{Thm : pp expansion : still quasivariety}\eqref{item : pp expansion : universal class} and $\SSS(\mathsf{K}[\mathscr{L}_{\mathcal{F}}])$ is also a pp expansion of $\M_2$, \cref{Prop : pp expansions and subreducts} implies that $\B\res_{\L_{\F_2}} \in \M_2$.

We now show that $\rho(\A) \in \M_2$. Observe that $\M_2$ is closed under subalgebras because it is a pp expansion of $\K$. Hence, it is sufficient to prove  $\rho(\A) \leq \B\res_{\L_{\F_2}}$ because $\B\res_{\L_{\F_2}} \in \M_2$. 
We have
\begin{equation}\label{eq:rho tau preserve reduct}
\rho(\C)\resLK =\C\resLK \text{ and } \tau(\D)\resLK =\D\resLK \text{ for all }\C \in \M_1 \text{ and }\D \in \M_2    
\end{equation}
because $\rho(f)=f(x_1, \dots, x_n)$ and $\tau(f)=f(x_1, \dots, x_n)$ for every $n$-ary $f \in \L_\K$ by the definition of $\rho$ and $\tau$.
Since $\A \leq \B\res_{\L_{\F_1}}$, 
from 
\eqref{eq:rho tau preserve reduct}
it follows that 
\begin{equation}\label{Eq : rho preserves the reduct : Beth companion : term equivalence}
    \rho(\A)\resLK =\A\resLK \leq (\B\res_{\L_{\F_1}})\resLK =\B\resLK = (\B\res_{\L_{\F_2}})\resLK.
\end{equation}
Therefore, it only remains to show that $g_f^{\rho(\A)}(a_1, \dots, a_n)=g_f^{\B\res_{\L_{\F_2}}}(a_1, \dots, a_n)$ for all $n$-ary $g_f \in \L_{\F_2}-\L_\K$ and $a_1, \dots, a_n \in A$. We have
\begin{align*}
g_f^{\rho(\A)}(a_1, \dots, a_n) &= \rho(g_f)^\A(a_1, \dots, a_n)\\
&= F_1(f)^\A(a_1, \dots, a_n)\\
&= F_1(f)^{\B\res_{\L_{\F_1}}}(a_1, \dots, a_n)\\
&= f^{(\B\res_{\L_{\F_1}})\resLK}(a_1, \dots, a_n)\\
&= f^{(\B\res_{\L_{\F_2}})\resLK}(a_1, \dots, a_n)\\
&= g_f^{\B\res_{\L_{\F_2}}}(a_1, \dots, a_n),
\end{align*}
where the first two equalities hold by the definitions of $\rho(\A)$ and $\rho(g_f)$, 
the third holds because $\A \leq \B\res_{\L_{\F_1}}$,
the fifth because $(\B\res_{\L_{\F_1}})\resLK = (\B\res_{\L_{\F_2}})\resLK$, and the last because $\B \in \K[\L_{\F}]$ implies $\B\res_{\L_{\F_2}} \in \K[\L_{\F_2}]$.
The fourth equality follows from \eqref{Eq : f and F_i} because $(\B\res_{\L_{\F_1}})\resLK = (\B\res_{\L_{\F_2}})\resLK$ and $\B\res_{\L_{\F_2}} \in \K[\L_{\F_2}]$ imply $\langle a_1, \dots, a_n \rangle \in \dom(f^{(\B\res_{\L_{\F_1}})\resLK})$.
 We conclude that $\rho(\A) \leq \B\res_{\L_{\F_2}}$, and hence $\rho(\A) \in \M_2$.

Next we prove that $\tau(\rho(\A)) =\A$. 
By 
\eqref{eq:rho tau preserve reduct}
we obtain
\[
\tau(\rho(\A))\resLK = \rho(\A)\resLK =\A\resLK.
\]
Therefore, it remains to show that $g_f^{\tau(\rho(\A))}=g_f^\A$ for every $n$-ary $g_f \in \L_{\F_1}-\L_K$. Let $a_1, \dots, a_n \in A$. Then
\begin{align*}
g_f^{\tau(\rho(\A))}(a_1, \dots, a_n) &= \tau(g_f)^{\rho(\A)}(a_1, \dots, a_n)\\
&= F_2(f)^{\rho(\A)}(a_1, \dots, a_n)\\
&= F_2(f)^{\B\res_{\L_{\F_2}}}(a_1, \dots, a_n)\\
&= f^{(\B\res_{\L_{\F_2}})\resLK}(a_1, \dots, a_n)\\
&= f^{(\B\res_{\L_{\F_1}})\resLK}(a_1, \dots, a_n)\\
&= g_f^{\B\res_{\L_{\F_1}}}(a_1, \dots, a_n)\\
&= g_f^{\A}(a_1, \dots, a_n),
\end{align*}
where the first equality follows from the definition of $\tau(\E)$ for an $\L_{\F_2}$-algebra $\E$, the second from the definition of $\tau(g_f)$, 
the third holds because 
$\rho(\A) \leq \B\res_{\L_{\F_2}}$ as we established above,
 the fifth because $(\B\res_{\L_{\F_2}})\resLK = (\B\res_{\L_{\F_1}})\resLK$, the sixth because $\B \in \K[\L_{\F}]$ implies $\B\res_{\L_{\F_1}} \in \K[\L_{\F_1}]$, and the last because $\A \leq \B\res_{\L_{\F_1}}$. 
 The fourth equality follows from \eqref{Eq : f and F_i} because $(\B\res_{\L_{\F_2}})\resLK = (\B\res_{\L_{\F_1}})\resLK$ and $\B\res_{\L_{\F_1}} \in \K[\L_{\F_1}]$ imply $\langle a_1, \dots, a_n \rangle \in \dom(f^{(\B\res_{\L_{\F_2}})\resLK})$. 
 We conclude that $\tau(\rho(\A)) =\A$.

Thus, conditions \eqref{item:term eq 1} and \eqref{item:term eq 3} in the definition of faithful term equivalence hold. The proof that conditions \eqref{item:term eq 2} and \eqref{item:term eq 4} hold is analogous.
 Since 
 $\tau(f)=f(x_1, \dots, x_n)$ and $\rho(f)=f(x_1, \dots, x_n)$ for every $n$-ary $f \in \L_\K$, 
 we conclude that $\mathsf{M}_1$ and $\mathsf{M}_2$ are faithfully term equivalent relative to $\K$.
\end{proof}

\section{Structure theory}

In this section, we address the following question:\ what do we gain by moving from a quasivariety to its Beth companion? We will answer it by showing that not only does the Beth companion of a quasivariety $\K$ inherit a significant portion of the structure theory of $\K$, but its structure theory is often much richer than that of $\K$. 

We recall that if $\M$ is a pp expansion of a quasivariety $\K$, then $\M$ is a quasivariety and $\K$ the class of $\L_\K$-subreducts of $\M$ (see \cref{Thm : pp expansion : still quasivariety}\eqref{item : pp expansion : quasivariety} and \cref{Prop : pp expansions and subreducts}).

\begin{Definition}\label{def:congruence preserving simple equational}
    A pp expansion $\M$ of a quasivariety $\K$ is said to be 
    \benroman
\item \emph{congruence preserving} when $\mathsf{Con}_{\mathsf{M}}(\A) = \mathsf{Con}_{\mathsf{K}}(\A \resLK)$ for every $\A \in \mathsf{M}$;
\item \emph{simple} when it is of the form $\K[\L_\F]$ for some $\mathcal{F} \subseteq \extpp(\K)$;
\item \emph{equational} when it is induced by some $\mathcal{F} \subseteq \exteq(\K)$.
\eroman
We say that a Beth companion of $\K$ is \emph{congruence preserving} (resp.\ \emph{simple}, or \emph{equational}) when it is so as a pp expansion of $\K$. 
\end{Definition}

\begin{Remark}
The properties introduced above are preserved by faithful term equivalence. More precisely, let $\M_1$ and $\M_2$ be two pp expansions of a quasivariety $\K$ that are faithfully term equivalent relative to $\K$. It is straightforward to prove that $\M_1$ is congruence preserving if and only if $\M_2$ is. In \cite[Cor.~3.6]{BethCat26}, it is shown that $\M_1$ is simple if and only if $\M_2$ is. It can also be proved that $\M_1$ is equational if and only if $\M_2$ is; since we will not rely on this fact, we omit its proof.
\qed
\end{Remark}

The next result compares the strength of the notions introduced in \cref{def:congruence preserving simple equational}. 
To this 
end, we recall that a quasivariety $\K$ has the \emph{relative congruence extension property} when for all $\A \leq \B \in \K$ and $\theta \in \mathsf{Con}_\K(\A)$ there exists $\phi \in \mathsf{Con}_\K(\B)$ such that $\theta = \phi{\upharpoonright}_A$. When $\K$ is a variety, we simply say that $\K$ has the \emph{congruence extension property}.\footnote{Although we will not rely on this fact, we remark that the amalgamation and the relative congruence extension properties persist in pp expansions of quasivarieties.} Lastly, we write ``AP'' as a shorthand for ``amalgamation property''.

\begin{Theorem}\label{Thm : implications : simple - equational}
Let $\M$ be a pp expansion of a quasivariety $\K$. Then the following holds:
\[
\K \text{ has the AP } \Longrightarrow \M \text{ is equational } \Longrightarrow \M \text{ is simple }\Longrightarrow\M \text{ is congruence preserving}.    
\]
Moreover, if $\K$ has the relative congruence extension property, then $\M$ is congruence preserving.
\end{Theorem}

\begin{Remark}
All the quasivarieties in (\ref{item : Beth companion : example : 1})--(\ref{item : Beth companion : example : 5}) of \cref{Thm : Beth companion : examples}  have the amalgamation property.\footnote{The class $\K$ has the amalgamation property when it is the quasivariety of  cancellative commutative monoids (see, e.g., \cite[pp.~100, 108]{SurvKissal}), the variety of (resp.\ bounded) distributive lattices  (see, e.g., \cite[p.\ 114]{BD74}), the variety of Hilbert algebras (see \cite[Thm.\ 1]{MR975871}), or
the variety of pseudocomplemented distributive lattices (see \cite[Thm.\ 5]{GLAP}).
} Therefore, their Beth companions are amenable to the chain of implications in \cref{Thm : implications : simple - equational}.
 On the other hand, we remark that the converses of the implications in \cref{Thm : implications : simple - equational} need not hold in general. A nonsimple congruence preserving Beth companion is exhibited in \cite[Thm.~3.3]{CKMEXAv2}. For a simple nonequational Beth companion, see \cref{Exa : Beth companion : meadows : zero characteristic}. The variety generated by the four-element linearly ordered Heyting algebra has the strong epimorphism surjectivity property (see \cite[Thm.~8.1]{Mak00}) and, therefore, is its own Beth companion by \cref{Thm : Beth companion : examples}(\ref{item : Beth companion : example : 6}). As such, it is equational. However, it lacks the amalgamation property (see \cite[Thm.\ 2]{Mak77}). Lastly, the variety of lattices is its own Beth  companion 
(see page~\pageref{lattices ES} and \cref{Thm : Beth companion : examples}(\ref{item : Beth companion : example : 6})) and, as such, it is congruence preserving. Nonetheless, it lacks the congruence extension property (see, e.g., \cite[p.~102]{SurvKissal}).
\qed
\end{Remark}

The congruence preserving pp expansions of a quasivariety $\K$ inherit a substantial portion of the structure theory of $\K$, namely, the one related to the structure of lattices of $\mathsf{K}$-congruences. The next concepts are instrumental to make this idea precise.  Let $\K$ be a quasivariety. A \emph{congruence equation} is a formal equation in the binary
symbols $\land, \lor$, and $\circ$. A congruence equation is \emph{valid} in an algebra $\A$ relative to $\K$ when it becomes true
whenever we interpret the variables of the equation as $\mathsf{K}$-congruences of $\A$, and for arbitrary binary relations $\alpha$ and $\beta$ on $A$, we interpret $\alpha \land \beta$, $\alpha \lor \beta$, and $\alpha \circ \beta$ as $\alpha \cap \beta$, $\mathsf{Cg}_\K^\A(\alpha \cup \beta)$, and $\alpha \circ \beta$, respectively.
We say that a congruence equation is \emph{valid} in $\K$ when it is valid relative to $\K$ in every member of $\K$ \cite{Pixley72,Wille70} (see also \cite{KearnesKiss12}). For instance, a quasivariety $\K$ is relatively congruence distributive precisely when the congruence equation $(x \lor y) \land (x \lor z) \thickapprox x \lor (y \land z)$ is valid in $\K$. Similarly, a variety $\K$ is congruence permutable if and only if the congruence equation $x \circ y  = y \circ x$ is valid in $\K$.

\begin{Theorem}\label{Thm: easy consequences of congruence preservation : congruence equations}
Let $\M$ be a congruence preserving pp expansion of a quasivariety $\K$. Then every congruence equation valid in $\K$ is valid in $\M$.
\end{Theorem}

\begin{proof}
        Immediate from the definitions of a congruence preserving pp expansion and of a  congruence equation.
\end{proof}

In addition, congruence preserving pp expansions preserve and reflect the property of ``being relatively (finitely) subdirectly irreducible'' (cf.\  \cref{Thm : pp expansion : still quasivariety}(\ref{item : pp expansion : quasivariety})) and preserve the property of ``being a variety''.
We recall that the latter fails for arbitrary pp expansions (see \cite[Thm.~2.1]{CKMEXAv2}). More precisely, we will prove the following.

\begin{Theorem}\label{Thm: easy consequences of congruence preservation : variety and RFSI}
Let $\M$ be a congruence preserving pp expansion of a quasivariety $\K$. Then 
\[
\mathsf{M}_\textsc{rfsi} = \{ \A \in \M : \A{\upharpoonright}_{\mathscr{L}_\K} \in \K_{\textsc{rfsi}}\} \, \, \text{ and } \, \, \mathsf{M}_\textsc{rsi} = \{ \A \in \M : \A{\upharpoonright}_{\mathscr{L}_\K} \in \K_{\textsc{rsi}}\}.
\] 
Moreover, if $\mathsf{K}$ is a variety, then $\mathsf{M}$ is a variety. 
\end{Theorem}

As we mentioned, the aim of this section is to explain what 
we gain by moving from a quasivariety to its Beth companion.\ Our main result states that not only does every congruence preserving Beth companion of a quasivariety $\K$ inherit a significant portion of the structure theory of $\K$ (see Theorems \ref{Thm: easy consequences of congruence preservation : congruence equations} and \ref{Thm: easy consequences of congruence preservation : variety and RFSI}), but it often gains remarkable properties in comparison to those of $\K$. 
 More precisely, we will establish the following.

\begin{Theorem}\label{Thm : Con preserving Beth comp : structure theorem : main}
Let $\K$ be a relatively congruence distributive quasivariety for which $\K_{\textsc{rfsi}}$ is closed under nontrivial subalgebras.\ Then every congruence preserving Beth companion $\M$ of $\K$ is an arithmetical variety with the congruence extension property such that $\M_\textsc{fsi}$ is closed under nontrivial subalgebras.
\end{Theorem}

In view of \cref{Thm : Con preserving Beth comp : structure theorem : main}, under reasonable assumptions, the structure theory of a congruence preserving Beth companion $\mathsf{M}$ of a quasivariety $\K$ enhances that of $\K$ as follows: $\mathsf{M}$ turns out to be a variety (as opposed to a quasivariety) which, moreover, is arithmetical (as opposed to relatively congruence distributive only) and possesses the congruence extension property.

Let us illustrate the effect of \cref{Thm : Con preserving Beth comp : structure theorem : main} in the setting of filtral quasivarieties. A variety $\K$ is said to be \emph{discriminator} when there exist a class of algebras $\M$ and a quaternary term $t$ such that $\K = \VVV(\M)$ and $t^{\A}$ is the quaternary discriminator function on $A$ for every $\A \in \M$ \cite{BurrisWener79,Werner78} (see also  \cite[Sec.~IV.9]{BuSa00}). Examples of discriminator varieties include the variety of Boolean algebras and for each $n \in \mathbb{N}$ the variety of rings satisfying the equation $x \thickapprox x^n$ (see, e.g., \cite[pp.\ 179–180]{Burris1992symbolic}). The importance of discriminator varieties derives from the fact that they admit a general representation theorem in terms of Boolean products with subdirectly irreducible factors \cite[Thm.\ IV.9.4]{BuSa00} (see also \cite{Vaggione2001}). While every discriminator variety is filtral (see, e.g., \cite[p.~101]{Ber87}),
the converse need not hold in general: for instance, the variety of (bounded) distributive lattices is filtral, but it is not discriminator. 
However, its Beth companion, the variety of Boolean algebras, is a discriminator variety. From \cref{Thm : Con preserving Beth comp : structure theorem : main} we will infer that this is true in general.

\begin{Corollary} \label{Cor : rel filtral -> discriminator}
    Every Beth companion of a relatively filtral quasivariety is a discriminator variety.
\end{Corollary}

Before proving these results, let us illustrate the applicability of Theorem \ref{Thm : Con preserving Beth comp : structure theorem : main} in the context of ring theory.

\begin{exa}[\textsf{Reduced commutative rings}]\label{Exa : Beth companion : meadows : zero characteristic}
We recall that in a field of prime characteristic $p$ every element has at most one $p$-th root
(see, e.g., \cite[F14 p.\ 71]{FLField}). The following concept originates in \cite{CKM25RCR}. A field $\A$ is said to be \emph{weakly rooted} when one of the following conditions holds:
\benroman
\item $\A$ has characteristic $0$;
\item $\A$ has prime characteristic $p$ and all its elements have a   $p$-th root.
\eroman
Given a prime $p$, the \emph{weak $p$-root} of an element $a$ of a weakly rooted field $\A$ is 
\[
 r_p(a) = \begin{cases*}
                    \text{the $p$-th root of $a$}
                    & if  $\A$ has characteristic $p$;  \\
                     0 & otherwise.
                 \end{cases*}
 \]
An \emph{implicitly closed field} is an algebra $\langle A; +, \cdot, -, 0, 1, (\,)^*, \{ r_p : \text{$p$ is prime}\}\rangle$, where the tuple $\langle A; +, \cdot, -, 0, 1, (\,)^*\rangle$ is a zero-totalized field (see \cref{exa:meadows}) that is weakly rooted and $r_p$ is the unary operation of taking weak $p$-roots.
Algebras that embed into direct products of implicitly closed fields will be called \emph{implicitly closed meadows}. 
The class $\mathsf{ICM}$ of implicitly closed meadows is a variety axiomatized by the axioms of meadows (see \cref{exa:meadows}) together with the equation 
\[
(r_p(x))^p \thickapprox  (1-p^*p)x
\]
for every prime $p$, where we denote by $p$ the result of summing $p$-times the unit $1$.

It is established in \cite[Thm.~6.1]{CKM25RCR} that $\mathsf{ICM}$ is a simple Beth companion of the quasivariety $\mathsf{RCRing}$ of reduced commutative rings. Thus, from  
\cref{Thm : implications : simple - equational}
it follows that $\mathsf{ICM}$ is a congruence preserving Beth companion of $\mathsf{RCRing}$.
We also recall that  $\mathsf{RCRing}$ is relatively congruence distributive (see \cite[Thm.~2]{Tom54}) and its relatively finitely subdirectly irreducible members are precisely the integral domains 
(see \cite[p.~410]{CampRaf}),
where an \emph{integral domain} is a commutative ring validating the formulas $0 \not \thickapprox 1$ and $xy\thickapprox 0 \to (x\thickapprox 0 \sqcup y\thickapprox 0)$. Therefore, the class of relatively finitely subdirectly irreducible members of $\mathsf{RCRing}$ is closed under nontrivial subalgebras. Consequently,  from  \cref{Thm : Con preserving Beth comp : structure theorem : main} it follows that $\mathsf{ICM}$ is an arithmetical variety with the congruence extension property. None of these facts holds for $\mathsf{RCRing}$, which is a proper quasivariety without the relative congruence extension property\footnote{To show that $\mathsf{RCRing}$ lacks the relative congruence extension property, consider the polynomial ring $\mathbb{Z}[x]$ and observe that it is an integral domain  and, therefore, belongs to $\mathsf{RCRing}$.
Then let $\theta$ be the congruence of $\mathbb{Z}[x]$ generated by the pair $\langle x, 0 \rangle$. As $\mathbb{Z}[x] / \theta \cong \mathbb{Z}$, we obtain that $\theta$ is an $\mathsf{RCRing}$-congruence of $\mathbb{Z}[x]$. On the other hand, observe that $\mathbb{Z}[x]$ is a subalgebra of a field $\A$ because it is an integral domain (see, e.g., \cite[Thm.~11.7.2]{Art10}). Since $\A$ is a field, it does not possess a congruence extending $\theta$.} (congruence permutability is usually understood as a property of varieties only).

Lastly, we prove that $\mathsf{ICM}$ is not an equational pp expansion of $\mathsf{RCRing}$, as we promised earlier. For suppose that $\mathsf{ICM}$ is equational, with a view to contradiction. By \cref{Thm : axiomatization of pp expansion (almost always)} there exists a family $\F = \{ f_*\} \cup \{ f_p : p \text{ is prime}\} \subseteq \exteq(\mathsf{RCRing})$ of unary implicit operations  corresponding to $\{ (\,)^* \} \cup \{ r_p : \text{$p$ is prime}\}$ such that $\mathsf{ICM} = \mathsf{RCRing}[\L_\F]$. Then consider the fields $\mathbb{Q}$ and $\mathbb{Z}/2\mathbb{Z}$, viewed as implicitly closed fields. As $f_2 \in \exteq(\mathsf{RCRing})$, there exists a conjunction of equations $\varphi(x, y)$ in $\L_\mathsf{RCRing}$ defining $f_2$. Since $\mathbb{Q}, \mathbb{Z}/2\mathbb{Z} \in \mathsf{ICM} = \mathsf{RCRing}[\L_\F]$ and $r_2^\mathbb{Q}(1) = 0$ and $r_2^{\mathbb{Z}/2 \mathbb{Z}}(1) = 1$, we have
\[
\mathbb{Q} \vDash \varphi(1, 0) \, \, \text{ and }\, \, \mathbb{Z} / 2 \mathbb{Z} \vDash \varphi(1, 1).
\]
Now, let $\mathbb{Z}$ be the ring of integers. As $\varphi(x, y)$ is a conjunction of equations in $\L_{\mathsf{RCRing}}$, from the left hand side of the above display and $\mathbb{Z} \leq \mathbb{Q}{\upharpoonright}_{\L_\mathsf{RCRing}}$, it follows that $\mathbb{Z} \vDash \varphi(1, 0)$. Let $\pi \colon \mathbb{Z}  \to \mathbb{Z}/ 2 \mathbb{Z}$ be the canonical ring homomorphism. Since $\varphi(x, y)$ is a conjunction of equations in $\L_{\mathsf{RCRing}}$ and $\mathbb{Z} \vDash \varphi(1, 0)$, we obtain $\mathbb{Z}/2 \mathbb{Z} \vDash \varphi(1, 0)$. Together with the right hand side of the above display, this contradicts the assumption that $\varphi(x, y)$ is functional in $\mathsf{RCRing}$.
\qed
\end{exa}

Next, we prove Theorems 
\ref{Thm : implications : simple - equational},
\ref{Thm: easy consequences of congruence preservation : variety and RFSI}, 
and \ref{Thm : Con preserving Beth comp : structure theorem : main} and Corollary \ref{Cor : rel filtral -> discriminator}. 
The proof of  
\cref{Thm : implications : simple - equational} hinges upon the next 
observation.

\begin{Proposition} \label{Prop : trivial inclusion f congruence preservation}
Let $\M=\SSS(\K[\mathscr{L}_\mathcal{F}])$ be a pp expansion of a quasivariety $\K$. Then the following conditions hold:
\benroman
\item\label{item : cp pp exp : 1} for every $\A \in \M$ we have $\mathsf{Con}_{\mathsf{M}}(\A) \subseteq \mathsf{Con}_{\mathsf{K}}(\A \resLK)$;
\item\label{item : cp pp exp : 2} for every $\A \in \K[\mathscr{L}_\mathcal{F}]$ we have $\mathsf{Con}_\M(\A) = \mathsf{Con}_\K(\A \resLK)$.
\eroman 
\end{Proposition}

\begin{proof}
 \eqref{item : cp pp exp : 1}:  Consider $\theta \in \mathsf{Con}_{\mathsf{M}}(\A)$. The inclusion $\mathscr{L}_{\mathsf{K}} \subseteq \mathscr{L}_{\mathsf{M}}$ guarantees that $\theta$ is a congruence of $\A \resLK$. 
    Moreover, since $\theta$ is an $\mathsf{M}$-congruence of $\A$, we have $\A/\theta \in \mathsf{M}$. 
    Thus, $(\A \resLK)/\theta = (\A/\theta) \resLK \in \mathsf{M} \resLK \subseteq \mathsf{K}$, where the last inclusion holds by \cref{Prop : pp expansions and subreducts}.
    Consequently,  $\theta \in \mathsf{Con}_{\mathsf{K}}(\A \resLK)$.

  \eqref{item : cp pp exp : 2}: Consider $\A \in \K[\mathscr{L}_\mathcal{F}]$.\ As the the inclusion $\mathsf{Con}_\M(\A) \subseteq  \mathsf{Con}_\K(\A \resLK)$ holds by \eqref{item : cp pp exp : 1}, it suffices to prove the reverse inclusion. Consider $\theta \in \mathsf{Con}_{\mathsf{K}}(\A \resLK)$.\ Then $(\A{\upharpoonright}_{\mathscr{L}_\K}) / \theta \in \K$. 
  By \cref{Prop : pp expansions : subreducts : extendable}  there exists $\B \in \K[\mathscr{L}_\mathcal{F}]$ such that $(\A{\upharpoonright}_{\mathscr{L}_\K}) / \theta \leq \B{\upharpoonright}_{\mathscr{L}_\K}$.  
Therefore, the canonical surjection $h \colon \A{\upharpoonright}_{\mathscr{L}_\K} \to \A{\upharpoonright}_{\mathscr{L}_\K} / \theta$ can be viewed as a homomorphism $h \colon \A{\upharpoonright}_{\mathscr{L}_\K} \to \B{\upharpoonright}_{\mathscr{L}_\K}$.\ As $\A, \B \in \K[\mathscr{L}_\mathcal{F}]$, Proposition \ref{Prop : small homs are big homs} guarantees that $h$ is also a homomorphism from $\A$ to $\B$. Together with $\B \in \M$ and $\mathsf{Ker}(h) = \theta$, this yields that $\theta$ is an $\M$-congruence of $\A$.
\end{proof}

We are now ready to prove 
\cref{Thm : implications : simple - equational}.

\begin{proof}
First, suppose that $\K$ has the amalgamation property. We need to prove that $\M$ is equational. As $\M$ is a pp expansion of $\K$, it is of the form  $\SSS(\K[\L_\mathcal{F}])$.\  From the assumption that $\K$ has the amalgamation property and Theorem \ref{Thm : existential elimination}(\ref{item : existential elimination : 1}) it follows that every $f \in \mathcal{F} \subseteq \extpp(\K)$ is interpolated by some $f_* \in \imp_{\textsc{eq}}(\K)$. Consider the set of implicit operations
\[
\mathcal{F}_* = \{ f_* : f \in \mathcal{F} \}
\]
of $\K$. We will prove that $\mathcal{F}_* \subseteq \ext_{\textsc{eq}}(\K)$.
As $\mathcal{F}_* \subseteq \imp_{\textsc{eq}}(\K)$ by definition, it suffices to show that $\mathcal{F}_* \subseteq \ext(\K)$. Consider an $n$-ary $f_* \in \mathcal{F}_*$,  $\A \in \K$, and $a_1, \dots, a_n \in A$. Since $f \in \mathcal{F} \subseteq \ext(\K)$, there exists $\B \in \K$ with $\A \leq \B$ such that $\langle a_1, \dots, a_n \rangle \in \mathsf{dom}(f^\B)$. Consequently, $\langle a_1, \dots, a_n \rangle \in \mathsf{dom}(f_*^\B)$ because $f_*$ interpolates $f$. Hence, $f_*$ is extendable, as desired. Therefore, $\mathsf{M}_* = \SSS(\mathsf{K}[\mathscr{L}_{\mathcal{F}_*}])$ is an equational pp expansion of $\mathsf{K}$. Therefore, to 
 prove that $\M$ is an equational pp expansion of $\K$, 
it only remains to show that $\M = \mathsf{M}_*$, under the assumption that the function symbols $g_f$ and $g_{f_*}$ coincide for each $f \in \F$.

Since both $\M$ and $\M_*$ are closed under subalgebras, it will be enough to prove
\[
\mathsf{K}[\mathscr{L}_{\mathcal{F}_*}] \subseteq \SSS(\mathsf{K}[\mathscr{L}_{\mathcal{F}}])\, \, \text{ and }\, \, \mathsf{K}[\mathscr{L}_{\mathcal{F}}] \subseteq \mathsf{K}[\mathscr{L}_{\mathcal{F}_*}].
\]
To prove $\mathsf{K}[\mathscr{L}_{\mathcal{F}_*}] \subseteq \SSS(\mathsf{K}[\mathscr{L}_{\mathcal{F}}])$, consider $\A \in \K$ for which $\A[\L_{\mathcal{F}_*}]$ is defined. As $\F \subseteq \extpp(\K)$, by \cref{Thm : extendable 1} there exists $\B \in \K$ with $\A \leq \B$ for which $\B[\L_\F]$ is defined. Let $f \in \F$. Since $f_*$ interpolates $f$ by assumption and $f^\B$ is total because $\B[\L_\F]$ is defined, we obtain $f^\B = f_*^\B$. Hence, $\B[\L_{\F_*}]$ is also defined and coincides with $\B[\L_\F]$. Together with $\A[\L_{\F_*}] \leq \B[\L_{\F_*}]$ (which holds by $\A \leq \B$ and \cref{Prop : implicit operations extend}), this yields $\A[\L_{\F_*}] \leq \B[\L_\F] \in \K[\L_\F]$, whence $\A[\L_{\F_*}] \in \SSS(\mathsf{K}[\mathscr{L}_{\mathcal{F}}])$. Lastly, we prove the inclusion $\mathsf{K}[\mathscr{L}_{\mathcal{F}}] \subseteq \mathsf{K}[\mathscr{L}_{\mathcal{F}_*}]$. Consider $\A \in \K$ for which $\A[\L_{\mathcal{F}}]$ is defined. As $f_*$ interpolates $f$ for each $f \in \F$, we obtain that $\A[\L_{\mathcal{F}_*}]$ is also defined and coincides with $\A[\L_{\mathcal{F}}]$, whence $\A[\L_{\mathcal{F}}] \in \mathsf{K}[\mathscr{L}_{\mathcal{F}_*}]$. This concludes the proof that $\M$ is equational.

The implication ``if $\M$ is equational, then it is simple'' follows from \cref{Thm : axiomatization of pp expansion (almost always)}. Next, suppose that $\M$ is simple, i.e., that $\M=\K[\L_\F]$ for some $\F \subseteq \extpp(\K)$.
    Therefore, $\mathsf{Con}_\M(\A) = \mathsf{Con}_\K(\A \resLK)$ for every $\A \in \M$ by \cref{Prop : trivial inclusion f congruence preservation}(\ref{item : cp pp exp : 2}). Hence, $\M$ is congruence preserving.

    It only remains to show that if $\K$ has the relative congruence extension property, then $\M$ is congruence preserving. Suppose that $\K$ has the relative congruence extension property.  In view of \cref{Prop : trivial inclusion f congruence preservation}(\ref{item : cp pp exp : 1}), it suffices to show that $\mathsf{Con}_{\mathsf{K}}(\A \resLK) \subseteq \mathsf{Con}_\M(\A)$ for each $\A \in \M$. To this end, consider $\A \in \mathsf{M}$ and $\theta \in \mathsf{Con}_{\mathsf{K}}(\A \resLK)$. As $\M$ is a pp expansion of $\K$, it is of the form $\SSS(\K[\L_\mathcal{F}])$. Since $\A \in \M$, this implies the existence of some $\B \in \K[\L_\mathcal{F}]$ such that $\A \leq \B$. Clearly,  $\A \resLK \leq \B \resLK \in \K$. Therefore, from the assumption that $\mathsf{K}$ has the relative congruence extension property it follows that there exists $\phi \in \mathsf{Con}_{\mathsf{K}}(\B \resLK)$ such that $\theta = \phi{\upharpoonright}_A$.
    Since $\B \in \mathsf{K}[\mathscr{L}_{\mathcal{F}}]$, we can apply \cref{Prop : trivial inclusion f congruence preservation}(\ref{item : cp pp exp : 2}), obtaining  $\phi \in \mathsf{Con}_{\mathsf{K}}(\B \resLK) = \mathsf{Con}_{\mathsf{M}}(\B)$.
    Together with $\A \leq \B$, this yields $\phi {\upharpoonright}_A \in \mathsf{Con}_{\mathsf{M}}(\A)$. As $\theta = \phi{\upharpoonright}_A$, we conclude that $\theta \in \mathsf{Con}_{\mathsf{M}}(\A)$.
\end{proof}

Next we prove \cref{Thm: easy consequences of congruence preservation : variety and RFSI}.

\begin{proof}
We begin by proving that $\mathsf{M}_\textsc{rfsi} = \{ \A \in \M : \A{\upharpoonright}_{\mathscr{L}_\K} \in \K_{\textsc{rfsi}}\}$, namely, that for every $\A \in \M$,
\begin{equation}\label{Eq : Beth completions : FSI members stay the same : iff}
\A \in \M_\textsc{rfsi} \iff \A{\upharpoonright}_{\mathscr{L}_\K} \in \K_\textsc{rfsi}.
\end{equation}
To this end, consider $\A\in \M$ and observe that $\A{\upharpoonright}_{\mathscr{L}_\K} \in \K$ by \cref{Prop : pp expansions and subreducts}.\  From \cref{Prop : RFSI} it follows that
\begin{align*}
\A \in \M_\textsc{rfsi} &\iff \textup{id}_A \in \mathsf{Irr}_\M(\A);\\
\A{\upharpoonright}_{\mathscr{L}_\K} \in \K_\textsc{rfsi} &\iff \textup{id}_A \in \mathsf{Irr}_\K(\A{\upharpoonright}_{\mathscr{L}_\K}).
\end{align*}
Observe that $\mathsf{Con}_\M(\A) = \mathsf{Con}_\K(\A{\upharpoonright}_{\mathscr{L}_\K})$ because $\M$ is a congruence preserving pp expansion of $\K$ by assumption. Therefore, $\mathsf{Irr}_\M(\A) = \mathsf{Irr}_\K(\A{\upharpoonright}_{\mathscr{L}_\K})$. Together with the above display, this establishes  (\ref{Eq : Beth completions : FSI members stay the same : iff}), as desired. The proof that $\mathsf{M}_\textsc{rsi} = \{ \A \in \M : \A{\upharpoonright}_{\mathscr{L}_\K} \in \K_{\textsc{rsi}}\}$ is analogous and, therefore, omitted.

It only remains to prove the last part of the statement. Suppose that $\K$ is a variety.  From Theorem \ref{Thm : pp expansion : still quasivariety}\eqref{item : pp expansion : quasivariety} it follows that $\mathsf{M}$ is a quasivariety. Therefore, it remains to show that $\mathsf{M}$ is closed under homomorphic images. 
 By \cref{Prop : variety iff Con=ConK} it suffices to prove that $\mathsf{Con}(\A) = \mathsf{Con}_{\mathsf{M}}(\A)$ for every $\A \in \M$. 
To this end, observe that for every $\A \in \M$,
\[
\mathsf{Con}(\A) \subseteq \mathsf{Con}(\A \resLK) = \mathsf{Con}_{\mathsf{K}}(\A \resLK) = \mathsf{Con}_{\mathsf{M}}(\A),
\]
    where the first step is straightforward,
    the second follows from $\A \resLK \in \K$ (see \cref{Prop : pp expansions and subreducts}) and the assumption that $\K$ is a variety, and the third from $\A \in \M$ and the assumption that  $\mathsf{M}$ is a congruence preserving pp expansion of $\mathsf{K}$.
 Thus, $\mathsf{Con}(\A) = \mathsf{Con}_{\mathsf{M}}(\A)$, as desired.
\end{proof}

Our next goal is to prove Theorem \ref{Thm : Con preserving Beth comp : structure theorem : main}. To this end, it is convenient to establish the following technical observation first.

\begin{Proposition}\label{Prop : con preserving Beth comp: structural UA trick}
Let $\K$ be a relatively congruence distributive quasivariety for which $\K_{\textsc{rfsi}}$ is closed under nontrivial subalgebras. Moreover, consider $\A \in \K$ and  $\B \leq \A \times \A$ with projection maps $p_1, p_2 \colon \B \to \A$ such that for every $\langle a, b \rangle \in B$ we have $\langle a, a\rangle, \langle b, b \rangle \in B$. 
Assume that $\K$ has the strong epimorphism surjectivity property. Then $B = \mathsf{Cg}_\K^\A(B) \cap (p_1[B] \times p_2[B])$.
\end{Proposition}

\begin{proof}
As the inclusion $B \subseteq \mathsf{Cg}_\K^\A(B) \cap (p_1[B] \times p_2[B])$ always holds, we detail the proof of the reverse inclusion. Consider $\langle a, b \rangle \in \mathsf{Cg}_\K^\A(B) \cap (p_1[B] \times p_2[B])$. We need to prove that $\langle a, b \rangle \in B$. Suppose the contrary, with a view to contradiction. 

From Proposition \ref{Prop : congruences are subalgebras} it follows that $\mathsf{Cg}_\K^\A(B)$ is the universe of a member $\mathsf{Cg}_\K^\A(B)^*$ of $\K$ such that $\mathsf{Cg}_\K^\A(B)^* \leq \A \times \A$ is a subdirect product. Together with 
 $B \subseteq \mathsf{Cg}_\K^\A(B)$ 
and the assumption that $\B \leq \A \times \A$, this yields $\B \leq \mathsf{Cg}_\K^\A(B)^*$. As $\B \leq \mathsf{Cg}_\K^\A(B)^* \in \K$ and $\langle a, b \rangle \in \mathsf{Cg}_\K^\A(B) - B$,
we can apply the assumption that $\K$ has the strong epimorphism surjectivity property, obtaining a pair of homomorphisms $g, h \colon \mathsf{Cg}_\K^\A(B)^* \to \C$ with $\C \in \K$ such that 
\begin{equation}\label{Eq : con preserving Beth comps : K variety : 1}
g{\upharpoonright}_B = h{\upharpoonright}_B \, \, \text{ and } \, \, g(\langle a, b\rangle) \ne h(\langle a, b \rangle).
\end{equation}
In view of Remark \ref{Rem : simplification of SES},
we may further assume that $\C \in \K_{\textsc{rfsi}}$.

Since $\langle a, b \rangle \in p_1[B] \times p_2[B]$ by assumption, there exist $c, d \in A$ such that $\langle a, c \rangle, \langle d, b \rangle \in B$. By the assumptions this yields $\langle a, a\rangle, \langle b, b \rangle \in B$. Consequently, the left hand side of (\ref{Eq : con preserving Beth comps : K variety : 1}) guarantees that $g(\langle a, a\rangle) = h(\langle a, a \rangle)$ and $g(\langle b, b\rangle) = h(\langle b, b \rangle)$. By the right hand side of (\ref{Eq : con preserving Beth comps : K variety : 1}) this implies that
\begin{equation}\label{Eq : con preserving Beth comps : K variety : 33}
\text{either }g(\langle a, b \rangle) \ne g(\langle a, a\rangle) \text{ or } h(\langle a, b \rangle) \ne h(\langle a, a\rangle)
\end{equation}
and 
\begin{equation}\label{Eq : con preserving Beth comps : K variety : 33b}
\text{either }g(\langle a, b \rangle) \ne g(\langle b, b\rangle) \text{ or } h(\langle a, b \rangle) \ne h(\langle b, b\rangle).
\end{equation}

We rely on the next observation.

\begin{Claim}\label{Claim : con preserving Beth comps : K variety : 1}
For every $\phi \in \{ \mathsf{Ker}(g), \mathsf{Ker}(h) \}$ there exists $\eta \in \mathsf{Con}_\K(\A)$ such that \[
\phi \in \{ (\eta \times A^2){\upharpoonright}_{\mathsf{Cg}_\K^\A(B)}, (A^2 \times \eta){\upharpoonright}_{\mathsf{Cg}_\K^\A(B)}\}.
\]
\end{Claim}

\begin{proof}[Proof of the Claim]
By symmetry it suffices to show that the statement holds for the case where $\phi = \mathsf{Ker}(g)$. Since $g \colon \mathsf{Cg}_\K^\A(B)^* \to \C$ is a homomorphism with $\C \in \K$, we have $\mathsf{Ker}(g) \in \mathsf{Con}_\K(\mathsf{Cg}_\K^\A(B)^*)$. Moreover, recall that $\C \in \textsf{K}_\textsc{rfsi}$ and that $\K_\textsc{rfsi}$ is closed under nontrivial subalgebras by assumption. Therefore, from $\mathsf{Cg}_\K^\A(B)^* / \mathsf{Ker}(g) \in \III\SSS(\C)$ it follows that $\mathsf{Cg}_\K^\A(B)^* / \mathsf{Ker}(g)$ is either trivial or belongs to $\K_\textsc{rfsi}$. Lastly, recall that $\mathsf{Cg}_\K^\A(B)^* \leq \A \times \A$ is a subdirect product with $\A \in \K$. As $\K$ is a relatively congruence distributive quasivariety by assumption, we can apply \cref{Cor : CD vs FSI product congruences}, obtaining the desired conclusion.
\end{proof}

By symmetry and \cref{Claim : con preserving Beth comps : K variety : 1} we may assume that there exist $\eta_1, \eta_2 \in \mathsf{Con}_\K(\A)$ such that $\mathsf{Ker}(g) = (\eta_1 \times A^2){\upharpoonright}_{\mathsf{Cg}_\K^\A(B)}$ and $\mathsf{Ker}(h) \in \{ (\eta_2 \times A^2){\upharpoonright}_{\mathsf{Cg}_\K^\A(B)}, (A^2 \times \eta_2){\upharpoonright}_{\mathsf{Cg}_\K^\A(B)}\}$. We have two cases depending on whether $\mathsf{Ker}(h)$ is $(\eta_2 \times A^2){\upharpoonright}_{\mathsf{Cg}_\K^\A(B)}$ or $(A^2 \times \eta_2){\upharpoonright}_{\mathsf{Cg}_\K^\A(B)}$. 

First, suppose that $\mathsf{Ker}(h) = (\eta_2 \times A^2){\upharpoonright}_{\mathsf{Cg}_\K^\A(B)}$. 
Then
\begin{equation}\label{Eq : con preserving Beth comps : K variety : 12}
\mathsf{Ker}(g) = (\eta_1 \times A^2){\upharpoonright}_{\mathsf{Cg}_\K^\A(B)} \, \, \text{ and } \, \, \mathsf{Ker}(h) = (\eta_2 \times A^2){\upharpoonright}_{\mathsf{Cg}_\K^\A(B)}.
\end{equation}
Recall that $\langle a, b \rangle \in \mathsf{Cg}_\K^\A(B)$ and, therefore, $a, b \in A$. Together with $\eta_1, \eta_2 \in \mathsf{Con}(\A)$, this implies $\langle a, a \rangle \in \eta_1 \cap \eta_2$ and $\langle a, b \rangle \in A^2$.  
 Thus,
$\langle \langle a, a \rangle, \langle a, b \rangle \rangle \in (\eta_1 \times A^2) \cap (\eta_2 \times A^2)$. Since  $\langle a, a \rangle, \langle a, b \rangle \in \mathsf{Cg}_\K^\A(B),$ 
we obtain
\[
\langle \langle a, a \rangle, \langle a, b \rangle  \rangle \in ((\eta_1 \times A^2) \cap (\eta_2 \times A^2)){\upharpoonright}_{\mathsf{Cg}_\K^\A(B)} = (\eta_1 \times A^2)_{\mathsf{Cg}_\K^\A(B)} \cap (\eta_2 \times A^2)_{\mathsf{Cg}_\K^\A(B)}.
\]
In view of (\ref{Eq : con preserving Beth comps : K variety : 12}), this amounts to $\langle \langle a, a \rangle, \langle a, b \rangle \rangle \in \mathsf{Ker}(g) \cap \mathsf{Ker}(h)$, that is,
\[
g(\langle a, a \rangle) = g(\langle a, b \rangle) \, \, \text{ and } \, \, h(\langle a, a \rangle) = h(\langle a, b \rangle),
\]
a contradiction with (\ref{Eq : con preserving Beth comps : K variety : 33}).

It only remains to consider the case where $\mathsf{Ker}(h) = (A^2 \times \eta_2){\upharpoonright}_{\mathsf{Cg}_\K^\A(B)}$. In this case, we have 
\begin{equation}\label{Eq : con preserving Beth comps : K variety : 123}
\mathsf{Ker}(g) = (\eta_1 \times A^2){\upharpoonright}_{\mathsf{Cg}_\K^\A(B)} \, \, \text{ and } \, \, \mathsf{Ker}(h) = (A^2 \times \eta_2){\upharpoonright}_{\mathsf{Cg}_\K^\A(B)}.
\end{equation}
We rely on the following observation.

\begin{Claim}\label{Claim : con preserving Beth comps : K variety : theta into eta}
We have $\mathsf{Cg}_\K^\A(B) \subseteq \eta_1$.
\end{Claim}

\begin{proof}[Proof of the Claim]
As $\eta_1 \in \mathsf{Con}_\K(\A)$, it will be enough to show that $B \subseteq \eta_1$. To this end, consider $\langle p, q \rangle \in B$. By assumption we also have $\langle q, q\rangle \in B$. Since $\B \leq \A \times \A$ by assumption, we have $p, q \in A$. Together with 
$\eta_2 \in \mathsf{Con}(\A)$, 
this yields $\langle q, q \rangle \in \eta_2$ and, therefore, $\langle \langle p, q \rangle, \langle q, q \rangle \rangle \in (A^2 \times \eta_2)$. As $\langle p, q \rangle \in B \subseteq \mathsf{Cg}_\K^\A(B)$ and $\langle q, q \rangle \in \mathsf{Cg}_\K^\A(B)$ (the latter because $\mathsf{Cg}_\K^\A(B)$ is a congruence of $\A$), we obtain
\[
\langle \langle p, q \rangle, \langle q, q \rangle \rangle \in (A^2 \times \eta_2){\upharpoonright}_{\mathsf{Cg}_\K^\A(B)} = \mathsf{Ker}(h),
\]
where the last equality holds by the right hand side of (\ref{Eq : con preserving Beth comps : K variety : 123}). 
 The left hand side of (\ref{Eq : con preserving Beth comps : K variety : 1}) implies
$\mathsf{Ker}(g){\upharpoonright}_B = \mathsf{Ker}(h){\upharpoonright}_B$. Together with the above display and $\langle p, q \rangle, \langle q, q\rangle \in B$, this implies $\langle \langle p, q \rangle, \langle q, q \rangle \rangle \in \mathsf{Ker}(g)$. Since $\mathsf{Ker}(g) = (\eta_1 \times A^2){\upharpoonright}_{\mathsf{Cg}_\K^\A(B)}$ by the left hand side of (\ref{Eq : con preserving Beth comps : K variety : 123}), we conclude that $\langle p, q \rangle \in \eta_1$.
\end{proof}

Now, recall that $\langle a, b \rangle \in \mathsf{Cg}_\K^\A(B)$. By \cref{Claim : con preserving Beth comps : K variety : theta into eta} we obtain $\langle a, b \rangle \in \eta_1$. 
 Since $b \in A$, 
this implies $\langle \langle a, b \rangle, \langle b, b \rangle \rangle \in (\eta_1 \times A^2)$. 
On the other hand, from $a, b \in A$ and 
$\eta_2 \in  \mathsf{Con}(\A)$
it follows that $\langle \langle a, b \rangle, \langle b, b \rangle \rangle \in (A^2 \times \eta_2)$. As $\langle a, b \rangle, \langle b, b\rangle \in \mathsf{Cg}_\K^\A(B)$, we obtain
\[
\langle \langle a, b \rangle, \langle b, b \rangle \rangle \in (\eta_1 \times A^2){\upharpoonright}_{\mathsf{Cg}_\K^\A(B)} \, \, \text{ and } \, \, \langle \langle a, b \rangle, \langle b, b \rangle \rangle \in (A^2 \times \eta_2){\upharpoonright}_{\mathsf{Cg}_\K^\A(B)}.
\]
By (\ref{Eq : con preserving Beth comps : K variety : 123}) this amounts to $\langle \langle a, b \rangle, \langle b, b \rangle \rangle \in \mathsf{Ker}(g) \cap \mathsf{Ker}(h)$, that is,
\[
g(\langle a, b \rangle) = g(\langle b, b \rangle) \, \, \text{ and } \, \, h(\langle a, b \rangle) = h(\langle b, b \rangle),
\]
a contradiction with (\ref{Eq : con preserving Beth comps : K variety : 33b}).
\end{proof}

We will also make use of the next observation from \cite[Thm.\ 6.1]{CKMWES}.

\begin{Theorem}\label{Thm : weak ES : structural result}
Let $\K$ be a relatively congruence distributive quasivariety for which $\K_{\textsc{rfsi}}$ is closed under nontrivial subalgebras.\ If $\K$ has the weak epimorphism surjectivity property, then the variety $\mathbb{V}(\K)$ is arithmetical.
\end{Theorem}

As a last step before proving Theorem \ref{Thm : Con preserving Beth comp : structure theorem : main}, we establish the following result on the strong epimorphism surjectivity property.

\begin{Theorem}\label{Thm : con preserving Beth comps : structural result}
Let $\K$ be a relatively congruence distributive quasivariety for which $\K_{\textsc{rfsi}}$ is closed under nontrivial subalgebras.\ If $\K$ has the strong epimorphism surjectivity property, then $\K$ is an arithmetical variety with the congruence extension property.
\end{Theorem}

\begin{proof}
Suppose that $\K$ has the strong epimorphism surjectivity property. 

We begin by showing that $\K$ is a variety. As $\K$ is a quasivariety, by \cref{Prop : variety iff Con=ConK} it suffices to show that $\mathsf{Con}(\A) = \mathsf{Con}_\K(\A)$ 
for every $\A \in \K$. 
To this end, consider $\A \in \K$ and $\theta \in \mathsf{Con}(\A)$. We will show that $\theta = \mathsf{Cg}_\K^\A(\theta)$ which, in turns, implies $\theta \in \mathsf{Con}_\K(\A)$, as desired. In view of \cref{Prop : congruences are subalgebras}, the congruence $\theta$ is the universe of a subalgebra $\B \leq \A \times \A$. Furthermore, for every $\langle a, b \rangle \in B = \theta$ we have $\langle a, a\rangle, \langle b, b \rangle \in \theta = B$ because $\theta$ is a congruence of $\A$. Therefore, we can apply \cref{Prop : con preserving Beth comp: structural UA trick}, obtaining
\[
\theta = B = \mathsf{Cg}_\K^\A(B) \cap (p_1[B]\times p_2[B]) = \mathsf{Cg}_\K^\A(\theta) \cap (p_1[\theta]\times p_2[\theta]).
\]
Since $\theta$ is a reflexive relation on $\A$, we have $p_1[\theta]\times p_2[\theta] = A \times A$. Therefore, the above display yields $\theta = \mathsf{Cg}_\K^\A(\theta) \cap (A \times A) = \mathsf{Cg}_\K^\A(\theta)$. It follows that $\theta \in \mathsf{Con}_\K(\A)$.  Hence, we conclude that $\K$ is a variety.

Next we prove that the variety $\K$ has the congruence extension property.\ It will be enough to show that $\theta = \mathsf{Cg}^\A(\theta){\upharpoonright}_C$ for all $\C \leq \A \in \K$ and $\theta \in \mathsf{Con}(\C)$. Accordingly, consider $\C \leq \A \in \K$ and $\theta \in \mathsf{Con}(\C)$. In view of Proposition \ref{Prop : congruences are subalgebras}, the congruence $\theta$ is the universe of a subalgebra $\B \leq \C \times \C$. Furthermore, for every $\langle a, b \rangle \in B = \theta$ we have $\langle a, a\rangle, \langle b, b \rangle \in \theta = B$ because $\theta$ is a congruence of $\C$. As $\C \leq \A$, we have $\C \times \C \leq \A \times \A$, whence $\B \leq \A \times \A$. Therefore, we can apply Proposition \ref{Prop : con preserving Beth comp: structural UA trick}, obtaining
\[
\theta = B = \mathsf{Cg}_\K^\A(B) \cap (p_1[B]\times p_2[B]) = \mathsf{Cg}_\K^\A(\theta) \cap (p_1[\theta]\times p_2[\theta]).
\]
Since $\theta$ is a reflexive relation on $\C$, we have $p_1[\theta]\times p_2[\theta] = C \times C$. Therefore, the above display yields $\theta = \mathsf{Cg}_\K^\A(\theta) \cap (C \times C) = \mathsf{Cg}^\A(\theta){\upharpoonright}_C$. Thus, we conclude that $\K$ has the congruence extension property.

It only remains to prove that the variety $\K$ is arithmetical. As $\K$ has the strong epimorphism surjectivity property by assumption, it has the weak epimorphism surjectivity property as well. Moreover, $\K = \mathbb{V}(\K)$ because $\K$ is a variety. Therefore, from Theorem \ref{Thm : weak ES : structural result} it follows that $\K$ is arithmetical.
\end{proof}

We are now ready to prove Theorem \ref{Thm : Con preserving Beth comp : structure theorem : main}.

\begin{proof}
Let $\M$ be a  congruence preserving Beth companion 
of $\K$. Since $\M$ is a pp expansion of $\K$ and $\K$ is a quasivariety by assumption, we obtain that $\M$ is also a quasivariety (see \cref{Thm : pp expansion : still quasivariety}(\ref{item : pp expansion : quasivariety})). In addition, as $\M$ is a Beth companion of $\K$ and $\K$ is a quasivariety, $\M$ has the strong epimorphism surjectivity property by \cref{Thm : Beth companions vs strong Beth}. From the assumption that $\K$ is relatively congruence distributive and \cref{Thm: easy consequences of congruence preservation : congruence equations} it follows that $\M$ is relatively congruence distributive as well. Lastly, we will prove that $\M_{\textsc{rfsi}}$ is closed under nontrivial subalgebras. Consider $\A \leq \B \in \M_{\textsc{rfsi}}$ with $\A$ nontrivial. By \cref{Thm: easy consequences of congruence preservation : variety and RFSI} we have $\B{\upharpoonright}_{\mathscr{L}_\K} \in \K_{\textsc{rfsi}}$. Furthermore, $\A \leq \B$ implies $\A{\upharpoonright}_{\mathscr{L}_\K} \leq \B{\upharpoonright}_{\mathscr{L}_\K}$. As $\A$ is nontrivial, so is $\A{\upharpoonright}_{\mathscr{L}_\K}$. Therefore, $\A{\upharpoonright}_{\mathscr{L}_\K} \leq \B{\upharpoonright}_{\mathscr{L}_\K} \in \K_{\textsc{rfsi}}$ and the assumption that $\K_{\textsc{rfsi}}$ is closed under nontrivial subalgebras guarantee that $\A{\upharpoonright}_{\mathscr{L}_\K} \in \K_{\textsc{rfsi}}$. Together with $\A \in \M$, this allows us to apply \cref{Thm: easy consequences of congruence preservation : variety and RFSI}, obtaining $\A \in \M_{\textsc{rfsi}}$. Hence, we conclude that $\M_{\textsc{rfsi}}$ is closed under nontrivial subalgebras.

Therefore, $\M$ is a quasivariety with the strong epimorphism surjectivity property that, moreover, is relatively congruence distributive  and such that $\M_{\textsc{rfsi}}$ is closed under nontrivial subalgebras. Thus, from \cref{Thm : con preserving Beth comps : structural result} it follows that $\M$ is an arithmetical variety with the congruence extension property. Since $\M$ is a variety, we obtain $\M_{\textsc{fsi}} = \M_{\textsc{rfsi}}$. As $\M_{\textsc{rfsi}}$ is closed under nontrivial subalgebras, we conclude that so is $\M_{\textsc{fsi}}$.
\end{proof}

Our last goal is to prove Corollary \ref{Cor : rel filtral -> discriminator}. To this end, we rely on the following consequence of Theorem \ref{Thm : con preserving Beth comps : structural result} which, in  the context of varieties, originates in \cite[Cor.\ 3(i)]{Ber87} (see also \cite[Example 6.5]{CKMWES}).

\begin{Corollary}\label{Cor : filtral vs discriminator : purely algebraic version}
Every relatively filtral quasivariety with the strong epimorphism surjectivity property is a discriminator variety.
\end{Corollary}

\begin{proof}
Let $\K$ be a relatively filtral quasivariety with the strong epimorphism surjectivity property. As $\K$ is relatively filtral, it is relatively congruence distributive and $\K_\textsc{rfsi}$ is closed under nontrivial subalgebras  (see, e.g., \cite[Cor. 6.5(i, iv)]{CampRaf}). 
Therefore, from \cref{Thm : con preserving Beth comps : structural result} it follows that $\K$ is an arithmetical variety. Since $\K$ is a  relatively  filtral quasivariety, this yields that $\K$ is a congruence permutable filtral variety. As congruence permutable filtral varieties coincide with discriminator varieties (see, e.g., \cite{BlokKohlerPigozzi84,FriedKiss83}), we conclude that $\K$ is a discriminator variety.
\end{proof}

Furthermore, we recall that a join semilattice $\langle A; \lor \rangle$ is said to be \emph{dually Brouwerian} when for all $a, b \in A$ there exists the smallest $c \in A$ such that $a \leq b \lor c$. Moreover, an element $a$ of a complete lattice $\A$ is \emph{compact} when for every $X \subseteq A$ such that $a \leq \bigvee X$ there exists a finite $Y \subseteq X$ such that $a \leq \bigvee Y$. With every quasivariety $\K$ and algebra $\A$ we associate a join semilattice $\mathsf{Comp}_\K(\A)$ whose universe is the set of compact elements of $\mathsf{Con}_\K(\A)$ and whose join operation $+$ is defined by the rule $\theta + \phi = \mathsf{Cg}^\A_\K(\theta \cup \phi)$. Lastly, a member $\A$ of a quasivariety $\K$ is said to be \emph{simple relative to} $\K$ when $\mathsf{Con}_\K(\A)$ has exactly two elements, and we say that $\K$ is \emph{relatively semisimple} when every member of $\K_\textsc{rsi}$ is  simple relative to $\K$.  When $\K$ is a variety, we simply say that $\A$ is \emph{simple} and $\K$ \emph{semisimple}. We will rely on the fact that a quasivariety $\K$ is relatively filtral if and only if it is relatively semisimple and $\mathsf{Comp}_\K(\A)$ is dually Brouwerian for every $\A \in \K$ (see \cite[Thm. 6.3]{CampRaf} and \cite[Thms.\ 5 \& 8]{KohlerPigozzi80}).\footnote{This description of filtrality originated in the context of varieties  (see \cite{FriedGraetzerQuackenbusgìh80,FriedKiss83}) and was later extended to quasivarieties in \cite{CampRaf}.} Bearing this in mind, we will now prove \cref{Cor : rel filtral -> discriminator}.

\begin{proof}
Let $\M$ be a 
Beth companion of a relatively filtral quasivariety $\K$. Since every relatively filtral quasivariety has the relative congruence extension property (see, e.g., \cite[Cor.\ 6.5(i)]{CampRaf}), so does $\K$. As $\M$ is a pp expansion of $\K$, we conclude that $\M$ is congruence preserving by 
\cref{Thm : implications : simple - equational}.

We will show that $\M$ is also a relatively filtral quasivariety. The fact that $\M$ is a quasivariety is a consequence of Theorem \ref{Thm : pp expansion : still quasivariety}(\ref{item : pp expansion : quasivariety}) and the assumption that $\M$ is a pp expansion of the quasivariety $\K$. Therefore, it only remains to prove that $\M$ is relatively filtral, i.e., that it is relatively semisimple and $\mathsf{Comp}_\M(\A)$ is dually Brouwerian for every $\A \in \M$. To show that $\M$ is relatively semisimple, consider $\A \in \M_\textsc{rsi}$. Since $\M$ is a congruence preserving pp expansion of $\K$, we can apply \cref{Thm: easy consequences of congruence preservation : variety and RFSI}, obtaining  $\A\resLK \in \K_\textsc{rsi}$. 
As the quasivariety $\K$ is relatively semisimple (because it is relatively filtral by assumption), we obtain that $\mathsf{Con}_\K(\A\resLK)$ has exactly two elements. Moreover, $\mathsf{Con}_\M(\A) = \mathsf{Con}_\K(\A\resLK)$ because $\A \in \M$ and $\M$ is a congruence preserving pp expansion of $\K$. Therefore, $\mathsf{Con}_\M(\A)$ has exactly two elements, whence $\A$ is simple relative to $\M$. Hence, we conclude that $\M$ is relatively semisimple, as desired. Next we prove that $\mathsf{Comp}_\M(\A)$ is dually Brouwerian for every $\A \in \M$. Consider $\A \in \M$ and observe that $\A\resLK \in \K$ by \cref{Prop : pp expansions and subreducts}. Since $\M$ is congruence preserving, we have $\mathsf{Con}_\M(\A) = \mathsf{Con}_\K(\A\resLK)$, whence $\mathsf{Comp}_\M(\A) = \mathsf{Comp}_\K(\A\resLK)$. As $\mathsf{Comp}_\K(\A\resLK)$ is dually Brouwerian (because $\K$ is relatively filtral and $\A\resLK \in \K$), we conclude that so is $\mathsf{Comp}_\M(\A)$. Thus, we conclude that $\M$ is a relatively filtral quasivariety.

Lastly, observe that $\M$ has the strong epimorphism surjectivity property because it is 
 a 
Beth companion of the quasivariety $\K$ (see Theorem \ref{Thm : Beth companions vs strong Beth}).\ Thus, $\M$ is a relatively filtral quasivariety with the strong epimorphism surjectivity property. By \cref{Cor : filtral vs discriminator : purely algebraic version} we conclude that $\M$ is a discriminator variety.
\end{proof}

\section{Absolutely closed and primal algebras}

The aim of this section is to provide additional criteria to establish whether a pp expansion of a quasivariety is a Beth companion. The two main results are Theorems \ref{Thm : Beth compl. iff eq absolutely closed} and \ref{Prop : Beth comp triangle trick} involving absolutely closed algebras and primal algebras, respectively. We will employ these results to determine the Beth companions of torsion-free Abelian groups (\cref{Exa : tfAB}), Abelian $\ell$-groups (\cref{Exa : lAB}), MV-algebras (\cref{Exa : MV}), and varieties of MV-algebras generated by a finite chain (\cref{Exa : MV gen by fin chains}).

We begin by recalling the definition of an absolutely closed algebra (see \cite[p.~236]{Isb65}).

\begin{Definition}
Let $\mathsf{K}$ be a class of algebras. A member $\A$ of $\K$ is called \emph{absolutely closed} in $\mathsf{K}$ when 
$\mathsf{d}_{\mathsf{K}}(\A, \B) = A$
for every $\B \in \mathsf{K}$ such that $\A \leq \B$.
The class of absolutely closed members of $\mathsf{K}$ will be denoted by $\KAC$.
\end{Definition}

On the one hand, absolutely closed algebras are related to the reducts of the members of Beth companions, as the following result states.

\begin{Theorem} \label{Thm : acl and Beth companion}
Let $\mathsf{K}$ be a quasivariety with a Beth companion $\mathsf{M} = \SSS(\mathsf{K}[\mathscr{L}_{\mathcal{F}}])$. 
Then 
\[
\mathsf{K}[\mathscr{L}_{\mathcal{F}}] \resLK \subseteq \KAC \subseteq 
\mathsf{M} \resLK.
\]
Moreover, if $\M$ is an equational Beth companion of $\K$, then
$\mathsf{M} \resLK = \KAC$.
\end{Theorem}

On the other hand, the next result shows that absolutely closed algebras provide a sufficient condition for a pp expansion of a quasivariety to be a Beth companion.

\begin{Theorem} \label{Thm : Beth compl. iff eq absolutely closed}
Let $\mathsf{M}$ be a pp expansion of a quasivariety $\mathsf{K}$ such that $\mathsf{M} \resLK \subseteq \KAC$. Then $\M$ is a Beth companion of $\mathsf{K}$.
\end{Theorem}

We postpone the proofs of Theorems \ref{Thm : acl and Beth companion} and \ref{Thm : Beth compl. iff eq absolutely closed} and proceed to describe some of their applications. To this end, it is convenient to introduce injective algebras and absolute retracts (see, e.g., \cite[pp.~80--81]{SurvKissal}) and relate them to absolutely closed algebras.

\begin{Definition}
A member $\A$ of a quasivariety $\K$ is said to be:
\benroman
\item \emph{injective in $\K$} when for all $\B, \C \in \mathsf{K}$ such that $\B \leq \C$ and homomorphism $f \colon \B \to \A$ there exists a homomorphism $g \colon \C \to \A$ such that $g\res_{B}=f$;
\[
\begin{tikzcd}[scale=1.6]
\B \arrow[d, "f"'] \arrow[r, hookrightarrow] & \C \arrow[dl, dashed, "g"] \\
\A &
\end{tikzcd}
\]
\item an \emph{absolute retract in $\K$} when for every $\B \in \mathsf{K}$ such that $\A \leq \B$ there exists a homomorphism $g \colon \B \to \A$ such that $g\res_{A}$ is the identity map on $A$. 
\[
\begin{tikzcd}[scale=1.6]
\A \arrow[d, "id"'] \arrow[r, hookrightarrow] & \B \arrow[dl, dashed, "g"] \\
\A &
\end{tikzcd}
\]
\eroman
\end{Definition}

\begin{Proposition}\label{Prop : injective implies abs closed}
The following conditions hold for a class of algebras $\K$ and $\A \in \K$:
\benroman
\item\label{Prop : injective implies abs closed: 1} if $\A$ is injective in $\K$, then it is an absolute retract in $\K$;
\item\label{Prop : injective implies abs closed: 2} if $\A$ is an absolute retract in $\K$, then it is absolutely closed in $\K$.
\eroman
\end{Proposition}

\begin{proof}
 \eqref{Prop : injective implies abs closed: 1}: See, e.g., \cite[Prop.~1.1]{SurvKissal}.

\eqref{Prop : injective implies abs closed: 2}: Suppose that $\A$ is an absolute retract in $\K$ and consider $\B \in \K$ such that $\A \leq \B$. Since $\A$ is an absolute retract in $\K$, there exists a homomorphism $g \colon \B \to \A$ such that $g\res_{A}$ is the identity on $A$.  Let $h = i \circ g$, where $i$ is the inclusion map of $\A$ into $\B$. Then $h \colon \B \to \B$ is a homomorphism such that  $h(a)=g(a)=a$ for every $a \in A$. Consider the identity map $id \colon \B \to \B$. As $h(a)=a=id(a)$ for every $a \in A$, we have $h\res_A = id\res_A$. Then $h(b)=id(b)=b$ for every $b \in \d_\K(\A,\B)$.  Since the image of $h$ is $A$, it follows that $\d_\K(\A,\B) = A$. Therefore, $\A$ is absolutely closed in $\K$.
\end{proof}

We are now ready to illustrate how  \cref{Thm : Beth compl. iff eq absolutely closed} can be applied to describe the Beth companions of concrete classes of algebras.

\begin{exa}[\textsf{Torsion-free Abelian groups}]\label{Exa : tfAB}
An Abelian group\footnote{In this and the next example, we temporarily switch to the additive notation for Abelian groups, as it will be more convenient.} $\A = \langle A; +, -, 0 \rangle$ is said to be \emph{torsion-free} when $0$ is its only element of finite order. 
Torsion-free Abelian groups form a quasivariety $\tfAb$ axiomatized relative to Abelian groups by the quasiequations $nx \thickapprox  0 \to x \thickapprox  0$ for every $n \in \mathbb{Z}^+$. Our aim is to describe the Beth companion of $\tfAb$.
    
To this end, let $\A \in \tfAb$, $a \in A$, and $n \in \Z^+$. We say that an element $b \in A$ is the result of \emph{dividing $a$ by $n$} if $nb=a$. An Abelian group $\A$ is called \emph{divisible} when for all $a \in A$ and $n \in \Z^+$ there exists some $b \in A$ such that $nb=a$. In view of the next result, ``dividing by $n$'' is an extendable implicit operation of $\tfAb$.
\begin{Proposition} \label{Prop : div is ext imp op in T}
For each $n \in \mathbb{Z}^+$ there exists a unary $f_n \in \exteq(\tfAb)$ such that for all $\A \in \tfAb$ and $a \in \mathsf{dom}(f_n^\A)$,
\begin{align*}
    \mathsf{dom}(f_n^\A) &= \{ c \in A : nb = c \text{ for some }b \in A \};\\
f_n^\A(a) &= \text{the result of dividing $a$ by $n$.}
\end{align*}
\end{Proposition}

\begin{proof} 
Let $\varphi_n(x, y) = x \thickapprox ny$. Moreover, let $\Z$ and $\mathbb{Q}$ be the additive groups of the integers and the rationals, respectively. The fundamental theorem of finitely generated Abelian groups (see, e.g., \cite[Thm.~5.2.3]{DF04}) implies that every finitely generated torsion-free Abelian group is isomorphic to $\Z^m$ for some $m \in \mathbb{N}$.  \cref{Prop : quasivariety = Q fin gen SI} implies that $\tfAb$ is generated as a quasivariety by its finitely generated members. Since $\Z \leq \mathbb{Q}$, we obtain that $\tfAb$ is also generated as a quasivariety by $\mathbb{Q}$. Observe that for all $\frac{a}{b},\frac{c}{d} \in \mathbb{Q}$ and $n \in \Z^+$ we have that $\mathbb{Q} \vDash \varphi_n(\frac{a}{b},\frac{c}{d})$ implies $\frac{c}{d}=\frac{a}{nb}$. Therefore, each $\varphi_n$ is functional in $\mathbb{Q}$.\ As $\tfAb$ is the quasivariety generated by $\mathbb{Q}$, \cref{Cor : functionality in Q(K)} yields that $\varphi_n$ defines a member $f_n$ of $\impeq(\tfAb)$.\ From the definition of $\varphi_n$ it follows that the two displays in the statement hold. Moreover, $f_n^{\mathbb{Q}}$ is total because every rational can be divided by $n$ in $\mathbb{Q}$. As $\tfAb$ is generated as a quasivariety by $\mathbb{Q}$, \cref{Prop : extendable : sufficient conditions}\eqref{item : extendable : sufficient conditions : 2} guarantees that $f_n$ is extendable.
\end{proof}

\begin{Corollary}\label{Cor : TAFG lacks the SES}
 $\tfAb$ lacks the strong epimorphism surjectivity property.
\end{Corollary}

\begin{proof}
Let $f_2$ be the member of $\exteq(\tfAb)$ given by   \cref{Prop : div is ext imp op in T}. Since $\Z \leq \mathbb{Q} \in \tfAb$ and $f_2^\mathbb{Q}(1)=\frac{1}{2} \notin \Z$, from  \cref{Thm : dominions : pp formulas} it follows that $\frac{1}{2} \in \d_\tfAb(\Z, \mathbb{Q})-\Z$. Therefore,  $\tfAb$ lacks the strong epimorphism surjectivity property by \cref{Prop : SES and dominions}.
\end{proof}

Let $\F=\{f_n : n \in \Z^+\}$ be the set of implicit operations given by \cref{Prop : div is ext imp op in T}.\ By the same proposition we have $\F \subseteq \exteq(\tfAb)$.
Denote by $\L$ the language of groups and let $\L_\F = \L \cup \{ d_n : n \in \Z^+\}$ be an $\F$-expansion of $\L$ in which the role of $g_{f_n}$ is played by $d_n$. 
 Then $\DAb = \SSS(\tfAb[\mathscr{L}_{\mathcal{F}}])$ is an equational pp expansion of $\tfAb$. We will prove that it is an equational Beth companion of $\tfAb$. To this end, we rely on the following observation.

\begin{Proposition}\label{Prop : divisible torsion free injective}
$\DAb\res_\L$ is the class of divisible torsion-free Abelian groups.\ Moreover, every member of $\DAb\res_\L$ is injective in $\tfAb$.
\end{Proposition}

\begin{proof}
Let $\A \in \tfAb$. As $\F \subseteq \exteq(\tfAb)$, from \cref{Thm : axiomatization of pp expansion (almost always)} it follows that $\DAb =\tfAb[\mathscr{L}_{\mathcal{F}}]$.\ Then $\A \in \DAb\res_{\L}$ if and only if $f_n^\A$ is total for every $n \in \Z^+$. Therefore, $\A \in \DAb\res_{\L}$ if and only if $\A$ is divisible. Since $\DAb\res_{\L} \subseteq \tfAb$, we conclude that the members of $\DAb\res_{\L}$ are exactly the divisible torsion-free Abelian groups.\
It only remains to prove that every member of $\DAb\res_{\L}$ is injective in $\tfAb$. 
It is well known that the injective members of the variety $\mathsf{AG}$ of Abelian groups are exactly the divisible Abelian groups (see, e.g., \cite[Cor.~2.3.2]{Wei94}). Since $\tfAb \subseteq \mathsf{AG}$, from the definition of an injective algebra it follows that every divisible member of $\mathsf{AG}$ that is torsion-free is also injective in $\tfAb$. As all the members of $\DAb\res_\L$ are divisible, we conclude they are injective in $\tfAb$.
\end{proof}

We are now ready to establish the desired description of the Beth companion of $\tfAb$.

\begin{Theorem}\label{Thm : Beth companion TF Abelian groups}
$\DAb$ is a variety and an equational Beth companion of $\tfAb$.
\end{Theorem}

\begin{proof}
Let $\Sigma$ be a set of equations axiomatizing the variety of Abelian groups.\ Since $\varphi_n = x \thickapprox ny$ is the equation defining $f_n$ (see the proof of \cref{Prop : div is ext imp op in T}), from \cref{Thm : axiomatization of pp expansion (almost always)} it follows that $\DAb$ is axiomatized by the set of quasiequations
\[
\Gamma = \Sigma \cup \{nx \thickapprox  0 \to x \thickapprox  0 : n \in \Z^+\} \cup \{ x \thickapprox nd_n(x) : n \in \Z^+\}.
\]

We will show that $\DAb$ is also axiomatized by the set of equations
\[
\Gamma' = \Sigma \cup \{x \thickapprox d_n(nx)  : n \in \Z^+\} 
\cup 
 \{ x \thickapprox nd_n(x) : n \in \Z^+\}.
\]
It will be enough to prove that $\Gamma$ and $\Gamma'$ have the same models. First, let $\A$ be a model of $\Gamma$.\ Then $\A\res_{\L}$ is a torsion-free Abelian group. Consider $a \in A$ and $n \in \Z^+$. Since  $\A \vDash x \thickapprox nd_n(x)$, we have $na = nd_n^\A(na)$. Then $0 = n(d_n^\A(na)-a)$, which implies $a =  d_n^\A(na)$ because $\A\res_{\L}$ is torsion-free. So, $\A \vDash \Gamma'$. Conversely, suppose that $\A$ is a model of $\Gamma'$. Let $a \in A$ and $n \in \Z^+$. If $na = 0$, then 
\[
a=d_n^\A(na)=d_n^\A(0)=d_n^\A(n0)=0,
\]
where the first and last equalities hold because $\A \vDash x \thickapprox d_n(nx)$, the second because $na=0$ by assumption, and the third because $0=n0$. Therefore, $\A \vDash nx \thickapprox  0 \to x \thickapprox  0$. Hence, $\A$ is a model of $\Gamma$. We conclude that the set of equations $\Gamma'$ axiomatizes $\DAb$, which is then a variety.

Lastly, by \cref{Prop : divisible torsion free injective} every member of $\DAb\res_{\L}$ is injective in $\tfAb$. Then \cref{Prop : injective implies abs closed} implies $\DAb\res_{\L} \subseteq \tfAb_\ac$.
Since $\DAb$ is a pp expansion of $\tfAb$, from \cref{Thm : Beth compl. iff eq absolutely closed} it follows that $\DAb$ is a Beth companion of $\tfAb$ which, moreover, is equational because $\DAb$ is an equational pp expansion of $\tfAb$.
\end{proof}

\end{exa}

For the next pair of examples, it is convenient to introduce the following class of algebras (see \cite[p.~236]{Isb65}). 
\begin{Definition}
A member $\A$ of a class of algebras $\K$ is called \emph{saturated} in $\K$ when there exists no $\B \in \K$ such that $\A$ is a proper $\K$-epic subalgebra of $\B$. 
\end{Definition}

In quasivarieties, saturated algebras are also called \emph{epicomplete} (see, e.g., \cite[p.~176]{BH89}). Saturated and absolutely closed algebras are related as follows.

\begin{Proposition} \label{Prop : AP implies acl = sat}
The following conditions hold for a class of algebras $\K$:
\benroman
\item\label{Prop : AP implies acl = sat: 1} every algebra that is absolutely closed in $\K$  is saturated in $\K$;
\item\label{Prop : AP implies acl = sat: 2} if $\K$ is a quasivariety with the amalgamation property, then every algebra that is saturated  in $\K$ is absolutely closed in $\K$.
\eroman
\end{Proposition}

\begin{proof}
\eqref{Prop : AP implies acl = sat: 1}: Suppose that $\A$ is absolutely closed in $\K$. Consider $\B \in \K$ such that $\A$ is a proper subalgebra of $\B$. Assume, with a view to contradiction, that $\A$ is a $\K$-epic subalgebra of $\B$. Then $f=g$ for every $\C \in \K$ and pair of homomorphisms $f,g \colon \B \to \C$ such that $f\res_A=g\res_A$. Therefore, $\d_\K(\A,\B)=B$. Since $\A$ is absolutely closed in $\K$, we have that $\d_\K(\A,\B) = A$. We conclude that $\A=\B$, which contradicts the assumption that $\A$ is a proper subalgebra of $\B$.

\eqref{Prop : AP implies acl = sat: 2}: Assume that $\K$ has the amalgamation property and that $\A$ is saturated in $\K$. Let $\B \in \K$ be such that $\A \leq \B$. To prove that $\A$ is absolutely closed in $\K$, we need to show that $\d_\K(\A,\B)=A$. Consider the subalgebra $\D$ of $\B$ with universe $\d_\K(\A,\B)$. Then $\A \leq \D \leq \B$. Since $\K$ has the amalgamation property and is a quasivariety by assumption, \cref{Prop : doms computable in subalgebras} implies 
 $\mathsf{d}_{\mathsf{K}}(\A, \D)=\mathsf{d}_{\mathsf{K}}(\A, \B) \cap D$. 
As $\mathsf{d}_{\mathsf{K}}(\A, \B) = D$,
we obtain  
 $\mathsf{d}_{\mathsf{K}}(\A, \D)=D$. 
 Furthermore, $\D \in \K$ because $\D \leq \B \in \K$ and $\K$ is a quasivariety. It follows that $\A$ is a $\K$-epic subalgebra of $\D$. Then  the assumption that $\A$ is saturated in $\K$ and $\D \in \K$ let us conclude that $\A=\D$. Therefore, $\d_\K(\A,\B)=A$, as desired.
\end{proof}

\begin{exa} [\textsf{$\ell$-groups}] \label{Exa : lAB}
An Abelian $\ell$-group is an algebra $\langle A; +,-,\land, \lor,0 \rangle$ such that 
$\langle  A; +,-, 0 \rangle$ is an Abelian group, 
$\langle A; \land, \lor \rangle$ is a lattice, and
\[
a \leq b \text{ implies } a + c \leq b + c
\]
for all $a,b,c \in A$, where $\leq$ denotes the partial order on $A$ induced by its lattice structure (see, e.g., \cite{MR1369091}).
The class $\lAb$ of Abelian $\ell$-groups forms a variety (see, e.g., \cite[Cor.~1 of Thm.~XIII.2]{BirkLT}). Our aim is to describe the Beth companion of $\lAb$.

To this end, given $\A \in \lAb$, $a,b \in A$, and $n \in \Z^+$, we say that $b$ is the result of \emph{dividing $a$ by $n$} if $nb=a$. We say that an Abelian $\ell$-group is \emph{divisible} when so is its group reduct.

\begin{Proposition} \label{Prop : div is ext imp op in T lgroup}
For each $n \in \mathbb{Z}^+$ there exists a unary $f_n \in \exteq(\lAb)$ such that for all $\A \in \lAb$ and $a \in \mathsf{dom}(f_n^\A)$,
\begin{align*}
    \mathsf{dom}(f_n^\A) &= \{ c \in A : nb = c \text{ for some }b \in A \};\\
f_n^\A(a) &= \text{the result of dividing $a$ by $n$}.
\end{align*}
\end{Proposition}

\begin{proof}
Let $\varphi_n(x, y) = x \thickapprox ny$. The proof of 
\cref{Prop : div is ext imp op in T} shows that $\varphi_n$ defines a unary implicit operation of $\tfAb$. As the class of group reducts of $\lAb$ is $\tfAb$ (see \cite[Cor.\ of Thm.~XIII.11]{BirkLT}), the equation $\varphi_n$ defines also a unary $f_n \in \impeq(\lAb)$. Clearly, the two displays in the statement hold for $f_n$. Therefore, it only remains to show that $f_n$ is extendable. Recall from \cite[Lem.~1, p.~317]{MR6550} that the members of $\lAb_\textsc{rfsi}$ are nontrivial and linearly ordered. Moreover, the class of nontrivial linearly ordered members of $\lAb_\textsc{rfsi}$ is $\UUU(\mathbb{Q})$, where $\mathbb{Q}$ denotes the additive group of the rationals equipped with the lattice structure induced by the standard order of $\mathbb{Q}$ by \cite{MR155753}.
Therefore, the Subdirect Decomposition Theorem \ref{Thm : Subdirect Decomposition} yields that $\lAb$ is generated as a quasivariety by $\mathbb{Q}$. 
Arguing as in the proof of \cref{Prop : div is ext imp op in T}, we conclude that each $f_n$ is extendable.
\end{proof}

The next observation can be traced back at least to \cite{MR574094}.

\begin{Corollary}
$\lAb$ lacks the strong epimorphism surjectivity property.
\end{Corollary}

\begin{proof}
Analogous to the proof of \cref{Cor : TAFG lacks the SES} with the sole difference that $\mathbb{Q}$ is the algebra employed in the proof of \cref{Prop : div is ext imp op in T lgroup} and $\mathbb{Z}$ the subalgebra of $\mathbb{Q}$ whose universe is the set of integers.
\end{proof}

Set $\F=\{f_n : n \in \Z^+\}$ be the set of  implicit operations given by \cref{Prop : div is ext imp op in T lgroup}.\ 
By the same proposition we have  $\F \subseteq \extpp(\tfAb)$.
Let $\L_\F$ be an $\F$-expansion of $\L_{\lAb}$. Then $\lDAb = \SSS(\lAb[\mathscr{L}_{\mathcal{F}}])$ is an equational pp expansion of $\lAb$.

\begin{Theorem} \label{Thm : div are Beth of tfAb and lAb}
$\lDAb$ is a variety and an equational Beth companion of $\lAb$.
\end{Theorem}

\begin{proof}
Since $\F \subseteq \exteq(\lAb)$ and $\lAb$ is a variety, \cref{Thm : pp expansion : still quasivariety}\eqref{item : pp expansion : variety} implies that $\lDAb$ is also a variety.
Moreover, from \cref{Thm : axiomatization of pp expansion (almost always)} it follows that $\lDAb =\lAb[\mathscr{L}_{\mathcal{F}}]$.
Therefore, $\lDAb\res_{\L_\lAb}$ is the class of divisible Abelian $\ell$-groups.
Then \cite[Thm.~2.1]{EpiLG}
yields that $\lDAb\res_{\L_\lAb}$ is the class of members of $\lAb$ that are saturated in $\lAb$. 
Since $\ell\mathsf{Ab}$ has the amalgamation property (see \cite[Thm.\ 2.3]{APLG}), we can apply \cref{Prop : AP implies acl = sat}\eqref{Prop : AP implies acl = sat: 2}, obtaining $\lDAb\res_{\L_\lAb} \subseteq \lAb_\ac$. Thus, \cref{Thm : Beth compl. iff eq absolutely closed} implies that $\lDAb$ is a Beth companion of $\lAb$ which, moreover, is equational because $\lDAb$ is an equational pp expansion of $\lAb$.
\end{proof}
\end{exa}

\begin{exa}[\textsf{MV-algebras}]\label{Exa : MV}
An \emph{MV-algebra} is an algebra $\A = \langle A; \oplus, \neg, 0 \rangle$ comprising a commutative monoid  $\langle A; \oplus, 0 \rangle$  and satisfying the equations
\[
\neg \neg x \thickapprox x, \qquad x \oplus \neg 0 \thickapprox \neg 0, \qquad \neg(\neg x \oplus y) \oplus y \thickapprox \neg(\neg y \oplus x) \oplus x.
\]    
From a logical standpoint, the interest of MV-algebras derives from the fact that they algebraize the infinite-valued \Luk ukasiewicz logic (see, e.g., \cite{COM00}).

The variety $\mathsf{MV}$ of MV-algebras is generated by the algebra $\I = \langle [0, 1], \oplus, \neg,0\rangle$ with universe the real unit interval $[0, 1] =\{ a \in \mathbb{R} : 0 \leq a \leq 1\}$ and equipped with the operations
\[
a \oplus b = \min\{a+b,1\} \quad
\text{ and }  \quad 
\neg a = 1-a
\]
for all $a,b \in [0, 1]$, where $+$ and $-$ denote the standard addition and subtraction in $\mathbb{R}$ (see, e.g., \cite[Prop.~8.1.1]{COM00}). Our aim is to describe the Beth companion of $\MV$. 

To this end, we will employ the following abbreviations
\[
1 = \neg 0 \qquad \text{ and } \qquad x \odot y = \neg (\neg x \oplus \neg y),
\]
and for every $n \in \mathbb{N}$ we recursively define $n.x$ by  setting
\[
0.x = 0\qquad \text{ and }\qquad (n+1).x = (n.x) \oplus x.
\]

Let $A \in \MV$ and $a \in A$. For every $n \in \Z^+$ we say that $b \in A$ is the result of \emph{dividing $a$ by $n$} when $n.b=a$ and $b \odot ((n-1).b) = 0$. An $\MV$-algebra $\A$ is called \emph{divisible} when for all $a \in A$ and $n \in \Z^+$ there exists $b \in A$ obtained as the result of dividing $a$ by $n$.

\begin{Proposition} \label{Prop : MV div formula}
For each $n \in \mathbb{Z}^+$ there exists a unary $f_n \in \exteq(\MV)$ such that for all $\A \in \MV$ and $a \in \mathsf{dom}(f_n^\A)$,
\begin{align*}
    \mathsf{dom}(f_n^\A) &= \{ c \in A : b \text{ is the result of diving $c$ by $n$ for some }b \in A \};\\
f_n^\A(a) &= \text{the result of dividing $a$ by $n$}.
\end{align*}
\end{Proposition}

\begin{proof}
Consider the  conjunction of equations
\[
\varphi_n(x,y) = (n.y \thickapprox  x) \sqcap (y \odot ((n-1).y)  \thickapprox  0).
\]
We will rely on the following observation.

\begin{Claim}\label{Claim: division in I}
For all $a,b \in [0, 1]$ the following conditions hold in $\I$:
\benroman
\item\label{Claim: division in I: 1} $n.b=a$ if and only if either $\left(a < 1 \text{ and } b=\frac{a}{n} \right)$ or $\left(a = 1 \text{ and } b \geq \frac{1}{n}\right)$;\\[-1.7ex]
\item\label{Claim: division in I: 2} $b \odot ((n-1).b)=0$ if and only if $b \leq \frac{1}{n}$;\\[-1.7ex]
\item\label{Claim: division in I: 3} $\I \vDash \varphi_n (a,b)$ if and only if $b = \frac{a}{n}$.
\eroman
\end{Claim}

\begin{proof}[Proof of the Claim]
The definitions of $\oplus$ and $\neg$ on $\I$ yield $n.c = \min \{ nc, 1\}$ and $c \odot d = \max \{ c+d-1, 0\}$ for all $c,d \in I$ and $n \in \mathbb{N}$.

\eqref{Claim: division in I: 1}: Suppose that $n.b=a$. Then $\min \{ nb, 1\} = a$. If $a < 1$, then $nb = a$. If $a = 1$, then $\min \{nb,1\}=1$. So, $nb \geq 1$, which yields $b \geq \frac{1}{n}$. To prove the reverse implication, we have to consider two cases. First, assume that $a < 1$ and $b=\frac{a}{n}$. Then $n.b =\min \{nb,1\} = \min \{a,1\}=a$. Next we consider the case where $a = 1$ and $b \geq \frac{1}{n}$. We have $nb \geq 1$, and hence $n.b = \min \{nb,1\}= 1 = a$. 

\eqref{Claim: division in I: 2}: We have that $b \odot ((n-1).b)=0$ if and only if $\max \{ b+((n-1).b)-1, 0\} = 0$, which is equivalent to $b+((n-1).b-1) \leq 0$. Moreover,
\begin{align*}
b+((n-1).b)-1 = b + \min\{ (n-1)b, 1\}-1=\min\{nb-1,b\}.
\end{align*}
Therefore, $b \odot ((n-1).b)=0$ if and only if $\min\{nb-1,b\} \leq 0$. As $b \geq 0$, the latter condition is equivalent to $nb -1 \leq 0$, and hence to $b \leq \frac{1}{n}$.

\eqref{Claim: division in I: 3}: Together with \eqref{Claim: division in I: 1} and \eqref{Claim: division in I: 2}, the definition of $\varphi_n$  yields that $\I \vDash \varphi_n (a,b)$ if and only if either $\left(a < 1 \text{ and } b=\frac{a}{n} \text{ and } b \leq \frac{1}{n} \right)$ or $\left(a = 1 \text{ and } b = \frac{1}{n}\right)$. Since $a \in [0, 1]$, we have $\frac{a}{n} \leq \frac{1}{n}$. Therefore, if $b = \frac{a}{n}$, we have $b \leq \frac{1}{n}$. We conclude that $\I \vDash \varphi_n (a,b)$ if and only if $b=\frac{a}{n}$.
\end{proof}

From \cref{Claim: division in I}\eqref{Claim: division in I: 3} it follows that $\varphi_n$ is functional and total in $\I$. Since $\MV = \QQQ(\I)$ (see, e.g., \cite[p.~84]{GT98}), \cref{Cor : functionality in Q(K)} and \cref{Prop : extendable : sufficient conditions}\eqref{item : extendable : sufficient conditions : 2} imply that  $\varphi_n$ defines a unary $f_n \in \exteq(\MV)$. Lastly, as $\varphi_n$ defines $f_n$, the two displays in the statement hold.
\end{proof}

As a consequence, we obtain the following observation from \cite{BlHoo06,MR1888129}.

\begin{Corollary}
$\MV$ lacks the strong epimorphism surjectivity property.
\end{Corollary}

\begin{proof}
    Let $f_2$ be the member of $\exteq(\MV)$ given by   \cref{Prop : MV div formula}. Moreover, let $\A$ be the subalgebra of $\I$ with universe $\{0,1\}$. Since $\A \leq \I \in \MV$ and $f_2^\I(1)=\frac{1}{2} \notin A$, from  \cref{Thm : dominions : pp formulas} it follows that $\frac{1}{2} \in \d_\MV(\A, \I)- A$. Therefore,  $\MV$ lacks the strong epimorphism surjectivity property by \cref{Prop : SES and dominions}.
\end{proof}

A \emph{DMV-algebra} (see \cite{Ger01} and \cite[Def.~5.1.1]{Ger01thesis}) is an algebra $\A = \langle A; \oplus, \neg, \{d_n\}_{n \in \Z^+}, 0 \rangle$ comprising an MV-algebra $\langle A; \oplus, \neg, 0 \rangle$ and a sequence of unary operations $\{ d_n \}_{n \in \mathbb{Z}^+}$ satisfying the  equations
\[
n.d_n(x) \thickapprox  x \quad \text{and} \quad  d_n(x) \odot ((n-1).d_n(x))  \thickapprox  0.
\]
Let $\DMV$ be the variety of DMV-algebras.

\begin{Theorem}  \label{Thm : Beth comp of MV}
$\DMV$ is an equational Beth companion of $\MV$. 
\end{Theorem}

\begin{proof}
Let $\F = \{f_n : n \in \mathbb{Z}^+\} \subseteq \exteq(\MV)$ be the family of operations given by \cref{Prop : MV div formula}.
Moreover, let $d_n$ be a unary function symbol for each $n \in \Z^+$. Then the language $\L_\F = \L_\MV \cup \{ d_n : n \in \Z^+\}$ is an $\F$-expansion of $\L_\MV$ in which the role of $g_{f_n}$ is played by $d_n$. From \cref{Thm : axiomatization of pp expansion (almost always)} and the fact that each $f_n$ is defined by the conjunction of equations $\varphi_n$ in the proof of \cref{Prop : MV div formula} it follows that $\SSS(\MV[\L_\F])=\MV[\L_\F]$ is an equational pp expansion of $\MV$ axiomatized by the axioms of MV-algebras plus the equations $n.d_n(x) \thickapprox  x$ and $d_n(x) \odot ((n-1).d_n(x))  \thickapprox  0$ for $n \in \Z^+$.
Clearly, every member of $\MV[\L_\F]$ is a DMV-algebra. On the other hand, every DMV-algebra  can be obtained by adding the implicit operations $f_n$ to its MV-algebra reduct, which belongs to $\MV[\L_\F]\res_{\L_\MV}$. Therefore,  $\MV[\L_\F]$ coincides with the variety $\DMV$ of DMV-algebras.

In view of \cref{Thm : Beth compl. iff eq absolutely closed}, to show that $\DMV$ is a Beth companion of $\MV$, it suffices to prove that $\DMV\res_\MV \subseteq \MV_\ac$. The definition of $\DMV$ yields that the members of $\DMV\res_\MV$ are divisible MV-algebras. Every divisible MV-algebra is saturated in $\MV$ (see \cite[Thm.~3.18(ii)]{epiDvZa}) and $\MV$ has the amalgamation property (see \cite[p.~91]{Mun88}). Therefore, from \cref{Prop : AP implies acl = sat}\eqref{Prop : AP implies acl = sat: 2} it follows that $\DMV\res_\MV \subseteq \MV_\ac$. Then \cref{Thm : Beth compl. iff eq absolutely closed} yields that $\DMV$ is a Beth companion of $\MV$ which, moreover, is equational because $\DMV$ is an equational pp expansion of $\MV$.
\end{proof}
\end{exa}

The next concept originates in \cite{MR57230,MR57231}.

\begin{Definition}
A finite algebra $\A$ is said to be \emph{primal} when for every function $f \colon A^n \to A$ of positive arity there exists a term $t(x_1, \dots, x_n)$ of $\L_\A$ such that for all $a_1, \dots, a_n \in A$,
\[
f(a_1, \dots, a_n) = t^\A(a_1, \dots, a_n).
\]
\end{Definition}

Examples of primal algebras include the two-element Boolean algebra and the rings of the form $\mathbb{Z}_p$ with $p$ prime (see, e.g., \cite{MR57230}). Primal algebras admit the following elegant characterization (see, e.g., \cite[Cor.\ IV.10.8]{BuSa00}), where \emph{rigid} means ``lacking nonidentity automorphisms''.

\begin{Theorem}\label{Thm : primal description : simple etc}
A finite algebra $\A$ is primal if and only if $\VVV(\A)$ is arithmetical and $\A$ is simple, rigid, and lacks proper subalgebras.
\end{Theorem}

The structure theory of varieties generated by a primal algebra is very rich, as a consequence of the fact that these are precisely the varieties categorically equivalent to the variety of Boolean algebras \cite{MR244130} (see also \cite{MR724265}). In particular, since the variety of Boolean algebras has the surjective epimorphism property (see \cref{Thm : SES : Boolean algebras}) and this property is preserved by categorical equivalences between varieties by \cref{Rem : SES = monos are regular}, we deduce the following.

\begin{Proposition}\label{Prop : primal algebras have the SES}
Varieties generated by a primal algebra have the strong epimorphism surjectivity property.
\end{Proposition}

In the next example, we will employ the following observation.

\begin{Theorem} \label{Prop : Beth comp triangle trick}
Let $\A$ be an $\mathscr{L}$-algebra, $\F \subseteq \imppp(\A)$, and $\L_\mathcal{F}$ an $\F$-expansion of $\L$ such that $\A[\mathscr{L}_{\mathcal{F}}]$ is defined. If $\A[\mathscr{L}_{\mathcal{F}}]$ is primal, then  $\VVV(\A[\mathscr{L}_{\mathcal{F}}])$ is a Beth companion of $\mathbb{Q}(\A)$. If, moreover, $\F \subseteq \impeq(\A)$, then the Beth companion $\VVV(\A[\mathscr{L}_{\mathcal{F}}])$ is  equational.
\end{Theorem}

\begin{proof}
Suppose that $\A[\mathscr{L}_{\mathcal{F}}]$ is primal. Then  $\VVV(\A[\mathscr{L}_{\mathcal{F}}])$ has the strong epimorphism surjectivity property by \cref{Prop : primal algebras have the SES}. Furthermore, recall from \cite[Thm.\ IV.9.4]{BuSa00} that $\QQQ(\A[\mathscr{L}_{\mathcal{F}}]) = \VVV(\A[\mathscr{L}_{\mathcal{F}}])$  because $\A[\mathscr{L}_{\mathcal{F}}]$ is primal. Together with \cref{Thm : Beth companions vs strong Beth}, this yields that, to conclude that $\VVV(\A[\mathscr{L}_{\mathcal{F}}])$ is a Beth companion of $\QQQ(\A)$, it only remains to show that $\QQQ(\A[\mathscr{L}_{\mathcal{F}}])$ is a pp expansion of $\QQQ(\A)$.

Since $\F \subseteq \imppp(\A)$, for every $f \in \F$ there exists a pp formula $\varphi_f$ functional in $\A$ that defines $f$.
By \cref{Cor : functionality in Q(K)} each $\varphi_f$ defines some $f^* \in \imppp(\QQQ(\A))$. Let $\F^* = \{ f^* : f \in \F\}$.
As $\A[\L_\F]$ is defined, $f^{\A}$ is total for every $f \in \F$. Consequently,  $f^\A=(f^*)^\A$ because $f$ and $f^*$ are both defined by $\varphi_f$. Therefore, $(f^*)^\A$ is total for every $f^* \in \F^*$. Then \cref{Prop : extendable : sufficient conditions}\eqref{item : extendable : sufficient conditions : 2} yields $\mathcal{F}^* \subseteq \extpp(\QQQ(\A))$. We can regard $\L_\F$ as an $\F^*$-expansion of $\L$ by setting $g_{f^*}=g_f$ for each $f \in \F$. Since $f^\A=(f^*)^\A$ for every $f \in \F$, the definition of $\A[\L_\F]$ is independent on whether $\L_\F$ is thought of as an $\F$-expansion or as an $\F^*$-expansion.
Thus, \cref{Thm : pp expansion : description in terms of Q and U} implies that $\QQQ(\A[\mathscr{L}_{\mathcal{F}}])$ is a pp expansion of $\QQQ(\A)$ induced by $\F^*$ and $\L_\F$.

The last part of the statement follows immediately from the construction described above.
\end{proof}

\begin{exa} [\textsf{Varieties of MV-algebras generated by a  finite chain}] \label{Exa : MV gen by fin chains}
For $n \in \mathbb{Z}^+$, we denote by $\Luk_n$ the  subalgebra of the real unit interval $\I$ (cf.\ \cref{Exa : MV}) with universe $\{\frac{m}{n} : m \in \mathbb{N}, \ m \leq n\}$. Notice that $\Luk_n$ is a finite $\MV$-algebra of $n+1$ elements. We consider the variety 
\[
\MV_n = \mathbb{V}(\Luk_n) = \QQQ(\Luk_n),
\]
where the second equality in the above display holds by \cite[Lem.\ 1.6]{MVGiTo}. One of the reasons the varieties $\MV_n$ are of interest is that they are precisely the proper nontrivial subvarieties of $\MV$ with the amalgamation property (see \cite[Thm.\ 13]{MVDNLet}) or, equivalently, the subvarieties of $\MV$ generated by a finite subdirectly irreducible algebra (see, e.g., \cite[Lem.~3]{Cha59} and \cite[Prop.~3.6.5]{COM00}).

For each $n \in \mathbb{Z}^+$ consider the conjunction of equations
\[
\psi_n(x, y) = (n.y \thickapprox  1) \sqcap (y \odot ((n-1).y) \thickapprox  0).
\]
As $\Luk_n \leq \I$, \cref{Claim: division in I}(\ref{Claim: division in I: 1}, \ref{Claim: division in I: 2}) implies that for all $a,b \in \Luk_n$,
\[
\Luk_n \vDash \psi_n(a,b) \iff b=\frac{1}{n}.
\]
Therefore, $\psi_n$ defines a total unary $c_n \in \impeq(\Luk_n)$ such that $c_n^{\scriptsize \Luk_n}$ is the constant map with value $\frac{1}{n}$.  Let $\L_n$ be a $c_n$-expansion of $\L_{\MV_n}$. Since $c_n^{\scriptsize\Luk_n}$ is total, the algebra $\Luk_n[\L_n]$ is defined and can thought of as the result of adding a constant for the element $\frac{1}{n}$ to $\Luk_n$. Then let $\DMV_n =\VVV(\Luk_n[\L_n])$.

\begin{Theorem} \label{Thm : Beth comp of V(MVn)}
$\DMV_n$ is a an equational Beth companion of $\MV_n$.
\end{Theorem}

\begin{proof}
Recall that $\MV_n = \QQQ(\Luk_n)$. Furthermore, $c_n \in \impeq(\Luk_n)$ and $\Luk_n[\L_n]$ is defined. Therefore, in view of \cref{Prop : Beth comp triangle trick}, it suffices to show that $\Luk_n[\L_n]$ is a primal algebra. To this end, we will employ \cref{Thm : primal description : simple etc}.

First, observe that $\Luk_n[\L_n]$ is finite because so is $\Luk_n$.\ Moreover, recall from \cite[Cor.~3.5.4]{COM00} that $\Luk_n$ is simple. As $\Luk_n[\L_n]$ is obtained by adding a constant operation to $\Luk_n$, we have $\mathsf{Con}(\Luk_n) = \mathsf{Con}(\Luk_n[\L_n])$. Hence, $\Luk_n[\L_n]$ is simple too. Similarly, since $\Luk_n[\L_n]$ is obtained by adding to $\Luk_n$ a constant operation with value $\frac{1}{n}$ and in $\Luk_n$ we have $\frac{m}{n} = m.\frac{1}{n}$ for every $0 \leq m \leq n$, the algebra $\Luk_n[\L_n]$ is rigid and lacks proper subalgebras.
Lastly, recall that $\MV_n = \VVV(\Luk_n)$ is arithmetical (see, e.g., \cite[Prop.~7.6]{GM05}). Therefore, $\VVV(\A)$ is arithmetical for every expansion $\A$ of $\Luk_n$ by \cite[Thm.\ II.12.5]{BuSa00}. In particular, $\VVV(\Luk_n[\L_n])$ is arithmetical, as desired.
\end{proof}
\end{exa}

We now turn our attention to proving Theorems \ref{Thm : acl and Beth companion} and \ref{Thm : Beth compl. iff eq absolutely closed}. We begin by establishing the following pair of results.

\begin{Proposition} \label{Prop : doms in K vs. M}\label{Prop : doms pp expansion ext}
    Let $\mathsf{M} = \SSS(\mathsf{K}[\mathscr{L}_{\mathcal{F}}])$ be a pp expansion of a universal class $\mathsf{K}$. Then the following conditions hold:
    \benroman
    \item\label{item : dom_K = dom_K[L_F]}\label{Thm : doms pp expansion ext : 1} $\mathsf{d}_{\mathsf{K}}(\A \resLK, \B \resLK) = \mathsf{d}_{\mathsf{M}}(\A, \B)$ for all $\A \leq \B \in \mathsf{K}[\mathscr{L}_{\mathcal{F}}]$; 
    \item\label{item : check SES of M in K}\label{Thm : doms pp expansion ext : 2} $\mathsf{M}$ is a Beth companion of $\K$ if and only if $\mathsf{d}_{\mathsf{K}}(\A \resLK, \B \resLK) = A$ for all $\A \leq \B \in \M$. 
    \eroman
\end{Proposition}

\begin{proof}
\eqref{item : dom_K = dom_K[L_F]}: Let $\A \leq \B \in \mathsf{K}[\mathscr{L}_{\mathcal{F}}]$. To establish the inclusion from left to right, consider $b \in B-\mathsf{d}_{\M}(\A, \B)$. Then there exists $\C \in \M$ and a pair of homomorphisms $g,h \colon \B \to \C$ such that $g\res_{A}=h\res_{A}$ and $g(b) \neq h(b)$. From \cref{Prop : pp expansions and subreducts} it follows that $\C\resLK \in \K$. Together with the fact that $g,h \colon \B\resLK \to \C\resLK$ are homomorphisms such that $g\res_{A}=h\res_{A}$ and $g(b) \neq h(b)$, this yields $b \in B-\mathsf{d}_{\mathsf{K}}(\A\resLK, \B\resLK)$. Hence, $\mathsf{d}_{\mathsf{K}}(\A\resLK, \B\resLK) \subseteq \mathsf{d}_{\M}(\A, \B)$. To prove the reverse inclusion, consider $b \in B-\mathsf{d}_{\mathsf{K}}(\A\resLK, \B\resLK)$. Then there exist $\C \in \K$ and a pair of homomorphisms $g,h \colon \B\resLK \to \C$ such that $g\res_{A}=h\res_{A}$ and $g(b) \neq h(b)$. Since $\M = \SSS(\K[\L_\mathcal{F}])$  is a pp expansion of the universal class $\K$ by assumption,  \cref{Prop : pp expansions : subreducts : extendable} guarantees that
$\mathsf{K}$ is the class of $\mathscr{L}_\mathsf{K}$-subreducts of $\mathsf{K}[\mathscr{L}_\mathcal{F}]$. So, there exists $\D \in \mathsf{K}[\mathscr{L}_\mathcal{F}]$ such that $\C \leq \D\resLK$. As $\B, \D \in \mathsf{K}[\mathscr{L}_\mathcal{F}]$, the homomorphisms  $g,h \colon \B\resLK \to \D\resLK$ can by viewed as homomorphisms $g,h \colon \B \to \D$ by \cref{Prop : small homs are big homs}. Therefore, from $\D \in \M$, $g\res_{A}=h\res_{A}$, and $g(b) \neq h(b)$ it follows that  $b \in B-\mathsf{d}_{\M}(\A, \B)$.\ Hence, conclude that $\mathsf{d}_{\M}(\A, \B) \subseteq \mathsf{d}_{\mathsf{K}}(\A\resLK, \B\resLK)$.

\eqref{item : check SES of M in K}: 
Assume that $\mathsf{M}$ is a Beth companion of $\K$ and consider $\A \leq \B \in  \mathsf{M}$. We will show that $\mathsf{d}_{\mathsf{K}}(\A \resLK, \B \resLK) = A$. Since $\mathsf{M} = \SSS(\mathsf{K}[\mathscr{L}_{\mathcal{F}}])$, there exists $\C \in \mathsf{K}[\mathscr{L}_{\mathcal{F}}]$ such that $\B \leq \C$. As $\C \in \mathsf{K}[\mathscr{L}_{\mathcal{F}}]$, from \eqref{item : dom_K = dom_K[L_F]} it follows that $\mathsf{d}_{\mathsf{K}}(\A \resLK, \C \resLK) = A$. 
 Moreover, \cref{Cor: dominions subalg quot}\eqref{Cor: dominions subalg quot: 1} yields $\d_\K(\A \resLK, \B \resLK) \subseteq \d_\K(\A \resLK, \C \resLK)$. Therefore, $\d_\K(\A \resLK, \B \resLK) = A$.

To prove the converse implication assume that $\d_{\mathsf{K}}(\A \resLK, \B \resLK) = A$ for all $\A \leq \B \in \mathsf{M}$. 
By \cref{Prop : SES and dominions,Thm : Beth companions vs strong Beth}, to show that $\M$ is a Beth companion of $\K$, it suffices to establish that $\mathsf{d}_{\mathsf{M}}(\C, \D) = C$ for all $\C \leq  \D\in \mathsf{M}$. Consider $\C \leq \D \in \M$.
As $\mathsf{M} = \SSS(\mathsf{K}[\mathscr{L}_{\mathcal{F}}])$, there exists $\E \in \mathsf{K}[\mathscr{L}_{\mathcal{F}}]$ such that $\D \leq \E$. 
Then 
\[
\mathsf{d}_{\mathsf{M}}(\C, \D) \subseteq \mathsf{d}_{\mathsf{M}}(\C, \E) =  \mathsf{d}_{\mathsf{K}}(\C \resLK, \E \resLK),
\]
where the first equality holds by \cref{Cor: dominions subalg quot}\eqref{Cor: dominions subalg quot: 1} and the second follows from \eqref{item : dom_K = dom_K[L_F]} because $\E \in \mathsf{K}[\mathscr{L}_{\mathcal{F}}]$. 
Our assumption implies that $\mathsf{d}_{\mathsf{K}}(\C \resLK, \E \resLK) = C$. Thus, $\mathsf{d}_{\mathsf{M}}(\C, \D) = C$. 
\end{proof}

\begin{Proposition} \label{Prop : Beth comp have same reducts}
    Let $\mathsf{M}_1$ and  $\mathsf{M}_2$ be a pair of Beth companions of a quasivariety $\mathsf{K}$. 
    Then $\mathsf{M}_1 \resLK = \mathsf{M}_2 \resLK$.
\end{Proposition}

\begin{proof}
Since $\M_1$ and $\M_2$ are Beth companions of $\K$, by \cref{Thm : Beth companions : term equivalence} there exists a pair of maps $\tau \colon \L_{\M_1} \to T_2$ and $\rho \colon \L_{\M_2} \to T_1$ witnessing that $\mathsf{M}_1$ and $\mathsf{M}_2$ are faithfully term equivalent relative to $\mathsf{K}$, where $T_i$ is the set of terms of $\M_i$ in a countably infinite set of variables for $i = 1,2$.
By symmetry it suffices to show $\mathsf{M}_1 \resLK \subseteq \mathsf{M}_2 \resLK$. To this end, consider $\A \in \mathsf{M}_1$.
The definition of a faithful term equivalence yields  $\rho(\A) \in \mathsf{M}_2$ and $\rho(\A) \resLK = \A \resLK$ (see \cref{Rem : properties preserved by term equivalence}\eqref{item: preservation term eq : reduct}).
Therefore,
$\A \resLK = \rho(\A) \resLK \in \mathsf{M}_2 \resLK$, 
as desired. 
\end{proof}

We are now ready to prove \cref{Thm : acl and Beth companion}.

\begin{proof}
We first prove the inclusion $\K[\L_\F]\resLK \subseteq \K_\ac$. Consider $\A \in \K[\L_\F]$. To show that $\A\resLK$ is absolutely closed in $\K$, let $\B \in \K$ with $\A \resLK \leq \B$. We need to prove that $\d_\K(\A \resLK, \B) = A$.
Since $\M$ is a pp expansion of $\K$, \cref{Prop : pp expansions and subreducts} guarantees the existence of $\C \in \M$ such that $\B \leq \C \resLK$. Then \cref{Cor: dominions subalg quot}\eqref{Cor: dominions subalg quot: 1} yields $\d_\K(\A \resLK, \B) \subseteq \d_\K(\A \resLK, \C \resLK)$. As $\M$ is a Beth companion of $\K$, from \cref{Prop : doms in K vs. M}\eqref{item : check SES of M in K} it follows that $\d_\K(\A \resLK, \C\resLK) =A$.  Therefore, $\d_\K(\A \resLK, \B) = A$, as desired

To prove the inclusion $\K_\ac \subseteq \M\resLK$, consider $\A \in \K_\ac$. Since $\mathcal{F} \subseteq \mathsf{ext}(\mathsf{K})$, \cref{Prop : pp expansions : subreducts : extendable} implies that there exists $\B \in \mathsf{K}[\mathscr{L}_{\mathcal{F}}]$ such that $\A \leq \B \resLK$. 
We show that $A$ is the universe of a subalgebra of $\B$. From $\A \leq \B \resLK$ it follows that $A$ is closed under the operations of the language of $\K$. 
Recall that every operation symbol of $\mathscr{L}_{\mathcal{F}} - \mathscr{L}_{\mathsf{K}}$ is of the form $g_f$ for some $f \in \F$. 
Consider an $n$-ary $f \in \F$. We will show that $A$ is closed under $g_f^\B$. Since $\B \in \mathsf{K}[\mathscr{L}_{\mathcal{F}}]$, we have  $g_f^\B=f^{\B\resLK}$. Let $a_1, \dots, a_n \in A$. Since $f^{\B\resLK}$ is total, \cref{Thm : dominions : pp formulas} yields  $f^{\B\resLK}(a_1, \dots, a_n) \in \mathsf{d}_{\mathsf{K}}(\A, \B \resLK)$. As $\A$ is absolutely closed in $\K$ by assumption, we have $\mathsf{d}_{\mathsf{K}}(\A, \B \resLK) = A$, and hence $f^{\B\resLK}(a_1, \dots, a_n) \in A$. We have shown that $A$ is the universe of a subalgebra of $\B$.
Therefore, we can expand $\A$ to an $\mathscr{L}_{\mathsf{M}}$-algebra $\A^*$ by setting $g_f^{\A^*} = g_f^{\B} \res_A$ for every $g_f \in \mathscr{L}_{\mathcal{F}} - \mathscr{L}_{\mathsf{K}}$. The definition of $\A^*$ guarantees that $\A^* \leq \B$. Since $\B \in \M$ and $\M$ is a universal class by \cref{Thm : pp expansion : still quasivariety}\eqref{item : pp expansion : universal class}, we obtain $\A^* \in \SSS(\M) \subseteq \mathsf{M}$.
Thus, we conclude that $\A = \A^* \resLK \in \mathsf{M} \resLK$. 

It remains to show that, when $\M$ is an equational Beth companion of $\K$, we have $\K_\ac = \M\res_{\L_\K}$. Suppose that $\M$ is an equational Beth companion of $\K$. Then $\mathsf{M}$ is 
faithfully term equivalent to a 
Beth companion $\mathsf{M}^*$ of $\K$ induced by a family of operations defined by conjunctions of equations.
Moreover, \cref{Thm : axiomatization of pp expansion (almost always)} yields that
$\mathsf{M}^*$ is of the form $\mathsf{K}[\mathscr{L}_{\mathcal{F}^*}]$ for some $\mathcal{F}^* \subseteq \mathsf{ext}_{\textsc{eq}}(\mathsf{K})$.  
Then $\K[\mathscr{L}_{\mathcal{F}^*}]\resLK =\M^*\resLK$.
Since $\K[\L_{\F^*}]\resLK \subseteq \K_\ac \subseteq \M^*\resLK$ by the first two paragraphs of this proof, we have $\K_\ac =\M^*\resLK$. As $\M$ and $\M^*$ are Beth companions of $\K$, \cref{Prop : Beth comp have same reducts} yields $\M\resLK = \M^*\resLK$. Hence, we conclude that $\K_\ac =\M\resLK$.
\end{proof}

Lastly, we prove \cref{Thm : Beth compl. iff eq absolutely closed}.

\begin{proof}
In view of \cref{Prop : doms in K vs. M}(\ref{item : check SES of M in K}), to prove that $\M$ is a Beth companion of $\K$, it suffices to show that $\mathsf{d}_{\mathsf{K}}(\A \resLK, \B \resLK) = A$ for all $\A \leq \B \in \M$. To this end, consider $\A \leq \B \in \M$. Since $\M$ is a universal class by 
\cref{Thm : pp expansion : still quasivariety}(\ref{item : pp expansion : universal class}), we obtain $\A \in \M$. Therefore, the assumption that $\mathsf{M} \resLK \subseteq  \K_\ac$ implies $\A\resLK \in \K_\ac$. Hence, $\mathsf{d}_{\mathsf{K}}(\A \resLK, \B \resLK) = A$.
\end{proof}

\section{Classes without a Beth companion} \label{Sec : Classes without Beth comp}

 We close this work by providing some examples of classes of algebras lacking a Beth companion.\
The two main results of the section concern the varieties of monoids and of commutative monoids (\cref{Thm : M lacks Beth comp}), and certain quasivarieties of Heyting algebras (\cref{Thm : quasivariety HA and C5}). We begin with the  result on monoids.

\begin{Theorem} \label{Thm : M lacks Beth comp}
The varieties of monoids and of commutative monoids lack a Beth companion. 
\end{Theorem}

The theorem above is the starting point of a complete description of the varieties of commutative monoids admitting a Beth companion. As the methods utilized in its proof go beyond the theory of implicit operations developed here,
we will provide such a description in the separate work \cite{CKMMON}. 

In order to prove 
 \cref{Thm : M lacks Beth comp}, 
we first need to introduce the notion of a dominion base and establish some technical results about dominion bases and implicit operations.

\begin{Definition}
    Let $\mathsf{K}$ be a class of algebras and $\Delta \subseteq \imppp(\mathsf{K})$. We say that $\Delta$ is a \emph{dominion base} for $\K$ when for all $\A \leq \B \in \K$ and $b \in \d_\K(\A, \B)$ there exist $f \in \Delta$ and $\langle a_1, \dots, a_n \rangle \in \dom({f^\B}) \cap A^n$ such that $f^\B(a_1, \dots, a_n)=b$.
\end{Definition}

\cref{Thm : dominions : pp formulas} states that $\imppp(\mathsf{K})$ is a dominion base for every elementary class $\K$, and Isbell's Zigzag Theorem \ref{Thm : Isbell Theorem} states that Isbell's operations (see \cref{Exa : Isbell operations}) form a dominion base for the varieties of monoids and of commutative monoids. The following result illustrates how having a concrete and transparent dominion base simplifies the task of finding interpolants for implicit operations.

\begin{Theorem}\label{Thm : interpolation dominion base terms}
Let $\K$ be a quasivariety with dominion base $\Delta$ and $f \in \imppp(\K)$ of arity~$n$. Then there exist $g \in \Delta$ and $n$-ary terms $t_1, \dots, t_m$ of $\K$
such that the composition $g(t_1^\K, \dots, t_m^\K)$ interpolates $f$ in $\K$. 
\end{Theorem}

\begin{proof}
Let $\varphi$ be a pp formula defining $f$. Then 
\[
\varphi (x_1, \dots, x_n,y) = \exists z_1, \dots, z_k \psi(z_1, \dots, z_k, x_1, \dots, x_n, y),
\]
where $\psi$ is a conjunction of equations.\ 
Let $X=\{z_1, \dots, z_k, x_1, \dots, x_n, y\}$. Since $\K$ is a quasivariety, the free algebra $\boldsymbol{F}_\K(X)$ belongs to $\K$ (see \cref{thm:quasivariety free algebra}).
We denote by $\theta$ the $\K$-congruence of $\boldsymbol{F}_\K(X)$ generated by the pairs $\langle s_1, s_2 \rangle$, where
$s_1 \thickapprox s_2$ is an equation in $\psi$. Consider $\B = \boldsymbol{F}_\K(X) / \theta$ and $\A = \mathsf{Sg}^\B( x_1 / \theta, \dots, x_n / \theta)$.
The definition of $\theta$ implies that 
\[
\B \vDash \psi(z_1/\theta, \dots, z_k/\theta, x_1/\theta, \dots, x_n/\theta, y/\theta).
\]
Therefore, $\langle x_1 / \theta, \dots, x_n / \theta \rangle \in \mathsf{dom}(f^\B)$ and
$f^\B(x_1/\theta, \dots, x_n/\theta)=y/\theta$ because $f$ is defined by $\varphi$.
As $f \in \imppp(\K)$ by assumption and $x_1/\theta, \dots, x_n/\theta \in A$ by the definition of $\A$,
it follows from \cref{Thm : dominions : pp formulas} that $y / \theta \in \mathsf{d}_{\mathsf{K}}(\A, \B)$.
Since $\Delta$ is a dominion base for $\mathsf{K}$ and $\A$ is generated by $x_1 / \theta, \dots, x_n / \theta$, there exist an $m$-ary $g \in \Delta$ and terms $t_1(x_1, \dots, x_n), \dots, t_{m}(x_1, \dots, x_n)$ of $\K$ such that 
\begin{equation}\label{Eq : the formula : precedent : pre}
\langle t_1^\B(x_1 / \theta, \dots, x_n/ \theta), \dots, t_{m}^\B(x_1/ \theta, \dots, x_n/ \theta) \rangle \in \mathsf{dom}(g^\B)    
\end{equation}
 and 
\begin{equation}\label{Eq : the formula}
g^\B(t_1^\B(x_1 / \theta, \dots, x_n/ \theta), \dots, t_{m}^\B(x_1/ \theta, \dots, x_n/ \theta))= y/ \theta.
\end{equation}

We will prove that $g(t_1^\K, \dots, t_m^\K)$ interpolates $f$ in $\K$. To this end, consider $\C \in \mathsf{K}$ and $c_1, \dots, c_n, d \in C$ such that $\langle c_1, \dots, c_n \rangle \in \dom(f^\C)$ and $f^\C(c_1,\dots,c_n)=d$.
We need to show that 
\[
\langle t_1^\C(c_1, \dots, c_n), \dots, t_{m}^\C(c_1, \dots, c_n) \rangle \in \mathsf{dom}(g^\C) \, \, \text{ and }\, \, 
g^\C(t_1^\C(c_1, \dots, c_n), \dots, t_{m}^\C(c_1, \dots, c_n))=d.
\]
Since $f$ is defined by $\varphi$, from $f^\C(c_1, \dots, c_n)=d$ it follows that $\C \vDash \psi(e_1, \dots, e_k, c_1, \dots,c_n, d)$ for some $e_1, \dots, e_k \in C$. Therefore, 
\begin{equation}\label{Eq : the formula 2}
\C \vDash s_1(e_1, \dots, e_k, c_1, \dots,c_n, d) \thickapprox s_2(e_1, \dots, e_k, c_1, \dots,c_n, d)
\end{equation}
for every equation $s_1 \thickapprox s_2$ in $\psi$. As  $X=\{z_1, \dots, z_k, x_1, \dots, x_n, y\}$ is a set of free generators for $\boldsymbol{F}_\K(X)$, there exists a homomorphism $h \colon \boldsymbol{F}_\K(X) \to \C$ such that $h(z_i)=e_i$ and $h(x_i)=c_i$ for each $i$, and $h(y)=d$. The definition of $\theta$ and~\eqref{Eq : the formula 2} yield $\theta \subseteq \Ker(h)$. 
Since $\B= \boldsymbol{F}_\K(X)/\theta$, \cref{Prop : homomorphisms : smaller congruences} implies that the homomorphism $h$ induces a homomorphism $k \colon \B \to \C$ such that $k(z_i/\theta)=e_i$ and $k(x_i/\theta)=c_i$ for each $i$, and $k(y/\theta)=d$.
Since $k$ is a homomorphism, we obtain
\[
t_i^\C(c_1, \dots, c_n) = t_i^\B(k(x_1 / \theta), \dots, k(x_n/ \theta)) = k(t_i^\B(x_1 / \theta, \dots, x_n/ \theta))
\]
for each $i \leq m$. So, 
\eqref{Eq : the formula : precedent : pre} implies  $\langle t_1^\C(c_1, \dots, c_n), \dots, t_{m}^\C(c_1, \dots, c_n) \rangle \in \dom(g^\C)$ because $g \in \imppp(\K)$.
We also have
\begin{align*}
g^\C(t_1^\C(c_1, \dots, c_n)&, \dots, t_{m}^\C(c_1, \dots, c_n)) \\
&= g^\C(t_1^\C(k(x_1/\theta), \dots, k(x_n/\theta)), \dots, t_{m}^\B(k(x_1/\theta), \dots, k(x_n/\theta)))\\
& = k(g^\B(t_1^\B(x_1 / \theta, \dots, x_n/ \theta), \dots, t_{m}^\B(x_1/ \theta, \dots, x_n/ \theta)))\\
& = k(y/ \theta)\\
& =d,
\end{align*}
where the first equality holds because $k(x_i/\theta)=c_i$ for each $i$, the second follows from 
the assumptions that  $k$ is a homomorphism and $g \in \imppp(\K)$, the third  from (\ref{Eq : the formula}), and the last from $k(y/\theta)=d$. Hence, we conclude that $g^\C(t_1^\C(c_1, \dots, c_n), \dots, t_{m}^\C(c_1, \dots, c_n))=d$, as desired.
\end{proof}

We now apply \cref{Thm : interpolation dominion base terms} and the fact that Isbell's operations form a dominion base to deduce a useful property of implicit operations of the variety of (commutative) monoids.

\begin{Proposition}\label{Prop: property unary extpp in Mon}
Let $\K$ be either the variety of monoids or the variety of commutative monoids. For every unary $f \in \imppp(\K)$ there exist $l,r \in \mathbb{N}$ such that $a^l f^\A(a) = a^r$ for all $\A \in \K$ and $a \in \dom(f^\A)$.
\end{Proposition} 

\begin{proof}
Consider a unary $f \in \imppp(\K)$. For each $n \in \mathbb{N}$ we denote by $g_n$ the implicit operation of $\K$ defined by the $n$-th  Isbell's
formula (see \cref{Exa : Isbell operations}). As Isbell's operations form a dominion base for $\K$, \cref{Thm : interpolation dominion base terms} yields $n \in \mathbb{N}$ and unary terms $t_1, \dots, t_{2n+1}$ such that 
$g_n(t_1^\K, \dots, t_{2n+1}^\K)$
interpolates $f$ in $\K$. 
As unary terms of $\K$ are equivalent to powers of a variable, there exist   $e_1, \dots, e_{n+1} \in \mathbb{N}$ such that $\K \vDash t_i(x) \thickapprox  x^{e_i}$ for each $i \leq 2n+1$.

Recall that $g_0^\A$ is the identity function for every $\A \in \K$. Thus, when $n=0$ we have $f^\A(a)= t_1^\A(a)=a^{e_1}$ for all $\A \in \K$ and $a \in \dom(f^\A)$, and hence we can take $l=0$ and $r=e_1$. So, in the rest of the proof we will assume that $n > 0$. 
Let $l=\sum_{i=1}^n e_{2i}$ and $r = \sum_{i=0}^n e_{2i+1}$. Consider $\A \in \K$ and $a \in \dom(f^\A)$. We show that $a^l f^\A(a) = a^r$. To this end, let $a_i = a^{e_i}$ for every $i \leq 2n+1$. Then $f^\A(a)= g_n^\A(a_1, \dots, a_{2n+1})$ because 
 $g_n(t_1^\K, \dots, t_{2n+1}^\K)$ 
interpolates $f$ in $\K$ and $\K \vDash t_i(x) \thickapprox  x^{e_i}$ for each $i \leq 2n+1$. Notice that $a_1, \dots, a_{2n+1}$ pairwise commute because they are powers of $a$ in $\A$.

We will rely on the next fact, which was established under the commutativity assumption in \cite[proof of 2.7]{Isb65}.

\begin{Claim}\label{Claim : useful a_i-conditions noncomm}
We have $(\prod_{i = 1}^n a_{2i}) f^\A(a) = \prod_{i = 0}^n a_{2i+1}$.
\end{Claim}

\begin{proof}[Proof of the Claim]
As $f^\A(a)= g_n^\A(a_1, \dots, a_{2n+1})$, the definition of $g_n$ implies that there exist $c_1, \dots, c_n \in A$ satisfying
        \benroman
        \item\label{Claim : useful a_i-conditions noncomm : 1} $f^\A(a) = a_1c_1$;
        \item\label{Claim : useful a_i-conditions noncomm : 2} $a_{2i} c_i = a_{2i+1} c_{i+1}$ for every $1 \leq i \leq n-1$;
        \item\label{Claim : useful a_i-conditions noncomm : 3} $a_{2n}c_n = a_{2n+1}$.
        \eroman
Let $c_{n+1} = 1^\A$. To conclude the proof of the claim, it suffices to show that for every positive $m \leq n$,
\begin{equation}\label{eq : induction claim}
\left(\prod_{i = 1}^m a_{2i}\right) f^\A(a) = \left(\prod_{i = 0}^m a_{2i+1}\right) c_{m+1}.   
\end{equation}
This is because we set $c_{n+1}=1^\A$ and, therefore, for $m=n$ we obtain $(\prod_{i = 1}^n a_{2i}) f^\A(a) = (\prod_{i = 0}^n a_{2i+1}) c_{n+1} = \prod_{i = 0}^n a_{2i+1}$, as desired.

The proof of the above display proceeds by induction on $m \leq n$. First, we have 
\[
a_2 f^\A(a) = a_2 a_1 c_1 = a_1 a_2 c_1 = a_1 a_3 c_2,
\]
where the first equality holds by \eqref{Claim : useful a_i-conditions noncomm : 1}, the second because $a_1$ and $a_2$ commute, and the third follows from \eqref{Claim : useful a_i-conditions noncomm : 2}.\
Therefore, \eqref{eq : induction claim} holds for $m=1$. 
Suppose now that $1 < m \leq n$. We show that \eqref{eq : induction claim} holds for $m$ under the assumption that it holds for $m-1$,
which means that $(\prod_{i = 1}^{m-1} a_{2i}) f^\A(a) = (\prod_{i = 0}^{m-1} a_{2i+1}) c_{m}$.
We have
\begin{align*}
\left(\prod_{i = 1}^m a_{2i}\right) f^\A(a) & = a_{2m} \left(\prod_{i = 1}^{m-1} a_{2i}\right) f^\A(a) = a_{2m} \left(\prod_{i = 0}^{m-1} a_{2i+1}\right) c_{m} = \left(\prod_{i = 0}^{m-1} a_{2i+1}\right) a_{2m} c_{m}\\
&= \left(\prod_{i = 0}^{m-1} a_{2i+1}\right) a_{2m+1} c_{m+1} = \left(\prod_{i = 0}^{m} a_{2i+1}\right)c_{m+1},
\end{align*}
where the first and third equalities hold because $a_{2m}$ commutes with every $a_i$, the second follows from the induction hypothesis, the fourth is a consequence of \eqref{Claim : useful a_i-conditions noncomm : 2} when $i \leq n-1$ and of \eqref{Claim : useful a_i-conditions noncomm : 3} and $c_{n+1} = 1^\A$  when $m=n$, and the last is straightforward. This shows that \eqref{eq : induction claim} holds for every $m \leq n$.
\end{proof}

Since $l=\sum_{i=1}^n e_{2i}$ and $r = \sum_{i=0}^n e_{2i+1}$, we have
\[
\prod_{i = 1}^n a_{2i} = \prod_{i = 1}^n a^{e_{2i}} =a^l \quad \text{and} \quad \prod_{i = 0}^n a_{2i+1} = \prod_{i = 0}^n a^{e_{2i+1}} = a^r.
\]
Therefore, \cref{Claim : useful a_i-conditions noncomm} yields 
\[
a^l f^\A(a) = \left(\prod_{i = 1}^n a_{2i}\right) f^\A(a) = \prod_{i = 0}^n a_{2i+1} = a^r.\qedhere
\]
\end{proof}

We will also need the following technical result.

\begin{Proposition} \label{Prop : Beth comp implies interpolation with extendable}
Let $\mathsf{K}$ be a quasivariety with a Beth companion of the form $\SSS(\K[\L_\F])$. 
Then for every $f \in \imppp(\K)$ there exists $g \in \extpp(\K)$ that interpolates $f$ in $\K[\L_\F]\resLK$.
\end{Proposition}

\begin{proof}
Since $f \in \imppp(\K)$, by \cref{Thm : Beth companions vs strong Beth} there exists a term $t$ of $\L_\F$ that interpolates $f$ in $\SSS(\K[\L_\F])$.
As $f$ has positive arity and $t$ interpolates $f$, it follows that $t$ is not a constant. 
So, \cref{Prop : equivalence in K[F]}\eqref{item : equivalence in K[F] : 2} 
yields $g \in \mathsf{ext}_{\textsc{pp}}(\mathsf{K})$ 
such that 
\begin{equation*}
t^\B=g^{\B\resLK} \textrm{ for every } \B \in \K[\L_\F].
\end{equation*}

We will show that $g$ interpolates $f$ in $\K[\L_\F]\resLK$.
Let $\A \in \K[\L_\F]\resLK$ and $\langle a_1, \dots, a_n \rangle \in \mathsf{dom}(f^\A)$. As $\A[\L_\F] \in \K[\L_\F]$ and $\A = \A[\L_\F]\resLK$, the above display implies that $t^{\A[\L_\F]}=g^{\A}$. It follows that $g^{\A}$ is total, and hence $\langle a_1, \dots, a_n \rangle \in \mathsf{dom}(g^{\A})$. Since $t$ interpolates $f$ in $\SSS(\K[\L_\F])$, we obtain
\[
f^{\A}(a_1, \dots, a_n) = f^{\A[\L_\mathcal{F}]\resLK}(a_1, \dots, a_n) = t^{\A[\L_\F]}(a_1, \dots, a_n) = g^\A(a_1, \dots, a_n).
\]
We conclude that $g$ interpolates $f$ in $\K[\L_\F]\resLK$.
\end{proof}

We are now ready to prove \cref{Thm : M lacks Beth comp}.

\begin{proof}
Let $\K$ be either the variety of monoids or the variety of commutative monoids and assume, with a view to contradiction, that $\K$ has a Beth companion. 
\cref{Cor : Beth companion : ext(K)} implies that $\K$ has also a Beth companion of the form $\SSS(\K[\L_\F])$. Let $f$ be the   operation of taking inverses in monoids. Then $f \in \imppp(\K)$ by \cref{Thm : inverses monoid : implicit operation}.
\cref{Prop : Beth comp implies interpolation with extendable} yields a unary $g \in \extpp(\K)$ that interpolates $f$ in $\K[\L_\F]\res_{\L_\K}$. By \cref{Prop: property unary extpp in Mon} there exist $l,r \in \mathbb{N}$ such that $a^l g^\A(a) = a^r$ for all $\A \in \K$ and $a \in \dom(g^\A)$.
Since $\K$ is a variety and $\SSS(\K[\L_\F])$ a pp expansion of $\K$, \cref{Prop : pp expansions and subreducts} implies that $\K = \SSS((\SSS(\K[\L_\F]))\res_{\L_\K})$.\ Hence, $\K \subseteq \SSS(\K[\L_\F]\res_{\L_\K})$. Therefore, there exists an extension $\B$ of the multiplicative monoid $\mathbb{Q}$ of the rationals such that $\B \in \K[\L_\F]\res_{\L_\K}$. Then $g^\B$ is total because $g \in \extpp(\K) = \F$.
Since $g$ interpolates $f$ in $\K[\L_\F]\res_{\L_\K}$, we obtain $g^\B(2)=f^\B(2)=f^\mathbb{Q}(2)=2^{-1}$. Therefore, $2^{l-1}=2^l g^\B(2)=2^r$. So, $2^{l-1}=2^r$ holds in $\mathbb{Q}$. In turn, this implies $l=r+1$. Consequently, 
\begin{equation}\label{Eq : r+1 is r multiplied by g : monoid}
a^{r+1} g^\A(a) = a^r\text{ for all }\A \in \K\text{ and }a \in \dom(g^\A).
\end{equation}

\begin{Claim} \label{Prop : special A noncomm}
There exist $\C \in \K$ and $c \in C$ such that $c^r \neq c^{r+1}$ and $c^{r+1}=c^{r+2}$.
\end{Claim}
\begin{proof}[Proof of the Claim]
Let $C$ be the set of symbols $\{c^i : 0 \leq i \leq r+1\}$
    and 
    define the following binary operation on $C$ 
    \[
    c^i \cdot c^j 
    = \begin{cases}
        c^{i+j} &\textrm{ if }  i + j < r+1;\\
        c^{r+1} &\textrm{ otherwise.}
    \end{cases}
    \]
    It is straightforward to verify that this defines a commutative monoid $\C$ with neutral element $c^0$. So, $\C \in \K$.
    Let $c=c^1$. Then $c^i$ is the $i$-th power of $c^1$ for every $i \leq r+1$.
    From the definition of $\C$ it then follows immediately that $c^r \neq c^{r+1}$ and $c^{r+1}=c^{r+2}$. 
\end{proof}

Let $\C$ and $c \in C$ be as in \cref{Prop : special A noncomm}. Since $g \in \extpp(\K)$, there exists $\D \in \K$ such that $\C \leq \D$ and $g^\D$ is total. From $c \in \dom(g^\D)$ and (\ref{Eq : r+1 is r multiplied by g : monoid}) it follows $c^{r+1} g^{\D}(c) = c^r$. Thus, using $c^{r+1}=c^{r+2}$, we deduce
\[
c^r = c^{r+1} g^{\D}(c) = c^{r+2} g^{\D}(c) = c \? c^{r+1} g^{\D}(c) = c \? c^{r} = c^{r+1},
\]
a contradiction with $c^r \neq c^{r+1}$ in $\C$.
\end{proof}

\begin{Remark}\label{Rem : semigroups and Beth comp}
    The proof of \cref{Thm : M lacks Beth comp} can easily be adapted to show that both the variety of semigroups and the variety of commutative semigroups lack a Beth companion. 
    To see this, recall that Isbell's formulas (see \cref{Exa : Isbell operations}) also form a dominion base for these varieties  (see \cite[Thm.\ 2.3]{Isb65} for semigroups and \cite[Thm.\ 1.1]{HoIsEpiII} for commutative semigroups). 
    The changes required for adapting the proof of \cref{Thm : M lacks Beth comp} are limited to the following. First, the role of the implicit operation $f$ defined by $\varphi = (x \cdot y \thickapprox 1) \sqcap (y \cdot x \thickapprox 1)$ should be taken over by the implicit operation $g$ defined by \[\psi = (x^2 \cdot y \thickapprox x) \sqcap (y^2 \cdot x \thickapprox y) \sqcap (x \cdot y \thickapprox y \cdot x),\] 
    which also defines inverses when they exist.
    In particular, $g$ coincides with $f$ on $\mathbb{Q}$. 
    Moreover, the monoids $\mathbb{Q}$ and $\C$ appearing in the proof of \cref{Thm : M lacks Beth comp} should be replaced by their semigroup reducts. Lastly, the proof of Claim \ref{Claim : useful a_i-conditions noncomm} 
   uses  
    the fact that $\A$ has a neutral element, which need not be the case if $\A$ is an arbitrary semigroup. This problem can be overcome easily by adding a neutral element to $\A$ in the proof of  that claim.

    On the other hand, the implicit operation $g$ becomes extendable when restricted to the quasivariety $\mathsf{CCS}$ of cancellative commutative semigroups. 
    In fact, an argument similar to the one detailed in the proof of \cref{Thm : Beth companion : examples}(\ref{item : Beth companion : example : 1}) shows that the pp expansion of $\mathsf{CCS}$ induced by $g$ is the Beth companion of $\mathsf{CCS}$ and, moreover,  is term equivalent to the variety of Abelian groups (inversion is given by $g$ and the neutral element is rendered as the unary operation $x \cdot g(x)$). 
    \qed
\end{Remark}

The second main result of the section gives a sufficient condition for a quasivariety of Heyting algebras to lack a Beth companion. 
In order to state it, 
we first need to recall some 
definitions. The first is the notion of a maximal relatively subdirectly irreducible algebra in a quasivariety (see, e.g., \cite[p.~81]{SurvKissal}).

\begin{Definition}
Let $\K$ be a quasivariety. We say that $\A \in \K_\rsi$ is \emph{maximal} when $\A$ cannot be properly embedded into any $\B \in \K_\rsi$. We denote by $\Mcal(\K_\rsi)$ the class of maximal members of $\K_\rsi$.
\end{Definition}

We will also make use of the ordered sum, an operation that has been extensively used to study implicit definability and surjectivity of epimorphisms in varieties of Heyting algebras (see, e.g., \cite[p.~87]{Mak00} and \cite[p.~9]{MVESHA}).
Intuitively, the  ordered sum of two Heyting algebras $\A$ and $\B$ is the result of pasting $\A$ below $\B$ and gluing the top element $1^\A$ of $\A$ to the bottom element $0^\B$ of $\B$. 

\begin{Definition}
Let $\A$ and $\B$ be a pair of Heyting algebras. Their  \emph{ordered sum}  (also known as  \emph{vertical sum} or just  \emph{sum}) $\A + \B$ is the unique Heyting algebra whose 
universe is the disjoint union of $A - \{ 1^\A\}$ and $B$ and whose lattice order is given by 
\begin{align*}
    c \leq d \iff &\text{ either }(c,d \in A - \{ 1^\A\}\text{ and }c \leq^\A d)\\
& \text{ or }(c,d \in B\text{ and }c \leq^\B d)\\
&\text{ or }(c \in A - \{ 1^\A\}\text{ and }d \in B), 
\end{align*}
where $\leq^\A$ and $\leq^\B$ denote the lattice orders of $\A$ and $\B$, respectively.
\end{Definition}

In the following, for each $n \in \mathbb{Z}^+$ we will denote by $\C_n$ the $n$-element linearly ordered Heyting algebra. 
We are now ready to state the second main result of the section.

\begin{Theorem}\label{Thm : quasivariety HA and C5}
Let $\K$ be a relatively congruence distributive quasivariety of Heyting algebras. If there exists a Heyting algebra $\A$ such that $\A + \C_5 \in \Mcal(\K_\rsi)$, then $\K$ lacks a Beth companion.
\end{Theorem}

To give a better understanding of the applicability of \cref{Thm : quasivariety HA and C5}, we rely on the next characterization of relatively congruence distributive quasivarieties of Heyting algebras, where $\mathsf{HA}$ stands for the variety of Heyting algebras.

\begin{Theorem}
    A quasivariety $\mathsf{K}$ of Heyting algebras is relatively congruence distributive if and only if $\mathsf{K} = \QQQ(\mathsf{M})$ for some universal class $\M$ such that $\mathsf{M} \subseteq \mathsf{HA}_\fsi$.
\end{Theorem}
\begin{proof}
Let $\K$ be a quasivariety of Heyting algebras. Since the variety $\mathsf{HA}$ is congruence distributive  by \cref{Thm : lattice reduct -> CD}, from \cite[Cor. 1.4]{CD90} it follows that $\mathsf{K}$ is relatively congruence distributive if and only if  $\K_\rfsi \subseteq \mathsf{HA}_\fsi$. Therefore, it only remains to prove that $\K_\rfsi \subseteq \mathsf{HA}_\fsi$ if and only if there exists a universal class $\M$ such that $\mathsf{K} = \QQQ(\mathsf{M})$ and $\mathsf{M} \subseteq \mathsf{HA}_\fsi$.

We first establish the implication from left to right. To this end, assume that $\K_\rfsi \subseteq \mathsf{HA}_\fsi$.
Let $\M=\UUU(\mathsf{K}_\rfsi)$. The \cref{Thm : Subdirect Decomposition} yields $\mathsf{K} = \QQQ(\K_\rfsi)$. It follows that $\mathsf{K} = \QQQ(\M)$ because $\K_\rfsi \subseteq \M \subseteq \K$.
It is well known (see, e.g., \cite[Prop.~A.4.3]{Lau19}) that a Heyting algebra is finitely subdirectly irreducible if and only if its greatest element is join irreducible, a property that can be expressed with a universal sentence. Therefore, 
$\mathsf{HA}_\fsi$ is a universal class by \cref{Thm : classes generation}\eqref{item : universal class generation}.
Then
\[
\M = \UUU(\mathsf{K}_\rfsi) \subseteq \UUU(\mathsf{HA}_\fsi) = \mathsf{HA}_\fsi.
\]
Thus, $\mathsf{M}$ has the desired properties. 
For the converse implication, assume that there exists a universal class $\M$ such that $\mathsf{K} = \QQQ(\mathsf{M})$ and $\mathsf{M} \subseteq \mathsf{HA}_\fsi$.\ \cref{Thm : easy Jonsson quasivarieties} implies that $\K_{\rfsi}=\QQQ(\M)_{\rfsi} \subseteq \III\SSS\PPU(\M)$. Then $\K_{\rfsi} \subseteq \UUU(\M)$ by \cref{Thm : quasivariety generation}.
As $\M$ is a universal class contained in $\mathsf{HA}_\fsi$, we obtain 
\[
\mathsf{K}_\rfsi \subseteq \UUU(\mathsf{M}) = \mathsf{M} \subseteq \mathsf{HA}_\fsi,
\]
as desired.
\end{proof}

Before presenting its proof, we first establish a series of consequences of \cref{Thm : quasivariety HA and C5}. As every variety of Heyting algebras is congruence distributive (see \cref{Thm : lattice reduct -> CD}), the following is an immediate corollary of \cref{Thm : quasivariety HA and C5}.

\begin{Corollary}\label{Cor : variety HA and C5}
Let $\K$ be a variety of Heyting algebras. If there exists a Heyting algebra $\A$ with $\A + \C_5 \in \Mcal(\K_\si)$, then $\K$ lacks a Beth companion.
\end{Corollary}

We will use the following consequence of \cref{Thm : quasivariety HA and C5} to show that infinitely many varieties of Heyting algebras lack a Beth companion. 

\begin{Corollary}\label{Cor : no Beth comp fin gen quasivarieties HA}
Let $\K$ be a finite set of finite Heyting algebras such that $\QQQ(\K)$ is relatively congruence distributive. Assume that there exists a Heyting algebra $\A$ such that $\A + \C_5 \in \K$ and $\A + \C_5$ cannot be properly embedded into any member of $\K$.
Then $\QQQ(\K)$ lacks a Beth companion.
\end{Corollary}

\begin{proof}
By \cref{Thm : quasivariety HA and C5} it is sufficient to show that $\A + \C_5 \in \Mcal(\QQQ(\K)_\rsi)$. 
Since $\K$ is a finite set of finite Heyting algebras, $\PPU(\K) \subseteq \III(\K)$ (see \cref{Prop : P_u trivial in finite setting}). Consequently,   from
\cref{Thm : easy Jonsson quasivarieties} it follows
that $\QQQ(\K)_{\rsi} \subseteq \III\SSS(\K)$. Therefore, if $\A + \C_5$  embeds properly into a member of $\QQQ(\K)_{\rsi}$, then it  also embeds properly  into a member of $\K$, but this contradicts our hypothesis. Thus, $\A + \C_5 \in \Mcal(\QQQ(\K)_\rsi)$. 
\end{proof}

We also obtain an analogue of \cref{Cor : no Beth comp fin gen quasivarieties HA} for finitely generated varieties of Heyting algebras.

\begin{Corollary}\label{Cor : lack Beth companion finitely generated varieties of HA}
Let $\K$ be a finite set of finite Heyting algebras. Assume that there exists a Heyting algebra $\A$ such that $\A + \C_5 \in \K$ and one of the following conditions holds:
\benroman
\item\label{Cor : lack Beth companion finitely generated varieties of HA : 1} $\A + \C_5 \in \HHH\SSS(\B)$ implies $\A + \C_5 \cong \B$ for each $\B  \in \K$;
\item\label{Cor : lack Beth companion finitely generated varieties of HA : 2} all members of $\K$ have size at most $\lvert \A + \C_5 \rvert$.
\eroman
Then $\VVV(\K)$ lacks a Beth companion.
\end{Corollary}

\begin{proof}
Suppose that \eqref{Cor : lack Beth companion finitely generated varieties of HA : 1} holds. By \cref{Cor : variety HA and C5} it suffices to show that $\A + \C_5 \in \Mcal(\VVV(\K)_\si)$. Suppose, with a view to contradiction, that there exists $\D \in \VVV(\K)_{\si}$ into which $\A + \C_5$ properly embeds.\
Since $\K$ is a finite set of finite Heyting algebras, $\PPU(\K) \subseteq \III(\K)$ (see \cref{Prop : P_u trivial in finite setting}). Since $\VVV(\K)$ is congruence distributive by \cref{Thm : lattice reduct -> CD},
from 
\cref{Thm : Jonsson} it follows 
that $\VVV(\K)_{\si} \subseteq \HHH\SSS(\K)$. 
 Then $\D \in \HHH\SSS(\K)$, and so there exists $\B \in \K$ such that $\D \in \HHH\SSS(\B)$. Since $\A + \C_5$ embeds into $\D$, we have
$\A + \C_5 \in \III\SSS\HHH\SSS(\B)$. 
For every class $\M$ we have $\SSS\HHH(\M) \subseteq \HHH\SSS(\M)$ (see, e.g., \cite[Lem.~II.9.2]{BuSa00}) and $\III \HHH(\M) = \HHH(\M)$.\ Consequently,  $\III\SSS\HHH\SSS(\B) = \HHH\SSS(\B)$, and hence
$\A + \C_5 \in \HHH\SSS(\B)$. Then \eqref{Cor : lack Beth companion finitely generated varieties of HA : 1} implies  $\A + \C_5 \cong \B$. Recall that $\A + \C_5, \B \in \K$ and all the members of $\K$ are finite. Therefore, since $\A+\C_5$ properly embeds into $\D$ and $\D \in \HHH\SSS(\B)$, we have $\lvert \A+\C_5 \rvert < \lvert \D \rvert \leq \lvert \B \rvert$.
Thus, we reached a contradiction because $\A + \C_5 \cong \B$. 

To conclude the proof, it is then sufficient to show that \eqref{Cor : lack Beth companion finitely generated varieties of HA : 2} implies \eqref{Cor : lack Beth companion finitely generated varieties of HA : 1}. Let $\B \in \K$ be such that $\A + \C_5 \in \HHH\SSS(\B)$.
Then $\lvert \A + \C_5 \rvert \leq \lvert \B \rvert$. So, \eqref{Cor : lack Beth companion finitely generated varieties of HA : 2} yields $\lvert \A + \C_5 \rvert = \lvert \B \rvert$. Since $\A + \C_5 \in \HHH\SSS(\B)$ and $\B$ is finite (the latter because $\B \in \K$ and $\K$ is a class of finite algebras by assumption), we obtain that $\A + \C_5 \cong \B$. Thus, \eqref{Cor : lack Beth companion finitely generated varieties of HA : 1} holds.
\end{proof}

\begin{exa}[\textsf{G\"odel algebras}]
A Heyting algebra is called a \emph{G\"odel algebra} when it satisfies the prelinearity axiom $(x \to y) \vee (y \to x) \thickapprox 1$. From a logical standpoint, the interest of G\"odel algebras is that they algebraize the G\"odel-Dummett logic (see, e.g., \cite{MR1464942,MR1900263}).

The variety $\GA$ of G\"odel algebras is generated by the class of all finite linearly ordered Heyting algebras or by any infinite linearly ordered Heyting algebra (see \cite[Thm.~1.5]{Hor69a}). 
Every proper subvariety of $\GA$ is of the form $\VVV(\C_n)$ for $n \geq 1$ and $\VVV(\C_n) \subseteq \VVV(\C_m)$ if and only if $n \leq m$ (see \cite{DM71,HK72}). 
 The next result governs the existence of a Beth companion for varieties of G\"odel algebras.

\begin{Theorem} \label{Thm : V(Cn) lacks Beth comp}
A variety $\V$ of G\"odel algebras lacks a Beth companion if and only if $\V=\VVV(\C_n)$ for $n \geq 5$. All the remaining varieties of G\"odel algebras are their own Beth companion.
\end{Theorem}

\begin{proof}
From \cite[Thm.~8.1]{Mak00} it follows that the varieties of G\"odel algebras with the strong epimorphism surjectivity property are exactly $\GA$ and $\VVV(\C_n)$ for $n \leq 4$. Hence, these varieties are their own Beth companions by \cref{Thm : Beth companion : examples}\eqref{item : Beth companion : example : 6}.  Let $n \geq 5$.
As $\C_n \cong \C_{n-4} + \C_5$, from \cref{Cor : lack Beth companion finitely generated varieties of HA} it follows that $\VVV(\C_n)$ lacks a Beth companion.
\end{proof}
\end{exa}

In order to prove \cref{Thm : quasivariety HA and C5}, we first establish a series of useful results.

\begin{Proposition} \label{Prop : extendables total on max SI}
Let $\K$ be a quasivariety and $f \in \extpp(\K)$.
Then $f^\A$ is total for every $\A \in  \HHH\PPP(\Mcal(\K_\rsi)) \cap \K$.
\end{Proposition}

\noindent \textit{Proof.} First, assume that $\A \in \Mcal(\K_\rsi)$. 
Since $f \in \extpp(\K)$ and $\A \in \K_\rsi$, by \cref{Thm : extendable 1} there exists $\B \in \K_\rsi$ such that $\A \leq \B$ and $f^{\B}$ is total. The maximality of $\A$ implies that $\A = \B$.
This shows that $f^{\A}$ is total for every $\A \in \Mcal(\K_\rsi)$.
Then let $\M = \{ \A \in \K : f^\A \text{ is total}\}$. We have $\Mcal(\K_\rsi) \subseteq \M$. Our goal is to show that $\HHH\PPP(\Mcal(\K_\rsi)) \cap \K \subseteq \M$. Since $\M \subseteq \K$ and $\K$ is a quasivariety, we have $\PPP(\M) \cap \K = \PPP(\M)$. Therefore, applying \cref{Prop : closure under op} twice with $\mathbb{O} =\PPP$ and $\mathbb{O} = \HHH$, we obtain
\[
\pushQED{\qed}\HHH\PPP(\Mcal(\K_\rsi)) \cap \K \subseteq \HHH\PPP(\M) \cap \K \subseteq \HHH(\PPP(\M) \cap \K) \cap \K \subseteq \HHH(\M) \cap \K \subseteq \M.\qedhere \popQED
\]

\begin{Proposition} \label{Prop : Trivial doms in HPP_U(maxSI)}
    Let $\mathsf{K}$ be a quasivariety with a Beth companion. Moreover, let $\A,\B \in \HHH\PPP(\mathcal{M}(\mathsf{K}_\rsi)) \cap \K$ be such that $\A \leq \B$. Then $\mathsf{d}_{\mathsf{K}}(\A, \B) = A$.
\end{Proposition}

\begin{proof}
We may assume that 
 $\K$ has a Beth companion
of  the form $\SSS(\K[\L_\mathcal{F}])$. \cref{Prop : extendables total on max SI} implies that $f^\A$ and $f^\B$ are total for every $f \in \F$. Therefore, $\A[\L_\F]$ and $\B[\L_\F]$ are defined and are members of $\K[\L_\mathcal{F}]$. It then follows from \cref{Prop : doms pp expansion ext}\eqref{Thm : doms pp expansion ext : 2} that
\[
\mathsf{d}_{\mathsf{K}}(\A, \B) =\mathsf{d}_{\K}(\A[\L_\F]\resLK, \B[\L_\F]\resLK) = A.\qedhere
\]
\end{proof}

\begin{Proposition} \label{Prop : d_K vs d_(K_RSI)}
Let $\K$ be a quasivariety. Then $\d_\K(\A, \B) = \d_{\K_\rsi}(\A, \B)$ for all $\L_\K$-algebras $\A, \B$ such that $\A \leq \B$.
\end{Proposition}

\begin{proof}
From the definition of a dominion it follows immediately that $\d_\K(\A, \B) \subseteq \d_{\K_\rsi}(\A, \B)$.
To prove the other inclusion, let $b \in B$ and assume that $b \notin \d_\K(\A, \B)$. 
Then there exist $\C \in \K$ and homomorphisms $g,h \colon \B \to \C$ such that $g\res_A = h\res_A$ and $g(b) \neq h(b)$. The \cref{Thm : Subdirect Decomposition} implies that there exists $\{\C_i : i \in I\} \subseteq \K_\rsi$ such that $\C \leq \prod_{i \in I} \C_i$ is a subdirect product. For each $i \in I$ let $p_i \colon \C \to \C_i$ be the restriction of the canonical projection. Since $g(b) \neq h(b)$, there exists $i \in I$ such that $(p_i \circ g)(b) \neq (p_i \circ h)(b)$. As $(p_i \circ g)\res_A = p_i \circ (g\res_A) = p_i \circ (h\res_A) = (p_i \circ h)\res_A$, the homomorphisms $p_i \circ g, p_i \circ h \colon \B \to \C_i$ witness that $b \notin \d_{\K_\rsi}(\A, \B)$. Thus, $\d_{\K_\rsi}(\A, \B) \subseteq  \d_\K(\A, \B)$.
\end{proof}

\begin{Proposition} \label{Prop : maps to SI factor through projections}
Let $\K$ be a relatively congruence distributive quasivariety of Heyting algebras. 
Let also $\A_1,\A_2,\B$ be Heyting algebras with $\B \in \K_\rfsi$ and 
$h \colon \A_1 \times \A_2 \to \B$ a homomorphism.
Then there exist $i \in \{1,2\}$ and a homomorphism $g \colon \A_i \to \B$ such that $h = g \circ \pi_i$, where $\pi_i \colon \A_1 \times \A_2 \to \A_i$ is the canonical projection map.
\end{Proposition}

\begin{proof}
Since $\K$ is relatively congruence distributive and $\B \in \K_\rfsi$, from \cite[Cor.~1.4]{CD90} it follows that $\B$ is a finitely subdirectly irreducible Heyting algebra. Therefore, the greatest element $1^{\B}$ of $\B$ is join irreducible (see, e.g., \cite[Prop.~A.4.3]{Lau19}).
Let $0^{\A_i}$ and $1^{\A_i}$ be the least and greatest elements of $\A_i$ for $i=1,2$. Then 
\[
1^{\B} = h(\langle 1^{\A_1},1^{\A_2} \rangle) = h(\langle 0^{\A_1},1^{\A_2} \rangle \vee \langle 1^{\A_1},0^{\A_2} \rangle) = h(\langle 0^{\A_1},1^{\A_2} \rangle) \vee h(\langle 1^{\A_1},0^{\A_2} \rangle).  
\]
Since $1^\B$ is join irreducible, we have
\[
h(\langle 0^{\A_1},1^{\A_2} \rangle) = 1^{\B} \quad \text{or} \quad h(\langle 1^{\A_1},0^{\A_2} \rangle) = 1^{\B}.
\]

By symmetry we may assume that $h(\langle 0^{\A_1},1^{\A_2} \rangle) = 1^{\B}$.
We will prove that 
\begin{equation}\label{Eq : the map h is a projection on some coordinate : HAs}
h(\langle a, c\rangle) = h(\langle b, c\rangle)\text{ for all }a, b \in A_1\text{ and }c \in A_2.
\end{equation}
To this end, observe that 
\begin{align*}
h(\langle a,c \rangle) \to h(\langle b,c \rangle) = h(\langle a \to b,c \to c \rangle) = h(\langle a \to b,1^{\A_2} \rangle) \geq h(\langle 0^{\A_1},1^{\A_2} \rangle) = 1^{\B},
\end{align*}
and, therefore, $h(\langle a,c \rangle) \leq h(\langle b,c \rangle)$. An analogous  argument shows that $h(\langle b,c \rangle) \leq h(\langle a,c \rangle)$, whence $h(\langle a, c\rangle) = h(\langle b, c\rangle)$, as desired.

Lastly, from \eqref{Eq : the map h is a projection on some coordinate : HAs} it follows that $\mathsf{ker}(\pi_2) \subseteq \mathsf{ker}(h)$.
As a straightforward consequence of \cref{Prop : homomorphisms : smaller congruences} (see, e.g., \cite[Ex.~1.26.8]{Ber11}), we
obtain
a homomorphism $g \colon \A_2 \to \B$ such that 
$h = g \circ \pi_2$.
\end{proof}

We are now ready to prove \cref{Thm : quasivariety HA and C5}.

\begin{proof}
To simplify the notation, we let $\B = \A +\C_5$.
\begin{Claim}\label{claim : existence D nontrivial dominion}
There exists $\D \in \HHH(\B)$ such that $\D \leq \B \times \B$ and $\d_\K(\D,\B \times \B) \neq D$.
\end{Claim}
\begin{proof}[Proof of the Claim]
Let the elements of $\C_5$ be $0^{\C_5}= c_1 < c_2 < c_3 < c_4 < c_5 = 1^{\C_5}$. Recall that $B=(A - \{1^{\A}\}) \cup C_5$. We define
\[
D = \{ \langle a,a \rangle : a \in A-\{1^{\A}\}\} \cup \{\langle c_1,c_1 \rangle, \langle c_2,c_3 \rangle, \langle c_4,c_4 \rangle, \langle c_5,c_5 \rangle\} \subseteq B \times B. 
\]
As $\B = \A +\C_5$ and $\C_5$ is linearly ordered, the implication $\to^\B$ of $\B$ can be described in terms of the implication $\to^{\A}$ of $\A$ as follows. For all $a,b \in B$ we have
\[
a \to^\B b = 
\begin{cases}
c_5 & \text{if } a \leq b;\\
a \to^{\A} b & \text{if } a \nleq b \text{ and } a,b \in A - \{1^{\A}\};\\
b & \text{otherwise.}\\
\end{cases}
\]
It is then immediate to verify that $\D \leq \B \times \B$.
Moreover, a straightforward verification yields  that the map  $k \colon \B \to \B \times \B$ defined as follows is a homomorphism whose image is $D$: for every $b \in B$,
\[
k(b) = 
\begin{cases}
\langle b, b \rangle & \text{if } b \in \{ c_1, c_5 \} \cup (A - \{1^{\A}\});\\
\langle c_2, c_3 \rangle & \text{if }b = c_2;\\
\langle c_4, c_4 \rangle & \text{if }b = c_3;\\
\langle c_5, c_5 \rangle & \text{if }b = c_4.
\end{cases}
\]
Thus, $\D \in \HHH(\B)$. 

Therefore, it only remains to show that $\d_\K(\D,\B \times \B) \neq D$.
By \cref{Prop : d_K vs d_(K_RSI)} we have  
\[
\d_\K(\D, \B \times \B) = \d_{\K_\rsi}(\D, \B \times \B).
\]
Since $\langle c_5, c_4 \rangle \in (B \times B) -  D$, it suffices to show  $\langle c_5, c_4 \rangle \in \d_{\K_\rsi}(\D, \B \times \B)$. 
Let $g,h \colon \B \times \B \to \E$ be a pair of homomorphisms such that $\E \in \K_{\rsi}$ and $g \res_D = h \res_D$. We need to prove that $g(\langle c_5, c_4 \rangle) = h(\langle c_5, c_4 \rangle)$.
Since $\E \in \K_\rsi \subseteq \K_{\rfsi}$, \cref{Prop : maps to SI factor through projections} yields that both $g$ and $h$ factor through a projection. 
We have two cases: either $g$ and $h$ factor through the same projection or not. 

First, suppose that $g$ and $h$ factor through $\pi_1$. Then there exists a pair of homomorphisms $g', h' \colon \B \to \E$ such that $g = g' \circ \pi_1$ and $h = h' \circ \pi_1$. 
Therefore,
\[
g(\langle c_5, c_4 \rangle) = g'(c_5) = 1^\E = h'(c_5) = h(\langle c_5, c_4 \rangle),
\]
where the second and third equalities hold because $c_5$ is the greatest element of $\B$, and the others because $g = g' \circ \pi_1$ and $h = h' \circ \pi_1$.

Next we consider the case where both $g$ and $h$ factor through $\pi_2$.\ Then there exists a pair of homomorphisms $g', h' \colon \B \to \E$ such that $g = g' \circ \pi_2$ and $h = h' \circ \pi_2$. Therefore,
\[
g(\langle c_5, c_4 \rangle) = g'(c_4) = g(\langle c_4, c_4 \rangle) = h(\langle c_4, c_4 \rangle)  = h'(c_4) = h(\langle c_5, c_4 \rangle),
\]
where the middle equality holds because $\langle c_4, c_4 \rangle \in D$ and $g\res_D = h\res_D$, and the others because $g = g' \circ \pi_2$ and $h = h' \circ \pi_2$.

Lastly, suppose that $g$ and $h$ factor through different projections. Without loss of generality, we may assume that $g$ factors through $\pi_1$ and $h$ factors through $\pi_2$.\ 
Then there exists a pair of homomorphisms $g', h' \colon \B \to \E$ such that $g = g' \circ \pi_1$ and $h = h' \circ \pi_2$. 
Since $g\res_D = h\res_D$ and $\langle c_2,c_3 \rangle,\langle c_4,c_4 \rangle \in D$, we obtain
\begin{align} 
g'(c_2) &= g(\langle c_2,c_3 \rangle) = h(\langle c_2,c_3 \rangle) = h'(c_3); \label{Eq : equality cond assymetry} \\ 
g'(c_4) &= g(\langle c_4,c_4 \rangle) = h(\langle c_4,c_4 \rangle) = h'(c_4). \label{Eq : equality cond a_2}
\end{align}
Assume, with a view to contradiction, that both $g'$ and $h'$ are injective. As $\B \in \Mcal(\K_\rsi)$ by assumption and $\E \in \K_\rsi$, it follows that $g'$ and $h'$ are isomorphisms. 
Since $c_3$ is the only element $a$ of $\B$ such that $\{ b \in B : a \leq b\}$ has size $3$, every automorphism of $\B$ must fix $c_3$. It follows that $g'(c_3)=h'(c_3)$ because $(h')^{-1} \circ g'$ is an automorphism of $\B$. Then \eqref{Eq : equality cond assymetry} yields $g'(c_2)= h'(c_3) = g'(c_3)$, which is impossible because $g'$ is injective. 
Therefore, either $g'$ or $h'$ is not injective.
 Suppose first that $g'$ is not injective. Then there exist $a,b \in B$ such that $a \nleq b$ and $g'(a)=g'(b)$. As $a \nleq b$ and $g'$ is a homomorphism, we obtain $a \to^\B b \neq 1^\B$ and $g'(a \to^\B b)=g'(a) \to^\E g'(b)=1^\E$. Note that every Heyting algebra homomorphism is order preserving because it is a lattice homomorphism. Therefore, since $c_4$ is the second greatest element of $\B$, we have $a \to^\B b \leq c_4$ and, consequently, $1^\E=g'(a \to^\B b) \leq g'(c_4)$. So, $g'(c_4)=1^\E$. Then we have
\[
g(\langle c_5, c_4 \rangle) = g'(c_5) = 1^\E = g'(c_4)  = h'(c_4) = h(\langle c_5, c_4 \rangle),
\]
where first and last equalities hold because $g = g' \circ \pi_1$ and $h = h' \circ \pi_2$, the second because $c_5$ is the greatest element of $\B$, the third because $g'(c_4)=1^\E$ as we just observed, and the fourth follows from \eqref{Eq : equality cond a_2}.
Next, suppose that $h'$ is not injective. An  argument similar to the one above shows that $h'(c_4)=1^\E$. Then
\[
h(\langle c_5, c_4 \rangle) = h'(c_4) = 1^{\E} = g'(c_5) = g(\langle c_5, c_4 \rangle),
\]
where first and last equalities hold because $g = g' \circ \pi_1$ and $h = h' \circ \pi_2$, the second because $h'(c_4)=1^\E$ as we just observed, and the third because $c_5$ is the greatest element of $\B$. 
We conclude that $g(\langle c_5, c_4 \rangle) = h(\langle c_5, c_4 \rangle)$ in all possible cases. Thus, $\langle c_5, c_4 \rangle \in \d_{\K_\rsi}(\D, \B \times \B)$, as desired.

\end{proof}

Let $\D$ be as in \cref{claim : existence D nontrivial dominion}. As $\B \in \K$, $\D \leq \B \times \B$, and $\K$ is a quasivariety, we obtain $\D, \B \times \B \in \K$. Since $\B \in \Mcal(\K_\rsi)$ and
$\D \in \HHH(\B)$, it follows that $\D, \B \times \B \in \HHH\PPP(\Mcal(\K_\rsi)) \cap \K$. If $\K$ had a Beth companion, then \cref{Prop : Trivial doms in HPP_U(maxSI)} would imply that $\d_\K(\D,\B \times \B) = D$, contradicting \cref{claim : existence D nontrivial dominion}. Thus, $\K$ lacks a Beth companion.
\end{proof}

\

\subsection*{Acknowledgments} We are grateful to Rodrigo Nicolau Almeida and Simon Santschi for comments on an earlier draft of this manuscript. We also thank Miguel Campercholi, Diego Castaño, Ivan Di Liberti, and Luca Reggio for many interesting conversations, and Marino Gran for drawing our attention to \cite{MR1268510}.


\begin{thebibliography}{100}

\bibitem{AHS06}
J.~Ad{\'a}mek, H.~Herrlich, and G.~E. Strecker.
\newblock Abstract and concrete categories: the joy of cats.
\newblock {\em Repr. Theory Appl. Categ.}, (17):1--507, 2006.

\bibitem{EpiLG}
M.~Anderson and P.~Conrad.
\newblock Epicomplete {$l$}-groups.
\newblock {\em Algebra Universalis}, 12(2):224--241, 1981.

\bibitem{And94}
G.~E. Andrews.
\newblock {\em Number theory}.
\newblock Dover Publications, Inc., New York, 1994.

\bibitem{Art10}
M.~Artin.
\newblock {\em Algebra}.
\newblock Pearson, second edition, 2010.

\bibitem{Bacs_epi}
P.~D. Bacsich.
\newblock An epi-reflector for universal theories.
\newblock {\em Canad. Math. Bull.}, 16:167--171, 1973.

\bibitem{Bacsich47}
P.~D. Bacsich.
\newblock Model theory of epimorphisms.
\newblock {\em Canad. Math. Bull.}, 17:471--477, 1974.

\bibitem{BD74}
R.~Balbes and P.~Dwinger.
\newblock {\em Distributive lattices}.
\newblock University of Missouri Press, Columbia, Mo., 1974.

\bibitem{BH89}
R.~N. Ball and A.~W. Hager.
\newblock Characterization of epimorphisms in {A}rchimedean lattice-ordered groups and vector lattices.
\newblock In {\em Lattice-ordered groups}, volume~48 of {\em Math. Appl.}, pages 175--205. Kluwer Acad. Publ., Dordrecht, 1989.

\bibitem{BF85}
J.~Barwise and S.~Feferman, editors.
\newblock {\em Model‑Theoretic Logics}.
\newblock Number~8 in Perspectives in Logic. Springer Verlag, 1985.

\bibitem{Ber87}
C.~Bergman.
\newblock Saturated algebras in filtral varieties.
\newblock {\em Algebra Universalis}, 24(1-2):101--110, 1987.

\bibitem{Ber11}
C.~Bergman.
\newblock {\em Universal Algebra: Fundamentals and Selected Topics}.
\newblock Chapman \& Hall Pure and Applied Mathematics. Chapman and Hall/CRC, 2011.

\bibitem{MedBHT}
J.~A. Bergstra, Y.~Hirshfeld, and J.~V. Tucker.
\newblock Meadows and the equational specification of division.
\newblock {\em Theoret. Comput. Sci.}, 410(12-13):1261--1271, 2009.

\bibitem{BT07}
J.~A. Bergstra and J.~V. Tucker.
\newblock The rational numbers as an abstract data type.
\newblock {\em J. ACM}, 54(2), 2007.

\bibitem{BMRES}
G.~Bezhanishvili, T.~Moraschini, and J.~G. Raftery.
\newblock Epimorphisms in varieties of residuated structures.
\newblock {\em J. Algebra}, 492:185--211, 2017.

\bibitem{MR6550}
G.~Birkhoff.
\newblock Lattice, ordered groups.
\newblock {\em Ann. of Math.}, 43(2):298--331, 1942.

\bibitem{BirkLT}
G.~Birkhoff.
\newblock {\em Lattice {T}heory}.
\newblock American Mathematical Society Colloquium Publications, Vol. 25. American Mathematical Society, New York, 1948.

\bibitem{BlHoo06}
W.~J. Blok and E.~Hoogland.
\newblock {The Beth property in Algebraic Logic}.
\newblock 83(1--3):49--90, 2006.

\bibitem{BlokKohlerPigozzi84}
W.~J. Blok, P.~K\"ohler, and D.~Pigozzi.
\newblock On the structure of varieties with equationally definable principal congruences. {II}.
\newblock {\em Algebra Universalis}, 18(3):334--379, 1984.

\bibitem{BP89}
W.~J. Blok and D.~Pigozzi.
\newblock {\em Algebraizable logics}, volume 396 of {\em Mem. Amer. Math. Soc.}
\newblock A.M.S., Providence, January 1989.

\bibitem{Burris1992symbolic}
S.~Burris.
\newblock Discriminator varieties and symbolic computation.
\newblock {\em J. Symbolic Comput.}, 13(2):175--207, 1992.

\bibitem{BuSa00}
S.~Burris and H.~P. Sankappanavar.
\newblock {\em A Course in Universal Algebra}.
\newblock 2012.
\newblock The millennium edition, available online.

\bibitem{BurrisWener79}
S.~Burris and H.~Werner.
\newblock Sheaf constructions and their elementary properties.
\newblock {\em Trans. Amer. Math. Soc.}, 248(2):269--309, 1979.

\bibitem{Camper18jsl}
M.~A. Campercholi.
\newblock Dominions and primitive positive functions.
\newblock {\em Journal of Symbolic Logic}, 83(1):40--54, 2018.

\bibitem{CampRaf}
M.~A. Campercholi and J.~G. Raftery.
\newblock Relative congruence formulas and decompositions in quasivarieties.
\newblock {\em Algebra Universalis}, 78(3):407--425, 2017.

\bibitem{CampVaggDisc}
M.~A. Campercholi and D.~J. Vaggione.
\newblock Implicit definition of the quaternary discriminator.
\newblock {\em Algebra Universalis}, 68(1-2):1--16, 2012.

\bibitem{CampVaggSemCon}
M.~A. Campercholi and D.~J. Vaggione.
\newblock Semantical conditions for the definability of functions and relations.
\newblock {\em Algebra Universalis}, 76:71--98, 2015.

\bibitem{CKMEXAv2}
L.~Carai, M.~Kurtzhals, and T.~Moraschini.
\newblock An addendum to ``{T}he theory of implicit operations''.
\newblock Available at \url{https://arxiv.org/pdf/2601.00006v2}, 2025.

\bibitem{CKMWES}
L.~Carai, M.~Kurtzhals, and T.~Moraschini.
\newblock Epimorphisms between finitely generated algebras.
\newblock {\em Indag. Math. (N.S.)}, 36(5):1336--1354, 2025.

\bibitem{BethCat26}
L.~Carai, M.~Kurtzhals, and T.~Moraschini.
\newblock A categorical description of simple {B}eth companions.
\newblock Available at \url{https://arxiv.org/pdf/2605.09141}, 2026.

\bibitem{CKM25RCR}
L.~Carai, M.~Kurtzhals, and T.~Moraschini.
\newblock A completion of reduced commutative rings.
\newblock Submitted, available at \url{https://arxiv.org/pdf/2605.12661}, 2026.

\bibitem{CKMMON}
L.~Carai, M.~Kurtzhals, and T.~Moraschini.
\newblock Implicit operations in varieties of commutative monoids.
\newblock Submitted, available at \url{https://arxiv.org/pdf/2603.13916}, 2026.

\bibitem{MR1268510}
A.~Carboni, G.~M. Kelly, and M.~C. Pedicchio.
\newblock Some remarks on {M}al'tsev and {G}oursat categories.
\newblock {\em Appl. Categ. Structures}, 1(4):385--421, 1993.

\bibitem{CCV25}
D.~Casta{\~n}o, M.~Campercholi, and D.~Vaggione.
\newblock Preservation theorems for {AE}-sentences.
\newblock {\em Journal of Symbolic Logic}, pages 1--17, 2025.

\bibitem{MR1464942}
A.~Chagrov and M.~Zakharyaschev.
\newblock {\em Modal logic}, volume~35 of {\em Oxford Logic Guides}.
\newblock The Clarendon Press, Oxford University Press, New York, 1997.
\newblock Oxford Science Publications.

\bibitem{Cha59}
C.~C. Chang.
\newblock A new proof of the completeness of the {\Luk}ukasiewicz axioms.
\newblock {\em Trans. Amer. Math. Soc.}, 93:74--80, 1959.

\bibitem{ModCK}
C.~C. Chang and H.~J. Keisler.
\newblock {\em Model theory}, volume~73 of {\em Studies in Logic and the Foundations of Mathematics}.
\newblock North-Holland Publishing Co., Amsterdam, third edition, 1990.

\bibitem{COM00}
R.~L.~O. Cignoli, I.~M.~L. D'Ottaviano, and D.~Mundici.
\newblock {\em Algebraic foundations of many-valued reasoning}, volume~7 of {\em Trends in Logic---Studia Logica Library}.
\newblock Kluwer Academic Publishers, Dordrecht, 2000.

\bibitem{MR1408738}
M.~\'{C}iri\'{c} and S.~Bogdanovi\'{c}.
\newblock Posets of {$\mathfrak C$}-congruences.
\newblock {\em Algebra Universalis}, 36(3):423--424, 1996.

\bibitem{Cla83}
G.~T. Clarke.
\newblock Semigroup varieties with the amalgamation property.
\newblock {\em J. Algebra}, 80(1):60--72, 1983.

\bibitem{MR21847244}
A.~H. Clifford and G.~B. Preston.
\newblock {\em The algebraic theory of semigroups. {V}ol. {I}}.
\newblock Mathematical Surveys, No. 7. American Mathematical Society, Providence, RI, 1961.

\bibitem{MR218472}
A.~H. Clifford and G.~B. Preston.
\newblock {\em The algebraic theory of semigroups. {V}ol. {II}}.
\newblock Mathematical Surveys, No. 7. American Mathematical Society, Providence, RI, 1967.

\bibitem{CD90}
L.~Czelakowski and W.~Dziobiak.
\newblock Congruence distributive quasivarieties whose finitely subdirectly irreducible members form a universal class.
\newblock {\em Algebra Universalis}, 27(1):128--149, 1990.

\bibitem{MR724265}
B.~A. Davey and H.~Werner.
\newblock Dualities and equivalences for varieties of algebras.
\newblock In {\em Contributions to lattice theory ({S}zeged, 1980)}, volume~33 of {\em Colloq. Math. Soc. J\'anos Bolyai}, pages 101--275. North-Holland, Amsterdam, 1983.

\bibitem{Dek20}
P.~M. Dekker.
\newblock {\em Interpolation and Beth definability in implicative fragments of IPC}.
\newblock PhD thesis, University of Amsterdam, 2020.

\bibitem{MVDNLet}
A.~Di~Nola and A.~Lettieri.
\newblock One chain generated varieties of {MV}-algebras.
\newblock {\em J. Algebra}, 225(2):667--697, 2000.

\bibitem{DHilbA}
A.~Diego.
\newblock {\em Sobre {\'a}lgebras de Hilbert}.
\newblock PhD thesis, Universidad de Buenos Aires, 1961.
\newblock Available at \url{http://digital.bl.fcen.uba.ar/Download/Tesis/Tesis_1092_Diego.pdf}.

\bibitem{DF04}
D.~S. Dummit and R.~M. Foote.
\newblock {\em Abstract algebra}.
\newblock John Wiley \& Sons, Inc., Hoboken, NJ, third edition, 2004.

\bibitem{DM71}
J.~M. Dunn and R.~K. Meyer.
\newblock Algebraic completeness results for {D}ummett's {${\rm LC}$} and its extensions.
\newblock {\em Z. Math. Logik Grundlagen Math.}, 17:225--230, 1971.

\bibitem{epiDvZa}
A.~Dvure\v{c}enskij and O.~Zahiri.
\newblock On epicomplete {$MV$}-algebras.
\newblock {\em J. Appl. Logics}, 5(1):165--183, 2018.

\bibitem{MR1322960}
D.~Eisenbud.
\newblock {\em Commutative algebra}, volume 150 of {\em Graduate Texts in Mathematics}.
\newblock Springer-Verlag, New York, 1995.
\newblock With a view toward algebraic geometry.

\bibitem{MR57230}
A.~L. Foster.
\newblock Generalized ``{B}oolean'' theory of universal algebras. {I}. {S}ubdirect sums and normal representation theorem.
\newblock {\em Math. Z.}, 58:306--336, 1953.

\bibitem{MR57231}
A.~L. Foster.
\newblock Generalized ``{B}oolean'' theory of universal algebras. {II}. {I}dentities and subdirect sums of functionally complete algebras.
\newblock {\em Math. Z.}, 59:191--199, 1953.

\bibitem{FriedGraetzerQuackenbusgìh80}
E.~Fried, G.~Gr\"atzer, and R.~Quackenbush.
\newblock Uniform congruence schemes.
\newblock {\em Algebra Universalis}, 10(2):176--188, 1980.

\bibitem{FriedKiss83}
E.~Fried and E.~W. Kiss.
\newblock Connections between congruence-lattices and polynomial properties.
\newblock {\em Algebra Universalis}, 17(3):227--262, 1983.

\bibitem{Ger01thesis}
B.~Gerla.
\newblock {\em Many-Valued Logics of Continuous t-Norms and Their Functional Representation}.
\newblock Ph.{D}. thesis, University of Milan, 2001.
\newblock Available at: \url{https://www.dicom.uninsubria.it/\string~bgerla/tesi.pdf}.

\bibitem{Ger01}
B.~Gerla.
\newblock Rational {\Luk}ukasiewicz logic and divisible {MV}-algebras.
\newblock {\em Neural Networks World}, 10, 2001.

\bibitem{GM05}
J.~Gispert and D.~Mundici.
\newblock M{V}-algebras: a variety for magnitudes with {A}rchimedean units.
\newblock {\em Algebra Universalis}, 53(1):7--43, 2005.

\bibitem{GT98}
J.~Gispert and A.~Torrens.
\newblock Quasivarieties generated by simple {MV}-algebras.
\newblock volume~61, pages 79--99. 1998.
\newblock Many-valued logics.

\bibitem{MVGiTo}
J.~Gispert and A.~Torrens.
\newblock Quasivarieties generated by simple {MV}-algebras.
\newblock {\em Studia Logica}, 61(1):79--99, 1998.

\bibitem{Go98a}
V.~A. Gorbunov.
\newblock {\em Algebraic theory of quasivarieties}.
\newblock Siberian School of Algebra and Logic. Consultants Bureau, New York, 1998.
\newblock Translated from the Russian.

\bibitem{Gra08}
G.~Gr\"{a}tzer.
\newblock {\em Universal algebra}.
\newblock Springer, New York, second edition, 2008.

\bibitem{GLAP}
G.~Gr\"{a}tzer and H.~Lakser.
\newblock The structure of pseudocomplemented distributive lattices. {II}. {C}ongruence extension and amalgamation.
\newblock {\em Trans. Amer. Math. Soc.}, 156:343--358, 1971.

\bibitem{MR155753}
Y.~Gurevich and A.~I. Kokorin.
\newblock Universal equivalence of ordered {A}belian groups.
\newblock {\em Algebra i Logika Sem.}, 2(1):37--39, 1963.

\bibitem{MR1900263}
Petr H\'ajek.
\newblock {\em Metamathematics of fuzzy logic}, volume~4 of {\em Trends in Logic---Studia Logica Library}.
\newblock Kluwer Academic Publishers, Dordrecht, 1998.

\bibitem{HK72}
T.~Hecht and T.~Katri\v{n}\'{a}k.
\newblock Equational classes of relative {S}tone algebras.
\newblock {\em Notre Dame J. Formal Log.}, 13:248--254, 1972.

\bibitem{HodModTh}
W.~Hodges.
\newblock {\em Model Theory}.
\newblock Cambridge University Press, 1993.

\bibitem{BethHoog}
E.~Hoogland.
\newblock Algebraic characterizations of various {B}eth definability properties.
\newblock {\em Studia Logica}, 65(1):91--112, 2000.

\bibitem{Hoo01}
E.~Hoogland.
\newblock {\em Definability and Interpolation: Model‑theoretic Investigations}.
\newblock Ph.{D}. thesis, University of Amsterdam, 2001.

\bibitem{Hor62}
A.~Horn.
\newblock The separation theorem of intuitionist propositional calculus.
\newblock {\em J. Symbolic Logic}, 27:391--399, 1962.

\bibitem{Hor69a}
A.~Horn.
\newblock Logic with truth values in a linearly ordered {H}eyting algebra.
\newblock {\em J. Symb. Logic}, 34:395--408, 1969.

\bibitem{HowZigzag}
J.~M. Howie.
\newblock Isbell's zigzag theorem and its consequences.
\newblock In {\em Semigroup theory and its applications ({N}ew {O}rleans, {LA}, 1994)}, volume 231 of {\em London Math. Soc. Lecture Note Ser.}, pages 81--91. Cambridge Univ. Press, Cambridge, 1996.

\bibitem{HoIsEpiII}
J.~M. Howie and J.~R. Isbell.
\newblock Epimorphisms and dominions. {II}.
\newblock {\em J. Algebra}, 6:7--21, 1967.

\bibitem{MR244130}
T.~Hu.
\newblock Stone duality for primal algebra theory.
\newblock {\em Math. Z.}, 110:180--198, 1969.

\bibitem{Isb65}
J.~R. Isbell.
\newblock Epimorphisms and dominions.
\newblock In {\em Proc. {C}onf. {C}ategorical {A}lgebra ({L}a {J}olla, {C}alif., 1965)}, pages 232--246. Springer-Verlag New York, Inc., New York, 1966.

\bibitem{IsbEpiIV}
J.~R. Isbell.
\newblock Epimorphisms and dominions. {IV}.
\newblock {\em J. London Math. Soc. (2)}, 1:265--273, 1969.

\bibitem{KP01}
K.~Kaarli and A.~F. Pixley.
\newblock {\em Polynomial completeness in algebraic systems}.
\newblock Chapman \& Hall/CRC, Boca Raton, FL, 2001.

\bibitem{KearnesKiss12}
K.~A. Kearnes and E.~W. Kiss.
\newblock The shape of congruence lattices.
\newblock {\em Mem. Amer. Math. Soc.}, 222(1046), 2013.

\bibitem{Kim57}
N.~Kimura.
\newblock {\em On Semigroups}.
\newblock Ph.{D}. thesis, Tulane University of Louisiana, 1957.

\bibitem{SurvKissal}
E.~W. Kiss, L.~M\'{a}rki, P.~Pr\"{o}hle, and W.~Tholen.
\newblock Categorical algebraic properties. {A} compendium on amalgamation, congruence extension, epimorphisms, residual smallness, and injectivity.
\newblock {\em Studia Sci. Math. Hungar.}, 18(1):79--140, 1982.

\bibitem{Koh81}
P.~K\"ohler.
\newblock Brouwerian semilattices.
\newblock {\em Trans. Amer. Math. Soc.}, 268(1):103--126, 1981.

\bibitem{KohlerPigozzi80}
P.~K\"ohler and D.~Pigozzi.
\newblock Varieties with equationally definable principal congruences.
\newblock {\em Algebra Universalis}, 11(2):213--219, 1980.

\bibitem{Kom75}
Y.~Komori.
\newblock The finite model property of the intermediate propositional logics on finite slices.
\newblock {\em J. Fac. Sci. Univ. Tokyo Sect. IA Math.}, 22(2):117--120, 1975.

\bibitem{MR1369091}
V.~M. Kopytov and N.~Ya. Medvedev.
\newblock {\em The theory of lattice-ordered groups}, volume 307 of {\em Mathematics and its Applications}.
\newblock Kluwer Academic Publishers Group, Dordrecht, 1994.

\bibitem{Kreisel60JSL}
G.~Kreisel.
\newblock Explicit definability in intuitionistic logic.
\newblock {\em J. Symb. Log.}, 25:389--390, 1960.

\bibitem{Kuz74}
A.~V. Kuznetsov.
\newblock Some classification problems for superintuitionistic logics.
\newblock In {\em Proceedings of the 3rd USSR Conference on Mathematical Logic}, pages 119--122, Novosibirsk, 1974.
\newblock (in Russian).

\bibitem{LakPDL}
H.~Lakser.
\newblock The structure of pseudocomplemented distributive lattices. {I}. {S}ubdirect decomposition.
\newblock {\em Trans. Amer. Math. Soc.}, 156:335--342, 1971.

\bibitem{Lan84}
S.~Lang.
\newblock {\em Algebra}.
\newblock Addison-Wesley Publishing Company, Advanced Book Program, Reading, MA, second edition, 1984.

\bibitem{Lau19}
F.~M. Lauridsen.
\newblock {\em Cuts and Completions: Algebraic aspects of structural proof theory}.
\newblock Ph.{D}. thesis, University of Amsterdam, 2019.
\newblock available at https://eprints.illc.uva.nl/id/eprint/2169.

\bibitem{FLField}
F.~Lorenz.
\newblock {\em Algebra. Volume I: Fields and Galois Theory}.
\newblock Springer New York, 2006.

\bibitem{LynHom}
R.~C. Lyndon.
\newblock Properties preserved under homomorphism.
\newblock {\em Pacific J. Math.}, 9:143--154, 1959.

\bibitem{Magari1969}
R.~Magari.
\newblock Variet\'a a quozienti filtrali.
\newblock {\em Ann. Univ. Ferrara}, Sez. VII:5--20, 1969.

\bibitem{Mak77}
L.~L. Maksimova.
\newblock Craig's theorem in superintuitionistic logics and amalgamable varieties.
\newblock {\em Algebra i Logika}, 16(6):643--681, 741, 1977.

\bibitem{MakpBeth}
L.~L. Maksimova.
\newblock Projective {B}eth properties in modal and superintuitionistic logics.
\newblock {\em Algebra Log.}, 38(3):316--333, 379, 1999.

\bibitem{Mak99}
L.~L. Maksimova.
\newblock Superintuitionistic logics and the projective {B}eth property.
\newblock {\em Algebra Log.}, 38(6):680--696, 769, 1999.

\bibitem{Mak00}
L.~L. Maksimova.
\newblock Intuitionistic logic and implicit definability.
\newblock {\em Ann. Pure Appl. Logic}, 105(1-3):83--102, 2000.

\bibitem{Mat83}
E.~Matlis.
\newblock The minimal prime spectrum of a reduced ring.
\newblock {\em Illinois J. Math.}, 27(3):353--391, 1983.

\bibitem{Mit78}
A.~Mitschke.
\newblock Near unanimity identities and congruence distributivity in equational classes.
\newblock {\em Algebra Universalis}, 8(1):29--32, 1978.

\bibitem{MRWepi}
T.~Moraschini, J.~G. Raftery, and J.~J. Wannenburg.
\newblock Epimorphisms, definability and cardinalities.
\newblock {\em Studia Logica}, 108(2):255--275, 2020.

\bibitem{MVESHA}
T.~Moraschini and J.~J. Wannenburg.
\newblock Epimorphism surjectivity in varieties of {H}eyting algebras.
\newblock {\em Ann. Pure Appl. Logic}, 171(9):102824, 31, 2020.

\bibitem{Mun88}
D.~Mundici.
\newblock Free products in the category of abelian {$l$}-groups with strong unit.
\newblock {\em J. Algebra}, 113(1):89--109, 1988.

\bibitem{Los55}
J.~{\Luk}o{\'s}.
\newblock On the extending of models. {I}.
\newblock {\em Fund. Math.}, 42:38--54, 1955.

\bibitem{MR975871}
K.~Pa{\l}asi\'{n}ska.
\newblock The failure of strong amalgamation property in certain classes of {BCK}-algebras.
\newblock {\em Math. Japon.}, 33(6):913--917, 1988.

\bibitem{MR574094}
F.~D. Pedersen.
\newblock Epimorphisms in the category of {A}belian {$l$}-groups.
\newblock {\em Proc. Amer. Math. Soc.}, 53(2):311--317, 1975.

\bibitem{APLG}
K.~R. Pierce.
\newblock Amalgamations of lattice ordered groups.
\newblock {\em Trans. Amer. Math. Soc.}, 172:249--260, 1972.

\bibitem{Pixley72}
A.~F. Pixley.
\newblock Local {M}alcev conditions.
\newblock {\em Canad. Math. Bull.}, 15:559--568, 1972.

\bibitem{MR344067}
H.~Rasiowa and R.~Sikorski.
\newblock {\em The mathematics of metamathematics}.
\newblock Monografie Matematyczne [Mathematical Monographs], Tom 41. PWN---Polish Scientific Publishers, Warsaw, third edition, 1970.

\bibitem{SM03}
E.~S\"{u}li and D.~F. Mayers.
\newblock {\em An introduction to numerical analysis}.
\newblock Cambridge University Press, Cambridge, 2003.

\bibitem{Tar55}
A.~Tarski.
\newblock Contributions to the theory of models. {III}.
\newblock {\em Indag. Math. (Proceedings)}, 58:56--64, 1955.

\bibitem{MR1888129}
D.~E. Tishkovski\u{i}.
\newblock On the {B}eth property in extensions of \l ukasiewicz logics.
\newblock {\em Sibirsk. Mat. Zh.}, 43(1):183--187, iv, 2002.

\bibitem{Tom54}
H.~Tominaga.
\newblock Some remarks on radical ideals.
\newblock {\em Math. J. Okayama Univ.}, 3:139--142, 1954.

\bibitem{Vaggione2001}
D.~Vaggione.
\newblock Characterization of discriminator varieties.
\newblock {\em Proc. Amer. Math. Soc.}, 129(3):663--666, 2001.

\bibitem{WmEpiL}
D.~R. Wasserman.
\newblock {\em Epimorphisms and dominions in varieties of lattices}.
\newblock Ph.{D}. thesis, University of California, Berkeley, 2001.
\newblock Available at https://www.proquest.com/docview/304684559.

\bibitem{Wei94}
C.~A. Weibel.
\newblock {\em An introduction to homological algebra}, volume~38 of {\em Cambridge Studies in Advanced Mathematics}.
\newblock Cambridge University Press, Cambridge, 1994.

\bibitem{Werner78}
H.~Werner.
\newblock {\em Discriminator-algebras}, volume~6 of {\em Studien zur Algebra und ihre Anwendungen [Studies in Algebra and its Applications]}.
\newblock Akademie-Verlag, Berlin, 1978.
\newblock Algebraic representation and model theoretic properties.

\bibitem{Wille70}
R.~Wille.
\newblock {\em Kongruenzklassengeometrien}, volume 113 of {\em Lecture Notes in Mathematics}.
\newblock Springer-Verlag, Berlin-New York, 1970.

\end{thebibliography}
\end{document}